\documentclass[11pt]{article}

\usepackage{geo}
\usepackage{graphicx}
\graphicspath{{images/}}
\usepackage{dsfont}
\usepackage{mathtools}

\usepackage{amsmath,amssymb}
\usepackage{thmtools}
\usepackage{thm-restate}

\usepackage{bm}
\usepackage{csquotes}
\usepackage{natbib}

\usepackage{tikz}

\newcommand\MYcurrentlabel{xxx}

\newcommand{\MYstore}[2]{\global\expandafter \def \csname MYMEMORY #1 \endcsname{#2}}

\newcommand{\MYload}[1]{\csname MYMEMORY #1 \endcsname}

\newcommand{\MYnewlabel}[1]{\renewcommand\MYcurrentlabel{#1}\MYoldlabel{#1}}

\newcommand{\MYdummylabel}[1]{}

\newcommand{\torestate}[1]{%
  \let\MYoldlabel\label%
  \let\label\MYnewlabel%
  #1%
  \MYstore{\MYcurrentlabel}{#1}%
  \let\label\MYoldlabel%
}
\newcommand{\restateprop}[1]{%
  \let\MYoldlabel\label
  \let\label\MYdummylabel
  \begin{proposition*}[Restatement of \pref{#1}]
    \MYload{#1}
  \end{proposition*}
  \let\label\MYoldlabel
}
\newcommand{\restatelemma}[1]{%
  \let\MYoldlabel\label
  \let\label\MYdummylabel
  \begin{lemma*}[Restatement of \pref{#1}]
    \MYload{#1}
  \end{lemma*}
  \let\label\MYoldlabel
}
\newcommand{\restateclaim}[1]{%
  \let\MYoldlabel\label
  \let\label\MYdummylabel
  \begin{claim*}[Restatement of \pref{#1}]
    \MYload{#1}
  \end{claim*}
  \let\label\MYoldlabel
}

\newcommand{\restatetheorem}[1]{%
  \let\MYoldlabel\label
  \let\label\MYdummylabel
  \begin{theorem*}[Restatement of \pref{#1}]
    \MYload{#1}
  \end{theorem*}
  \let\label\MYoldlabel
}

\usepackage{amsbsy}

\usepackage[
        noabbrev,
        capitalise,
        nameinlink,
    ]
    {cleveref}

\newcommand{\bbS}{\mathbb S}

\newcommand{\R}{\mathbb R}

\newcommand{\Z}{\mathbb Z}

\newcommand{\Expected}[1]{\mathbb{E} \left[ #1 \right] }
\newcommand{\Expectednop}[1]{\mathbb{E} [ #1 ] }
\newcommand{\Var}[1]{\mathbf{Var} \left[ #1 \right] }
\newcommand{\Varsub}[2]{\underset{#1}{\mathbf{Var}}\left[ #2 \right] }

\newcommand{\bigO}[1]{ O \left( #1 \right) }
\newcommand{\bigOnop}[1]{ O ( #1 ) }
\newcommand{\bigOtilde}[1]{ \widetilde{O} \left( #1 \right) }
\newcommand{\bigOtildenop}[1]{ \widetilde{O} ( #1 ) }

\renewcommand{\Pr}[1]{\mathbb{P} \left( #1 \right) }
\newcommand{\Prd}[2]{\mathbb{P}_{#1} \left( #2 \right) }
\renewcommand{\top}{\mathsf{T} }
\DeclareMathOperator{\unif}{\textup{Unif}}
\newcommand{\unifhalf}[1]{\textup{Unif} \left( \left[ -1/2, 1/2 \right] \right)^{#1} }

\newcommand{\sone}{a_\textup{pure}}
\newcommand{\stwo}{a_\textup{mix}}
\newcommand{\bone}{\boldsymbol{1}}
\renewcommand{\bone}{b_\textup{pure}}
\newcommand{\btwo}{b_\textup{mix}}
\newcommand{\abparamsum}{\sone\bone + \stwo\btwo}

\newcommand{\Kpure}{K_{(\textup{pure})}}
\newcommand{\rpure}{r_{(\textup{pure})}}

\newcommand{\Kmix}{K_{(\textup{mix})}}
\newcommand{\rmix}{r_{(\textup{mix})}}

\newcommand{\Kpuredaggers}{K_{(\textup{pure})}^{\dagger \bone}}
\newcommand{\Kmixdaggers}{K_{(\textup{mix})}^{\dagger \btwo}}
\newcommand{\hattau}{\hat{\tau}}

\newcommand{\kappapure}{\boldsymbol{\kappa}^{(\textup{pure})}}
\newcommand{\kappapuretop}{\boldsymbol{\kappa}^{(\textup{pure})\top}}
\newcommand{\kappakmix}{\boldsymbol{\kappa}^{(\textup{mix})}}
\newcommand{\kappamixtop}{\boldsymbol{\kappa}^{(\textup{mix})\top}}

\newcommand{\Aeltabold}[1]{\boldsymbol{\alpha}_{#1}}
\newcommand{\Deltabold}[1]{\boldsymbol{\Delta}_{#1}}
\newcommand{\Deltaboldhat}[1]{\mathbf{r}_{#1}}
\newcommand{\Deltaboldhatintro}[1]{\mathbf{z}_{#1}}

\newcommand{\etahat}{\hat{\boldsymbol{\eta}}}

\newcommand{\kappapurej}[1]{\boldsymbol{\kappa}_{#1}^{(\textup{pure})}}
\newcommand{\kappapuretopj}[1]{\boldsymbol{\kappa}_{#1}^{(\textup{pure})\top}}
\newcommand{\kappakmixj}[1]{\boldsymbol{\kappa}_{#1}^{(\textup{mix})}}
\newcommand{\kappamixtopj}[1]{\boldsymbol{\kappa}_{#1}^{(\textup{mix})\top}}

\newcommand{\Expectedsub}[2]{\underset{#1}{\mathbb{E}}\left[#2\right]}
\newcommand{\Expectedsubnop}[2]{\underset{#1}{\mathbb{E}} #2 }
\newcommand{\Prsub}[2]{\underset{#1}{\mathbb{P}}\left(#2\right)}
\newcommand{\Prnop}[1]{\mathbb{P}(#1)}

\newcommand{\TV}[2]{\ensuremath{\mathrm{TV}\left(#1, #2\right)}}
\newcommand{\RGG}{\ensuremath{\mathrm{RGG}_q\left(n, d, p\right)}} 
\newcommand{\RGGinfty}{\ensuremath{\mathrm{RGG}_\infty\left(n, d, p\right)}} 
\newcommand{\RGGp}[3]{\ensuremath{\mathrm{RGG}_q\left(#1, #2, #3\right)}} 
\newcommand{\RGGinftyp}[3]{\ensuremath{\mathrm{RGG}_\infty\left(#1, #2, #3\right)}} 
\newcommand{\Gnp}{\ensuremath{G\left(n, p\right)}}

\newcommand{\Gbf}{\ensuremath{\mathbf{G}}}
\newcommand{\Xbf}{\mathbf{X}}

\newcommand{\Adj}{\ensuremath{\mathbf{A}}}
\newcommand{\Adjp}{\ensuremath{\mathbf{A}_p}}
\newcommand{\Adjpp}{\ensuremath{\Adj - p\mathds{1}\mathds{1}^{\top}}}
\newcommand{\Adjppt}[1]{\ensuremath{ \left(\Adjpp \right)^{#1} }}

\newcommand{\EW}{\ensuremath{\textsc{Ew}}}

\newcommand{\SW}{\ensuremath{\textsc{Sw}}}

\newcommand{\law}[1]{\mathcal{L}( #1 )}

\newcommand{\T}{\mathbf{T}}
\newcommand{\Tijk}[1]{\mathbf{T}_{#1}}
\newcommand{\barT}{\overline{\mathbf{T}}}
\newcommand{\barTijk}[1]{\overline{\mathbf{T}}_{#1}}

\newcommand{\Hardcore}{H^\circ}

\newcommand{\tr}[1]{\ensuremath{\textup{tr}\left( #1 \right)}}

\newcommand{\find}{f_{\textup{ind}}}

\newcommand{\dist}{\ensuremath{\rho}}

\newcommand{\spa}{\ensuremath{\Omega}}
\newcommand{\distro}{\ensuremath{\mu}}
\newcommand{\thresh}{\ensuremath{\tau}}
\newcommand{\ind}[1]{\ensuremath{\mathds{1}\left(#1\right)}}
\newcommand{\vecx}{\ensuremath{\mathbf{x}}}
\newcommand{\vecy}{\ensuremath{\mathbf{y}}}
\newcommand{\bbT}{\ensuremath{\mathbb{T}}}

\DeclareMathOperator{\sign}{\mathsf{sign}}

\newcommand{\absolute}[1]{\ensuremath{\left\lvert #1 \right\rvert}}

\newcommand{\ER}{Erd\H{o}s--R\'enyi}

\newcommand{\dd}{\mathrm{d}}
\newcommand{\say}[1]{`#1'\xspace}
\newcommand{\cramer}{Cramér's condition\xspace}

\newcommand{\sw}{\mathrm{\textsc{Sw}}}

\newcommand{\eps}{\varepsilon}

\renewcommand{\le}{\leqslant}
\renewcommand{\leq}{\leqslant}
\renewcommand{\ge}{\geqslant}
\renewcommand{\geq}{\geqslant}

\renewcommand{\emph}{\textit}

\newcommand{\poly}{\mathrm{poly}}

\usepackage{scalerel}
\usepackage{stackengine,wasysym}

\usepackage[nottoc]{tocbibind}

\usepackage{todonotes}

\title{Testing Thresholds and Spectral Properties of High-Dimensional Random Toroidal Graphs via Edgeworth-Style Expansions\footnote{Accepted for presentation at the
Conference on Learning Theory (COLT) 2025}.}

\author{Samuel Baguley \and Andreas Göbel \and Marcus Pappik \and Leon Schiller}
\date{\small Hasso Plattner Institute, University of Potsdam\\ \texttt{\{firstname.lastname\}@hpi.de}}

\begin{document}

\maketitle

\begin{abstract}
    We study high-dimensional random geometric graphs (RGGs) of edge-density $p$ with vertices uniformly distributed on the $d$-dimensional torus and edges inserted between \say{sufficiently close} vertices with respect to an $L_q$-norm. In this setting, we focus on distinguishing an RGG from an Erd\H{o}s--R\'enyi graph if both models have the same marginal edge probability $p$. So far, most results in the literature considered either spherical RGGs with $L_2$-distance or toroidal RGGs under $L_\infty$-distance. However, for general $L_q$-distances, many questions remain open, especially if $p$ is allowed to depend on $n$. The main reason for this is that RGGs under $L_q$-distances can not easily be represented as the logical \say{AND} of their 1-dimensional counterparts, as is the case for $L_\infty$ geometries. To overcome this difficulty, we devise a novel technique for quantifying the dependence between edges based on a modified version of Edgeworth expansions.

Our technique yields the first tight algorithmic upper bounds for distinguishing toroidal RGGs under general $L_q$ norms from Erd\H{o}s--R\'enyi graphs for any fixed $p$ and $q$. We achieve this by showing that the signed triangle statistic can distinguish the two models when $d\ll n^3p^3$ for the whole regime of edge probabilities $\frac{c}{n}<p<1$. Additionally, our technique yields an improved information-theoretic lower bound for this task, showing that the two distributions converge in total variation whenever $d=\tilde{\Omega}(n^3p^2)$, which is just as strong as the currently best known lower bound for spherical RGGs in case of general $p$ from Liu et al. [STOC'22]. Finally, our expansions allow us to tightly characterize the spectral properties of toroidal RGGs both under $L_q$-distances for fixed $1 \le q < \infty$, and $L_\infty$-distance. We find that these are quite different for $q<\infty$ vs. $q=\infty$. Our results partially resolve a conjecture of Bangachev and Bresler [COLT'24] and prove that the distance metric, rather than the underlying space, is responsible for the observed differences in the behavior of high-dimensional spherical and toroidal RGGs.
\end{abstract}

\newpage
\tableofcontents
\newpage

%!TEX root = TRGGTesting.tex
\section{Introduction}

Random geometric graphs (RGGs) are a cornerstone for the mathematical modeling of a wide range of real-world phenomena.
While they were first introduced by Gilbert~\cite{gilbert1961random} to model communication networks of radio stations, RGGs have found applications in various settings including protein interactions~\cite{higham2008fitting}, opinion dynamics~\cite{estrada2016consensus} and motion planning~\cite{solovey2018new}. 
In particular RGGs have been proven to reproduce numerous properties of real-world networks~\cite{almagro2022detecting,friedrich2024real}, making them a valuable tool in various disciplines.
Moreover, they were shown to be well-suited to represent high-dimensional feature vectors, encountered in statistical mechanics \cite{gorban2018blessing} and in data science \cite{erba2020random}. 

% It is known that as the dimensionality of underlying space of an RGG goes to infinity, RGGs converge to \ER{} graphs, where edges are drawn independently \cite{friedrich2024cliques,Devorye2011HD}. This raises the question of when can we detect the geometry? See \cite{duchemin2023random} for great survey. To discuss the literature in more detail we first  define RGGs formally.\agob{Not happy about this part so far.}

Given a metric space $(\spa,\dist)$, a probability measure $\distro$ over $\spa$, a size $n\in\mathbb N$, and a connection threshold $\thresh \in [0,\infty)$,
a \emph{random geometric graph} $\Gbf$ 
comprises vertices $v\in[n] \coloneqq \{1, \dots, n\}$ with independent position vectors $\vecx_v \sim\mu$ that are connected if and only if their distance (as defined by $\rho$) is below $\tau$.
Formally, $\Gbf$
has the following distribution:
%over graphs with vertex set $[n] \coloneqq \{1, \dots, n\}$:
%
%Given a metric space $(\spa,\dist)$, a probability measure $\distro$ over $\spa$ and a connection threshold $\thresh \in \R$. 
%each vertex $v\in[n]$ is associated to an independent $\vecx_v\sim\distro$, and then 
for any undirected graph $G$ with adjacency matrix $A\in\{0,1\}^{n\times n}$, 
%a random geometric graph $\Gbf$ is sampled according to 
\[
\Pr{\Gbf=G}=\Expectedsub{\substack{\vecx_v\sim\mu,\\ v\in[n]}}{\prod_{1\leq u<v\leq n}\left(\ind{\dist(\vecx_u,\vecx_v)\leq \thresh}\right)^{A_{u,v}}\left(\ind{\dist(\vecx_u,\vecx_v)>\thresh}\right)^{1-A_{u,v}}}.
\]
%where $\ind{\mathcal E}$ denotes the indicator random variable of the event $\mathcal E$. 
% Both RGGs and \ER{} graphs are special case of graphon random graphs (see \cite[Chapter~13]{lovasz2012book}).
We denote the law of $\Gbf$ by $\mathrm{RGG}_\dist(n,\spa,p,\distro)$, where $p \coloneqq \Pr{\rho(\vecx_u,\vecx_v)\le\tau)}$ is the connection probability.

Depending on the parametrization of the model, RGGs can differ significantly from the classical random graph model of \ER{} graphs (denoted by $\Gnp$), where edges are sampled independently with probability $p$. For example, when the underlying metric space is not homogeneous (e.g., a $d$-dimensional hypercube), Dall and Christensen~\cite{dall2002random} have shown that for any choice of $d$, RGGs differ significantly from the $\Gnp$ model. However, when the metric space is homogeneous (e.g., $\bbS^{d-1}$ or $\bbT^d$) with a natural notion of dimensionality, it has been shown that edges lose their dependency and RGGs converge to \ER{} graphs as $d \to \infty$ \cite{friedrich2024cliques,Devorye2011HD}. 
This observation raises the question of how to determine the values of $d$ for which the two distributions differ, or as Duchemin and De Castro~\cite{duchemin2023random} phrase it, ``determine the point where geometry is lost in the dimension.'' 
% More formally, let $\text{RGG}_\dist(n,\spa,p,\distro)$ denote the law of an RGG over the metric space $(\spa,\dist)$ with probability measure $\distro$ and threshold $\thresh$ chosen such that each edge is present with probability $p$. 
% The main focus of this article is the toroidal case $\bbT^d$ (with its natural notion of dimensionality),
%the metric space has a natural notion of dimensionality (e.g. $\bbS^{d-1}$ or $\bbT^d$),
%, the highlighted question, which will be the main focus of this article, is to determine the regimes of $d$ for which, when given a graph $\Gbf$ on $n$ vertices, we can distinguish between the hypothesis
The main focus of this article is to 
identify the values of $d$ for which we can distinguish---via an efficient statistical test---between the hypotheses
\[
H_0:\Gbf\sim \Gnp \qquad \text{vs.} \qquad H_1:\Gbf\sim \mathrm{RGG}_\dist(n,\spa,p,\distro).    
\]

This question is relatively well-understood for $(\bbS^{d-1}, L_2)$, i.e.  the surface of the $d$-dimensional unit sphere with the $L_2$-norm, in which the position vectors associated with the vertices are sampled from the uniform measure. Research in this direction was initiated by Devroye et al.~\cite{Devorye2011HD}, who showed that the clique number is indistinguishable from that of the $\Gnp$ when $d=\Omega(\log^3 n)$, and also that the two models are completely indistinguishable when $d\gg\exp(n^2)$. 
To the best of our knowledge, the first algorithmic results on testing for geometry were given by Bubeck et al.~\cite{bubeck2016testing}, showing a threshold of $d=\Theta(n^3)$ for the dense regime of $p=\Theta(1)$ when $d=\poly(n)$. In particular, when $d= \omega(n^3)$ the total variation distance of the two graph distributions goes to 0 as $n\to\infty$, while when $d\in o(n^3)$, the number of signed triangles serves as a good statistic for detecting geometry. In addition, Bubeck et al.~\cite{bubeck2016testing} showed that for the sparse regime of $p=o(1)$, when $d=o(\log^3n)$, signed triangles distinguish RGGs from \ER{} graphs. Subsequently, Brennan, Bresler, and Nagaraj~\cite{brennan2020phase} obtained an upper bound on the testing threshold by showing that the two random graph models are statistically indistinguishable for $d=\Omega(n^{3/2})$. Most recently, Liu et al.~\cite{liu2022testing} improved the indistinguishability bound to $d=\Omega(\log^{36}n)$.

Setting aside the polylogarithmic gap in the above threshold, a consequent open question---see also the survey \cite[Section 3.4, Question 3]{duchemin2023random}---is how the geometry detection thresholds behave when other metric spaces are used to sample the vertices in RGGs. 
Since we believe that the indistinguishability result in~\cite{dall2002random} generalizes to other non-homogeneous metric spaces, the question becomes more interesting for homogeneous metric spaces.\footnote{For our purposes, we call a metric space $(\Omega, d)$ equipped with the Borel algebra $\mathcal{F}$ and a measure $\mu$ \emph{homogeneous} if, for all radii $r \ge 0$, a ball of radius $r$ under distance $d$ has the same measure regardless of its center $x \in \Omega$.} 
In this article we focus on the $d$-dimensional torus with $L_q$-distances, for any fixed $q\in\Z_{>0}\cup\{\infty\}$ (i.e., $(\spa,\dist)=(\bbT^d,L_q)$). 
More formally, the vertices of the RGG are associated with vectors in $[-1/2,1/2]^d$ and their distance is given by 
    $\|\mathbf{x}_u - \mathbf{x}_v \|_{q} = \sum_{i =1}^d (|\mathbf{x}_u(i) - \mathbf{x}_v(i)|_C )^q)^{\frac1q}$,
where $|x - y|_C \coloneqq \min\{ |x - y|, 1 - |x - y| \}$ is the distance on the unit circle. 
We write $\RGG$ for the associated law, where, as before, $p\in[0,1]$ uniquely determines the threshold value $\tau$ such that each pair of vertices is connected by an edge with probability $p$.

% $(\bbT^d,L_\infty)$ is a natural space, used for generating geometric inhomogeneous random graphs~\cite{bringmann2019geometric}, which are, essentially, random geometric graphs with powerlaw-weights that determine the degree distribution of the graph. 
The first to study high-dimensional RGGs with $(\bbT^d,L_q)$ as underlying space were Friedrich, et.~al.~\cite{friedrich2024cliques} who established that \RGG{} and \ER{} graphs converge in total variation as $d\to\infty$ (and their inhomogeneous variants), for all fixed values of $q\in\Z_{>0}\cup\{\infty\}$---this includes the maximum norm $L_\infty$. 
In subsequent work, the same authors \cite{friedrich2024real} showed that a triangle statistic can be used to determine the dimensionality of an RGG for $(\bbT^d,L_\infty)$, when $d=o(\log n)$. Bangachev and Bresler~\cite{Bangachev_Bresler_2024} showed that counting (signed) triangles is not the optimal statistic for detecting geometry in $(\bbT^d,L_\infty)$, when $d=\poly(n)$, which is surprising because it is for the sphere. They consider the whole edge probability regime of $n^{-1}<p\leq 1/2$ and show that the two distributions are indistinguishable when $d=\tilde\omega(\max\{n^{3/2}p, n\})$. Additionally they showed that the signed triangle statistic distinguishes the two graph distributions when $d=\tilde o((np)^{3/4})$, while the signed 4-cycle statistic does so when $d=\tilde o(np)$, and moreover that the later statistic is optimal in some sense \cite[Theorem 1.5]{Bangachev_Bresler_2024}.

All aforementioned results on geometry testing with $(\bbS^{d-1}, L_2)$ and $(\bbT^d,L_\infty)$ use the fact that that the metrics factorize over coordinates. In $(\bbS^{d-1}, L_2)$, for example, $\|\vecx-\vecy\|_2=\langle\vecx,\vecy\rangle=\sum_{i=1}^d|x_i|\cdot|y_i|\cdot\sign(x_i,y_i)$, in which case $\sign(x_i,y_i)$ is a one-dimensional edge independent from the other coordinates and the $|x_i|\cdot|y_i|$ are weights. Furthermore, in $(\bbT^d,L_\infty)$, $\|\vecx-\vecy\|_\infty\leq\tau$ if for all $i\in[d]$, it holds that $|x_i-y_i|_C\leq\tau$. This has been a crucial component in all analyses so far: as Bangachev and Bresler~\cite[Section 1.2.1]{Bangachev_Bresler_2024} admit ``the analysis of $L_q$ models, however, turns out to be much more difficult when $q < \infty$ as the factorization over 1-dimensional random geometric graphs does not hold any longer. In particular, this makes the
computation of signed subgraph counts much more difficult.''
%and we have not succeeded to perform such a computation even for triangles''. 
Despite this technical challenge those authors were able to obtain results for the special case of $p=1/2$ for all values of $q$. We highlight their results for $q\in o(d/\log d)$: in that case $\RGGp{n}{d}{1/2}$ is indistinguishable from $G(n,1/2)$ when $dq=\tilde\omega(n^3)$ and the distributions of the two models diverge for all $1/n\leq p\leq 1/2$, when $d=o(np/log n)$. Furthermore, they conjecture that the signed triangle statistic is able to distinguish $\RGGp{n}{d}{1/2}$ from $G(n,1/2)$ only when $dq\in\tilde o(n^3)$ \cite[Conjecture A.3]{Bangachev_Bresler_2024}. 

% To summarize, for $(\bbS^{d-1}, L_2)$ we have a threshold at $\Theta(n^3)$ for the dense regime and a threshold between $\Omega(log^{32}n)$ and $o(\log^3n)$ for the sparse regime, while for $(\bbT^d,L_\infty)$ we have a threshold between $\tilde\omega(\max\{n^{3/2}p, np\})$ and $\tilde o(np)$ for all edge densities. This raises the question of whether the latent space or the metric is responsible for the location of this threshold? To answer this question we need to overcome the technical challenges identified by Bangachev and Bresler and study $\RGG$ for $q<\infty$.

\subsection{Our Results}

We advance the understanding of high-dimensional 
($d=\poly(n)$, precisely, we assume throughout this paper that $d \ge n^{\gamma}$ for $\gamma > 0$ arbitrarily small) 
random geometric graphs on the torus under an $L_q$-norm (where $q < \infty$ is fixed) and for the whole range of edge-densities $n^{-\bigO{1}} \le p \le 1 - \varepsilon$. Specifically, we devise a novel technique for quantifying the dependence between the coordinates of a given random vector, which, as we explain in the following, allows us to overcome the hurdle of $L_q$ norms don't factorize to one-dimensional components when $q<\infty$. Using this approach, we obtain the following results.

\paragraph{Tight testing thresholds via signed triangles.}
Our first result is a tight algorithmic threshold for the signed triangle statistic in case that $p= \Theta(1)$. It also yields improved algorithmic upper bounds for $p = o(1)$. We define the number of signed triangles in a graph $\Gbf$ as \begin{align*}
    \T(\Gbf) \coloneqq  \sum_{i < j< k \in [n]} \Tijk{ijk} \text{ where } \Tijk{ijk} \coloneqq (\Gbf_{ij} - p)(\Gbf_{ik} - p)(\Gbf_{jk} - p).
\end{align*} where $\Gbf_{ij}$ denotes the indicator variable of the edge $\{i,j\}$. With this, we prove the following.
\begin{theorem}\label{thm:signedtriangles}
   For any fixed $1 \le q < \infty$, $d \ge n^{\gamma}$ and any $\frac{\alpha}{n}\le p \le 1 - \varepsilon$ for constants $\epsilon, \alpha > 0$, the number of signed triangles distinguishes $\Gbf_1 \sim \RGG$ and $\Gbf_2 \sim \Gnp$ with probability $1 - o(1)$ whenever $d = o(n^3p^3)$, and fails at doing so whenever $d = \tilde{\omega}(n^3p^3)$. 
    Precisely, \begin{align*}
         \left|\Expectedsub{\Gbf \sim \RGG}{\T(\Gbf)}\right| - \max \left\{ \sqrt{\Varsub{\Gbf \sim \RGG}{\T(\Gbf)}}, \sqrt{\Varsub{\Gbf \sim \Gnp}{\T(\Gbf)}} \right\} = \begin{cases}
             \omega(1) &\text{if } d = o(n^3p^3)\\
             o(1) &\text{if } d = \tilde{\omega}(n^3p^3).
         \end{cases}
    \end{align*}
\end{theorem}
\noindent This result improves upon the upper bound given by Theorem 1.11 of \cite{Bangachev_Bresler_2024}, which only asserts that the total variation distance 
tends to $1$ (i.e., $\TV{\RGG}{\Gnp} = 1- o(1)$) 
if $d = \tilde{o}(np)$. 
Their result does not rely on signed triangles but instead uses an entropy-based argument, which is built upon an explicit $\varepsilon$-net of $\mathbb{T}^d$ with specific properties. This has the further drawback of being an existential arguments which does not yield a computationally efficient test.\footnote{However, it works even if $q$ is allowed to depend on $n$.}
For any fixed $q$, \Cref{thm:signedtriangles} yields a test that works in a wider range of possible values for $d$ and $p$, and is further computationally efficient. 
At the same time, it proves \cite[Conjecture A.3]{Bangachev_Bresler_2024} for fixed $1 \le q < \infty$ and matches the information-theoretic threshold given by \cite[ Theorem 1.10]{Bangachev_Bresler_2024} if $p = \Theta(1)$ (up to polylogarithmic factors). We remark that \Cref{thm:signedtriangles} yields the same computational power as the signed triangle test on spherical RGGs under $L_2$-norm, as proved in~\cite{liu2022testing}. 

\paragraph{Improved information-theoretic lower bounds for all densities.}

In addition to improved algorithmic upper bounds, we also obtain an improved information-theoretic lower bounds for distinguishing $\RGG$ from $\Gnp$ that works for all relevant edge-densities $p$. 
\begin{restatable}{theorem}{statisticalindistinguishability}\label{thm:indistinguishability}
    For any fixed $1 \le q < \infty, \alpha > 0$ and any $\frac{\alpha}{n}\le p \le 1 - \varepsilon$, we have that the total variation distance $\TV{\RGG}{\Gnp} = o(1)$ whenever $d = \omega( \max\{ n\log(n), p^2n^3\log^6(n) \})$.
\end{restatable}
\Cref{thm:indistinguishability} improves over the previously known information-theoretic lower bound for distinguishing toroidal RGGs under $L_q$-norm from a $\Gnp$, given by \cite[Theorem 1.10]{Bangachev_Bresler_2024}, which only applies for $p = 1/2$ (but allows for $q = \omega(1)$). This bound is based on an intricate combination of anti-concentration and Fourier-theoretic methods but does not easily generalize to the case of $p = o(1)$. Our bound does not have this limitation and yields the same threshold as the currently best known lower bound for $(\bbS^{d-1},L_2)$ from \cite{liu2022testing}, while relying on entirely different techniques and extending to a more general regime of possible distance functions (all $L_q$-norms instead of only $L_2$-norm).\footnote{We remark that \cite{liu2022testing} proves a stronger lower bound for spherical RGGs in case $p=\alpha/n$ for constant $\alpha$ that relies on the locally tree-like properties of sparse, spherical RGGs. However, these techniques break down if $p = \omega(1/n)$. In this case, their lower-bound is just as strong as the one of \cref{thm:indistinguishability}.} We find it remarkable that the underlying technical heart of the above results is powerful enough to yield improved algorithmic upper \emph{and} lower-bounds that are precisely as strong as the currently best known bounds for spherical RGGs in the case of general $p$, which rely on quite different techniques. Precisely, the lower-bound in \cite{Liu_Mohanty_Schramm_Yang_2021} uses concentration bounds for the intersection of random spherical caps and anti-caps via transportation inequalities, and the upper bound from \cite{Liu_Mohanty_Schramm_Yang_2021} and \cite{bubeck2016testing} relies on a more direct integration that makes heavy use of the spherical symmetry. We generalize both of these results to a wider class of possible distance functions while relying on the same underlying way of quantifying dependencies between edges for both our upper- and lower-bound. 

% Our \Cref{thm:fouriercoefficients} (for chains) allows for a strengthening of this result applying to all fixed $1 \le q < \infty$ and -- in particular -- to all $\frac{\alpha}{n} \le p \le 1 - \varepsilon, \alpha > 0$. 
\paragraph{Spectral Properties of $\RGG$ and $\RGGinfty$.}

Finally, quantifying the dependence between edges allows us to tightly (i.e. up to lower order factors) characterize the spectrum of a $\Gbf \sim \RGG$ both for fixed $q < \infty$ and $q = \infty$. Spectral properties of spherical RGGs have recently gained a lot of attention in a series of works \cite{Bangachev_Bresler_2024_fourier, Bangachev_Bresler_2024_sandwich, Li_Schramm_2024, Liu_Mohanty_Schramm_Yang_2023}, where they were further used algorithmically for latent vector recovery, community detection, and robust testing. 
A key takeaway from these works (in particular from \cite{Li_Schramm_2024}) is that the adjacency matrix of a spherical RGG has $\tilde{\Theta}(d)$ large eigenvalues, each of order $\tilde{\Theta}(np/\sqrt{d})$, provided that $d \ll np$, while all remaining eigenvalues are of lower order. 
Furthermore, these $d$ eigenvectors are close to the $d$ rows of the $d \times n$ matrix $\mathbf{X} = (\vecx_1, \vecx_2, \ldots, \vecx_n)$, where $\vecx_i \in \mathbb{S}^{d-1}$ is the latent vector associated to vertex $i$. On the other hand, if $d \gg np$, the second largest eigenvalue is of order $\tilde{\Theta}(\sqrt{np})$, as in a $\Gnp$. 
Accordingly, $d = np$ marks the so-called \emph{spectral threshold}, that is, 
%$d = np$ marks a phase transition such that $d \ll np$ implies that a spherical RGG
a spherical RGG with $d \ll np$
is spectrally very different from a $\Gnp$, while for $d \gg np$ both models are spectrally similar. These observations (in case $d \ll np$) were also used for designing efficient spectral algortihms for latent vector recovery and clustering in \cite{Li_Schramm_2024}, and they also constitute the foundation of the sum-of-squares based spectral refutation algorithm for distinguishing an RGG from a $\Gnp$ under a small constant fraction of adversarial edge corruptions from \cite{Bangachev_Bresler_2024_sandwich}. Our technique yields a similar characterization of the spectrum of toroidal RGGs and extends the existing results known for the sphere under $L_2$ norm to all $1 \le q < \infty$.
\begin{restatable}[Spectral properties of $\Gbf \sim \RGG$]{theorem}{spectralstuff}\label{thm:spectralstuff}
    Let $\Gbf \sim \RGGp{n}{d}{p}$ have adjacency matrix $\Adj$ and eigenvalues $\lambda_1(\Adj) \ge \lambda_2(\Adj) \ge \ldots \ge \lambda_n(\Adj)$. For any $\frac{1}{n} \le p \le 1 -\varepsilon$ and any fixed $1 \le q < \infty$, there is a constant $C> 0 $ such that with probability $1 - o(1)$, \begin{align*}
         C \max\left\{ \sqrt{np}, \frac{np}{\sqrt{d}} \right\} \le  \max\{ |\lambda_2(\Adj)|, |\lambda_n(\Adj)| \} \le n^{o(1)} \max\left\{ \sqrt{np}, \frac{np}{\sqrt{d}} \right\}.
    \end{align*} Moreover, in case $d = \tilde{o}(np)$, there are $\tilde{\Omega}   (d)$ eigenvalues of order $\Theta(np/\sqrt{d})$. If $d \le np^{1 + \varepsilon}$, then for any $a \ge 1$, the number of eigenvalues of magnitude $\ge \frac{np}{a\sqrt{d}}$ is at most $\bigOtildenop{ d a^{4 + 2/\varepsilon} }$.
\end{restatable}
\noindent \cref{thm:spectralstuff} shows that a $\Gbf \sim \RGG$ for any fixed $1 \le q < \infty$ has similar spectral properties to a spherical RGG (with $L_2$-norm) and, in particular, obeys the same spectral threshold at $d \approx np$. 
Furthermore, the  theorem shows that when $d \ll np$, \RGG{} has at most $d^{1+o(1)}$ large eigenvalues of order $\Omega(np/\sqrt{d})$, while all the remaining $\Theta(n)$ eigenvalues are of order $o(np/\sqrt{d})$ and therefore of lower order. As in the case of a spherical RGG, we can explicitly construct $d$ nearly orthogonal eigenvectors with associated eigenvalue $\Omega(np/\sqrt{d})$ by defining one vector per dimension, based on the latent positions of all vertices in said dimension. We refer the reader to \Cref{sec:lowerboundeigenvalues} for all details underlying this construction. 
On a high level, each such vector is defined over $\{0, \pm 1\}^{n}$ and supported on vertices corresponding to two opposing circular arcs in a fixed dimension. Coordinates corresponding to vertices in the first arc are all set to $1$, while coordinate corresponding to the second arc are $-1$. Since vertices within the same arc have a slightly larger probability of being adjacent than other pairs, this yields the desired properties. To analyze how these slightly perturbed connection probabilities, we also use (univariate) Edgeworth expansions.

\paragraph{A different bound for $L_\infty$-norm} 
A simplification of the techniques used to prove \Cref{thm:spectralstuff} further allows us to study the spectral properties of $\RGGinfty$. Here, the spectrum behaves rather differenctly  differently. 

\begin{restatable}[Spectral properties of $\Gbf \sim \RGGinfty$]{theorem}{spectralstuffinfty}\label{thm:spectralstuffinfty}
    Let $\Gbf \sim \RGGinfty$ have adjacency matrix $\Adj$ and eigenvalues $\lambda_1(\Adj) \ge \lambda_2(\Adj) \ge \ldots \ge \lambda_n(\Adj)$. For any $\frac{1}{n} \le p \le 1 -\varepsilon$, there is a constant $C> 0 $ such that with probability $1 - o(1)$, \begin{align*}
         C \max\left\{ \sqrt{np}, \frac{np}{d} \right\} \le \max\{ |\lambda_2(\Adj)|, |\lambda_n(\Adj)| \} \le \log(n)^{11} \max\left\{ \sqrt{np}, \frac{np}{d} \right\}.
    \end{align*} Moreover, in case $d = \tilde{o}(\sqrt{np})$, there are $\tilde{\Omega}(d^2)$ eigenvalues of order $\Theta(np/d)$. If $d \le p^{\varepsilon}\sqrt{np}$, then for any $a$, the number of eigenvalues of magnitude $\ge \frac{np}{ad}$ is at most $\bigOtilde{ d^2 a^{4 + 1/\varepsilon} }$.
\end{restatable} 
\cref{thm:spectralstuffinfty} shows that a $\Gbf \sim \RGGinfty$ obeys a spectral threshold at $d \approx \sqrt{np}$ instead of $d \approx np$. Furthermore, for $d \ll \sqrt{np}$, $\Gbf\sim\RGGinfty$ typically has $\bigOnop{d^{2+o(1)}}$ large eigenvalues of order $\Omega(np/d)$, while all remaining eigenvalues have order $o(np/d)$. As in the $q<\infty$ case, we explicitly construct $\Theta(d^2)$ nearly orthogonal approximate eigenvectors with eigenvalue $\Theta(np/d)$, this time by finding roughly $d$ such vectors per dimension. Each vector has a similar structure as in the $q<\infty$ case, but is more sparse, which results in a larger number of vectors that can be constructed.

While we do not explicitly study algorithmic applications of the above results, we believe that they can be extended to yield efficient algorithms for similar problems as studied in \cite{Li_Schramm_2024, Bangachev_Bresler_2024_sandwich}, namely recovery of latent vectors or an approximate distance matrix associated to a given $\Gbf \sim \RGG$, clustering, and robust testing. In particular, we think that our explicit construction of approximate eigenvectors can serve as a certificate for underlying geometry that is robust to a certain number of adversarial edge-corruptions and can efficiently be checked for using a sum-of-squares based approach as in \cite{Bangachev_Bresler_2024_sandwich}. We think that a formal investigation of these ideas constitutes an interesting direction for future research.

\subsection{Discussion and Outlook}

While this work is focuses on $\RGG$, the techniques are general and could prove helpful in understanding RGGs on different metric spaces, even when sampling vertices non-uniformly, we further discuss this on \cref{sec:applicability}. Furthermore, we believe our approach to be applicable when $q = \omega(1)$, but this adaptation involves more cumbersome calculations and technicalities.

A promising direction for future research is to study the signed weight of more general patterns $H$. This was very recently achieved for spherical RGGs by \cite{Bangachev_Bresler_2024_fourier}, however their techniques heavily rely on leveraging spherical symmetry and inner products, and thus an adaptation to the torus (even for $L_2$-norm) seems challenging. We believe a combination of our approach together with the \say{cluster expansion} technique of \cite{Bangachev_Bresler_2024} could be fruitful. Moreover, it would be interesting to find algorithmic applications related to the spectral properties of $\RGG$. On the sphere, this was achieved by \cite{Li_Schramm_2024} for latent vector recovery, clustering, and community detection. Our explicit construction of eigenvectors could serve as an efficiently refutable certificate of geometry, as recently used in the context of robust testing for RGGs on the sphere \cite{Bangachev_Bresler_2024_sandwich}.

Finally, it would be interesting to improve upon the gap between information-theoretic lower bounds and algorithmic upper bounds in case of $p = o(1)$. This is open even for spherical RGGs, and resolving it would require a better understanding of \say{average-case} versus \say{worst-case} vector embeddings of a given $\Gbf \sim \RGG$ as already observed in \cite{Liu_Mohanty_Schramm_Yang_2021}. 

\section{Technical Overview}

This section contains an overview of the techniques used to prove \cref{thm:signedtriangles,thm:indistinguishability,thm:spectralstuff,thm:spectralstuffinfty}; full proofs can be found in the subsequent sections. We introduce a novel technique for quantifying the dependence between coordinates of a particular class of random vectors. In our case, each entry of this vector $\Deltaboldhatintro{} \in \mathbb{R}^k$ represents the (rescaled) distances associated to a collection $H = \{\{u_1, v_1\}, \ldots \{u_k, v_k\}\}$ of potential edges in an 
% $\Gbf \sim \RGG$. 
RGG.
Accordingly, $\Deltaboldhatintro{}$ determines which edges in $H$ are present in $\Gbf$ and which not, so dependencies between the entries of $\Deltaboldhatintro{}$ typically tell us about the dependencies between the edges in $H$. A crucial fact about $\Deltaboldhatintro{}$ is that it can be represented as a sum $\Deltaboldhatintro{} = \sum_{i=1}^d \Deltaboldhatintro{i}/\sqrt{d}$ of independent (although not necessarily i.i.d.) vectors $\Deltaboldhatintro{i}$, one for each dimension in $\mathbb{T}^d$. 
The $\Deltaboldhatintro{i}$ contain the (centered and rescaled) distances associated to the vertex pairs from $H$ in dimension $i$. While it is illustrative to think about $\Deltaboldhatintro{}$ in terms of this interpretation, we remark that our technique extends beyond this application.

Now, in each $\Deltaboldhatintro{i}$, there are certain dependencies between the individual coordinates that is \say{inherited} by $\Deltaboldhatintro{}$. For example, in case all the $\Deltaboldhatintro{i}$ are independent Gaussian vectors $\Deltaboldhatintro{i} \sim \mathcal{N}(\mathbf{0}, \boldsymbol{\Sigma})$ for some covariance matrix $\boldsymbol{\Sigma}$, it follows that $\Deltaboldhatintro{}$ also has distribution $\mathcal{N}(\mathbf{0}, \boldsymbol{\Sigma})$ and the dependencies between its coordinates are fully captured by all the off-diagonal entries in $\boldsymbol{\Sigma}$. More generally, it is intuitive that the \emph{mixed moments} in multiple coordinates of $\Deltaboldhatintro{i}$ naturally appearing in its Fourier transform should characterize how \say{interdependent} the coordinates are. Again, in case of Gaussian vectors, the entries in $\boldsymbol{\Sigma}$ are precisely all the mixed moments of order $2$, fully describing the dependencies within $\Deltaboldhatintro{}$. However, when considering distributions other than the Gaussian, the entries of $\Deltaboldhatintro{}$ can be dependent while still being pairwise uncorrelated, implying that moments of order $2$ (i.e. the covariances) contain no information about the dependencies of interest. In such cases, it is necessary to take higher order mixed moments into account. 

\subsection{Joint Cumulants Capture Dependencies}\label{sec:cumulantscapturedependencies}

% It turns out that---while the above intuition is true---
It is typically easier to work with \emph{cumulants} instead of mixed moments, especially when dealing with the higher-order versions of such quantities. Cumulants are related to mixed moments, but obey much more convenient additive and homogeneous properties, which are especially useful when working with sums of independent random vectors. Cumulants are defined via the complex-valued \emph{characteristic function} (CF) $C_{\Deltaboldhatintro{}}(\mathbf{t}) \coloneqq \Expected{ \exp\left(i(\mathbf{t}^\top \Deltaboldhatintro{})\right) }$, $\mathbf{t} \in \mathbb{R}^k$,  where $i = \sqrt{-1}$. Taking the logarithm yields the \emph{cumulant generating function} (CGF) \begin{align*}
    K_{\Deltaboldhatintro{}}((t_1, \ldots, t_k)^\top) = \log\left( C_{\mathbf{z}}((t_1, \ldots, t_k)^\top) \right) = \sum_{\substack{s = (s_1, \ldots, s_k) \in \mathbb{N}^{\times k}}} \kappa_{s}(\Deltaboldhatintro{}) \frac{(it_1)^{s_1}(it_2)^{s_2}\cdots (it_k)^{s_k}}{s_1!s_2!\cdots s_k!},
\end{align*} where $\mathbb{N}^{\times k}$ is the set of all $k$-tuples over the natural numbers, and the quantities $\kappa_{s}(\Deltaboldhatintro{})$ are the \emph{cumulants} of $\Deltaboldhatintro{}$. It is further convenient to define a \say{vectorized} version of the CGF by setting $K_{\mathbf{z}}(\mathbf{t}) = \sum_{j=1}^\infty \frac{1}{j!} \boldsymbol{\kappa}_j^\top(i\mathbf{t})^{\otimes j}$ where $\boldsymbol{\kappa}_j$ is a vector containing all (suitably rescaled) cumulants.

\paragraph{\say{Pure} and \say{mixed} cumulants.}

There is a simple intuition for why certain cumulants characterize the dependencies between the entries of $\Deltaboldhatintro{}$, which is obtained by comparing the individual (univariate) cumulants of the entries of $\Deltaboldhatintro{}$ to the joint cumulants of $\Deltaboldhatintro{}$. If the entries of $\Deltaboldhatintro{}$ were all independent, then the CGF of $\Deltaboldhatintro{}$ would have special structure, captured in the following observation.

\begin{observation}[CGFs are additive for independent coordinates]\label{obs1}
    Assume that $\Deltaboldhatintro{} \in \mathbb{R}^k$ is a random vector in which all coordinates $\{\Deltaboldhatintro{}(j)\}_{j\in[k]}$ are independent. Then, 
        $K_{\Deltaboldhatintro{}}(\mathbf{t}) = \sum_{j=1}^k K_{\Deltaboldhatintro{}(j)}(\mathbf{t}(j))$.
\end{observation} 
% \begin{proof}
%     If all $k$ entries of $\Deltaboldhatintro{}$ were independent, we could split \begin{align*}
%     C_{\Deltaboldhatintro{}}(\mathbf{t}) \coloneqq \Expected{ \exp\left(i(\mathbf{t}^\top \Deltaboldhatintro{})\right) } = \prod_{j=1}^k \Expected{ \exp\left(i\mathbf{t}(j) \Deltaboldhatintro{}(j\right) } \text{ and thus }  K_{\Deltaboldhatintro{}}(\mathbf{t}) = \sum_{j=1}^k K_{\Deltaboldhatintro{}(j)}(\mathbf{t}(j)).
% \end{align*}
% \end{proof}

\noindent \cref{obs1} shows that the CGF of the vector $\Deltaboldhatintro{}$ is just the sum of the CGFs of its individual entries if these are independent. Another way to phrase this is to say that $K_{\Deltaboldhatintro{}}$ has the form \begin{align*}
    K_{\Deltaboldhatintro{}}(\mathbf{t}) =\sum_{j=1}^k \sum_{s=1}^\infty \frac{(i\mathbf{t}(j))^{s}}{s!}\kappa_s(\Deltaboldhatintro{}(j)),
\end{align*} which implies that $\kappa_s(\Deltaboldhatintro{}) = 0$ whenever $s = (s_1,s_2,\ldots,s_k)$ has more than two non-zero coordinates. This suggests a partitioning of the cumulants of $\Deltaboldhatintro{}$ into two classes.
%: so-called \emph{pure} and \emph{mixed} cumulants according to the following definition. 
\begin{definition}[Pure and mixed cumulants]
  For some $s = (s_1, s_2, \ldots, s_k)$, we call the cumulant $\kappa_s(\Deltaboldhatintro{})$ \emph{pure} if exactly one of the $s_i$ is non-zero, and \emph{mixed} otherwise.
\end{definition}
\noindent Pure cumulants are exactly those that also appear in the CGFs of the individual entries of $\Deltaboldhatintro{}$, while mixed cumulants only appear in $K_{\Deltaboldhatintro{}}$. Since \cref{obs1} shows that all mixed cumulants are zero if the entries of $\Deltaboldhatintro{}$ are independent, this suggests that perturbations from $0$ in the mixed cumulants of $\Deltaboldhatintro{}$ quantify the dependence between individual entries.

\paragraph{Mixed cumulants capture deviations from the \say{ground state}.}

A different way to make sense of \cref{obs1} is to view pure cumulants as those describing the \say{ground state} of the vector $\Deltaboldhatintro{}$ in which all entries are distributed according to their marginal distribution in $\Deltaboldhatintro{}$, independently of each other. 
On the other hand, all the mixed cumulants of a given order $r$ (i.e. all mixed cumulants $\kappa_s(\Deltaboldhatintro{})$ for $s = (s_1, s_2, \ldots, s_k)$ with $s_1 + s_2 +\ldots = r$) capture the \say{$r$-th order deviation} from this ground state. All such deviations combined contain all information regarding the dependence between multiple coordinates. Conveniently, due to the homogeneity of cumulants in case of sums of independent random vectors such as $\Deltaboldhatintro{} = \sum_{i=1}^d \Deltaboldhatintro{i}/\sqrt{d}$, cumulants of order $r$ are typically of magnitude $\approx (1/\sqrt{d})^{r-2}$, so they rapidly decay in $r$. 
Intuitively, this 
%implies that the cumulants of higher order account for dependency-inducing corrections that become smaller and smaller with increasing $r$. This 
suggests that the lowest-order non-zero mixed cumulants should contain the largest amount of information about dependencies.

\paragraph{Mixed cumulants for cycles and chains.}

It turns out that this point of view is particularly helpful when trying to account for the dependencies between edges in a cycle or a chain of vertices appearing in a $\Gbf \sim \RGG$. 
Assume that $\Deltaboldhatintro{}$ represents the distances associated to a cycle $H$ of length $k$, 
% i.e., assume there is a sequence of fixed vertices $v_1, \ldots, v_k$ in $\Gbf$ such that $H = \{ \{v_j, v_{j+1}\} \mid j \in [k] \}$, and $\Deltaboldhatintro{i}(j)$ represents the distance associated to $\{v_j, v_{j+1}\}$ in dimension $i$, 
% i.e., $\Deltaboldhatintro{i}(j) = (\Delta_i(\{v_j, v_{j+1}\}) - \mu)/\sigma$, with $\Delta_i(\{v_j, v_{j+1}\}) \coloneqq |\mathbf{x}_{v_j}(i) - \mathbf{x}_{v_{j+1}}(i)|_C^q$ where $\mu, \sigma$ are such that $\Deltaboldhatintro{i}(j)$ has mean $0$ and variance $1$. 
%
i.e., assume there is a sequence of fixed vertices $v_1, \ldots, v_k$ in $\Gbf$ such that $H = \{ \{v_j, v_{j+1}\} \mid j \in [k] \}$, and $\Deltaboldhatintro{i}(j) = (\Delta_i(\{v_j, v_{j+1}\}) - \mu)/\sigma$ represents the distance associated to $\{v_j, v_{j+1}\}$ in dimension $i$, 
where $\mu, \sigma$ are such that $\Deltaboldhatintro{i}(j)$ has mean $0$ and variance $1$. 
In this case, for every proper subset $S \subset [k]$ of coordinates, the random variables $\{\Deltaboldhatintro{i}(j)\}_{j \in S}$ correspond to the distances associated to the edges in a forest (since deleting any edge from a cycle yields a forest), and in this setting all the distances are independent. By the relationship between cumulants and mixed moments (cf. \eqref{eq:cumulantrelationship}), this forces all mixed cumulants of order $< k$ and all cumulants of order $k$ except for $\kappa_{(1,1,\ldots,1)}(\Deltaboldhatintro{})$ to be $0$. The same holds if $H$ represents a chain with endpoints $u, v$, conditional on the fixed positions of $u, v$. This is captured in the following observation.
\begin{observation}[Mixed cumulants for cycles and chains]\label{obs:mixedcumulants}
    Assuming that $\Deltaboldhatintro{} \in \mathbb{R}^k$ represents a cycle or a chain with $k$ edges, all mixed cumulants of order $\le k$ except for $\kappa_{(1,1,\ldots,1)}(\Deltaboldhatintro{})$ are zero. In case of a chain, this even holds conditional on the position of the endpoints $u,v$ of the chain. 
\end{observation}

\noindent Per the discussion above, this suggests that $\kappa_{(1,1,\ldots,1)}(\Deltaboldhatintro{})$ characterizes the highest-order perturbations from the \say{ground state} and thus captures most of the dependence between the edges of $H$.

While \cref{obs:mixedcumulants} is tailored towards cycles and chains, we remark that it also helps us understand the dependencies between edges in any pattern $H$ that merely contain cycles and chains, which we crucially use for characterizing the spectrum of $\Gbf$.%, which is enough to characterize the spectrum of $\Gbf$, make a statement about the signed weight of sufficiently sparse pattern graphs, and to give improved information theoretic lower bounds. 
In the following, we sketch how to turn \cref{obs:mixedcumulants} into an actual quantification of the mentioned dependencies.

\subsection{Classical Edgeworth Expansions}\label{sec:classicedgeworth}

A common approach to quantifying the influence of cumulants on the distribution of a random variable is an \emph{Edgeworth Expansion}, the underlying idea of which is to apply a Fourier inversion to an approximation of the CF of $\Deltaboldhatintro{}$ obtained by using a suitable Taylor series. This has the advantage that---after applying Fourier inversion---the leading term is equal to the density of a standard Gaussian (denoted by $\phi$) plus a series of correction terms, explicitly phrased in terms of cumulants of order $\ge 3$ and derivatives of $\phi$ (or alternatively in terms of Hermite polynomials). More precisely, $C_{\Deltaboldhatintro{}}(\mathbf{t}) = \exp\left( K_{\Deltaboldhatintro{}}(\mathbf{t}) \right)$ is approximated by 
\begin{align*}
    \tilde{C}_{\Deltaboldhatintro{}}(\mathbf{t}) = 
    % \exp\left( \tilde{K}_{\Deltaboldhatintro{}}(\mathbf{t}) \right) =
    \exp\Big( -\frac{\mathbf{t}^\top\mathbf{t}}{2} + \sum_{j=3}^s \frac{1}{j!}\boldsymbol{\kappa}_j^\top (i\mathbf{t})^{\otimes j} \Big) = \exp\Big( -\frac{\mathbf{t}^\top\mathbf{t}}{2}\Big)\exp\Big(\sum_{j=3}^b \frac{1}{j!}\boldsymbol{\kappa}_j^\top (i\mathbf{t})^{\otimes j} \Big)
\end{align*} 
where $b \ge 3$ is a parameter 
% controlling the accuracy of our approximation by 
governing how many cumulants are taken into account. Above, we used that the first term in $K_{\Deltaboldhatintro{}}(\mathbf{t})$ is $0$ when centering $\Deltaboldhatintro{}$ so that it has mean $0$, and that the second term is $\mathbf{t}^\top\mathbf{t}$ when rescaling $\Deltaboldhatintro{}$ to have identity covariance. Then, after applying an $(s-2)$-order Taylor expansion to the second exponential appearing above, we have that  
\begin{align*}
    \tilde{C}_{\Deltaboldhatintro{}}(\mathbf{t}) 
    \approx \exp\Big( -\frac{\mathbf{t}^\top\mathbf{t}}{2}\Big)\sum_{\ell=0}^{s-2}\Big(\sum_{j=3}^s \frac{1}{j!}\boldsymbol{\kappa}_j^\top (i\mathbf{t})^{\otimes j} \Big)^\ell 
    \!\!\!= \exp\Big( -\frac{\mathbf{t}^\top\mathbf{t}}{2}\Big) + \sum_{j = 3}^{s(s-2)} \exp\Big( -\frac{\mathbf{t}^\top\mathbf{t}}{2}\Big)\boldsymbol{\alpha}_j^\top(i\mathbf{t})^{\otimes j}
\end{align*} for a collection of vectors $\boldsymbol{\alpha}_j$ that depend on cumulants of order $\ge 3$. Applying the above transformation has the advantage that the expression can now explicitly be Fourier inverted so that the first term $\exp( -\mathbf{t}^\top\mathbf{t}/2)$ becomes exactly the density $\phi$ of a standard Gaussian, and every subsequent term becomes a multiple of $\phi$ involving the $\boldsymbol{\alpha}_j$ and a product of Hermite polynomials. Collecting terms of roughly the same magnitude and defining an appropriate \say{vectorized} version of Hermite polynomials $H_j$ as done in \cite{Kundhi_Rilstone_2020}, one obtains approximate densities 
such as \begin{align*}
    f_{\Deltaboldhatintro{}}(\mathbf{x}) 
    &= \phi(\mathbf{x}) \Big( 1 + \frac{\boldsymbol{\kappa}_3^{\top}}{6\sqrt{d}} H_3(\mathbf{x}) \Big),\text{ or}\\
    % \text{or } 
     f_{\Deltaboldhatintro{}}(\mathbf{x}) 
     &= \phi(\mathbf{x}) \Big( 1 + \frac{\boldsymbol{\kappa}_3^{\top}}{6\sqrt{d}} H_3(\mathbf{x}) + \frac{1}{d}\Big( \frac{\boldsymbol{\kappa}_4^{\top}}{24}H_4(\mathbf{x}) + \frac{\boldsymbol{\kappa}_3^{\otimes 2 \top}}{72}H_6(\mathbf{x}) \Big) \Big),
\end{align*} 
using an order $3$ and order $4$ expansion, respectively. We refer the interested reader to \cite{Kundhi_Rilstone_2020} for further details on this approach.
% and the explicit errors arising from the mentioned approximations. 

\subsection{The Need for Specialized Expansions}

In principle, the expansions in \cref{sec:classicedgeworth} capture all information about dependence between the coordinates of $\Deltaboldhatintro{}$, provided that the parameter $s$ is chosen sufficiently large. 
However, 
% it turns out that for our purposes, working with expansions of the form given above is quite cumbersome and impractical. 
classical Edgeworth expansions turn out to be impractical for quantifying
% The reason for this is that---as described above---we wish to quantify 
the influence of \emph{mixed} cumulants on top of $\Deltaboldhatintro{}$'s \say{ground state} dictated by all the \emph{pure} cumulants. 
% In case of classical Edgeworth expansions, this is difficult, for essentially two reasons.
This is for two reasons.

First, in classical expansions, pure and mixed cumulants are not treated separately; instead many terms (especially those of higher order) contain products of pure and mixed cumulants, making it difficult to quantify the influence of one on the other. Secondly, the \say{ground state} considered by classical expansions is not identical to the one we consider. Instead, the first term is always the density of a standard Gaussian, and subsequent terms are corrections to this ground state %whose magnitude decays with increasing order. 
While our ground state is roughly Gaussian (by a CLT), this crucially only holds approximately and replacing the true ground state of $\Deltaboldhatintro{}$ by a Gaussian therefore introduces an additional source of error whose magnitude is non-negligible when compared to the small perturbations caused by mixed cumulants we wish to measure. %Furthermore, trying to \say{untangle} the higher order correction terms to obtain a better approximation to our true ground state is cumbersome due to the complex structure of the higher order terms involving tensor products of cumulants in which pure and mixed cumulants appear together.

Fortunately, there is a more elegant and more direct approach to obtaining an approximate density capturing the correct \say{ground state} and using correction terms directly linked to the mixed cumulants of $\Deltaboldhatintro{}$, which we describe in the following section.
We remark that, while it might in principle be possible to obtain a result comparable to ours by relying on the expansions given in previous work \cite{Kundhi_Rilstone_2020} and \say{untangling} and rearranging terms, the technical complexity of a such approach would be comparable (or worse), while being less illustrative and interpretable.
Moreover, since the result in \cite{Kundhi_Rilstone_2020} holds under more restrictive assumptions than here,\footnote{Specifically, summands are i.i.d., the CFs are integrable, and all terms satisfy \cramer.} certain technical modifications (see \Cref{sec:technicalchallenges}) would have to be worked into their proof in any case. 

\subsection{Main Theorem---Approximating the Density of $\mathbf{z}$ above its Ground State}\label{sec:main thm}

% We use a different way of approximating the CF of $\mathbf{z}$ that -- after a Fourier inversion -- yields a more convenient approximate density related directly to the true \say{ground state} of $\mathbf{z}$. 
Our approach to approximating the CF of $\mathbf{z}$ yields -- after a Fourier inversion -- an approximate density related directly to the true ground state of $\mathbf{z}$, that is more convenient than those in \cref{sec:classicedgeworth}.
The \say{ground state} of $\mathbf{z}$ is defined as the density associated to the product measure $\otimes_{j=1}^k \mathcal{L}(\mathbf{z}(j))$ 
% where $\mathcal{L}(\mathbf{z}(j))$ denotes the law of the $j$-th entry of $\mathbf{z}$. 
over the laws of the entries of $\mathbf{z}$. 
Alternatively, we can 
% represent $\find(\mathbf{x})$ as the product of the marginal densities of $\mathbf{z}(j)$, i.e., 
directly define
$\find(\mathbf{x}) \coloneqq \prod_{j = 1}^k f_{\mathbf{z}(j)}(\mathbf{x}(j))$,
where $f_{\mathbf{z}(j)}$ is the (marginal) density of $\mathbf{z}(j)$. 
We further define 
$\kappa \coloneqq \frac{1}{d} \sum_{j=1}^d \kappa_{(1,1,\ldots,1)}(\mathbf{z}_j)$ to be the average of all cumulants of order $(1,1,\ldots,1)$ over the individual $\mathbf{z}_j$. 
In case $\mathbf{z}$ represents a cycle $H$, all these cumulants are equal and $\kappa = \kappa_{(1,1,\ldots,1)}(\mathbf{z}_j)$ for any $j$. 
However, in case of a chain $H$ conditional on the positions of its endpoints, the 
% $\kappa_{(1,1,\ldots,1)}(\mathbf{z}_j)$ are not the same for all $j$, 
terms are not all the same,
and $\kappa$ is a function of the position of the two endpoints associated to $H$. 
Nevertheless, in both cases the approximate joint density of $\mathbf{z}$ can be expressed in terms of $\kappa$. 
This is formally captured in our main technical result.

\begin{restatable}[Main theorem, the joint density of $\mathbf{z}$ above its ground state]{theorem}{maindensityapprox}\label{thm:maindensityapprox}
Let $H$ be a cycle or a chain of length $k \ge 2$. Let $f$ denote the density of $\mathbf{z}$ and assume that a constant fraction of the $\mathbf{z}_i$ satisfy \cramer (\Cref{def:cramer}) for some $\varepsilon(\delta)$.
Then, for any choice of integer valued parameters $\bone \ge 3, \btwo \ge k, \sone, \stwo \ge 1$, there is a sequence of vectors $\left(\boldsymbol{\alpha}_j\right)_{j=k+1}^{\abparamsum}$ with $\boldsymbol{\alpha}_j \in \mathbb{R}^{k^j}$ such that the density $f$ of $\mathbf{z}$ is approximated by 
\begin{align*}
      \tilde{f}(\mathbf{x}) \coloneqq \find(\mathbf{x}) + (-1)^k\kappa \left( \frac{1}{\sqrt{d}} \right)^{k-2} \prod_{j=1}^k\phi^{(1)}(\mathbf{x}(j)) + \sum_{j=k+1}^{\abparamsum} \boldsymbol{\alpha}_i^\top \left( \boldsymbol{\nabla}^{\otimes j} \phi(\mathbf{x}) \right)
\end{align*} 
where $\phi(\mathbf{x})$ is the density of the standard Gaussian distribution and $\phi^{(1)}$ is its first derivative. The approximation error satisfies 
\begin{align*}
    \sup_{\mathbf{x} \in \mathbb{R}^k} |f(\mathbf{x}) - \tilde{f}(\mathbf{x})| = \bigO{ \max\left\{ \frac{1}{\sqrt{d}}, |\kappa| \right\}^{\!\stwo+1}\!\! \left( \frac{1}{\sqrt{d}}\right)^{\!(k-2)(\stwo+1)}\!\!\!\!\!\!\!\! + \left( \frac{1}{\sqrt{d}}  \right)^{\!\min\left\{\sone + 1, \bone-2, \btwo-2\right\}}}.
\end{align*} 
In case of a chain of length $k = 2$, we require $|\kappa| = o(1)$ for the above to hold. Furthermore, the absolute value of the coefficients in the remainder terms $\boldsymbol{\alpha}_j$ is bounded as 
\begin{align*}
    \|\boldsymbol{\alpha}_{k+j}\|_\infty \le C \left( \frac{1}{\sqrt{d}} \right)^{k-2} 
    \begin{cases}
        \max \left\{ |\kappa| d^{-\frac{1}{6}}, \frac{1}{\sqrt{d}} \right\} d^{-\frac{j-1}{6}} & \textup{if } k \ge 3 \text{ or } (k = 2 \text{ and } |\kappa| < d^{-\frac{1}{3}}),\\
        \max \left\{ |\kappa|^{1 + \frac{j}{2}}, d^{-\frac{j-1}{6}} \right\} & \textup{if } k = 2 \text{ and } |\kappa| \ge d^{-\frac{1}{3}},
    \end{cases}
\end{align*} 
where $C > 0$ is a constant.
\end{restatable}

\noindent In \cref{thm:maindensityapprox} the higher order corrections are represented not in terms of Hermite polynomials, but in terms of (vectors of) partial derivatives of $\phi$. We believe this simplifies the exposition, but remark that it is simple to obtain an equivalent formulation involving Hermite polynomials instead.

We now sketch the proof of \cref{thm:maindensityapprox} and explain the technical difficulties.
The key idea is to split the CGF of $\mathbf{z}$ into \say{pure} and \say{mixed} cumulants as 
$K(\mathbf{t}) = \Kpure(\mathbf{t}) + \Kmix(\mathbf{t})$, where $\Kpure$ and $\Kmix$, contain precisely those terms in the power series representation of $K_{\mathbf{z}}(\mathbf{t})$ which correspond to pure or mixed cumulants, respectively. 
To quantify the influence of $\kappa$ on the distribution of $\mathbf{z}$, we define an approximation to the CGF and the CF of $\mathbf{z}$ that takes all pure cumulants and the first mixed cumulant into account. 
For this
we introduce the parameters $\bone, \btwo$ as mentioned in \cref{thm:maindensityapprox} and define 
\begin{align*}
    \Kpuredaggers(\mathbf{t})\! =\!\! \sum_{j=3}^{\bone} \frac{1}{j!} \kappapuretopj{j} (i\mathbf{t})^{\otimes j} \text{ and } \Kmixdaggers(\mathbf{t})\! = \! \kappa \left( \frac{1}{\sqrt{d}} \right)^{k-2}\prod_{j=1}^k(i\mathbf{t}(j))\! +\!\!\!\!  \sum_{j=k+1}^{\btwo} \frac{1}{j!} \kappamixtopj{j} (i\mathbf{t})^{\otimes j}.
\end{align*} 
That is, $\Kpuredaggers$ contains all terms corresponding to pure cumulants of order between $3$ and $\bone$ while $\Kmixdaggers$ contains all terms corresponding to mixed cumulants of order between $k$ and $\btwo$. Then the CF of $\mathbf{z}$ is approximated as
\begin{align*}
     % &= \exp\left( K_{(\text{pure})}(\mathbf{t})\right) + \exp\left( K_{(\text{pure})}(\mathbf{t})\right) \left(\exp\left( K_{(\text{mix})}(\mathbf{t})\right) - 1 \right)\\
  C(\mathbf{t})  &\approx \exp\left( K_{(\text{pure})}(\mathbf{t})\right) + \exp\Big(-\frac{\mathbf{t}^\top\mathbf{t}}{2}\Big)\exp\left( \Kpuredaggers(\mathbf{t})\right) \left(\exp\left( \Kmixdaggers(\mathbf{t})\right) - 1 \right),
\end{align*} 
leaving out all pure cumulants of order $> \bone$ and all mixed cumulants of order $> \btwo$ in the second term. 
% To make the exponentials above useful for getting a sum of suitable correction terms, 
To transform this expression into a sum of tractable correction terms,
we expand the exponentials as Taylor series of a certain order.
%, like it is customary in the context of Edgeworth expansions. 
By applying an $\sone$-th order expansion to the first, and a $\stwo$-th order expansion to the second term, we obtain 
\begin{align*}
    \exp\left( \Kpuredaggers(\mathbf{t})\right) \approx \underbrace{\sum_{\ell = 0}^{s_1} \frac{1}{\ell!} \left( \Kpuredaggers(\mathbf{t}) \right)^\ell}_{\eqqcolon P_1(\mathbf{t})} \text{ and } \exp\left( \Kmixdaggers(\mathbf{t})\right) - 1\approx \underbrace{\sum_{\ell = 1}^{s_2} \frac{1}{\ell!} \left( \Kmixdaggers(\mathbf{t}) \right)^\ell}_{\eqqcolon P_2(\mathbf{t})}. 
\end{align*} 
%%%%%% Explicit form not needed now -- let's just hint to the reader
%If we were to expand $P_1, P_2$ further we would end up with a complex polynomial of the form \begin{align*}
%     P_1(\mathbf{t})  = 1 + \sum_{\ell = 1}^{\sone} \frac{1}{\ell!}\sum_{j_1=3}^{\bone} \sum_{j_2=3}^{\bone} \ldots \sum_{j_\ell=3}^{\bone} \frac{1}{j_1!j_2!\cdots j_\ell!}\left( \otimes_{m=1}^\ell \kappapurej{j_m} \right)^\top (i \mathbf{t})^{\otimes \sum_{m=1}^\ell j_m},
% \end{align*} 
% where the tensor products are correct up to reordering of the $j_m$. These polynomials are similar in shape to the expansions used in \cite{Kundhi_Rilstone_2020}. This expression of $P_1, P_2$ is important in \Cref{sec:higherorderterms} where the explicit bound on $\|\boldsymbol{\alpha}_{k+j}\|_\infty$ is proved. 
%
It is possible expand $P_1, P_2$ to end up with a complex polynomial of a particular form, similar in shape to the expansions used in \cite{Kundhi_Rilstone_2020}, which expansion is important in \Cref{sec:higherorderterms} where the explicit bound on $\|\boldsymbol{\alpha}_{k+j}\|_\infty$ is proved. 
For now, it is only necessary to 
% infer from the above 
note
that the product $P_1P_2$ can be written in the form 
\begin{align*}
    P_1(\mathbf{t})P_2(\mathbf{t}) = \kappa\left( \frac{1}{\sqrt{d}} \right)^{k-2} \prod_{j=1}^k(i\mathbf{t}(j)) + \sum_{j=k+1}^{\abparamsum} \boldsymbol{\alpha}_j^\top (i\mathbf{t})^{\otimes j}
\end{align*} for a sequence of vectors $\left(\boldsymbol{\alpha}_j\right)_{j=k+1}^{\abparamsum}$ with $\boldsymbol{\alpha}_j \in \mathbb{R}^{k^j}$ for which the absolute value of all components can be explicitly bounded. 
This leads to explicit bounds on the approximation to the CF of $\mathbf{z}$ that can be Fourier inverted, eventually yielding \Cref{thm:maindensityapprox}. 

\subsection{Technical Challenges}\label{sec:technicalchallenges}

% It turns out that the all of the above is much easier said than done. 
% The main reason for this is that the individual $\mathbf{z}_i$ are not quite as \say{well-behaved} as typically assumed in the context of Edgeworth expansions. 
% Concretely, there are two major obstacles we have to overcome: (1) our individual $\mathbf{z}_i$ are in fact \emph{not continuous} (i.e. they have no density), and (2) a prerequisite of the inversion theorem (\Cref{thm:inversion}) is that the CF of $\mathbf{z}$ is \emph{integrable}, which is a priori not guaranteed. We briefly describe how we overcome these obstacles. 
%
The two major obstacles in the method described in \cref{sec:main thm} are: (1) the individual $\mathbf{z}_i$ have no density, and (2) a prerequisite of the inversion theorem (\Cref{thm:inversion}) is that the CF of $\mathbf{z}$ is integrable, which is not guaranteed \emph{a priori}. We briefly describe how to handle these obstacles. 

\paragraph{\cramer ensures validity for vectors without density.}
As already mentioned, the literature usually assumes the $\mathbf{z}_i$ to have a density, which in our case is not given. To see this, consider the simple example in which $\mathbf{z}_i$ represents the distances associated to a cycle of length $3$ (in dimension $i$). Then, conditional on the first two coordinates $\mathbf{z}_i(1), \mathbf{z}_i(2)$, it holds that $\mathbf{z}_i(3)$ can only take two specific \emph{discrete} values, implying that $\mathbf{z}_i(1)$ has no density. 

However, a sum of independent random vectors that each have no density can itself have a density that can be approximated using the same tools as described above. 
%An simple example of such a situation is given when each random vector $\mathbf{z}_i \in \mathbb{R}^k$ in our sum is distributed uniformly in the direction of a given basis vector. As long as there is at least one term distributed in direction of each of the $k$ basis (vectors associated to an arbitrary basis of $\mathbb{R}^k$), the result is a continuous distribution in $\mathbb{R}^k$. 
A sufficient criterion for this phenomenon to occur is \cramer.
\begin{restatable}[\cramer]{definition}{cramercond}\label{def:cramer}
A random vector $\mathbf{z} \in \mathbb{R}^k$ satisfies \cramer if for all $\delta > 0$ and all $\mathbf{t} \in \mathbb{R}^k$ with $\|\mathbf{t}\| \ge \delta$, we have $
        |C_{\mathbf{z}}(\mathbf{t})| \le 1 - \varepsilon(\delta),
    $ where $\varepsilon(\delta) > 0$ for all $\delta > 0$.
\end{restatable}

\noindent \cramer is commonly used in the literature as a prerequisite for Edgeworth expansions. While \cramer is always met for continuous random vectors (which follows from the Riemann--Lebesgue lemma), and while it is easy to construct examples of non-continuous random vectors that do not meet \cramer (e.g. lattice distributions), there are random vectors that have no density but still satisfy the above condition. Our $\mathbf{z}_i$ are an example for this, which is formally shown in \Cref{lem:cramer}. This statement holds but it makes an exception for chains of length $k = 2$. Here, we can only guarantee that \cramer holds with a constant probability over the draw of the endpoints of our chain. However, this still means that with extremely high probability, there will be a constant fraction of the terms in the sum satisfying \cramer, and this is enough to obtain the same error guarantees as claimed previously.

\paragraph{Adding Gaussian noise to ensure integrability of the CF.}
The second obstacle is that both our error analysis in \Cref{sec:larget} and the inversion theorem (\Cref{thm:inversion}) require a CF which is integrable.
To overcome this hurdle, we use a simple fix that consists of adding independent, isotropic Gaussian noise $\boldsymbol{\eta} \sim \mathcal{N}(0, d^{-2\eta}\mathbf{I}_k)$ to $\mathbf{z}$, where $\eta$ is a parameter. Since $\boldsymbol{\eta}$ has an integrable CF, and since $\int |C_{\mathbf{z} + \boldsymbol{\eta}}(\mathbf{t})| \dd \mathbf{t} = \int |C_{\mathbf{z}}(\mathbf{t})||C_{\boldsymbol{\eta}}(\mathbf{t})| \dd \mathbf{t} \le \int |C_{\boldsymbol{\eta}}(\mathbf{t})| \dd \mathbf{t}$, this ensures integrability of the resulting $\mathbf{z} + \boldsymbol{\eta}$. Intuitively, taking the convolution of $\mathbf{z}$ with an independent Gaussian random variable ensures that all densities involved are smooth and continuous. While this adds another source of error and requires careful attention in the error analysis in \Cref{sec:larget}, it is quite easy to quantify by how much the probability of events defined in terms of the coordinates of $\mathbf{z}$ is influenced by our noise (\Cref{lem:influenceofnoise}). In particular this additional error can be made small enough for our purposes when choosing $\eta$ to be a sufficiently large constant. Furthermore, any constant $\eta$ only leads to a \say{blow-up} that is polynomial in $d$ in \Cref{sec:larget}, which is compensated by a factor exponentially small in $d$, arising due to \cramer.

\subsection{Applications to $\RGG$}

We continue by describing how the expansions from \Cref{thm:maindensityapprox} are helpful for understanding $\RGG$ in high dimensions. A generally helpful quantity is the so-called \emph{signed weight} of small pattern graphs $H$ (represented as a set of edges), defined as 
    $\SW(H) \coloneqq \prod_{e \in H} (\mathds{1}(e) - p)$,
 where -- given an edge $e = \{u, v\}$ -- we write $\mathds{1}(e)$ for the indicator that the edge is present in a given $\Gbf \sim \RGG$. The expected signed weight $\Expected{\SW(H)}$, intuitively tells us something about the dependence between the edges in $H$. In fact, many recent breakthroughs in understanding random geometric graphs and their information-computation landscape as well as their spectral properties rely on computing the signed weight of suitable patterns graphs $H$ (or alternatively the very closely related \emph{Fourier Coefficients}). We mention \cite{Bangachev_Bresler_2024,Bangachev_Bresler_2024_fourier} as very recent examples. Crucially, studying the signed weight of suitable patterns is helpful for both designing algorithmic upper and lower bounds (both computational and information-theoretic).

\paragraph{Bounding the signed weight of cycles and chains}

Our three applications all rely on turning \Cref{thm:maindensityapprox} into a suitable bound on $\Expected{\SW(H)}$ in case $H$ is a cycle or a chain. As mentioned above, this bound is not limited to only cycles and chains, but also makes a statement about the signed weight of a pattern $H$ in which cycles and chains appear. By a careful integration of the expansions in \Cref{thm:maindensityapprox}, we obtain the following bound.

\begin{restatable}{theorem}{fouriercoefficients}\label{thm:fouriercoefficients}
    Let $C_k$ be a cycle of length $k \ge 3$. Then 
    \begin{align*}
            \Expected{\textsc{Sw}(C_k)}  = \bigO{ p^k \left(\frac{\log(n)}{\sqrt{d}}\right)^{k-2} }.
        \end{align*}
    If $H$ is a chain of length $k \ge 2$ on the set of vertices $\{v_1, \ldots, v_{k+1}\}$ with positions $\mathbf{x}_1, \ldots, \mathbf{x}_k \in \mathbb{T}^d$, then conditional on any position $\mathbf{x}_1, \mathbf{x}_{k+1}$ of the endpoints of $H$, 
    \begin{align*}
        \Expectedsub{\mathbf{x}_2, \ldots, \mathbf{x}_k}{\textsc{Sw}(H) \mid \mathbf{x}_1, \mathbf{x}_{k+1}} = \bigO{ p^{k}|\kappa|  \left(\frac{\log(n)}{\sqrt{d}}\right)^{k-2} + p^k \log^2(n) \left(\frac{\log(n)}{\sqrt{d}}\right)^{k-1} }.
    \end{align*} 
    where $\kappa \coloneqq \frac{1}{d} \sum_{j=1}^d \kappa_{(1,1,\ldots,1)}(\mathbf{z}_j)$ (as in \Cref{thm:maindensityapprox}), which is a function of $\mathbf{x}_1$ and $\mathbf{x}_{k+1}$. If $k = 2$, we require $H$ to be $(\alpha, d^{-
    1/3})$--good (cf. \Cref{def:goodchains}) for an arbitrarily small constant $\alpha > 0$.
\end{restatable}

The bound from \Cref{thm:fouriercoefficients} (combined with further, specialized arguments) enables the aforementioned three applications regarding the properties of $\RGG$.

\subsubsection{Application 1: Tight Testing Thresholds via Signed Triangles}

To prove \cref{thm:signedtriangles}, we analyze the statistical power of the \emph{signed triangle} test on $\RGG$, defined as
\begin{align*}
    \T(\Gbf) \coloneqq  \sum_{i < j< k \in [n]} \Tijk{ijk} \text{ where } \Tijk{ijk} \coloneqq (\Gbf_{ij} - p)(\Gbf_{ik} - p)(\Gbf_{jk} - p).
\end{align*} where $G_{ij}$ denotes the indicator variable of the edge $\{i,j\}$.
 We bound the $ \T(\Gbf)$ from below, and its variance from above, then the result follows by Chebyshev's inequality. While the variance bound follows from \Cref{thm:fouriercoefficients}, the main challenge here is to bound $\Expected{\T(\Gbf)}$ \emph{from below}. To this end we will use \Cref{thm:maindensityapprox}, however, in contrast to the upper bound on signed weights from \Cref{thm:fouriercoefficients} we have to control the quantity $\kappa$ much more carefully. In fact, we must assert that $\kappa$ is a \emph{negative} constant for all $q \ge 1$ such that the first correction term from \Cref{thm:maindensityapprox}, proportional to $(-1)^k\kappa/\sqrt{d}$ is positive which results in an increased probability of forming a triangle. To do this we rely on the following technical lemma, which we prove in \Cref{sec:deferredtriangles}. 
\begin{restatable}{lemma}{correlation} \label[lemma]{lem:correlation}
    Consider a triangle $C_3$ with vertices $v_1, v_2, v_3$ and edges $e_1, e_2, e_3$. For all $L_q$-norms with $q \ge 1$, we have that  $\Expected{\gamma(e_1)\gamma(e_2)\gamma(e_3)} < 0$ where $\gamma(\{u,v\}) \coloneqq |x_u-x_v|_C^q - \Expected{|x_u-x_v|_C^q}$ and $x_{v_1}, x_{v_2}, x_{v_3} \sim \unifhalf{}$ are the positions of $v_1, v_2, v_3$ in a fixed dimension.
\end{restatable}
\noindent Proving the above relies on splitting $\Expected{\gamma(e_1)\gamma(e_2)\gamma(e_3)} = \mathbb{E}_{y}\left[\gamma(e_1)\mathbb{E}[\gamma(e_2)\gamma(e_3) \mid y]\right]$ where $y$ is the distance associated to the endpoints of $e_1$ and then employing a monotonicity argument to show that $\gamma(e_1)$ is strictly increasing in $y$ and that $\mathbb{E}[\gamma(e_2)\gamma(e_3) \mid y]$ is strictly decreasing in $y$. Intuitively, this asserts that the correlation of $\gamma(e_2)$ and $\gamma(e_3)$ is decreasing in $\gamma(e_1)$, so $\gamma(e_1)$ and $\mathbb{E}[\gamma(e_2)\gamma(e_3) \mid \gamma(e_1)]$ are negatively correlated. The main technical challenge is to prove monotonicity of $\mathbb{E}[\gamma(e_2)\gamma(e_3) \mid y]$, which we achieve by carefully estimating the derivative of $\mathbb{E}[\gamma(e_2)\gamma(e_3) \mid y]$ using the Leibnitz integral rule and suitable substitutions. 

\subsubsection{Application 2: Improved Information-Theoretic Lower Bounds for all Densities }

% We find it remarkable that the expansions from \Cref{thm:maindensityapprox} do not only yield improved algorithmic upper bounds, but also allow us to improve the information-theoretic lower bounds for all edge densities $p$. This improves over the previously only known information-theoretic lower bound for distinguishing toroidal RGGs under $L_q$-norm from a $\Gnp$ in \cite[Theorem 1.10]{Bangachev_Bresler_2024} which only makes a statement about the case of $p = 1/2$ (although it also allows for $q = \omega(1)$). Our \Cref{thm:fouriercoefficients} (for chains) allows for a strengthening of this result applying to all fixed $1 \le q < \infty$ and -- in particular -- to all $\frac{\alpha}{n} \le p \le 1 - \varepsilon, \alpha > 0$. 
% \begin{restatable}{theorem}{statisticalindistinguishability}\label{thm:indistinguishability}
%     For any fixed $1 \le q < \infty$ and any $\frac{\alpha}{n}\le p \le 1 - \varepsilon$, we have that \begin{align*}
%         \TV{\RGG}{\Gnp} = o(1)
%     \end{align*} whenever $d = \omega( \max\{ n\log(n), p^2n^3\log^6(n) \})$.
% \end{restatable}

\noindent Both our \Cref{thm:indistinguishability} and \cite[Theorem 1.10]{Bangachev_Bresler_2024} are based on a bound from Liu and Ràcz \cite{Liu_Racz_2023} that relates the total variation distance between $\RGG$ and $\Gnp$ to a sum over the moments of the self-convolution of two specific edge indicator variables. Precisely, Liu and Ràcz showed that
\begin{align*}
    \TV{\RGG}{\Gnp} \le \sum_{k=0}^{n-1} \log \left( \mathbb{E}_{\mathbf{x}, \mathbf{y}} \left[ \left(1 + \frac{\gamma(\mathbf{x},\mathbf{y})}{p(1-p)} \right)^k\right] \right), \end{align*} 
    where $\gamma(\mathbf{x},\mathbf{y}) = \mathbb{E}_{\mathbf{z}}[ (\mathds{1}_e(\mathbf{x}, \mathbf{z}) - p)(\mathds{1}_e(\mathbf{y}, \mathbf{z}) - p) ]$
 and $\mathds{1}_e(\mathbf{x}, \mathbf{z})$ denotes the indicator random variable for the event that $\|\mathbf{x} - \mathbf{z}\|_q \le \tau$.
In \cite{Bangachev_Bresler_2024}, the above moments of $\gamma$ are bounded using the Bernstein--McDiarmid concentration inequality together with an involved combination of anti-concentration and Fourier theoretic arguments aimed at computing marginal increments and variances. While this allows to explicitly quantify the influence of $q$, adapting the heavy machinery underlying this proof to the case of general $p$ seems out of reach. In contrast to this, we simply use a convenient expansion of the moments arising in the bound of Liu and Ràcz into a sum over signed weights of complete bipartite graphs $K_{2,k}$ given by \begin{align*}
    \mathbb{E}_{\mathbf{x}, \mathbf{y}} \left[ \left(1 + \frac{\gamma(\mathbf{x},\mathbf{y})}{p(1-p)} \right)^k\right] = \sum_{j=0}^{k} \binom{k}{j} \frac{1}{{p^j(1-p)^j}} \Expected{ \textsc{Sw}(K_{2,j}) },
\end{align*} such that we can subsequently apply our bound from \Cref{thm:fouriercoefficients} ultimately resulting in the bound from \Cref{thm:indistinguishability}. This is possible since a $K_{2,k}$ is just the union of $k$ chains of length $2$, so \Cref{thm:fouriercoefficients} is applicable (modulo some technical precautions). 

\paragraph{Comparison with the bounds known for spherical RGGs}
On top of \Cref{thm:fouriercoefficients}, all of this can be done in a remarkably short proof, while still resulting in a lower bound that is just as strong as the best known information-theoretic lower bound for spherical RGGs in the case of general $p$ from \cite[Theorem 1.2]{liu2022testing}. This is despite the fact that the bound in \cite{liu2022testing} relies on much heavier machinery based on concentration results of the measure of the intersection of random spherical caps and anti-caps using optimal transport. We remark that the authors of \cite{liu2022testing} also prove a stronger lower bound in case of $p = \alpha/n$ for constant $\alpha > 0$ which relies on a specialized argument exploiting the locally-tree-like properties of sparse RGGs to analyze an instance of the belief propagation algorithm for computing the marginal distribution of latent vectors conditioned on forming a specific graph, reminiscent of the \say{cavity method} from statistical physics. However, this analysis breaks down in the case of general $p$ as the locally-tree-like properties are lost. In this case, their approach based on \say{concentration via optimal transport} yields the same bound as in \Cref{thm:indistinguishability} (up to polylogarithmic factors). Closing the gap between algorithmic upper bounds and information-theoretic lower bounds for all $p=o(1)$, i.e., the gap between $n^3p^3$ and $n^3p^2$, is an open problem both for spherical RGGs and our model. It arises because all techniques for proving lower-bounds in the general $p$ case rely on \say{worst-case} vector embeddings of a given graph $\Gbf$, and going beyond this barrier requires new approaches that allow for analyzing \say{average-case} vector embeddings instead, this was also observed in \cite{liu2022testing}.

%\leon{todo: talk about the gap in comparision to \Cref{thm:signedtriangles}}

\subsubsection{Application 3: Spectral Properties of Random Toroidal Graphs}\label{sec:intro_psketch_spectral} \label{sec:spectralexplanation}

The upper bound of \cref{thm:spectralstuff} and \Cref{thm:spectralstuffinfty} is proved via the trace method. Since we are interested in the second largest eigenvalue in absolute value, the matrix we shall consider is the \emph{centered} adjacency matrix $\Adjp = \Adj - p \mathds{1}\mathds{1}^\top$, where $\mathds{1}$ is the $n$-dimensional all-ones vector. This eliminates the eigenspace associated to the largest eigenvector of $\Adj$ (which is well approximated by $\mathds{1}$). We thus reduced our problem to bounding the largest eigenvector of $\Adjpp$, which we do by bounding the trace of $\Adjppt{m}$ for an even, positive integer $m$. This is helpful for bounding the spectral norm since
\begin{align*}
    \tr{\Adjpp} = \sum_{i = 1}^n \lambda_i\left( \Adjpp \right),
\end{align*} 
so considering the $m$-th power of $\Adjpp$, we see that 
\begin{align*}
    \tr{\Adjppt{m}} =  \sum_{i = 1}^n \lambda_i\left( \Adjpp \right)^m \text{, so }
    \left|\lambda_1\left( \Adjpp \right)\right|^m \le \tr{\Adjppt{m}}
\end{align*} for every even $m$. By Markov's inequality, this means that for any $a = \omega(1)$, 
\begin{align}\label{eq:traceMethod}
    \Pr{\left|\lambda_1\left( \Adjpp \right)\right| \ge a^{1/m} \Expected{ \tr{\Adjppt{m} }}^{1/m}}\leq 1/a = o(1)
\end{align}
and we are left with bounding the above expectation. Conveniently, $\tr{\Adjppt{m}}$ has a combinatorial interpretation as the sum over all directed, closed walks of length $m$ on the complete graph $K_n$ with vertex set $[n]$, while each term in the sum corresponds to the product of the entries of $\Adjpp$ along the edges of the walk.
Formally, writing $\EW(n, m) \coloneqq \{(v_1, \dots, v_{m + 1}) \in [n]^{m + 1} \mid v_1 = v_{m+1}\}$ for the set of closed walks of length $m$, we have
\begin{align} \label{eq:traceAsWalks} 
    \tr{\Adjppt{m} }  = \sum_{(v_1, \dots, v_{m + 1}) \in \EW(n, m)} \prod_{j = 1}^{m} \big(\Adj_{v_j, v_{j+1}} - p \big) .
\end{align}
Now it is easy to see that the right hand side of the above can be represented as a sum over signed weights of suitable Eulerian multigraphs.
To this end, we associate a multigraph $H_{\pmb{v}}$ to a given walk $\pmb{v} = (v_1, \dots, v_{m+1}) \in \EW(n, m)$, defined as the multigraph on the vertex set $\{v_1, \dots, v_{m + 1}\}$ (duplications of vertices are removed), where we add an edge $\{v_j, v_{j+1}\}$ for every $j \in [m]$ with $v_j \neq v_{j+1}$ to the multiset $E(H_{\pmb{v}})$ (we preserve multiplicities but we do not allow self loops). Then, it is not hard to see that
$$
    \tr{\Adjppt{m} }  \le \sum_{\pmb{v} \in \EW(n, m)} \absolute{\SW(H_{\pmb{v}})},
$$ where we define the signed weight of a multigraph $H$ as $\SW(H) \coloneqq \prod_{e \in E(H)} (\Adj_e - p)$ such that multiplicites of edges are respected in the product.

Now, to bound the right-hand side of the above, we employ a strategy similar to the one used in~\cite{Li_Schramm_2024} for the Gaussian mixture block model. Namely, we bound the expected signed weight of a given Eulerian graph $H$ with $\ell \le m$ edges and no self-loops by \say{contracting} all cycles and chains in $H$ to reduce $H$ to a \say{core} $\Hardcore$ allowing us to treat edges inside and outside of $\Hardcore$ separately. Crucially, the core $\Hardcore$ is still a Eulerian multigraph but it now has minimum degree $4$, limiting the number of vertices it can contain to $m/2$. Moreover, conditional on the positions of vertices in $\Hardcore$, any removed cycle or contracted chain appears independently and its contribution to $\Expected{\SW(H)}$ can be bounded using \Cref{thm:fouriercoefficients}. The core $\Hardcore$ itself only has a limited contribution because the number of vertices in $\Hardcore$ is limited by $m/2$, which in turn limits the number of possible embeddings of a such $\Hardcore$ in $K_n$. On a high level, this already explains the upper bound from \Cref{thm:spectralstuff} and the phase transition occurring at $d \approx n$ (assuming that $p$ is a constant): 

\begin{itemize}
    \item If $d \gg n$, then the sum \eqref{eq:traceMethod} is dominated by those $H$ that contain as few contractable edges on chains or cycles of length $\ge 3$ as possible. Given a fixed number $m$ of edges in $H$, this property implies that $H$ contains few vertices, i.e., that our sum is dominated by \emph{dense} $H$. Precisely, it is not hard to see that a Eulerian multigraph $H$ with $m$ edges and without chains or cycles of length $\ge 3$ that can be contracted contains at most $1 + m/2$ vertices; this number is achieved exactly if $H$ is the union of $m/2$ cycles of length $2$ arranged in a tree-like fashion. Since there are $\bigO{n^{1 + m/2}}$ embeddings of a such $H$ in $K_n$ each of which roughly makes at most a constant contribution to the sum, we get that $$\Expected{ \tr{\Adjppt{m} }}^{1/m} \approx \left(n \cdot n^{m/2}\right)^{1/m} = n^{1/m}\sqrt{n},$$ which tends to $\sqrt{n}$ if $m$ is sufficiently large compared to $n$. The above further suggests that there are $\approx n$ eigenvalues of order $\sqrt{n}$, which is the same as we would expect in a $\Gnp$. 
    \item If $d \ll n$, then the opposite occurs and the contribution of a single $H$ to the sum becomes larger the more contractable chains and cycles it contains. This means that the sum is essentially dominated by sparse $H$, i.e., those which contain the same number of vertices and edges. Hence, the largest contribution to the sum in \eqref{eq:traceMethod} is achieved simply by cycles with $m$ edges.
    There are $\le n^{m}$ ways of embedding a such cycle into $K_n$ each of which has a contribution of $\approx (1 / \sqrt{d})^{m -2}$ to the sum such that
    $$\Expected{ \tr{\Adjppt{m} }}^{1/m} \approx \left(n^{m} \left(\frac{1}{\sqrt{d}}\right)^{m -2} \right)^{1/m} = d^{1/m}\left(\frac{n}{\sqrt{d}}\right),$$ which is $n^{o(1)}(n / \sqrt{d})$ when letting $m$ tend to $\infty$. This further suggests that now there are $\approx d$ large eigenvalues each of order $n/\sqrt{d}$.
\end{itemize} 

We remark that the factor of $n^{o(1)}$ in our upper bound comes from the fact that our expansions in their current form only allow us to handle the case of constant but arbitrarily large $m$. In order to improve this error to a polylogarithmic factor, we would need $m = \Omega(n)$. While we stronly believe our techniques to be adaptable to the regime $m = \Theta(n)$ as well, we decided to limit our attention to constant $m$ here for the sake of technical simplicity.

However, we are able to generalize the described results to the case of general $p$, namely to any $\frac{1}{n} \le p \le 1 - \varepsilon$. 
In this case, a similar phenomenon occurs but we have to be more careful with the contribution of edges in the core $\Hardcore$ and edges on degenerate cycles (i.e. cycles with only two edges). Our bounds on $\Expected{ \tr{\Adjppt{m} }}^{1/m}$ further allow us to make a statement about the number of large eigenvalues in a $\RGG$ since we are able to show that \begin{align}\label{eq:goodtrace}
\Expected{ \tr{\Adjppt{m} }} \le \begin{cases}
    \bigOtilde{ d\left(\frac{np}{\sqrt{d}}\right)^m } &\text{if } d \ll np \log^2(n)\\
    \bigOtilde{ n\sqrt{np}^m } &\text{if } d \gg np \log^2(n)
\end{cases}
\end{align} for every fixed even $m$. If there were $\Omega(d^{1+\varepsilon})$ eigenvalues of order $\Omega(np/\sqrt{d})$ while $d \ll np$, this would contradict the above.

While our proof is similar to the idea used in \cite{Li_Schramm_2024}, we need to be significantly more careful in estimating the contribution of different Eulerian multigraphs $H$ to our sum \eqref{eq:traceMethod} in order to arrive at the statement in \eqref{eq:goodtrace}. In \cite{Li_Schramm_2024} a weaker bound suffices since the authors are only interested in an upper bound on the spectral norm of the considered matrix. To prove our bounds on the number of large eigenvalues, we require the stronger statement from \eqref{eq:goodtrace}, which we achieve by proving that the sum in \eqref{eq:traceMethod} is indeed dominated by cycles of length $m$ while all other Eulerian multigraphs $H$ only contribute a lower-order term. To do this, we need a bound on the contribution of non-contracted edges which is stronger than the one used in \cite{Li_Schramm_2024}, and we explicitly distinguish between those $H$ which have a trivial and non-trivial core, respectively, where a core is trivial if it only consists of a single, isolated vertex.

\subsection{Application to Other Metric Spaces}\label{sec:applicability}

We note that our technique is not limited to the torus, but can be applied to more general metric spaces for which the cumulants associated to analogous random vectors have similar properties as in our case. For example, they almost instantly carry over to other product spaces, provided only that the CGFs in each dimension continue to meet Cramer's condition \cref{def:cramer}) and the space remains homogeneous. For example, this is the case when taking the cartesian product of $d$ spheres $\mathbb{S}^{k}$ (each of constant dimension $k$) instead of the circle $\mathbb{S}^1$, which yields the torus.

However, the presented ideas are not limited to such product spaces and we expect our techniques to carry over more generally if (1) the space is homogeneous, and (2) the magnitude of cumulants decays exponentially in their order $s$ while (3) the total cumulant generating function (the CGF after taking the sum over all dimensions) satisfies an analog of Cramér's condition (cf. \cref{def:cramer}) in the sense that its absolute value decays exponentially in $d$ for all $\mathbf{t}$ with $\|\mathbf{t}\| \ge \varepsilon \sqrt{d}$.
For example, we expect these conditions to be met when replacing the torus by a $d$-dimensional sphere equipped with $L_2$ norm, and thus believe that our method can be used to recover the known results for spherical RGGs.

Moreover, we expect our techniques to give meaningful results even for non-homogeneous spaces like a hypercube, Gaussian space or when using non-uniform distributions on the torus or the sphere. In such cases, as mentioned earlier, convergence to Erdős--Rényi in the high-dimensional limit is not always given, so we cannot always expect analogous results as presented here for the torus. This is mathematically reflected in the fact that certain mixed cumulants of low order are no longer zero and will therefore introduce additional error terms. However, we expect that suitable modifications of our expansions that explicitly take such additional terms into account can be used to study the properties of RGGs on such spaces as well. In particular, we would expect that we can recover known results for Gaussian RGGs since these only ``mildly'' violate homogeneity, as was already noted in \cite{Bangachev_Bresler_2024_fourier}.

% \newpage

% \tableofcontents

\section{Preliminaries}\label{sec:prelims}
We start with some preliminaries on characteristic functions, cumulants and Taylor expansions. 

\subsection{Characteristic Functions, Cumulants, Vector Calculus}

For two distributions $\nu, \mu$, we denote their total variation distance by $\TV{\nu}{\mu}$. We write vectors as bold, lower-case letters. For a column vector $\mathbf{z} \in \mathbb{R}^k$, we denote by $\mathbf{z}(i)$ its $i$-th entry, and by $\mathbf{z}^\top$ its transpose, so that $\mathbf{x}^\top\mathbf{y}$ denotes the inner product. We denote the set of $k$--tuples by $\mathbb{N}^{\times k}$ and -- given a $k$--tuple $s = (s_1, s_2,\ldots,s_k)$ -- we define $|s| = s_1 +s_2+\ldots+s_k$ and $s! = s_1!s_2!\cdots s_k!$.  
Using these conventions, we define the characteristic function (CF) of a random vector $\mathbf{z}\in \mathbb{R}^k$ as \begin{align*}
    C_{\mathbf{z}}(\mathbf{t}) \coloneqq \Expected{ \exp\left(i(\mathbf{t}^\top \mathbf{z})\right) } \text{ for any } \mathbf{t} \in \mathbb{R}^k
\end{align*} where $i = \sqrt{-1}$. Taking the logarithm, we obtain the cumulant generating function (CGF, also known as the characteristic exponent) that can be expanded as a power series to yield the expression  \begin{align}\label{eq:cumulants}
    K_{\mathbf{z}}((t_1, \ldots, t_k)^\top) = \log\left( C_{\mathbf{z}}((t_1, \ldots, t_k)^\top) \right) = \sum_{\substack{s = (s_1, \ldots, s_k) \in \mathbb{N}^{\times k}}} \kappa_{s}(\mathbf{z}) \frac{(it_1)^{s_1}(it_2)^{s_2}\cdots (it_k)^{s_k}}{s_1!s_2!\cdots s_k!}.
\end{align} where the sum runs over all $k$--tuples and the  $\kappa_{s}(\mathbf{z})$ are the coefficients of the power series expansion of $K_{\mathbf{z}}$, which we call \emph{cumulants}. The cumulants $\kappa_{s}(\mathbf{z})$ for which $s$ contains only exactly one non-zero entry are called \emph{pure}, all other cumulants are \emph{mixed}. The \emph{order} of $\kappa_s(\mathbf{z})$ is $|s|$. Note further that we can express cumulants as an $s$-th order partial derivative of $K_{\mathbf{z}}$ as \begin{align*}
    \kappa_{s}(\mathbf{z}) &= (-i)^{|s|} \left. \frac{\partial^{|s|}}{\partial t_1^{s_1} \partial t_2^{s_2}\cdots\partial t_k^{s_k}} \log\left( C_{\mathbf{z}}((t_1, \ldots, t_k)^\top) \right) \right|_{t_1, t_2, \ldots, t_k = 0}.
\end{align*} For any collection of random variables $X_1, X_2, \ldots, X_r$, it is convenient to define the \emph{ordinary joint cumulant} \begin{align*}
    \kappa(X_1, X_2, \ldots, X_r) = (-i)^{r} \left. \frac{\partial^{r}}{\partial t_1 \partial t_2\cdots\partial t_k} \log\left( \Expected{\exp\left( i\sum_{j=1}^rt_jX_j\right)} \right) \right|_{t_1, t_2, \ldots, t_k = 0}
\end{align*} which then allows us to equivalently define \begin{align*}
    \kappa_{s}(\mathbf{z}) = \kappa(\underbrace{\mathbf{z}(1), \ldots, \mathbf{z}(1)}_{s_1 \text{ times}}, \underbrace{\mathbf{z}(2), \ldots, \mathbf{z}(2)}_{s_2 \text{ times}}, \ldots, \underbrace{\mathbf{z}(k), \ldots, \mathbf{z}(k)}_{s_k \text{ times}}).
\end{align*} That is, we can define the joint cumulants of $\mathbf{z}$ by repeating the $i$-th entry of $\mathbf{z}$ for $s_i$ times as an argument of $\kappa$. In particular, this definition allows us to relate cumulants to mixed moments by invoking a Möbius inversion over the partition lattice \begin{align}\label{eq:cumulantmoebiusinversion}
    \kappa(X_1, X_2, \ldots, X_k) = \sum_\pi (|\pi|-1)!(-1)^{|\pi|-1}\prod_{B \in \pi}\Expected{\prod_{j\in B} X_j}
\end{align} where the sum goes over all partitions $\pi$ of $[k] = \{1, \ldots, k\}$ and $|\pi|$ denotes the number of blocks of said partition. We refer the interested reader to \cite[Chapter 3]{Peccati_Taqqu_2011} for more details and a rigorous proof of the above. Cumulants further have many nice properties, the most important of which we capture in the following propositions.
\begin{proposition}[Additivity of cumulants]\label{prop:additivity}
    Consider a random vector $\mathbf{z} = \sum_{i=1}^d \mathbf{z}_i$ where the $\mathbf{z}_i$ are i.i.d. Then for all $s \in \mathbb{N}^{\times k}$,\begin{align*}
        \kappa_{s}(\mathbf{z}) = \sum_{i=1}^d \kappa_{s}(\mathbf{z}_i).
    \end{align*} 
\end{proposition}
\begin{proposition}[Homogeneity of cumulants]\label{prop:homogeneity}
    Consider a random vector $\mathbf{z}$. Then for all $s \in \mathbb{N}^{\times k}$ and all $c\in \mathbb{R}$ \begin{align*}
        \kappa_{s}(c\cdot\mathbf{z}) = c^{|s|} \kappa_{s}(\mathbf{z}).
    \end{align*} 
\end{proposition}
\noindent Both properties are standard and can easily be derived from the definition of the CGF above. However, working with cumulants in this form can be tedious. To remedy this we also use a \say{vectorized} version of (\ref{eq:cumulants}) which relies on the Kronecker product $\mathbf{x} \otimes \mathbf{y}$ between two vectors $\mathbf{x}, \mathbf{y} \in \mathbb{R}^k$, defined as the $k^2$ dimensional vector \begin{align*}
    \mathbf{x} \otimes \mathbf{y} = \begin{pmatrix}
        \mathbf{x}(1) \mathbf{y} \\
        \mathbf{x}(2) \mathbf{y} \\
        \cdots\\
        \mathbf{x}(k) \mathbf{y}
    \end{pmatrix} \in \mathbb{R}^{k^2}.
\end{align*} We further denote by $\mathbf{x}^{\otimes j}$ the $j$-times repeated Kronecker product with itself, i.e., 
\begin{align*}
    \mathbf{x}^{\otimes j} = \underbrace{\mathbf{x} \otimes \mathbf{x} \otimes \ldots \otimes \mathbf{x}}_{j \text{ times}}.
\end{align*} 
Then we can put all cumulants of order $j$ into a vector $\boldsymbol{\kappa}_j \in \mathbb{R}^{k^j}$ and express $K_{\mathbf{z}}$ as\footnote{When putting cumulants of order $j$ into a vector, we further have to multiply them by a suitable constant to compensate the quotient of $j!$ and $s_1!s_2!\cdots s_k!$ for $s = (s_1, \ldots, s_k)$ with $|s| = j$.}
\begin{align*}
    K_{\mathbf{z}}(\mathbf{t}) = \sum_{j=1}^\infty \frac{1}{j!} \boldsymbol{\kappa}_j^\top(i\mathbf{t})^{\otimes j}.
\end{align*} Occasionally, we shall refer to $K_{\mathbf{z}}(\mathbf{t})$ as $K(\mathbf{t}, \mathbf{z})$ instead, for the sake of easier readability. These two notations refer to the same function. We further use the following version of Taylor's theorem. 

\begin{theorem}[Proposition 6 in \cite{Kundhi_Rilstone_2020}]\label{thm:taylor}
    For $f : \mathbb{R}^k \rightarrow \mathbb{C}$, $J$ times continuously differentiable, we have \begin{align*}
        f(\mathbf{x}) = r_J(\mathbf{x}) + \sum_{j = 0}^J \frac{1}{j!} f^{(j)}(\mathbf{0})^\top\mathbf{x}^{\otimes j}
    \end{align*}
    where 
    $
        r_J(\mathbf{x}) = \frac{1}{J!} \left(f^{(J)}(c\mathbf{x}) - f^{(J)}(\mathbf{0})\right)^\top\mathbf{x}^{\otimes J}
    $ 
    and $c$ is some value in $[0,1]$ that can depend on $\mathbf{x}$. Here, $f^{(j)}$ denotes the \say{$j$-th derivative of $f$}, i.e., the vector containing all order $j$ partial derivatives of $f$, or alternatively $
        f^{(j)} = \boldsymbol{\nabla}^{\otimes j} f(x).
    $
\end{theorem}
\noindent It is also convenient to note the following proposition about the Kronecker product.
\begin{observation}\label{prop:kroneckernorm}
    For every $\mathbf{t} \in \mathbb{R}^k$, we have $\|(i\mathbf{t})^{\otimes j} \| \le (\mathbf{t}^\top\mathbf{t})^{\frac{j}{2}}$.
\end{observation}
\noindent Another important theorem we rely on is the inversion theorem, which states that given the CF of a random vector with continuous density we can reverse the Fourier transform to recover its density.
\begin{theorem}[Inversion Theorem, \cite{Cramer_1999}]\label{thm:inversion}
Let $\mathbf{z} \in \mathbb{R}^k$ be a random vector with integrable characteristic function, i.e., assume that \begin{align*}
    \int_{\mathbb{R}^k}|C_{\mathbf{z}}|(\mathbf{t}) \dd \mathbf{t} < \infty.
\end{align*} Then the density $f$ of $\mathbf{z}$ exists and
\begin{align}\label{eq:inversion}
    f(\mathbf{x}) = \frac{1}{(2\pi)^k} \int_{\mathbb{R}^k} e^{-i \mathbf{x}^\top \mathbf{t}} C_{\mathbf{z}}(\mathbf{t}) \dd \mathbf{t} = \frac{1}{(2\pi)^k} \int_{\mathbb{R}^k} e^{-i \mathbf{x}^\top \mathbf{t}} \exp\left( K(\mathbf{t})\right) \dd \mathbf{t}.
\end{align}    
\end{theorem}
\noindent Furthermore, we define \cramer, which intuitively requires that the absolute value of the CF of a random vector is bounded away from $1$ for all $\mathbf{t}$ with $\|\mathbf{t}\| \ge \delta$, for all $\delta > 0$. This condition is required for our error estimation as it ensures that the CF of our normalized sum is exponentially small for large enough $\|\mathbf{t}\|$. 
\cramercond
% \begin{definition}[\cramer]\label{def:cramer}
%     We say that a random vector $\mathbf{z} \in \mathbb{R}^k$ satisfies Cramér's condition if for all $\delta > 0$ and all $\mathbf{t} \in \mathbb{R}^k$ with $\|\mathbf{t}\| \ge \delta$, it holds that $
%         |C_{\mathbf{z}}(\mathbf{t})| \le 1 - \varepsilon(\delta),
%     $ where $\varepsilon(\delta) > 0$ for all $\delta > 0$.
% \end{definition}

\noindent Given a random graph $\Gbf$ and an edge $e = \{u, v\}$, we write $\mathds{1}(e)$ or alternatively $\Gbf_{uv}$ for the indicator that this edge is present.
Formally, this refers to the random variable $\mathds{1}(e) \coloneqq \mathds{1}_{\Delta(e) \le \tau}$, where $\tau$ is the connection threshold associated with the desired marginal edge probability $p$.
Moreover, given a arbitrary set of edges $H$, we define the signed weight of $H$ as the random variable
\begin{align*}
    \textsc{Sw}(H) \coloneqq    \prod_{e \in H} (\mathds{1}(e) - p)  .
\end{align*}

\subsection{RGGs and the Concrete Random Vectors of Interest}\label{sec:concrete}

Let $\ind{\mathcal E}$ denote the indicator random variable of the event $\mathcal E$. 
% We define random geometric graphs on ($\bbT^d$, $L_q$), which is the graph distribution we are interested in.
% \begin{definition}[RGGs]\label{def:TRGG}
   A \emph{random geometric graph} (RGG) is the following distribution over graphs with vertex set $[n]$, where $n\in\Z_{\geq0}$. Let $q\in\Z_{\geq0}\cup\{\infty\}$, let $d\in\Z_{\geq0}$, let $p\in[0,1]$ and let $\unif$ denote the uniform measure over $\bbT^d=[0,1]^d$. For $q<\infty$, let
\begin{align*}
    \Delta(\{u,v\}) \coloneqq (\|\mathbf{x}_u - \mathbf{x}_v \|_{q})^q = \sum_{i =1}^d (|\mathbf{x}_u(i) - \mathbf{x}_v(i)|_C )^q,
\end{align*} where $
   |x - y|_C \coloneqq \min\{ |x - y|, 1 - |x - y| \}
$ is the distance on the unit circle, while for $q=\infty$ let
\begin{align*}
    \Delta(\{u,v\}) \coloneqq (\|\mathbf{x}_u - \mathbf{x}_v \|_{\infty}) = \max_{i \in[d]} \big\{|\mathbf{x}_u(i) - \mathbf{x}_v(i)|_C \big\}.
\end{align*}
   
\noindent For each $v\in[n]$, let $\vecx_v\sim\unif(\bbT)$. Then for any graph $\mathbf A$ with vertices in $[n]$ $\RGG$, is the following probability measure.
\[
\Prd{\Gbf\sim\RGG}{\Gbf=\mathbf{A}}=\Expectedsub{\substack{\vecx_v\sim\unif(\bbT),\\ v\in[n]}}{\prod_{1\leq u<v\leq n}\left(\ind{ \Delta(\{u,v\})\leq \thresh}\right)^{A_{u,v}}\left(\ind{ \Delta(\{u,v\})>\thresh}\right)^{1-A_{u,v}}},
\]
where $\tau\in\R$ is (uniquely) chosen such that 
$
\Pr{\ind{ \Delta(\{u,v\})\leq \thresh}}=p.
$
% \end{definition}

We now introduce the concrete random vectors studied in this work, which represent a cycle of length $k$ in a toroidal RGG.% as defined in \cref{def:TRGG}.
% We (mostly) consider the model for some $L_q$ norm with $q < \infty$. We denote the probability that there is an edge between two fixed vertices $u,v$ by $p$ and denote the associated connection threshold by $\tau$; thus $u$ and $v$ are connected by an edge whenever the $L_q$ distance between them raised to the power of $q$ is at most $\tau$, which can be expressed as 
% \begin{align*}
%     \Delta(\{u,v\}) \coloneqq (\|\mathbf{x}_u - \mathbf{x}_v \|_{q})^q = \sum_{i =1}^d (|\mathbf{x}_u(i) - \mathbf{x}_v(i)|_C )^q\le \tau
% \end{align*} where $
%     |x - y|_C \coloneqq \min\{ |x - y|, 1 - |x - y| \}
% $ is the distance on the unit circle.
Considering $k$ vertices $v_1, v_2, \ldots, v_k$ we are interested in the vector $\mathbf{\Delta} \in \mathbb{R}^k$ containing the pairwise $L_q$ distances (raised to the power of $q$) corresponding to all $k$ edges of our cycle, i.e., \begin{align*}
    \Deltabold{ } = \sum_{i=1}^d \Deltabold{i}, \qquad \Deltabold{i} = \begin{pmatrix}\Delta_i(\{v_1, v_2\})\\\Delta_i(\{v_2, v_3\})\\ ... \\\Delta_i(\{v_k, v_1\})\end{pmatrix}, \qquad \Delta_i(\{u, v\}) = |\mathbf{x}_u(i) - \mathbf{x}_v(i)|_C^q.
\end{align*}
%That is, $\mathbf{\Delta}$ is a $k$-dimensional vector where each entry corresponds to the $L_q$ distance (raised to the power of $q$) of one edge $\{v_j, v_{j+1}\}$ in the cycle. 
$\Delta(\{u,v\})$ is a sum of independent random variables, one for each of the $d$ dimensions of the torus. In dimension $i$, $\Delta_i(\{u,v\})$ is the contribution of said dimension to the $L_q$ distance (raised to the power of $q$); the contributions of all edges of the cycle in dimension $i$ are captured in $\mathbf{\Delta}_i$. We are now interested in making statements about the distribution of $\mathbf{\Delta}$. To this end, it is convenient to normalize the random vector to ensure mean $0$ and variance $1$ of all components. Thus we consider
\begin{align*}
    \Deltaboldhat{} \coloneqq \sum_{i=1}^d \frac{\Deltabold{i} - \mu}{\sqrt{d}\sigma} = \sum_{i=1}^d \frac{1}{\sqrt{d}} \Deltaboldhat{i} \text{ where } \Deltaboldhat{i} = \frac{\Deltabold{i} - \mu}{\sigma}
\end{align*} where $\mu, \sigma^2$ are mean and variance of a single component of $\Deltabold{i}$.
We also let $\hattau$ be the rescaled connection threshold $\hattau$ defined such that the $i$-th edge in the cycle exists if and only if $\Deltaboldhat{}(i) \le \hattau$.

\subsubsection{Adding Gaussian Noise}\label{sec:noise}

When approximating the density of $\Deltaboldhat{}$ using an Edgeworth-style expansion, we encounter the problem that each $\Deltabold{i}$ individually has no density. Although the sum of multiple $\Deltabold{i}$ does have a density their individual CFs are not guaranteed to be integrable, which means that \Cref{thm:inversion} is potentially not applicable. We resolve this issue by adding Gaussian noise with very small variance polynomial in $d$ (say $d^{-2\eta}$ where $\eta$ is a parameter) to our $\Deltabold{i}$. Although this changes the distribution to be approximated, the fluctuations incurred by the noise are very small, so our approximation still allows us to make sufficiently strong statements about the distribution of the underlying \say{signal} $\Deltaboldhat{}$. This statement is made formal in \Cref{lem:influenceofnoise}. 

We denote the Gaussian noise we add to each $\Deltabold{i}$ by $\boldsymbol{\eta}_i \sim \mathcal{N}(0, d^{-2\eta}\mathbf{I}_k)$ where $\mathbf{I}_k$ is the $k \times k$ identity matrix and $\eta > 0$ is a parameter. We then work with \begin{align}\label{eq:z}
    \mathbf{z} \coloneqq \sum_{i=1}^d \frac{1}{\sqrt{d}} \underbrace{ \left(\frac{\Deltabold{i} + \boldsymbol{\eta}_i - \mu}{\zeta\sigma} \right)}_{\eqqcolon \mathbf{z}_i}\text{ with } \zeta \coloneqq \sqrt{1 + \sigma^{-2}d^{-2\eta}}.
\end{align} Note that $\mathbf{z}$ has mean zero and covariance $\mathbf{I}_k$ because the variance of a single component of $\Deltabold{i} + \boldsymbol{\eta}_i$ is $\zeta^2\sigma^2 = \sigma^2 + d^{-2\eta}$. Since we are typically interested in making statements about $\Deltaboldhat{}$ instead of $\mathbf{z}$, we make use of the following decomposition of $\mathbf{z}$ into \say{signal} $\Deltaboldhat{}$ and \say{noise} $\boldsymbol{\eta}$:
\begin{align}\label{eq:znoise}
        \mathbf{z} = \frac{\Deltaboldhat{}}\zeta + \boldsymbol{\eta}, \quad\text{where}\quad \Deltaboldhat{} \coloneqq \sum_{i=1}^d \frac{\Deltabold{i} - \mu}{\sqrt{d}\sigma} \text{ and }  \boldsymbol{\eta} \coloneqq \sum_{i=1}^d \frac{\boldsymbol{\eta}_i}{\zeta\sqrt{d}\sigma} \sim \mathcal{N}\left(\mathbf{0}, \frac{d^{-2\eta}}{\zeta^2\sigma^2} \mathbf{I}_k \right).
    \end{align}
The following lemma quantifies the influence of the Gaussian noise on the probability of the event that a subset of entries of $\mathbf{z}$ and $\Deltabold{}$ are all at most $x$. 

\begin{lemma}[Quantifying the influence of Gaussian noise]\label{lem:influenceofnoise}
    There is some absolute constant $C > 0$ such that for any $S\subseteq [k]$ and every $x \in \mathbb{R}$, every $\eta \ge 1$ and sufficiently large $d$, we have
    \begin{align*}
        \left| \Pr{\bigcap_{i \in S} (\Deltaboldhat{}(i) \le x) } - \Pr{\bigcap_{i \in S} ( \mathbf{z}(i) \le x) } \right| \le C(1 + |x|) d^{-\eta + 1}.
    \end{align*}
\end{lemma} \begin{proof}
    We denote by $\mathcal{E}$ the event that $\|\boldsymbol{\eta}\|_\infty > d^{-\eta+1}$ and note that $\Pr{\mathcal{E}} \le \exp(-d)$ by elementary calculations using the Gaussian density. Moreover, note that by a Taylor series expansion, we have that 
    \begin{align*}
        |\zeta^{-1} - 1| \le Cd^{-2\eta} \quad\text{and thus}\quad \zeta^{-1} = 1 \pm Cd^{-2\eta}
    \end{align*} for some constant $C>0$ and all sufficiently large $d$. This shows that $\boldsymbol{\eta}$ is very small w.h.p.~and that $\zeta$ is very close to $1$. It therefore follows that 
    \begin{align*}
        \Pr{\bigcap_{i \in S} (\Deltaboldhat{}(i) \le x) } &\le \Pr{\left(\bigcap_{i \in S} \left(\mathbf{z}(i) \le  x/\zeta + d^{-\eta + 1}\right) \right) \cap \overline{\mathcal{E}} } + \Pr{\mathcal{E}} \\
        &\le \Pr{\bigcap_{i \in S} \left(\mathbf{z}(i) \le  x/\zeta + d^{-\eta + 1}\right)  } + \Pr{\mathcal{E}} \\
        &\le \Pr{ \bigcap_{i \in S} \left( \mathbf{z}(i) \le x + (|x| + 1)d^{-\eta + 1} \right) } + \Pr{\mathcal{E}} \\
        &\le \Pr{\bigcap_{i \in S} (\mathbf{z}(i) \le x) } + \Pr{ \bigcup_{i \in S} \left( x \le \mathbf{z}(i) \le x + (|x| + 1)d^{-\eta + 1} \right) } + \exp(-d).
    \end{align*} To estimate the second term above, we use that 
    \begin{align*}
        &\Pr{ \bigcup_{i \in S} \left( x \le \mathbf{z}(i) \le x + (|x| + 1)d^{-\eta + 1} \right) } \\
        &\hspace{1cm}\le \sum_{i=1}^k \Pr{x \le \mathbf{z}(i) \le x + (|x| + 1)d^{-\eta + 1}}
        \le C' k (|x| + 1)d^{-\eta + 1},
    \end{align*} since the density of each $\mathbf{z}(i)$ is bounded by some absolute constant $C'$ by \Cref{lem:locallimit} because each $\mathbf{z}(i)$ is the sum of independent i.i.d.~continuous random variables of bounded density. Similarly, we have that
    \begin{align*}
        \Pr{ \bigcap_{i \in S}  \left(\Deltaboldhat{}(i) \le x\right) } &\ge \Pr{ \left( \bigcap_{i \in S}  \left(\mathbf{z}(i) \le x/\zeta - d^{-\eta + 1} \right) \right) \cap \overline{\mathcal{E}} }\\
        &\ge \Pr{ \bigcap_{i \in S} \left(\mathbf{z}(i) \le x - (|x| + 1)d^{-\eta + 1}\right) } - \Pr{\mathcal{E}}\\
        &\ge \Pr{\bigcap_{i \in S} ( \mathbf{z}(i) \le x ) } - \underbrace{\Pr{ \bigcup_{i \in S} \left( x - (|x| + 1 )d^{-\eta + 1} \le \mathbf{z}(i) \le x\right) }}_{= \bigO{(|x|+1)d^{-\eta + 1}}} - \Pr{\mathcal{E}},
    \end{align*} 
    which finishes the proof.
\end{proof}

\noindent We further use the following tail bound on the individual components of $\Deltaboldhat{} =  \sum_{i=1}^d \frac{\Deltabold{i} - \mu}{\sqrt{d}\sigma}$ and $\mathbf{z} = \sum_{i=1}^d \frac{\Deltabold{i} + \boldsymbol{\eta}_i - \mu}{\sqrt{d}\zeta\sigma} $.

\begin{restatable}{lemma}{tailofz}
\label{lem:tailofz}
    For $d = \omega(1)$, every $c > 0$ and every $j \in [k]$,
    \begin{align*}
        &\Pr{\left| \sum_{i=1}^d \frac{\Deltabold{i}(j) - \mu}{\sqrt{d}\sigma} \right| \ge c\log(n)} \le  n^{-\omega(1)} \text{ and } \Pr{\left| \sum_{i=1}^d \frac{\Deltabold{i}(j) + \boldsymbol{\eta}_i(j) - \mu}{\sqrt{d}\zeta\sigma} \right| \ge c\log(n)} \le n^{-\omega(1)}.
    \end{align*} 
    Thus $\Pr{ \max\{ |\Deltaboldhat{}(j)|, |\mathbf{z}(j)| \} \ge c \log(n) } \le n^{-\omega(1)}$ for all $c > 0$.
\end{restatable}

\subsubsection{Properties of Cumulants for Cycles and Chains}\label{sec:cumulatnsofz}

We denote the cumulants of $\mathbf{z}$ by $\kappa_{s_1, s_2, \ldots, s_k}$, and the cumulants of an individual $\mathbf{z}_i$ by $\rho_{s_1, s_2, \ldots, s_k}$. By definition of the CGF we have the relationship \begin{align}\label{eq:cumulantrelationship}
    K(\mathbf{t}) &= d K_{\mathbf{z}_i}\left( \frac{\mathbf{t}}{\sqrt{d}} \right), \quad\text{and thus}\quad \kappa_{s} = d \frac{\rho_{s}}{\sqrt{d}^{|s|}} = d^{-\frac{|s| - 2}{2}} \rho_{s}.
\end{align}
Cumulants of $\mathbf{z}$ (if $\mathbf{z}$ represents the distances within a cycle or a chain of length $k$) have some particularly convenient properties, based on the observation below that whenever a mixed moment involving entries of $\mathbf{z}$ excludes at least one entry of $\mathbf{z}$, the entries involved the considered moment are uncorrelated. This can be seen by noting that after removing at least one edge from a cycle, what remains is a forest within which the distances associated to the remaining edges are independent.
\begin{observation}\label{obs:mixedmoments}
    For any $\alpha = (\alpha_1, \alpha_2,\ldots, \alpha_k) \in \mathbb{N}^k$ where at least one $\alpha_j = 0$, we have \begin{align*}
        \Expected{\prod_{j=1}^k \mathbf{z}(j)^{\alpha_j}} = \prod_{j=1}^k\Expected{ \mathbf{z}(j)^{\alpha_j}}.
    \end{align*} An analogous statement also holds for each individual $\mathbf{z}_i$. 
\end{observation}
\noindent An important consequence of this is the following. 
\begin{lemma}\label{lem:lowerordercumulantsarezero}
For any $s \in \mathbb{N}^{\times k}$ with $s= (s_1,s_2,\ldots,s_k)$ it holds that $\kappa_{s} = 0$ whenever at least one $s_j = 0$ and at least two $s_j$ are non-zero.
\end{lemma} 
\begin{proof}
    Denote by $\overline{\kappa}_{s}$ the respective cumulant of order $s$ associated to the random vector $\overline{\mathbf{z}}$ in which all entries are independent and distributed according to their marginal distribution in $\mathbf{z}$. It is immediate that 
    \begin{align*}
    K_{\mathbf{\overline{\mathbf{z}}}}(\mathbf{t}) = \log\left( \Expected{i\sum_{\ell=1}^k\mathbf{t}(\ell)\mathbf{\overline{\mathbf{z}}}(\ell)} \right) = \log\left(\prod_{\ell=1}^k \Expected{i\mathbf{t}(\ell)\mathbf{\overline{\mathbf{z}}}(\ell)} \right) = \sum_{\ell=1}^k\log\left( \Expected{i\mathbf{t}(\ell)\mathbf{\overline{\mathbf{z}}}(\ell)} \right) = \sum_{i=1}^k K_{\overline{\mathbf{z}}(j)}(\mathbf{t}(j))
    \end{align*} since the entries of $\overline{\mathbf{z}}$ are independent. Thus all mixed cumulants of $\overline{\mathbf{z}}$ are zero. Now, notice that when we expand $\kappa_{d}$ and $\overline{\kappa}_{s}$ according to \eqref{eq:cumulantmoebiusinversion}, if at least one $s_j = 0$ then all the mixed moments that appear factorize due to \Cref{obs:mixedmoments}. We therefore obtain identical expansions for both $\kappa_{s}$ and $\overline{\kappa}_{s}$, implying that $\kappa_{s}= \overline{\kappa}_{s} = 0$ as desired.
\end{proof}
\cref{lem:lowerordercumulantsarezero} implies in particular that the first (i.e. lowest order) mixed non-zero cumulant of $\mathbf{z}$ is $\kappa_{(1,1,\ldots, 1)}$, and that this is the only non-zero cumulant of order $k$. Moreover, this cumulant has a particularly simple expression in terms of its moments, presented in the following lemma.
Recall that $\kappa_{(1,1,\ldots,1)}$ is a joint cumulant of $\mathbf{z} = \sum_{i=1}^d \mathbf{z}_i$ where the $\mathbf{z}_i$ are i.i.d., and that $\rho_{(1,1,\ldots,1)}$ is the corresponding cumulant of a single $\mathbf{z}_i$.
\begin{lemma}\label{lem:cumulantasnicemixedmoment}
     It holds that $\kappa_{(1,1,\ldots,1)}$ is equal to the mixed moment of the components of $\mathbf{z}$, i.e.,
    \begin{align*}
        \kappa_{(1,1,\ldots,1)} = \Expected{\prod_{j=1}^k \mathbf{z}(j) }  \text{ and } \rho_{(1,1,\ldots,1)} = \Expected{\prod_{j=1}^k \mathbf{z}_1(j) }
    \end{align*}
\end{lemma}\begin{proof}
    The statement follows directly from the fact that 
    \begin{align*}
        \kappa_{(1,1,\ldots,1)} = \kappa(\mathbf{z}(1), \mathbf{z}(2), \ldots, \mathbf{z}(k)) = \sum_\pi (|\pi|-1)!(-1)^{|\pi|-1}\prod_{B \in \pi}\Expected{\prod_{j\in B} \mathbf{z}(j)}
    \end{align*} 
    by (\ref{eq:cumulantmoebiusinversion}). Note that all the mixed moments above that correspond to a partition $\pi$ of $[k]$ that has more than one block are equal to zero, since each of them is the product of at most $k-1$ entries of $\mathbf{z}$ and these are independent. Hence, the only term surviving in the above sum is the one corresponding to the trivial partition with only one block, i.e., the only partition such that $|\pi|  = 1$.
\end{proof}

\subsection{Auxiliary Statements}

We require the following standard local limit theorem, which is used to assert that normalized sums of continuous random variables have bounded density. 
\begin{theorem}[Local limit theorem, Theorem 19.1 (iii) $\Rightarrow$ (i) in \cite{Bhattacharya_Rao_2010}]\label{lem:locallimit}
    Let $X_1, \ldots, X_n$ be a sequence of i.i.d. continuous random variables of almost everywhere bounded density and variance $\sigma^2$. Define $S_n = \frac{1}{\sqrt{n}}\sum_{i=1}^n X_i$, denote the density function of $S_n$ by $f$ and denote the density of a $\mathcal{N}(0, \sigma)$ random variable by $\phi_{\sigma}$. Then \begin{align*}
        \lim_{n \rightarrow \infty} \sup_{x \in \mathbb{R}} |f(x) - \phi_\sigma(x)| = 0.
    \end{align*}
\end{theorem}
\noindent Furthermore, we use the following standard result about univariate Edgeworth expansions.
%\begin{theorem}[\emph{Univariate Edgeworth--Expansion for Distributions}]
%    For a sequence of zero-mean i.i.d. continuous random variables with finite third moment $X_1, \ldots X_d$, with variance $\sigma^2$, we have that \begin{align*}
%%        \left| \Pr{ a \le \sum_{i=1}^d\frac{X_i}{\sqrt{d}\sigma} \le b } - \int_a^b \left( \phi(x) + \frac{\kappa_3}{6\sqrt{d}}\phi^{(3)}(x) \right) \dd x \right| =  o\left( \frac{1}{\sqrt{d}} \right)
%    \end{align*} where $\phi$ is the density of the standard Gaussian distribution.
%\end{theorem} \begin{proof}
%    This follows by Theorem 1 in Chapter XVI, Section 4 in \cite{Feller_1991}. The statement there holds for all non-lattice distributions which is given here since the $X_i$ are continuous.
%\end{proof}

\begin{theorem}[Univariate Edgeworth expansion for densities]\label{thm:univariateedgeworthdensity}
    For a sequence $X_1, \ldots X_d$ of zero-mean i.i.d. continuous random variables with bounded density, finite $s+2$-th moment, and variance $\sigma^2$, define $X \coloneqq \sum_{j=1}^d \frac{X_j}{\sigma\sqrt{d}}$. Then the density $f$ of $X$ exists and satisfies \begin{align*}
        \left| f(x) - \left( \phi(x) + \phi(x)\sum_{j=1}^{s-2} \frac{p_j(x)}{d^{j/2}} \right) \right| =  o\left( \frac{1}{d^{(s-2)/2}} \right),
    \end{align*} 
    where $\phi$ is the density of the standard Gaussian distribution and the $p_1, \ldots, p_s$ are explicit polynomials where $p_j$ is of degree $3j$.
\end{theorem} \begin{proof}
    The statement follows from \cite[Theorem 2, Chapter XVI.2]{Feller_1991}. The statement there assumes additionally that the $|C_X|^v$ is integrable for some $v \ge 1$ where $|C_X|^v$. This is given here, since the density $f_i$ of each $X_i$ is bounded, so $f_i \in : L^1(\mathbb{R}^k) \cap L^2(\mathbb{R}^k)$, implying that $C_{X_i} \in L^2(\mathbb{R}^k)$ by \cite[Theorem 4.1 (iv)]{Bhattacharya_Rao_2010}
\end{proof}

\noindent We further use the following standard version of Bernstein's inequality. 
\begin{theorem}[Bernstein's inequality]\label{thm:bernstein}
    Let $X_1, \ldots, X_d$ be independent, zero-mean random variables and suppose that $|X_i| \le M $ almost surely. Then for all $t \ge 0$, \begin{align*}
        \Pr{ \sum_{i=1}^d X_i \ge t } \le \exp\left( \frac{t^2}{2\sum_{i=1}^d \Expected{X_i^2} + \frac{2}{3}Mt} \right).
    \end{align*}
\end{theorem}

\subsection{Probabilistic Estimates for the Connection Threshold}\label{sec:probestimates}

We need the following estimates for comparing the Gaussian PDF and CDF, and for giving sufficiently precise estimates on the rescaled connection threshold $\tilde{\tau}$. 

\begin{restatable}[The Gaussian CDF and PDF]{proposition}{cdfpdf}
\label{prop:cdfpdf}
    Let $\Phi$ denote the standard Gaussian CDF and let $\phi$ be the standard Gaussian PDF. Then, for every $x \le -1$, we have \begin{align*}
       \frac{1 - e^{-3/2}}{2|x|} \phi(x) \le \Phi(x) \le \phi(x).
    \end{align*} Moreover, for every $C > 0$, there are constants $C_1, C_2 > 0$ such that for all $x \le C$, \begin{align*}
        C_1\phi(x) \le \Phi(x) \le C_2\phi(x).
    \end{align*}
\end{restatable}\begin{proof}
    We have 
    \begin{align*}
        \Phi(x) = \frac{1}{\sqrt{2\pi}} \int_{-\infty}^x e^{-t^2/2} \text{d} t \le \frac{1}{\sqrt{2\pi}} \int_{-\infty}^x -te^{-t^2/2} \text{d} t = \frac{1}{\sqrt{2\pi}}e^{-x^2/2} = \phi(x)
    \end{align*} 
    where the inequality holds since $x \le -1$. Similarly, 
    \begin{align*}
        \Phi(x) = \frac{1}{\sqrt{2\pi}} \int_{-\infty}^x e^{-t^2/2} \text{d} t &\ge \frac{1}{\sqrt{2\pi}} \int_{2x}^x e^{-t^2/2} \text{d} t \\
        &\ge \frac{1}{\sqrt{2\pi}2|x|} \int_{2x}^x -t e^{-t^2/2} \text{d} t \\
        &= \frac{1}{2|x|} \left(\phi(x) - \phi(2x) \right) = \frac{1}{2|x|\sqrt{2\pi}} e^{-x^2/2}(1 - e^{-3x^2/2}) \ge \frac{1 - e^{-3/2}}{2|x|} \phi(x)
    \end{align*} as desired. Since for every $x \in [-1, C]$, both $\phi(x)$ and $\Phi(x)$ are bounded from above and below by some constants $C_1, C_2$, the second part of the statement also follows.
\end{proof}

\noindent We use the above combined with a univariate Edgeworth expansion to show that for all densities $1/n^c \le p$, $\hattau$ is close to $\Phi^{-1}(p)$, which is the threshold we would expect if all rescaled distances were distributed according to a standard Gaussian.

\begin{restatable}[Estimating the rescaled connection threshold]{lemma}{thresholdestimate}\label{lem:thresholdestimate}
    There is a constant $C > 0$ such that for every $p$ such that $\frac{1}{n^c} \le p \le 1-\varepsilon$ where $c, \varepsilon > 0$ are arbitrary constants, we have that \begin{align*}
        \left| \hattau - \Phi^{-1}(p) \right| \le \frac{Cp\log^{4}(n)}{\sqrt{d}}
    \end{align*}
\end{restatable}
\begin{proof}
 The proof is deferred to \Cref{sec:deferredprelims}.    
\end{proof}

\noindent As a corollary, it holds that whenever $1/p$ is bounded by some polynomial, $\hattau$ is small. This allows us to focus our attention to a region close to the origin when integrating.
\begin{restatable}{corollary}{tauissmall}
    \label{cor:tauissmall}
    For any $p$ such that $\frac{1}{n^c} \le p \le 1 - \varepsilon$ where $c, \varepsilon > 0$ are arbitrary constants, we  have that \begin{align*}
        |\Phi^{-1}(p)|, |\hattau| \le \begin{cases}
            \Theta(1) &\text{if } p = \Omega(1), \\
            \log(n)  &\text{if } p = o(1).
        \end{cases}
    \end{align*}
\end{restatable}

\noindent As a further corollary of \Cref{lem:thresholdestimate} we see that the Gaussian PDF and CDF evaluated at $\hattau$ are close to $p$, which is used later to show that the errors incurred by approximating the joint density of $\mathbf{z}$ can be phrased relative to $p$.

\begin{restatable}{corollary}{phip}
    \label{cor:phip}
    For all $p$ such that $\frac{1}{n^c} \le p = o(1)$ where $c > 0$ is an arbitrary constant, we have that \begin{align*}
        (1-o(1))p \le \phi(\hattau) \le (1 + o(1))p 
        \quad\text{and}\quad
        \Phi(\hattau) \le \phi(\hattau) \le (1 + o(1))p.
    \end{align*} Furthermore, for any constant $\varepsilon > 0$ there are constants $C_1,C_2$ such that for all $\frac{1}{n^c} \le p \le 1 - \varepsilon$, \begin{align*}
        C_1p \le \phi(\hattau) \le C_2p \text{ and } C_1p \le \Phi(\hattau) \le C_2p.
    \end{align*}
\end{restatable}

%\noindent Furthermore, we can assert that $\tilde{\tau}$ is at most logarithmic in absolute value for all densities $p$ that are of interest to us.

\section{Approximating the Joint Density of z}\label{sec:approxdensity}

In this section we prove our main theorem, which is at the heart of all subsequent sections and later allows us to make statements about the number of signed cycles and chains, the signed weight of sparse graphs, spectral properties, and improved algorithmic lower bounds for the $L_q$ model. We start by formally introducing our result, and then present the proof.

\subsection{Our Result}
Our main technical result states that the density of $\mathbf{z}$ (as defined in \Cref{sec:noise}) representing a cycle or a chain is well approximated by its \say{ground state} $\find$, plus a series of correction terms, phrased in terms of derivatives of $\phi$. We recall that the \say{ground state} is defined as the density associated to the product measure $\otimes_{j=1}^k \mathcal{L}(\mathbf{z}(j))$ where $\mathcal{L}(\mathbf{z}(j))$ denotes the law of the $j$-th entry of $\mathbf{z}$. Alternatively, \begin{align*}
    \find(\mathbf{x}) \coloneqq \prod_{j = 1}^k f_{\mathbf{z}(j)}(\mathbf{x}(j))
\end{align*} where $f_{\mathbf{z}(j)}$ is the (marginal) density of $\mathbf{z}(j)$. We further define the quantity \begin{align*}
    \kappa \coloneqq \frac{1}{d} \sum_{j=1}^d \kappa_{(1,1,\ldots,1)}(\mathbf{z}_j)
\end{align*} as the average of all cumulants of order $(1,1,\ldots,1)$ over the individual $\mathbf{z}_j$. In case $\mathbf{z}$ represents a cycle $H$, all these cumulants are equal and $\kappa = \kappa_{(1,1,\ldots,1)}(\mathbf{z}_j)$ for any $j$. However, in case $H$ is a chain, the $\kappa_{(1,1,\ldots,1)}(\mathbf{z}_j)$ are not the same for all $j$ and $\kappa$ is a function of the position of the two endpoints associated to $H$.
Furthermore, in case $H$ is a chain with $k$ edges on the vertex set $v_1,\ldots, v_{k+1}$, it is important to note that all our randomness comes from the positions of $v_2,\ldots, v_k$ and is to be understood conditional on the positions of  the endpoints $v_1, v_{k+1}$, whose position determine $\kappa$. If $k = 2$, we further have the technical problem that not all $\mathbf{z}_j$ are guaranteed to satisfy \cramer. However, we can still prove that each $\mathbf{z}_j$ satisfies \cramer with a constant probability over the draw of positions of $v_1, v_{k+1}$. For our error analysis, this is in fact sufficient because it implies that a constant fraction of all dimensions satisfies \cramer with high probability. However, since it is not guaranteed to hold almost surely, we require the following definition specifying when a chain of length $k = 2$ is sufficiently \say{good} for our purposes.
\begin{definition}[Good Chains]\label{def:goodchains}
    A chain $H$ of length $k = 2$ is $(\alpha, \beta)$--\emph{good} if its endpoints $u,v$ are positioned in $\mathbb{T}^d$ such that following two conditions are met. \begin{enumerate}
        \item[(i)] At least $\alpha d$ of the $\mathbf{z}_j$ satisfy Cramèr's condition as stated in \Cref{def:cramer}. That is, there is a set $S \subseteq [d]$ with $|S| \ge \alpha d$ such that for all $\mathbf{t}$ with $\|\mathbf{t}\| \ge \delta$, there is a constant $\varepsilon(\delta) > 0$ such that for all $j \in S$, \begin{align*}
            |C_{\mathbf{z}_j}(\mathbf{t})| \le 1 - \varepsilon(\delta).
        \end{align*}
        \item[(ii)] We have that $|\kappa| = \frac{1}{d} \left|\sum_{j=1}^d \kappa_{(1,1,\ldots, 1)}(\mathbf{z}_j)\right| \le \beta$. 
    \end{enumerate}
\end{definition}
\noindent We prove the following lemmas asserting that a chain does typically have small $|\kappa|$ and is typically $(\alpha, \bigOtildenop{1/\sqrt{d}})$-good. 
\begin{lemma}[Chains have small $|\kappa|$]\label{lem:chainshavesmallkappa}
    Assume that $H$ is a chain with $k$ edges on the set of vertices $v_1, \ldots, v_{k+1}$ with positions $\vecx_1, \vecx_2 \ldots, \vecx_{k+1}$, where $\vecx_1, \vecx_{k+1}$ are the positions of the endpoints. Then, there is a constant $M > 0$ such that
    \begin{align*}
        \Prsub{ \mathbf{x}_1, \mathbf{x}_{k+1} }{ |\kappa_e| \ge \frac{t}{d}} \le 2\exp\left( \max\left\{ \frac{-t^2}{4M^2 d}, \frac{-t}{\frac{4}{3}M} \right\} \right).
    \end{align*}
\end{lemma} \begin{proof}
    Assume that $H$ is a chain with $k$ edges on the set of vertices $v_1, \ldots, v_{k+1}$ with positions $\vecx_1, \ldots, \vecx_{k+1}$. Using \Cref{lem:cumulantasnicemixedmoment}, we get that
    \begin{align*}
        \kappa = \frac{1}{d} \sum_{i = 1}^d \kappa_{(1,1,\ldots,1)}(\mathbf{z}_i) = \frac{1}{d} \sum_{i = 1}^d \Expectedsub{\vecx_1, \ldots, \vecx_{k+1}}{\mathbf{z}_i(1)\cdots\mathbf{z}_i(k)} = \frac{1}{\zeta^kd} 
        \sum_{i = 1}^d\Expectedsub{\vecx_1, \ldots, \vecx_{k+1}}{\Deltaboldhat{i}(1)\cdots\Deltaboldhat{i}(k)}
    \end{align*} where we further used the decomposition from \eqref{eq:znoise} in \Cref{sec:noise}, expanded the expectation and omitted all the resulting terms that turn out to be zero.
    Furthermore, it is easy to note that \begin{align*}
        d \Expectedsub{\mathbf{x}_1, \mathbf{x}_{k+1}}{ \kappa } = \sum_{i=1}^d \Expectedsub{\mathbf{x}_1, \mathbf{x}_{k+1}}{ \Expectedsub{\mathbf{x}_2, \ldots, \mathbf{x}_k}{\left. \prod_{j=1}^{k} \Deltaboldhat{i}(j) \right| \mathbf{x}_1, \mathbf{x}_{k+1} } } = 0.
    \end{align*} Now, since all the $\Deltaboldhat{i}(j)$ are bounded almost surely in absolute value by some $M$, we apply \Cref{thm:bernstein} to obtain \begin{align*}
        \Prsub{ \mathbf{x}_1, \mathbf{x}_{k+1} }{d |\kappa| \ge t} \le 2\exp\left( \max\left\{ \frac{-t^2}{4M^2 d}, \frac{-t}{\frac{4}{3}M} \right\} \right),
    \end{align*} as desired.
\end{proof}\noindent The above shows that $|\kappa|$ is typically of order $1/\sqrt{d}$ and very unlikely to exceed this value by a lot.
\begin{lemma}[Chains are probably good]\label{lem:goodchains}
    There is a constant $\alpha > 0$ such that at least $\alpha d$ out of the $d$ terms $\mathbf{z}_i$ associated to a chain of length $k = 2$ satisfy \cramer with probability $1- \exp(-\Omega(d))$ over the draw of the endpoints of $H$. For chains and cycles of length $k \ge 3$, this holds for all $\mathbf{z}_i$ with probability $1$.
\end{lemma} \begin{proof}
    We get from \Cref{lem:cramer} (in \Cref{sec:deferrederror}) that for every dimension $j \in [d]$, there is an event $\mathcal{E}$ defined in terms of the positions of the endpoints of $H$ in dimension $j$ under which \cramer is met, and which occurs with probability at least $p_\mathcal{E} = \Omega(1)$. Denoting by $X$ the number of dimensions for which $X$ occurs, we get from a Bernstein bound (\Cref{thm:bernstein}) that \begin{align*}
        \Pr{|X - \Expected{X}| \ge t } \le 2\exp\left( -\min\left\{ \frac{t^2}{2d}, \frac{3t}{2} \right\} \right),
    \end{align*} so for $t = \frac{1}{2}\Expected{X} \ge \frac{p_{\mathcal{E}}}{2}d$, we get that $\Pr{X \le \Expected{X}/2} \le 2\exp( -dp_{\mathcal{E}}^2/8 ) = \exp(-\Omega(d))$, as desired. The statement for chains and cycles of length $k \ge 3$ also follows directly by \Cref{lem:cramer}.
\end{proof}

\noindent This leads us to the following theorem, which is our main technical result.
\maindensityapprox*

\noindent The rest of this section is devoted to proving the above. A rough outline is given at the beginning of \Cref{sec:maindensityapproxproof}.

\subsection{Auxiliary Statements}

For the proof of \Cref{thm:maindensityapprox}, we need the following lemmas concerning a certain type of integral that frequently appears in our error analysis.

\begin{lemma}\label{lem:integral} For every fixed $a \in \mathbb{R}$,
\begin{align*}
    \int_{\mathbb{R}^k} (\mathbf{t}^\top\mathbf{t})^{a}\exp\left( -\frac{\mathbf{t}^\top\mathbf{t}}{4} \right) \text{d} \mathbf{t} < \infty.
\end{align*}
\end{lemma}\begin{proof}
    Exploiting the spherical symmetry of the integrand, a change of variables yields 
    \begin{align*}
        \int_{\mathbb{R}^k} (\mathbf{t}^\top\mathbf{t})^{a} \exp\left(-\frac{\mathbf{t}^\top \mathbf{t}}{4}\right)\text{d} \mathbf{t} = \int_{0}^\infty r^{2a} \exp\left(-\frac{r^2}{4}\right) V^{(1)}(r) \text{d} r
    \end{align*} 
    where $V(r) = c(k) r^k$ is the volume of a $k$-dimensional sphere of radius $r$ ($c(k)$ is some  constant that depends on $k$ and is equal to the volume of a $k$-dimensional unit sphere). For every $R \ge 0$, we then get 
    \begin{align*}
        \int_{\mathbb{R}^k} (\mathbf{t}^\top\mathbf{t})^{a} \exp\left(-\frac{\mathbf{t}^\top \mathbf{t}}{4}\right)\text{d} \mathbf{t} 
        &\le R^{2a} V(R) + \int_{R}^\infty r^{2a} \exp\left(-\frac{r^2}{4}\right) V^{(1)}(r)\text{d} r\\
        &= R^{2a} V(R) + kc(k) \int_{R}^\infty r^{k-1+2a} \exp\left(-\frac{r^2}{4}\right)\text{d} r .
    \end{align*}
    In particular, for $R \coloneqq R(k, a)$ large enough, we have
    \begin{align*}
        R^{2a} V(R) + kc(k) \int_{R}^\infty r^{k-1+2a} \exp\left(-\frac{r^2}{4}\right)\text{d} r 
        &\le R^{2a} V(R) + kc(k) \int_{R}^\infty r \exp\left(-\frac{r^2}{8}\right)\text{d} r  \\
        &\le R^{2a} V(R) + 2kc(k) = \bigO{1}.
\end{align*}
\end{proof}

\noindent Occasionally, we further need to evaluate the above integral only over  a restricted domain in which case it is much smaller.
\begin{lemma}\label{lem:integrallimiteddomain}
    for any fixed $\delta > 0$ and $a \in \mathbb{R}^{+}$, define $\overline{\mathcal{B}(\delta)} = \{ \mathbf{t} \in \mathbb{R}^k \mid \|\mathbf{t}\| \ge \delta \sqrt{d} \}$. Then, \begin{align*}
        \int_{\overline{\mathcal{B}(\delta)}} (\mathbf{t}^\top\mathbf{t})^{a} \exp\left( -\frac{\mathbf{t}^\top\mathbf{t}}{2} \right) \text{d}\mathbf{t} = \exp\left(-\Omega(d) \right).
    \end{align*}
\end{lemma} \begin{proof}
Again, exploiting the spherical symmetry of the integrand to apply a change of variables and using the fact that $\|\mathbf{t}\| \ge \delta d$ for any $\mathbf{t} \in \overline{\mathcal{B}(\delta)}$ we obtain 
\begin{align*}
    \int_{\overline{\mathcal{B}(\delta)}} (\mathbf{t}^\top\mathbf{t})^{a} \exp\left(-\frac{\mathbf{t}^\top \mathbf{t}}{2}\right)\text{d} \mathbf{t} = \int_{\delta \sqrt{d}}^\infty r^{2a} \exp\left(-\frac{r^2}{2}\right) V^{(1)}(r) \text{d}
\end{align*} 
where we recall that $V(r) = c(k) r^k$ is the volume of a $k$-dimensional sphere of radius $r$ and $c(k)$ is some  constant that only depends on $k$ and is equal to the volume of the $k$-dimensional unit sphere. 
Since $V^{(1)}(r) = kc(k)r^{k-1} $, we get  \begin{align*}
    \int_{\overline{\mathcal{B}(\delta)}} (\mathbf{t}^\top\mathbf{t})^{a} \exp\left(-\frac{\mathbf{t}^\top \mathbf{t}}{2}\right)\text{d} \mathbf{t} &= kc(k) \int_{\delta \sqrt{d}}^\infty r^{2a + k-1} \exp\left(-\frac{r^2}{2}\right)\text{d} r\\
    &= kc(k) \int_{\delta \sqrt{d}}^\infty r \exp\left((2a + k-2)\ln(r)-\frac{r^2}{2}\right)\text{d} r\\
    &\le kc(k) \int_{\delta \sqrt{d}}^\infty r \exp\left(-\frac{r^2}{4}\right)\text{d} r\\
    &= 2kc(k) \exp\left(-\frac{\delta^2d}{4}\right) = \exp\left(-\Omega(d) \right)
\end{align*} where the penultimate step holds for sufficiently large $d$ depending on $a$ and $k$.
\end{proof} 

\subsection{Estimating the Error -- Proof of \Cref{thm:maindensityapprox}}\label{sec:maindensityapproxproof}

This section is dedicated solely to the proof of \Cref{thm:maindensityapprox}. Most of it (namely \Cref{sec:smallt} and \Cref{sec:larget}) is concerned with estimating the approximation error given by $\sup_{\vecx \in \mathbb{R}^k}|\tilde{f}(\vecx) - f(\vecx)|$. Afterwards, we prove our bound on the magnitude of the coefficients associated to the higher order correction terms (i.e. our bound on $\|\Aeltabold{k+j}\|_\infty$) in \Cref{sec:higherorderterms}.
The proof follows a similar high-level ideas as used in \cite{Feller_1991} and \cite{Kundhi_Rilstone_2020}, based on bounding the error \begin{align*}
    \text{Err}(\mathbf{x}) \coloneqq (2\pi)^k|f(\mathbf{x}) - \tilde{f}(\mathbf{x})| = \left|\int_{\mathbb{R}^k} e^{-i\mathbf{x}^\top \mathbf{t}}(C(\mathbf{t}) - \tilde{C}(\mathbf{t})) \text{d} \mathbf{t} \right|
\end{align*} by defining the set \begin{align*}
    \mathcal{B}(\delta) = \left\{ \mathbf{t} \in \mathbb{R}^k \mid \| \mathbf{t} \| \le \delta \sqrt{d} \right\}
\end{align*} for a small constant $\delta > 0$, and then splitting the integral into two parts based on whether $\mathbf{t} \in \mathcal{B}(\delta)$ or not. That is, we split \begin{align*}
    \text{Err}(\mathbf{x}) \le \underbrace{\int_{\overline{\mathcal{B}(\delta)}} |C(\mathbf{t})| \text{d} \mathbf{t} + \int_{\overline{\mathcal{B}(\delta)}} |\tilde{C}(\mathbf{t})| \text{d} \mathbf{t}}_{\eqqcolon \textup{Err}\uparrow(\mathbf{x})} + \underbrace{\int_{\mathcal{B}(\delta)} |C(\mathbf{t}) - \tilde{C}(\mathbf{t})| \text{d} \mathbf{t}}_{\eqqcolon \textup{Err}\downarrow(\mathbf{x})}.
\end{align*} where $\textup{Err}\downarrow$ accounts for \say{small values of $\mathbf{t}$} and $\textup{Err}\uparrow$ accounts for \say{large values of $\mathbf{t}$}. We prove the following in \Cref{sec:smallt}. \begin{proposition}\label{prop:smallt}
    For all $\mathbf{x} \in \mathbb{R}^k$,
    $$ \textup{Err}\downarrow(\mathbf{x}) = \bigO{ \max\left\{ \frac{1}{\sqrt{d}}, |\kappa| \right\}^{\stwo+1} \left( \frac{1}{\sqrt{d}}\right)^{(k-2)(\stwo+1)} + \left( \frac{1}{\sqrt{d}}  \right)^{\min\left\{\sone + 1, \bone-2, \btwo-2\right\}}}.$$
\end{proposition}
\noindent The proof relies on the fact that approximating $|C(\mathbf{t}) - \tilde{C}(\mathbf{t})|$ using a suitable combination of Taylor series is sufficiently accurate when $\|\mathbf{t}\|$ is small.
Afterwards, we turn to large $\mathbf{t}$ and prove the following in \Cref{sec:larget}. 
\begin{proposition}\label{prop:larget}
    For all $\mathbf{x} \in \mathbb{R}^k$, $\textup{Err}\uparrow(\mathbf{x}) \le \exp\left( -\Omega(d) \right).$
\end{proposition} Here, we rely on the fact that a constant fraction of the $\mathbf{z}_i$ satisfies Cramèr's condition, making the integral over $|C(\mathbf{t})|$ exponentially small in $d$ because $\|\mathbf{t}\|$ is large. To bound the integral over $|\tilde{C})(\mathbf{t})|$, we further use our explicit expression for $\tilde{C})(\mathbf{t})$.
Finally, in \Cref{sec:higherorderterms}, we show that the $\Aeltabold{j}$ arising after applying an inverse Fourier transform to $P_1(\mathbf{t})P_2(\mathbf{t})$ satisfy the following. \begin{proposition}\label{prop:higherorder}
    For all $k + 1 \le j \le \abparamsum$, \begin{align*}
        \|\boldsymbol{\alpha}_{k+j}\|_\infty \le C \left( \frac{1}{\sqrt{d}} \right)^{k-2} \begin{cases}
        \max \left\{ |\kappa| d^{-\frac{1}{6}}, \frac{1}{\sqrt{d}} \right\} d^{-\frac{j-1}{6}} & \textup{if } k \ge 3 \text{ or } (k = 2 \text{ and } |\kappa| < d^{-\frac{1}{3}})\\
        \max \left\{ |\kappa|^{1 + \frac{j}{2}}, d^{-\frac{j-1}{6}} \right\} & \textup{if } k = 2 \text{ and } |\kappa| \ge d^{-\frac{1}{3}}\end{cases}
    \end{align*} where $C > 0$ is a constant.
\end{proposition}
\noindent Stacking the above propositions together proves \Cref{thm:maindensityapprox}. \begin{proof}[Proof of \Cref{thm:maindensityapprox}.]
    We use the fact that the CF of $\mathbf{z}$ is integrable, which follows from the fact that we constructed $\mathbf{z}$ by adding Gaussian noise to $\mathbf{r}$ (cf. \Cref{sec:noise}); this integrability condition is formally verified in \Cref{lem:integrability} in \Cref{sec:larget} and also needed for the proof of \Cref{prop:larget}. Hence, the inversion theorem \Cref{thm:inversion} is applicable and yields that \begin{align*}
        |\tilde{f}(\vecx) - f(\vecx)| = \left| \frac{1}{(2\pi)^k} \int_{\mathbb{R}^k} e^{-i \mathbf{x}^\top \mathbf{t}} (\tilde{C}(\mathbf{t}) - C_{\mathbf{z}}(\mathbf{t})) \dd \mathbf{t} \right| = \frac{1}{(2\pi)^k} \textup{Err}(\mathbf{x}).
    \end{align*} Applying \Cref{prop:smallt} and \Cref{prop:larget} then yields the desired bound on $\sup_{\vecx \in \mathbb{R}^k}|\tilde{f}(\vecx) - f(\vecx)|$. The bound on $\|\boldsymbol{\alpha}_{k+j}\|_\infty$ follows by \Cref{prop:higherorder}.
\end{proof}

\subsubsection{Small Values of $\mathbf{t}$}\label{sec:smallt} 
For small $\mathbf{t}$, we use Taylor series around $\mathbf{t} = \mathbf{0}$ to show that the error incurred is sufficiently small. 
We first recall that by the definition of the approximation $\tilde C$,
\begin{align}\label{holy water}
	C(\mathbf{t}) -  \tilde{C}(\mathbf{t})
	&= \exp\left( K_{(\text{pure})}(\mathbf{t})\right) \left(\exp\left( K_{(\text{mix})}(\mathbf{t})\right) - 1 \right) 
	- \exp\left(-\frac{\mathbf{t}^\top\mathbf{t}}{2}\right) P_1(\mathbf{t})P_2(\mathbf{t}).
\end{align}
We shall rearrange this into a sum of terms that can each be individually bounded, using the following expressions for $\Kpure$ and $\Kmix$, obtained by applying \Cref{thm:taylor}:
\begin{align*}
	\Kpure(\mathbf{t}) 
	&= -\frac{\mathbf{t}^{\top}\mathbf{t}}{2} + \Kpuredaggers(\mathbf{t}) + \underbrace{ \frac{1}{b_1!}\boldsymbol{\nabla}^{\otimes b_1} \left( \Kpure(c\mathbf{t}) - \Kpure(\mathbf{0}) \right)^\top (i\mathbf{t})^{\otimes \bone} }_{\eqqcolon \rpure(\mathbf{t})},\\
	\Kmix(\mathbf{t}) 
	&= \Kmixdaggers(\mathbf{t}) + \underbrace{ \frac{1}{\btwo!}\boldsymbol{\nabla}^{\otimes \btwo} \left( \Kmix(c\mathbf{t}) - \Kmix(\mathbf{0}) \right)^\top (i\mathbf{t})^{\otimes \btwo} }_{\eqqcolon \rmix(\mathbf{t})}.
\end{align*}
Here $\Kpuredaggers$ contains all terms corresponding to cumulants of order between $3$ and $\bone$ of $\Kpure$, and $\Kmixdaggers$ contains all terms corresponding to cumulants of order between $k$ and $\btwo$ of $\Kmix$.
Substituting these two expressions yields
\begin{align}\label{politik}
	\exp&\left( K_{(\text{pure})}(\mathbf{t})\right) \left(\exp\left( K_{(\text{mix})}(\mathbf{t})\right) - 1 \right) \nonumber \\
	& = \exp\left(-\frac{\mathbf{t}^\top\mathbf{t}}{2}\right)
	\Bigg(
	\exp\left(\Kpuredaggers(\mathbf{t}) + \Kmixdaggers(\mathbf{t}) + \rpure(\mathbf{t}) + \rmix(\mathbf{t})\right) \nonumber
	\\&\hspace{8cm}- \exp\left(\Kpuredaggers(\mathbf{t}) + \rpure(\mathbf{t})\right) 
	\Bigg) \nonumber \\
	& = \exp\left(-\frac{\mathbf{t}^\top\mathbf{t}}{2}\right)
	\left(
	r_1(\mathbf{t})
	+ r_2(\mathbf{t})
	+ \exp\left({\Kpuredaggers(\mathbf{t})}\right)
    \left(\exp\left({\Kmixdaggers(\mathbf{t})}\right)-1\right)
	\right),
\end{align}
where
\begin{align*}
	r_1(\mathbf{t}) &\coloneqq \exp\left(\Kpuredaggers(\mathbf{t}) + \Kmixdaggers(\mathbf{t}) + \rpure(\mathbf{t})\right)\left(\exp\left({\rmix(\mathbf{t})}\right) - 1\right),
	\\
	r_2(\mathbf{t}) &\coloneqq \exp\left(\Kpuredaggers(\mathbf{t})\right)\left(\exp\left(\Kmixdaggers(\mathbf{t})\right)-1\right)\left(\exp\left(\rpure(\mathbf{t})\right)-1\right).
\end{align*}

\noindent Now, we can expand the right-hand term in \eqref{politik} by applying an $\sone$-th order Taylor series to $\exp\left( \Kpuredaggers(\mathbf{t})\right)$ and an $\stwo$-th order series to $\left(\exp\left( \Kmixdaggers(\mathbf{t})\right) - 1 \right)$ separately:
\begin{align*}
	\exp\left(\Kpuredaggers(\mathbf{t})\right) 
	& = \underbrace{\sum_{\ell = 0}^{s_1} \frac{1}{\ell!} \left( \Kpuredaggers(\mathbf{t}) \right)^\ell}_{\eqqcolon P_1(\mathbf{t})} + \underbrace{\left( \Kpuredaggers(\mathbf{t})\right)^{s_1+1}\exp\left(c \Kpuredaggers(\mathbf{t})\right)}_{\eqqcolon r_3(\mathbf{t})},\\
	\exp\left( \Kmixdaggers(\mathbf{t})\right) - 1 
	& = \underbrace{\sum_{\ell = 1}^{s_2} \frac{1}{\ell!} \left( \Kmixdaggers(\mathbf{t}) \right)^\ell}_{\eqqcolon P_2(\mathbf{t})} + \underbrace{\left( \Kmixdaggers(\mathbf{t})\right)^{s_2+1}\exp\left(c \Kmixdaggers(\mathbf{t})\right)}_{\eqqcolon r_4(\mathbf{t})}.
\end{align*}
Combining this with \eqref{holy water} and \eqref{politik} yields
\begin{align*}
	C(\mathbf{t}) -  \tilde{C}(\mathbf{t})
	&= \exp\left(-\frac{\mathbf{t}^\top\mathbf{t}}{2}\right)
	\left(
	r_1(\mathbf{t}) 
	+ r_2(\mathbf{t}) 
	+ P_1(\mathbf{t})r_4(\mathbf{t}) 
	+ P_2(\mathbf{t})r_3(\mathbf{t}) 
	+ r_3(\mathbf{t})r_4(\mathbf{t})
	\right).
\end{align*}
%%%%
Having found an explicit expression for the difference $C(\mathbf{t}) - \tilde{C}(\mathbf{t})$, it now remains to bound its absolute value. It now remains to bound the absolute value of all the five error terms appearing in $C(\mathbf{t}) - \tilde{C}(\mathbf{t})$. To this end, we need the following claim which proof we defer for now.
\begin{claim}\label{clm:restbounds}
        There is a sufficiently small $\delta > 0$ and a constant $C > 0$ such that for all $\mathbf{t} \in \mathcal{B}(\delta)$, \begin{align*}
            |P_1(\mathbf{t})| &\le 1 + \sum_{\ell=1}^{\sone}(\mathbf{t}^\top\mathbf{t})^\ell \\|P_2(\mathbf{t})| &\le \sum_{\ell=1}^{\stwo}(\mathbf{t}^\top\mathbf{t})^\ell\\
            |r_1(\mathbf{t})| &\le (\mathbf{t}^\top\mathbf{t})^{\frac{\btwo}{2}} \left( \frac{1}{\sqrt{d}}  \right)^{\btwo-2}\exp\left(\frac{\mathbf{t}^\top\mathbf{t}}{8}\right) \\
            |r_2(\mathbf{t})| &\le (\mathbf{t}^\top\mathbf{t})^{\frac{\bone}{2}} \left( \frac{1}{\sqrt{d}}  \right)^{\bone-2}\exp\left(\frac{\mathbf{t}^\top\mathbf{t}}{8}\right)\\
            |r_3(\mathbf{t})| &\le \left((\mathbf{t}^\top\mathbf{t})^{\frac{3}{2}}\frac{C}{\sqrt{d}}\right)^{\sone+1} \exp\left(\frac{\mathbf{t}^\top\mathbf{t}}{8}\right)
            \\|r_4(\mathbf{t})| &\le \left( (\mathbf{t}^\top\mathbf{t})^{\frac{k}{2}} |\kappa| \left(\frac{1}{\sqrt{d}}\right)^{k-2}  +  (\mathbf{t}^\top\mathbf{t})^{\frac{k+1}{2}} C \left(\frac{1}{\sqrt{d}}\right)^{k-1} \right)^{\stwo+1} \exp\left(\frac{\mathbf{t}^\top\mathbf{t}}{8}\right)
        \end{align*}
\end{claim}

\noindent Using the above claim, we bound the error for $\mathbf{t} \in \mathcal{B}(\delta)$ as \begin{align}\label{eq:errorsmallt}
    \textup{Err}\downarrow(\mathbf{x}) &= \int_{\mathcal{B}(\delta)} |C(\mathbf{t}) - \tilde{C}(\mathbf{t})| \text{d} \mathbf{t} \\& = \int_{\mathcal{B}(\delta)} 
    \exp\left( -\frac{\mathbf{t}^\top\mathbf{t}}{2} \right) \left( |P_1(\mathbf{t})r_4(\mathbf{t})| + |P_2(\mathbf{t})r_3(\mathbf{t})| + |r_3(\mathbf{t})r_4(\mathbf{t})| + |r_1(\mathbf{t})| + |r_2(\mathbf{t})|  \right) \text{d} \mathbf{t}.\notag
\end{align} We sketch how each of the five integrals resulting from the above expression can be bounded by considering the example associated to $r_1$, i.e., \begin{align*}
    I_1 \coloneqq \int_{\mathcal{B}(\delta)} \exp\left( -\frac{\mathbf{t}^\top\mathbf{t}}{2} \right)|r_1(\mathbf{t})|\text{d} \mathbf{t} \le \left( \frac{1}{\sqrt{d}}  \right)^{\btwo-2}\int_{\mathcal{B}(\delta)}(\mathbf{t}^\top\mathbf{t})^{\frac{\btwo}{2}} \exp\left(-\frac{\mathbf{t}^\top\mathbf{t}}{4}\right) \dd \mathbf{t} = \bigO{\left( \frac{1}{\sqrt{d}}  \right)^{\btwo-2}}
\end{align*} where we used \Cref{lem:integral}. Proceeding similarly for the other integrals appearing in (\ref{eq:errorsmallt}) and using the bounds from \Cref{clm:restbounds}, we get that 
\begin{align*}
    \textup{Err}\downarrow(\mathbf{x}) &=\int_{\mathcal{B}(\delta)} |C(\mathbf{t}) - \tilde{C}(\mathbf{t})| \text{d} \mathbf{t} \\ 
    &= \bigO{ \max \left\{ |\kappa|, \frac{1}{\sqrt{d}}\right\}^{\sone+1} \left(\frac{1}{\sqrt{d}}\right)^{(k-2)(\sone+1)} + \left( \frac{1}{\sqrt{d}}  \right)^{\bone-2} + \left( \frac{1}{\sqrt{d}}  \right)^{\btwo-2} + \left( \frac{1}{\sqrt{d}}  \right)^{\sone+1} }\\
    &= \bigO{ \max \left\{ |\kappa|, \frac{1}{\sqrt{d}}\right\}^{\sone+1} \left(\frac{1}{\sqrt{d}}\right)^{(k-2)(\sone+1)} + \left( \frac{1}{\sqrt{d}}  \right)^{\min\{\bone-2, \btwo-2, \sone+1\}}}
\end{align*} as claimed by the theorem. It only remains to prove \Cref{clm:restbounds} to conclude the proof.

\paragraph{ Proving \Cref{clm:restbounds}}
Before proving \Cref{clm:restbounds}, we need the following auxiliary lemmas and start by showing that we can represent $\Kpuredaggers$ and $\Kmixdaggers$ as a suitable polynomial in $\mathbf{t}^\top \mathbf{t}$.
\begin{lemma}\label{lem:boundingkdagger} There is a constant $C > 0$ such that \begin{align*}
        |\Kpuredaggers(\mathbf{t})| &\le C\sum_{j=3}^{\bone}  \left(\frac{1}{\sqrt{d}}\right)^{j-2} (\mathbf{t}^\top\mathbf{t})^{\frac{j}{2}}\\\text{ and } |\Kmixdaggers(\mathbf{t})| &\le |\kappa| \left(\frac{1}{\sqrt{d}} \right)^{k-2}  (\mathbf{t}^\top\mathbf{t})^{\frac{1}{2}} + C\sum_{j=k+1}^{\btwo} \left(\frac{1}{\sqrt{d}}\right)^{j-2} (\mathbf{t}^\top\mathbf{t})^{\frac{j}{2}}
    \end{align*}
    \end{lemma}\begin{proof}
        We note that $\Kpuredaggers(\mathbf{t}) = \sum_{j=3}^{\bone} \frac{1}{j!} \kappapuretop_j (i\mathbf{t})^{\otimes j}$,
        so by Cauchy--Schwarz,
        \begin{align*}
            |\Kpuredaggers(\mathbf{t})| &\le \sum_{j=3}^{\bone} \frac{1}{j!} \|\kappapure_j\| \|(i\mathbf{t})^{\otimes j}\| =  \sum_{j=3}^{\bone} \frac{1}{j!} \|\kappapure_j\| (\mathbf{t}^\top\mathbf{t})^{\frac{j}{2}}.
        \end{align*} By additivity and homogeneity of cumulants (\Cref{prop:additivity}, \Cref{prop:homogeneity}) and the fact that the individual cumulants of each $\mathbf{z}_i$ are bounded by some absolute constant, we have $\|\kappapure_j\| \le C\left(\frac{1}{\sqrt{d}}\right)^{j-2}$ for all $3 \le j \le \bone$, and thus  \begin{align*}
            |\Kpuredaggers(\mathbf{t})|\le C \sum_{j=3}^{\bone} \left(\frac{1}{\sqrt{d}}\right)^{j-2} (\mathbf{t}^\top\mathbf{t})^{\frac{j}{2}}.
            \end{align*} Using a similar argument combined with the fact that all cumulants of order $< k$ are zero and that $\kappa_{(1,1,\ldots,1)}$ is the only non-zero mixed cumulant of order $k$ by \Cref{lem:lowerordercumulantsarezero} we get that \begin{align*}
            \Kmixdaggers(\mathbf{t}) = \kappa \left(\frac{1}{\sqrt{d}} \right)^{k-2} \prod_{j=1}^k(i\mathbf{t}(j)) + \sum_{j=k+1}^{\btwo} \frac{1}{j!} \kappamixtop_j (i\mathbf{t})^{\otimes j}, \end{align*} so since $\|\kappakmix_j\|_\infty \le C\left(\frac{1}{\sqrt{d}}\right)^{j-2}$, \begin{align*}
            |\Kmixdaggers(\mathbf{t})| & \le |\kappa| \left(\frac{1}{\sqrt{d}} \right)^{k-2} (\mathbf{t}^\top\mathbf{t})^{\frac{k}{2}} + \sum_{j=k+1}^{\btwo}\frac{1}{j!}  \|\kappakmixj{j}\|(\mathbf{t}^\top\mathbf{t})^{\frac{j}{2}}\\ & = |\kappa| \left(\frac{1}{\sqrt{d}} \right)^{k-2} (\mathbf{t}^\top\mathbf{t})^{\frac{k}{2}} + C \sum_{j=k+1}^{\btwo} \left(\frac{1}{\sqrt{d}}\right)^{j-2} (\mathbf{t}^\top\mathbf{t})^{\frac{j}{2}}. 
        \end{align*}
    \end{proof}
    \noindent The following simplifies the above polynomials by asserting that their first term is asymptotically dominant. 
    \begin{corollary}\label{cor:kbound}
        There is a constant $C>0$ such that for sufficiently small $\delta > 0$ and all $\mathbf{t} \in \mathcal{B}(\delta)$, we have \begin{align*}
        |\Kpuredaggers(\mathbf{t})| \le \frac{C}{\sqrt{d}} (\mathbf{t}^\top\mathbf{t})^{\frac{3}{2}} \text{ and } |\Kmixdaggers(\mathbf{t})| \le |\kappa| \left( \frac{1}{\sqrt{d}} \right)^{k-2}  (\mathbf{t}^\top\mathbf{t})^{\frac{k}{2}} +  C \left(\frac{1}{\sqrt{d}}\right)^{k-1} (\mathbf{t}^\top\mathbf{t})^{\frac{k+1}{2}}
    \end{align*}
    \end{corollary} \begin{proof}
        This follows from the previous lemma by noting that for any $j$ and any $\ell < j$, we have 
        \begin{align*}
            (\mathbf{t}^\top \mathbf{t})^{\frac{j}{2}} \left(\frac{1}{\sqrt{d}}\right)^{j-2} 
            &= (\mathbf{t}^\top \mathbf{t})^{\frac{\ell}{2}} \left(\frac{1}{\sqrt{d}}\right)^{\ell-2} (\mathbf{t}^\top \mathbf{t})^{\frac{j - \ell}{2}} \left(\frac{1}{\sqrt{d}}\right)^{j-\ell} \\
            &= (\mathbf{t}^\top \mathbf{t})^{\frac{\ell}{2}} \left(\frac{1}{\sqrt{d}}\right)^{\ell-2}\left(\frac{\mathbf{t}^\top \mathbf{t}}{d}\right)^{\frac{j-\ell}{2}}\\
            &\le (\mathbf{t}^\top \mathbf{t})^{\frac{\ell}{2}} \left(\frac{1}{\sqrt{d}}\right)^{\ell-2},
        \end{align*} 
        since for $\mathbf{t} \in \mathcal{B}(\delta)$, we get $\mathbf{t}^\top\mathbf{t} \le \delta^2 d$ and thus $\frac{\mathbf{t}^\top \mathbf{t}}{d} < 1$ for sufficiently small $\delta$. Hence, 
        \[
            |\Kpuredaggers(\mathbf{t})| \le \sum_{j =3}^{\bone} \left(\frac{1}{\sqrt{d}}\right)^{j-2} (\mathbf{t}^\top\mathbf{t})^{\frac{j}{2}}
            = \left(\frac{1}{\sqrt{d}}\right)^{\ell-2}(\mathbf{t}^\top \mathbf{t})^{\frac{\ell}{2}}
            \sum_{j =3}^{\bone} \left(\frac{\mathbf{t}^\top \mathbf{t}}{d}\right)^{\frac{j-\ell}{2}} 
            \le C' (\mathbf{t}^\top \mathbf{t})^{\frac{\ell}{2}} \left(\frac{1}{\sqrt{d}}\right)^{\ell-2}.
        \]
         Choosing $\ell$ as in the bounds of
         \Cref{lem:boundingkdagger} and absorbing $C'$ into the constant $C$ concludes the proof. The same argument also applies to $|\Kmixdaggers(\mathbf{t})|$.
    \end{proof}

    \noindent We proceeed by bounding the remainder terms arising from leaving out higher order pure and mixed cumulants. These decay exponentially in $\bone, \btwo$, respectively with base $1/\sqrt{d}$, which follows by the chain rule and the fact that the CGF of all $\mathbf{z}_j$ is a function uniformly bounded in a neighborhood of $\mathbf{0}$ for all $j \in [d]$. 
    \begin{lemma}\label{lem:rbound}
    For every $\varepsilon > 0$ there is some sufficiently small $\delta> 0$ such that for all $\mathbf{t} \in \mathcal{B}(\delta)$,
    \begin{align*}
        |\rpure(\mathbf{t})| \le \varepsilon (\mathbf{t}^\top \mathbf{t})^{\frac{\bone}{2}} \left( \frac{1}{\sqrt{d}}  \right)^{\bone-2} \text{ and } |\rmix(\mathbf{t})| \le \varepsilon (\mathbf{t}^\top \mathbf{t})^{\frac{\btwo}{2}} \left( \frac{1}{\sqrt{d}}  \right)^{\btwo-2}.
        \end{align*} 
    \end{lemma} \begin{proof}
        We recall that \begin{align*}
            \rpure(\mathbf{t}) = \frac{1}{\bone!}\boldsymbol{\nabla}^{\otimes \bone} \left( \Kpure(c\mathbf{t}) - \Kpure(\mathbf{0}) \right)^\top (i\mathbf{t})^{\otimes \bone} 
        \end{align*} for some $c \in [0,1]$ (that can depend on $\mathbf{t}$) by \Cref{thm:taylor}. Now, using that \begin{align*}
            \Kpure(\mathbf{t}) = \sum_{j=1}^d \Kpure\left(\frac{\mathbf{t}}{\sqrt{d}}, \mathbf{z}_j\right)
        \end{align*} by additivity due to independence and since we normalize each summand by $1/\sqrt{d}$, we get from the chain rule that \begin{align*}
            &\boldsymbol{\nabla}^{\otimes \bone} \left( \Kpure(c\mathbf{t}) - \Kpure(\mathbf{0}) \right) \\
            &\hspace{2em}= \left(\frac{1}{\sqrt{d}}\right)^{\bone} \sum_{j = 1}^d \left( \left.\boldsymbol{\nabla}^{\otimes \bone}\Kpure\left(\mathbf{x}, \mathbf{z}_j\right)\right|_{\mathbf{x} = \frac{c\mathbf{t}}{\sqrt{d}}} - \left.\boldsymbol{\nabla}^{\otimes \bone}\Kpure\left(\mathbf{x}, \mathbf{z}_j\right)\right|_{\mathbf{x} = \mathbf{0}} \right).
        \end{align*} An analogous expression can be obtained for $\boldsymbol{\nabla}^{\otimes \btwo} \left( \Kmix(c\mathbf{t}) - \Kmix(\mathbf{0}) \right)$. Now, we rely on the following statement asserting that each term in the above sum is small. \begin{restatable}{claim}{cgfderivative}
        \label{lem:cgfderivative}
            For every $b \in \mathbb{N}$, every $c \in [0,1]$, and every $\varepsilon > 0$ there is a $\delta > 0$ such that for all $\mathbf{t} \in \mathcal{B}(\delta)$ and all $j \in [d]$,
            \begin{align*}
                &\left\| \left.\boldsymbol{\nabla}^{\otimes b}\Kpure\left(\mathbf{x}, \mathbf{z}_j\right)\right|_{\mathbf{x} = \frac{c\mathbf{t}}{\sqrt{d}}} - \left.\boldsymbol{\nabla}^{\otimes b}\Kpure\left(\mathbf{x}, \mathbf{z}_j\right)\right|_{\mathbf{x} = \mathbf{0}} \right\|\leq\eps \text{ and } \\&\left\| \left.\boldsymbol{\nabla}^{\otimes b}\Kmix\left(\mathbf{x}, \mathbf{z}_j\right)\right|_{\mathbf{x} = \frac{c\mathbf{t}}{\sqrt{d}}} - \left.\boldsymbol{\nabla}^{\otimes b}\Kmix\left(\mathbf{x}, \mathbf{z}_j\right)\right|_{\mathbf{x} = \mathbf{0}} \right\| \le \varepsilon.
            \end{align*}
        \end{restatable} \begin{proof}
            The proof of this claim is deferred to \Cref{sec:deferrederror} because it is very technical.
        \end{proof}
        Applying the above, we get that \begin{align*}
            |\rpure(\mathbf{t})| &\le \frac{1}{\bone!} \left(\frac{1}{\sqrt{d}}\right)^{\bone}\sum_{j = 1}^d  \left\| \left.\boldsymbol{\nabla}^{\otimes \bone}\Kpure\left(\mathbf{x}, \mathbf{z}_j\right)\right|_{\mathbf{x} = \frac{c\mathbf{t}}{\sqrt{d}}} - \left.\boldsymbol{\nabla}^{\otimes \bone}\Kpure\left(\mathbf{x}, \mathbf{z}_j\right)\right|_{\mathbf{x} = \mathbf{0}} \right\|(\mathbf{t}^\top \mathbf{t})^{\frac{\bone}{2}} \\
            &\le \varepsilon (\mathbf{t}^\top \mathbf{t})^{\frac{\bone}{2}} \left( \frac{1}{\sqrt{d}}  \right)^{\bone-2}.
        \end{align*} An analogous argument applies to $\rmix$.
    \end{proof}
    \noindent In order to tame the exponentials appearing in the quantities we wish to bound, we further prove the following lemma. This will later allow us to assert that the exponentials are dominated by the factor of $\exp\left(-\frac{\mathbf{t}^\top\mathbf{t}}{2}\right)$ which is additionally part of all integrals we wish to bound (cf \eqref{eq:errorsmallt}). \begin{corollary}\label{cor:epsilonforexponentials}
        For every $\varepsilon > 0$, there is some $\delta> 0$ such that for all $\mathbf{t} \in \mathcal{B}(\delta)$, \begin{align*}
            |\Kpuredaggers(\mathbf{t})|, |\Kmixdaggers(\mathbf{t})|, |\rpure(\mathbf{t})|, |\rmix(\mathbf{t})| \le \varepsilon ( \mathbf{t}^\top\mathbf{t} ).
        \end{align*}
    \end{corollary} \begin{proof}
        We start with $\Kmixdaggers$. By \Cref{cor:kbound}, we have \begin{align*}
            |\Kmixdaggers(\mathbf{t})|\le |\kappa| \left( \frac{1}{\sqrt{d}} \right)^{k-2}  (\mathbf{t}^\top\mathbf{t})^{\frac{k}{2}} +  C \left(\frac{1}{\sqrt{d}}\right)^{k-1} (\mathbf{t}^\top\mathbf{t})^{\frac{k+1}{2}}.
        \end{align*} Now for all $k\ge 3$ since  $|\kappa| = \bigO{1}$ we get \begin{align}\label{eq:epsilonexample}
            |\kappa| \left( \frac{1}{\sqrt{d}} \right)^{k-2}  (\mathbf{t}^\top\mathbf{t})^{\frac{k}{2}} +  C \left(\frac{1}{\sqrt{d}}\right)^{k-1} (\mathbf{t}^\top\mathbf{t})^{\frac{k+1}{2}} &= |\kappa| \left(\frac{\mathbf{t}^\top \mathbf{t}}{d}\right)^{\frac{k-2}{2}} (\mathbf{t}^\top \mathbf{t}) + C \left(\frac{\mathbf{t}^\top \mathbf{t}}{d}\right)^{\frac{k-1}{2}} (\mathbf{t}^\top \mathbf{t})\nonumber 
            \\ & \le \varepsilon (\mathbf{t}^\top \mathbf{t})
        \end{align} since $\frac{\mathbf{t}^\top \mathbf{t}}{d} \le \delta^2$ for $\mathbf{t} \in \mathcal{B}(\delta)$, so the above inequality follows for sufficiently small $\delta > 0$. If $k = 2$, same follows even more directly since we assumed that $|\kappa| = o(1)$ in this case. Using an entirely analogous argument as in (\ref{eq:epsilonexample}), we can prove our claim also for $\Kpuredaggers, \rpure,$ and $\rmix$ by relying on the bounds from \Cref{cor:kbound} and \Cref{lem:rbound}.
    \end{proof}

    \noindent The above sequence of lemmas can be stacked together to prove \Cref{clm:restbounds}.

    \begin{proof}[Proof of \Cref{clm:restbounds}]
        We only show our claim for $|r_1(\mathbf{t})|$ since the arguments for the other terms are analogous. We start by noting that by (the univariate version of) \Cref{thm:taylor} and any function $g(\mathbf{t})$, we have \begin{align*}
            \exp\left(g(\mathbf{t})\right) = 1 + g(\mathbf{t}) + g(\mathbf{t})(\exp\left(c g(\mathbf{t})\right) - 1) = 1 +  g(\mathbf{t})\exp\left(c g(\mathbf{t})\right)
         \end{align*} for some $c \in [0,1]$ (possibly dependent on $\mathbf{t}$). With this, we can rewrite \[\exp(\rpure(\mathbf{t})) = 1 + \rpure(\mathbf{t})\exp(c\rpure(\mathbf{t})).\]
        Now, we have that \begin{align*}
            r_1(\mathbf{t}) = \rpure(\mathbf{t})\exp\big( \underbrace{c \rpure(\mathbf{t}) + \Kpuredaggers(\mathbf{t}) + \Kmixdaggers(\mathbf{t}) + \rpure(\mathbf{t})}_{\eqqcolon u(\mathbf{t})}\big)
        \end{align*} for some $c \in [0,1]$ that might depend on $\mathbf{t}$. Now, due to \Cref{cor:epsilonforexponentials}, we get that $u(\mathbf{t}) \le \frac{\mathbf{t}^\top\mathbf{t}}{8}$ for some sufficiently small $\delta > 0$. Thus, applying the bound on $|\rpure(\mathbf{t})|$ from \Cref{lem:rbound}, we obtain \begin{align*}
            |r_1(\mathbf{t})| \le |\rpure(\mathbf{t})|\exp\left( \frac{ \mathbf{t}^\top\mathbf{t}}{8} \right) \le (\mathbf{t}^\top \mathbf{t})^{\frac{\btwo}{2}} \left( \frac{1}{\sqrt{d}}  \right)^{\btwo-2}\exp\left(\frac{\mathbf{t}^\top\mathbf{t}}{8}\right)
        \end{align*} as desired. The same strategy (i.e. using \Cref{cor:epsilonforexponentials} to tame the exponential terms, and using \Cref{lem:rbound} and \Cref{cor:kbound} for the rest) also applies to $r_2, r_3,$ and $r_4$. For the bound on $|P_1(\mathbf{t})|, |P_2(\mathbf{t})|$, \Cref{cor:epsilonforexponentials} suffices. 
    \end{proof}

\subsubsection{Large Values of $\mathbf{t}$}\label{sec:larget}

We proceed by bounding the second part of the Error, namely \begin{align*}
    \text{Err}\uparrow (\mathbf{x}) \coloneqq \int_{\overline{\mathcal{B}(\delta)}} |C(\mathbf{t})| \text{d} \mathbf{t} + \int_{\overline{\mathcal{B}(\delta)}} |\tilde{C}(\mathbf{t})| \text{d} \mathbf{t}.
\end{align*} Recalling that \begin{align*}
    \tilde{C}(\mathbf{t}) &= \exp\left( \Kpure \right) + \exp\left( -\frac{\mathbf{t}^\top\mathbf{t}}{2} \right)P_1(\mathbf{t})P_2(\mathbf{t})\\ 
    &= \exp\left( \Kpure \right) + \exp\left( -\frac{\mathbf{t}^\top\mathbf{t}}{2} \right)\left(\sum_{\ell_1=0}^{\sone} \sum_{\ell_2=1}^{\stwo} \frac{1}{\ell_1!\ell_2!} \left(\Kpuredaggers(\mathbf{t})\right)^{\ell_1} \left(\Kmixdaggers(\mathbf{t})\right)^{\ell_2} \right)
\end{align*} yields that \begin{align*}
    \text{Err}\uparrow (\mathbf{x}) \le \underbrace{\int_{\overline{\mathcal{B}(\delta)}} |C(\mathbf{t})|\text{d} \mathbf{t}}_{\eqqcolon \text{Err}_1(\mathbf{x})} + \underbrace{\int_{\overline{\mathcal{B}(\delta)}} \left|\exp\left( \Kpure(\mathbf{t}) \right)\right| \text{d} \mathbf{t}}_{\eqqcolon \text{Err}_2(\mathbf{x})} + \underbrace{\int_{\overline{\mathcal{B}(\delta)}} \exp\left( -\frac{\mathbf{t}^\top\mathbf{t}}{2} \right)|P_1(\mathbf{t})P_2(\mathbf{t})| \text{d} \mathbf{t}}_{\eqqcolon \text{Err}_3(\mathbf{x})}.
\end{align*} We bound these terms separately in the following paragraphs.

\paragraph{Bounding $\text{Err}_1(\mathbf{x})$.}
Here we use that $C$ as the characteristic function (CF) of our random vector $\mathbf{z}$ is the product of the individual CFs of the individual $\mathbf{z}_i$ out of which a constant fraction has absolute value strictly smaller than $1$, which follows from the fact that we assume a constant fraction of the $\mathbf{z}_i$ to meet Cramèr's condition. 
Since by independence, we have $C(\mathbf{t}) = \prod_{i=1}^d C_{\mathbf{z}_i}\left(\frac{\mathbf{t}}{\sqrt{d}} \right)$, this means that $|C(\mathbf{t})|$ is exponentially small in $d$. To ensure that this remains true even after integrating, we need the following lemma which ensures that the individual $C_{\mathbf{z}_i}$ are integrable. This follows because of the Gaussian noise we added (cf. \Cref{sec:noise}) and further ensures that the inversion theorem (\Cref{thm:inversion}) is applicable. 

\begin{lemma}[Integrability of the CF]\label{lem:integrability}
    For every $j$, \begin{align*}
        \int_{\mathbb{R}^k} |C_{\mathbf{z}_j}(\mathbf{s})| \text{d} \mathbf{s} \le \left( d^{\eta}\sqrt{2\pi}\zeta\sigma \right)^k < \infty
    \end{align*} where $\zeta = \sqrt{1 + \sigma^{-2}d^{-2\eta}}$ as in \Cref{sec:noise}.
\end{lemma} \begin{proof}
Similarly to the decomposition of $\mathbf{z}_i$ into \say{signal} and \say{noise} from \Cref{sec:noise} and specifically from \eqref{eq:znoise}, we decompose each individual $\mathbf{z}_j$ by setting \begin{align*}
    \mathbf{z}_j = \Deltaboldhat{j}/\zeta + \boldsymbol{\eta}_j \text{ where } \Deltaboldhat{j} \coloneqq \frac{\boldsymbol{\Delta}_j - \mu}{\sigma} \text{ and }  \boldsymbol{\eta}_j \coloneqq  \frac{\boldsymbol{\eta}_j}{\zeta\sigma} \sim \mathcal{N} \left( 0, \frac{d^{-2\eta}}{(\zeta\sigma)^2} \mathbf{I}_k \right)
\end{align*} Now, since the noise $\boldsymbol{\eta}_j$ is independent of $\hat{\mathbf{z}}_j$, we can write $C_{\mathbf{z}_j}(\mathbf{s}) = C_{\Deltaboldhat{j}}\left(\mathbf{s}/\zeta\right)C_{\boldsymbol{\hat{\eta}}_j}(\mathbf{s})$ and we can note that \begin{align*}
    \int_{\mathbb{R}^k} |C_{\mathbf{z}_j}(\mathbf{s})| \text{d} \mathbf{s} &\le \int_{\mathbb{R}^k} |C_{\Deltaboldhat{j}}(\mathbf{s}/\zeta)||C_{\boldsymbol{\hat{\eta}}_j}(\mathbf{s})| \text{d} \mathbf{s} \\
    &\le \int_{\mathbb{R}^k} |C_{\boldsymbol{\hat{\eta}}_j}(\mathbf{s})| \text{d} \mathbf{s}\\
    &= \int_{\mathbb{R}^k} \exp\left( -\frac{ d^{-2\eta}}{2(\zeta\sigma)^2} \mathbf{s}^\top \mathbf{s} \right) \text{d} \mathbf{s}
\end{align*} 
since the CF of a $\mathcal{N}(0, \mathbf{\Sigma})$ RV is $\exp\left( -\frac{\mathbf{s}^\top\mathbf{\Sigma}\mathbf{s}}{2} \right)$. Factorizing the integral yields
\begin{align*}
    \int_{\mathbb{R}^k} |C_{\mathbf{z}_j}(\mathbf{s})| \text{d} \mathbf{s}&\le \left(\int_{\mathbb{R}^k} \exp\left( -\frac{ d^{-2\eta}}{2(\zeta\sigma)^2} s^2 \right) \text{d} s \right)^k\\
    &= \left( d^{\eta}\zeta\sigma \int_{\mathbb{R}^k} \exp\left( -\frac{ r^2}{2} \right) \text{d} r \right)^k\\
    &= \left( d^{\eta}\sqrt{2\pi}\zeta\sigma \right)^k < \infty
\end{align*} where we substituted $s = d^{\eta}\sqrt{\sigma}r$.
\end{proof}

\noindent Proceeding, we further use the fac that at least $cd$ of the $\mathbf{z}_i$ meet Cramèr's condition, i.e., there is a set $S \subseteq [d]$ with $|S| \ge cd$ such that \begin{align*}
    |C_{\mathbf{z}_i}(\mathbf{t})| \le 1 - \varepsilon(\delta) \text{ for all } i\in S. 
\end{align*} Using this, we bound \begin{align*}
    \text{Err}_1(\mathbf{t}) =\int_{\overline{\mathcal{B}(\delta)}} |C(\mathbf{t})| \text{d} \mathbf{t} &= \int_{\overline{\mathcal{B}(\delta)}} \left|C_{\mathbf{z}_1}\left(\frac{\mathbf{t}}{\sqrt{d}}\right)\right| \prod_{j=2}^d\left|C_{\mathbf{z}_j}\left(\frac{\mathbf{t}}{\sqrt{d}}\right)\right| \text{d} \mathbf{t}\\
    &\le  (1 - \varepsilon(\delta))^{cd - 1} \int_{\mathbb{R}^k} \left|C_{\mathbf{z}_1}\left(\frac{\mathbf{t}}{\sqrt{d}}\right)\right| \text{d} \mathbf{t}\\
    &= (1 - \varepsilon(\delta))^{cd - 1} d^{k/2} \int_{\mathbb{R}^k} \left|C_{\mathbf{z}_1}\left(\mathbf{s}\right)\right| \text{d} \mathbf{s}\\
    &\le (1 - \varepsilon(\delta))^{cd - 1} d^{k/2} \left( d^{\eta}\sqrt{2\pi\sigma} \right)^k \\
    &= \exp \left( -\Omega(d) + O(\log(d))  \right) = \exp( -\Omega(d)).
\end{align*}

\paragraph{Bounding $\text{Err}_2(\mathbf{x})$.} 
To bound \begin{align*}
    \text{Err}_2(\mathbf{x}) = \int_{\overline{\mathcal{B}(\delta)}} \left|\exp\left( \Kpure(\mathbf{t}) \right)\right| \text{d} \mathbf{t},
\end{align*} We note that $\exp\left( \Kpure(\mathbf{t}) \right)$ is in fact the CF of a random vector $\tilde{\mathbf{z}}$ in which each entry has the same marginal distribution as in $\mathbf{z}$ but the entries are all independent. Using this fact, it is not hard to conclude that the same argument used in the previous paragraph is applicable here as well and yields that $\text{Err}_2(\mathbf{x}) = \exp( -\Omega(d))$ as desired.

\paragraph{Bounding $\text{Err}_3(\mathbf{x})$.}
To bound the third term, we use that  \begin{align}\label{eq:intcfapprox}
    \text{Err}_3(\mathbf{x}) &= \int_{\overline{\mathcal{B}(\delta)}} |\tilde{C}(\mathbf{t})| \text{d} \mathbf{t} = \int_{\overline{\mathcal{B}(\delta)}} \exp\left( -\frac{\mathbf{t}^\top\mathbf{t}}{2} \right)|P_1(\mathbf{t})P_2(\mathbf{t})| \text{d} \mathbf{t} \nonumber\\
    &\hspace{2cm}\le \sum_{\ell_1=0}^{\sone} \sum_{\ell_2=1}^{\stwo}  \frac{1}{\ell_1!\ell_2!} \int_{\overline{\mathcal{B}(\delta)}} \exp\left( -\frac{\mathbf{t}^\top\mathbf{t}}{2} \right) \left| \Kpuredaggers(\mathbf{t})\right|^{\ell_1}\left|\Kmixdaggers(\mathbf{t})\right|^{\ell_2} \text{d} \mathbf{t}.
\end{align} We expand $ \left|\Kpuredaggers(\mathbf{t})\right|^{\ell_1}$ and $\left|\Kmixdaggers(\mathbf{t})\right|^{\ell_2}$ as a polynomial using that \begin{align*}
\left|\Kpuredaggers(\mathbf{t})\right|^{\ell_1} \le \left(\sum_{j=3}^{\bone} \| \boldsymbol{\kappa}_j^{(\textup{pure})} \|\| (i\mathbf{t})^{\otimes j} \|\right)^{\ell_1} \le \left(\sum_{j=3}^{\bone} \| \boldsymbol{\kappa}_j^{(\textup{pure})} \|(\mathbf{t}^\top\mathbf{t})^{\frac{j}{2}}\right)^{\ell_1}
\end{align*} by Cauchy-Schwarz. Applying the same idea to $\left|\Kmixdaggers(\mathbf{t})\right|^{\ell_2}$ and using the fact that all cumulants up to order $\max\{\bone,\btwo\}$ are bounded in absolute value by some constant that only depends on $k, \bone, \btwo$, we get that there are coefficients $a_1, a_2,\ldots, a_{\bone\sone + \btwo\stwo}$ all bounded in absolute value by some constant only dependent on $k, \bone, \btwo, \ell_1, \ell_2$, such that \begin{align*}
    \left| \Kpuredaggers(\mathbf{t})\right|^{\ell_1}\left|\Kmixdaggers(\mathbf{t})\right|^{\ell_2} = \sum_{j=1}^{\ell_1\sone + \ell_2\stwo} a_j(\mathbf{t}^\top\mathbf{t})^{\frac{j}{2}}. 
\end{align*} Therefore, we get that \begin{align}\label{eq:erroralmostfinished}
    \text{Err}_3(\mathbf{x}) = \int_{\overline{\mathcal{B}(\delta)}} |\tilde{C}(\mathbf{t})| \text{d} \mathbf{t} \le \sum_{\ell_1=0}^{\sone} \sum_{\ell_2=1}^{\stwo}\sum_{j=1}^{\ell_1\sone + \ell_2\stwo} \frac{|a_j|}{\ell_1!\ell_2!} \int_{\overline{\mathcal{B}(\delta)}} (\mathbf{t}^\top\mathbf{t})^{\frac{j}{2}} \exp\left( -\frac{\mathbf{t}^\top\mathbf{t}}{2} \right) \text{d} \mathbf{t}.
\end{align} To bound the remaining integral, we use \Cref{lem:integrallimiteddomain} the fact that $k, \sone,\stwo,\bone,\btwo$ are constants, we get that $\text{Err}_3(\mathbf{x}) = d^{-\omega(1)}$. This implies that \begin{align*}
    \textup{Err}\uparrow (\mathbf{x}) \le \int_{\overline{\mathcal{B}(\delta)}} |C(\mathbf{t})| \text{d} \mathbf{t} + \int_{\overline{\mathcal{B}(\delta)}} |\tilde{C}(\mathbf{t})| \text{d} \mathbf{t} = \exp( -\Omega(d))
\end{align*} as desired.% and thus finishes the proof of \Cref{thm:maindensityapprox}.

\subsubsection{Bounding the Higher-Order Correction Terms}\label{sec:higherorderterms}

The last crucial aspect of \Cref{thm:maindensityapprox} that we need to take care of is the bound on the infinity-norm of the $\boldsymbol{\alpha}_j$ that appear in our higher-order correction terms. We sketch once again, how these correction terms are obtained and afterwards prove that \begin{align*}
    \|\boldsymbol{\alpha}_{k+j}\|_\infty \le C d^{-\frac{k-2}{2}} \begin{cases}
        \max \left\{ |\kappa|, \frac{1}{\sqrt{d}} \right\} d^{-\frac{j-1}{6}} & \textup{if } k \ge 3 \text{ or } (k = 2 \text{ and } |\kappa| < d^{-\frac{1}{3}})\\
        |\kappa|\max \left\{ |\kappa|^{\frac{j}{2}}, d^{-\frac{j-1}{6}} \right\} & \textup{if } k = 2 \text{ and } |\kappa| \ge d^{-\frac{1}{3}}\end{cases} .
\end{align*} Recall that our approximate density $\tilde{f}$ arises from inverting the Fourier Transform associated to the product of polynomials $P_1(\mathbf{t})P_2(\mathbf{t})$ where \begin{align*}
    P_1(\mathbf{t}) \coloneqq 1 + \sum_{\ell = 1}^{s_1} \frac{1}{\ell!} \left( \Kpuredaggers(\mathbf{t}) \right)^\ell \text{ and } P_2(\mathbf{t}) \coloneqq \sum_{\ell = 1}^{s_2} \frac{1}{\ell!} \left( \Kmixdaggers(\mathbf{t}) \right)^\ell.
\end{align*} We note that when expanding $P_2$, we get a polynomial of the form \begin{align*}
    P_2(\mathbf{t}) = \sum_{\ell = 1}^{s_2} \frac{1}{\ell!}\sum_{j_1=k}^{b_2} \sum_{j_2=k}^{b_2} \ldots \sum_{j_\ell=k}^{b_2} \frac{1}{j_1!j_2!\cdots j_\ell!}\left( \otimes_{m=1}^\ell \kappakmixj{j_m} \right)^\top (i \mathbf{t})^{\otimes \sum_{m=1}^\ell j_m} = \sum_{j=k}^{b_2s_2} \boldsymbol{\gamma}_{j}^\top (i \mathbf{t})^{\otimes j}
\end{align*} where the $\boldsymbol{\gamma}_{j} \in \mathbb{R}^{k^j}$ collect all the terms appearing in front of $(i\mathbf{t})^{\otimes j}$, i.e. \begin{align}\label{eq:gammaj}
    \boldsymbol{\gamma}_{j} 
    =  \sum_{\ell =1}^{j/k} \frac{1}{\ell!} \hspace{.2cm} \sum_{\substack{(j_1,j_2,\ldots j_\ell) \in \mathbb{N}^{\times \ell} \\ \sum_{m} j_m 
    = j\\j_m \ge k}} \frac{1}{j_1!\ldots j_\ell!} \left(\otimes_{m=1}^\ell \kappakmixj{j_m}\right).
\end{align} 
Note that we can assume that the first sum above stops at $\ell/k$ instead of $s_2$ since the second sum assumes that each $j_m \ge k$ and that all the $j_m$ sum to $j$, which is only possible if $\ell \le j/k$. With this and the fact that all mixed cumulants of order $k \le j_m \le b_2$ are bounded by \begin{align*}
    \|\kappakmixj{j_m}\|_\infty \le \begin{cases}
       |\kappa| d^{-\frac{j_m-2}{2}} &\text{if } j_m = k\\
       \hspace{.1cm}C_1 d^{-\frac{j_m-2}{2}} &\text{if } j_m > k
    \end{cases} \text{ we get that } 
    \|\kappakmixj{j_m}\|_\infty^{1/j_m} \le \begin{cases}
       |\kappa|^{\frac{1}{k}} d^{-\frac{k-2}{2k}} &\text{if } j_m = k\\
       C_1^{\frac{1}{k+1}} d^{-\frac{k-1}{2(k+1)}} &\text{if } j_m > k
    \end{cases}
\end{align*} 
for some constant $C_1 > 0$.
Using that all $j_m$ in \eqref{eq:gammaj} sum to $j$, we get for all $1 \le \ell \le j/k$ and $j_1, \dots, j_\ell$
\begin{align*}
    \|\otimes_{m=1}^\ell \kappakmixj{j_m}\|_\infty 
    &\le \prod_{m=1}^{\ell} \max \left\{ |\kappa|^{\frac{j_m}{k}} d^{-j_m\frac{k-2}{2k}}, C_1^{\frac{j_m}{k+1}} d^{-j_m\frac{k-1}{2(k+1)}} \right\} \\
    &\le C_1^{\frac{j}{k+1}}  \max \left\{ |\kappa|^{\frac{j}{k}} d^{-j\frac{k-2}{2k}}, d^{-j\frac{k-1}{2(k+1)}} \right\}, 
\end{align*}
where we assume $C_1 \ge 1$ in the last inequality.
Further, from \eqref{eq:gammaj} and the multinomial theorem, we get
\begin{align*}
    \| \boldsymbol{\gamma}_{j} \|_\infty \le C_1^{\frac{j}{k+1}} \max \left\{ |\kappa|^{\frac{j}{k}} d^{-j\frac{k-2}{2k}}, d^{-j\frac{k-1}{2(k+1)}} \right\} \sum_{j=1}^{\ell/k} \frac{\ell^{j}}{\ell! j!}
\end{align*}
and by setting $C \coloneqq C_1^{\frac{j}{k+1}} \sum_{\ell=1}^{j/k} \frac{\ell^{j}}{\ell! j!}$ (note that $j \le b_2 s_2$ is bounded by a constant) we have 
\begin{align*}
    \| \boldsymbol{\gamma}_{j} \|_\infty 
    &\le C \max\left\{ |\kappa|^{\frac{j}{k}} d^{-\frac{j(k-2)}{2k}}, d^{-\frac{j(k-1)}{2(k+1)}} \right\}  \\
    &= C \max\left\{ |\kappa| d^{-\frac{(k-2)}{2}} |\kappa|^{\frac{j-k}{k}} d^{-(j-k)\frac{k-2}{2k}}, d^{-\frac{(k-1)}{2}}d^{-(j-k-1)\frac{k-1}{2(k+1)}} \right\} 
\end{align*} or alternatively  \begin{align*}
    \| \boldsymbol{\gamma}_{k+j} \|_\infty &\le C d^{-\frac{k-2}{2}} \begin{cases}
        \max \left\{ |\kappa| d^{-\frac{1}{6}}, \frac{1}{\sqrt{d}} \right\} d^{-\frac{j-1}{6}} & \textup{if } k \ge 3 \\
        \max \left\{ |\kappa|^{1 + \frac{j}{2}}, \frac{1}{\sqrt{d}} d^{- \frac{j-1}{6}} \right\} & \textup{if } k = 2 \end{cases}
\end{align*} for all $j \ge 1$, while $\| \boldsymbol{\gamma}_{k} \|_\infty = |\kappa|d^{-\frac{k-2}{2}}$. Similarly, we can note that $\| \kappapurej{j} \|_\infty \le C d^{-\frac{j-2}{2}}$ for all $3 \le j \le b_1$ and \begin{align*}
    P_1(\mathbf{t}) 
    &= 1 + \sum_{\ell = 1}^{s_1} \frac{1}{\ell!}\sum_{j_1=k}^{b_1} \sum_{j_2=k}^{b_1} \ldots \sum_{j_\ell=j}^{b_1} \frac{1}{j_1!j_2!\cdots j_\ell!}\left( \otimes_{m=1}^\ell \kappapurej{j_m} \right)^\top (i \mathbf{t})^{\otimes \sum_{m=1}^\ell j_m} \\
    &= 1 + \sum_{j=3}^{b_1s_1} \boldsymbol{\beta}_j^\top (i\mathbf{t})^{\otimes j}
\end{align*} 
with \begin{align*}
    \boldsymbol{\beta}_{j} =  \sum_{\ell =1}^{j/3} \frac{1}{\ell!} \hspace{.2cm} \sum_{\substack{(j_1,j_2,\ldots j_\ell) \in \mathbb{N}^{\times \ell} \\ \sum_{m}^\ell j_m = j\\j_m \ge 3}} \frac{1}{j_1!\ldots j_\ell!} \left(\otimes_{m=1}^\ell \kappapurej{j_m}\right), \text{ so } \| \boldsymbol{\beta}_{j} \|_\infty = \sum_{\ell =1}^{j/3} \bigO{ d^{-\frac{j - 2\ell}{2}} } = \bigO{ d^{-\frac{j}{6} } }.
\end{align*} 
Again, we can stop summing after $j/3$ since we requrie the $j_m$ to sum to $j$ but also know that $j_m\ge 3$ for all $m$. Combining this, we get that \begin{align*}
    P_1(\mathbf{t})P_2(\mathbf{t}) 
    &= \kappa \left(\frac{1}{\sqrt{d}}\right)^{k-2}\prod_{j=1}^k (i\mathbf{t}(j)) +  \sum_{j = k+1}^{s_1b_1 + s_2b_2} \boldsymbol{\alpha}_j^\top (i\mathbf{t})^{\otimes j} 
    \text{ with }\\
    \boldsymbol{\alpha}_{k + j} 
    &= \boldsymbol{\gamma}_{k + j} + \sum_{\substack{(j_1, j_2) \in \mathbb{N}^{\times 2} \\ j_1 + j_2 = j\\ j_1 \ge 3, j_2 \ge 1}} \boldsymbol{\beta}_{j_1} \otimes \boldsymbol{\gamma}_{k + j_2},
\end{align*} where \begin{align*}
    \| \boldsymbol{\beta}_{j_1} \otimes \boldsymbol{\gamma}_{k + j_2} \|_\infty \le C d^{-\frac{k-2}{2}} \begin{cases}
        \max \left\{ |\kappa| d^{-\frac{1}{6}}, \frac{1}{\sqrt{d}} \right\} d^{-\frac{j_1-1}{6}} d^{-\frac{j_2}{6}} & \textup{if } k \ge 3 \\
        \max \left\{ |\kappa|^{1 + \frac{j_1}{2}}, \frac{1}{\sqrt{d}} d^{- \frac{j_1-1}{6}} \right\} d^{-\frac{j_2}{6}} & \textup{if } k = 2 \end{cases}
\end{align*} so since $j_1 + j_2 = j$, \begin{align*}
    \|\boldsymbol{\alpha}_{k+j}\|_\infty \le C d^{-\frac{k-2}{2}} \begin{cases}
        \max \left\{ |\kappa| d^{-\frac{1}{6}}, \frac{1}{\sqrt{d}} \right\} d^{-\frac{j-1}{6}} & \textup{if } k \ge 3 \text{ or } (k = 2 \text{ and } |\kappa| < d^{-\frac{1}{3}})\\
        \max \left\{ |\kappa|^{1 + \frac{j}{2}}, d^{-\frac{j-1}{6}} \right\} & \textup{if } k = 2 \text{ and } |\kappa| \ge d^{-\frac{1}{3}}\end{cases}
\end{align*} as desired.

\section{Bounding The Signed Weight of Cycles and Chains}

Our main technical application of \Cref{thm:maindensityapprox} is to bound the signed weight of cycles and chains from above (although it also allows us to bound certain signed weights from below, see \Cref{sec:triangles}). We remark that our result is not limited to the case of cycles or chains, but also yields a bound on the signed weight of a pattern $H$ that contains cycles or chains as induced subgraphs (see \Cref{sec:fouriereuler} for a concrete example). We believe that \Cref{thm:maindensityapprox} or its extensions are also useful for bounding the signed weight of more general patterns $H$, potentially via expanding $\Expected{\SW(H)}$ into an alternating sum over terms of the form $\Expected{\prod_{e\in A}\mathds{1}(e)}$ over all $A \subseteq H$ and then apply a \say{cluster expansion} approach to quantify the deviation of this sum from its \say{ground state} corresponding to a $\Gnp$. A similar approach for bounding the signed weight of general $H$ was used in \cite{Bangachev_Bresler_2024} for the special case of $L_\infty$-norm.

\subsection{Our Result}

\noindent We obtain the following bound on Fourier coefficients for cycles and chains, which we phrase in terms of the asymptotic behavior of their expected signed weights (see \Cref{sec:prelims} for the necessary definitions and notation). 

\fouriercoefficients*

\subsection{Proving \Cref{thm:fouriercoefficients}}

To prove \Cref{thm:fouriercoefficients}, we first use the following convenient reformulation of Fourier coefficients. We remark that this only holds for cycles and chains but breaks down in the general case.

\begin{lemma}\label{lem:rearrangecycle}
    For a cycle or a chain $H$ of length $k$ with edges $e_1, \ldots, e_k$, we have
    \begin{align*}
        \Expected{\textsc{Sw}(H)} = \Expected{\prod_{e \in H}\mathds{1}(e)} - p^k.
    \end{align*}
\end{lemma} \begin{proof}
    First of all, note that multiplying out $\textsc{Sw}(H) = \prod_{e \in H}(\mathds{1}(e) - p)$ yields that \begin{align*}
        \Expected{\textsc{Sw}(H)} &= \sum_{S \subseteq H} (-1)^{k-|S|}\Expected{\prod_{e \in S} \mathds{1}(e)}p^{k-|S|}
        =\Expected{\prod_{e \in H}\mathds{1}(e)} + \sum_{S \subset H} (-p)^{k-|S|} p^{|S|},
    \end{align*} where we used that $\Expected{\prod_{e \in S} \mathds{1}(e)} = p^{|S|}$ for all $S \neq H$ because after removing any edge from a cycle or a chain, the remaining edges are independent. Now, using the binomial theorem, we get \begin{align*}
        \sum_{S \subset H} (-p)^{k-|S|} p^{|S|}
        &= - p^k + \sum_{S \subseteq H} (-p)^{k-|S|} p^{|S|}
        = - p^k + \sum_{j=0}^k \binom{k}{j} (-p)^j p^{k-j} \\
        &= - p^k + (p - p)^k = - p^k,
    \end{align*} which finishes the proof.
\end{proof}

\noindent Hence, to prove \Cref{thm:fouriercoefficients}, it suffices to show that \begin{align*}
    \Expected{\prod_{e\in C_k} \mathds{1}(e)} & = p^k + \bigO{ p^k \left(\frac{\log(n)}{\sqrt{d}}\right)^{k-2} } \text{ and } \\\Expected{\prod_{e\in H} \mathds{1}(e) } & = p^k + \bigO{ p^{k}|\kappa|  \left(\frac{\log(n)}{\sqrt{d}}\right)^{k-2} + p^k \log^2(n) \left(\frac{\log(n)}{\sqrt{d}}\right)^{k-1} }, 
\end{align*} respectively. Now, we can use that $\Expected{\prod_{e\in H} \mathds{1}(e) }$ is the probability that all edges in our cycle or our chain are present. If we choose our noise parameter $\eta$ (cf. \Cref{sec:noise}) large enough, this probability is sufficiently close to the probability that all components of $\mathbf{z}$ are at most $\hattau$, i.e.  \begin{align*}
  \Expected{\prod_{e\in C_k} \mathds{1}(e)} \approx \Pr{\bigcap_{i=1}^k(\mathbf{z}(i) \le \hattau)}.
\end{align*} Thus, we are mainly concerned with bounding the above probability. Using our approximation $\tilde{f}$ to the joint density of $\mathbf{z}$ from our main theorem \Cref{thm:maindensityapprox}, we can do this by integrating \begin{align*}
    \tilde{f}(\mathbf{x}) =  \find(\mathbf{x}) + (-1)^k \kappa \left( \frac{1}{\sqrt{d}} \right)^{k-2} \prod_{j=1}^k\phi^{(1)}(\mathbf{x}(j)) +\sum_{j=k+1}^{\sone\bone + \stwo\btwo} \boldsymbol{\alpha}_i^\top \left( \boldsymbol{\nabla}^{\otimes j} \phi(\mathbf{x}) \right) 
\end{align*} in a region close to the origin. Due to our tail bound \Cref{lem:tailofz}, it in fact suffices to define the interval $D = [-\log(n), \hattau]$ and to integrate $\int_{D^k} \tilde{f}(\mathbf{x}) \text{d} \mathbf{x}$. The main amount of work then lies in bounding the integrals over the above three terms, which we do in the following sequence of claims. We start by showing that the first term is close to $p^k$.

\begin{claim}[The First Term]\label{clm:thefirstclaim}
    For any $\eta > 0$, we have \begin{align*}
        \int_{D^k}  \find(\mathbf{x}) \textup{d}\mathbf{x} = \left( \int_{D}  f_{\mathbf{z}(1)}(x) \textup{d}x \right)^{k} \le \left(p + \bigO{\frac{\log(n)}{d^{\eta - 1}}}\right)^k.
    \end{align*}
\end{claim}

\noindent We further prove that the second correction term is essentially as large as the correction term we wish to have. 

\begin{claim}[The Second Term]\label{clm:thesecondclaim}
    We have \begin{align*}
        |\kappa|
        \left(\frac{1}{\sqrt{d}}\right)^{k-2}\int_{D^k} \prod_{j=1}^k \phi^{(1)}(\mathbf{x}(j)) \text{d} \mathbf{x} = |\kappa|
        \left(\frac{1}{\sqrt{d}}\right)^{k-2} \left( \int_{D}\phi^{(1)}(x) \textup{d} x \right)^k \le Cp^k|\kappa|
        \left(\frac{1}{\sqrt{d}}\right)^{k-2}
    \end{align*}
\end{claim}

\noindent Now, it only remains to prove that the third term is $\bigO{ p^{k}/d^{(k-2)/2}}$. This follows by the fact that we integrate close to the origin where the derivatives of $\phi(x)$ only differ from $\phi(x)$ by polylogarithmic factors, and by our bound on $\|\boldsymbol{\alpha}_j\|_\infty$ from \Cref{thm:maindensityapprox}.

\begin{claim}[The Third Term]\label{clm:thethirdclaim}
    We have \begin{align*}
        \int_{D^k}\sum_{j=k+1}^{\sone\bone + \stwo \btwo} \boldsymbol{\alpha}_i^\top \left( \boldsymbol{\nabla}^{\otimes j} \phi(\mathbf{x}) \right) \textup{d} \mathbf{x} = \bigO{ p^k|\kappa| \left(\frac{\log(n)}{\sqrt{d}} \right)^{k-2}  + p^k \log^2(n) \left(\frac{\log(n)}{\sqrt{d}} \right)^{k-1} }
    \end{align*}
\end{claim}

\noindent Putting everything together and quantifying the other sources of error, we can prove \Cref{thm:fouriercoefficients}.
\begin{proof}[Proof of \Cref{thm:fouriercoefficients}]
   Using \Cref{lem:rearrangecycle}, we get that every cycle or chain $H$ satisfies \begin{align*}
            \Expected{\textsc{Sw}(H)} = \Expected{\prod_{e\in H}\mathds{1}(e)} - p^k
        \end{align*} and by \Cref{lem:influenceofnoise}, \begin{align*}
            \Expected{\prod_{e\in H}\mathds{1}(e)} 
            &= \Pr{ \bigcap_{i = 1}^k (\Deltaboldhat{}(i) \le \hattau) } \le \Pr{ \bigcap_{i = 1}^k(\mathbf{z}(i) \le \hattau) } + \bigO{\frac{1}{d^{\eta-1}}} \\
            &\le \Pr{ \bigcap_{i = 1}^k(\mathbf{z}(i) \le \hattau) } + \bigO{\frac{p^k}{d^{k}}}
        \end{align*} 
        since we can choose $\eta$ to be a sufficiently large constant such that $\frac{1}{d^{\eta-1}} \le \frac{p^k}{d^k}$ since $p \ge n^{-\bigO{1}}$ and $d \ge n^{\gamma}$. Now we wish to use our approximation $\tilde{f}$ to the density $f$ of $\mathbf{z}$ from \Cref{thm:maindensityapprox}. To this end, we must first control some tail behavior of $\mathbf{z}$ that subsequently allows us to integrate only over a small region around the origin. We define the interval $D = [-\log(n), \hattau]$ and note that \begin{align*}
            \Pr{ \bigcap_{i = 1}^k(\mathbf{z}(i) \le \hattau) } &= \Pr{ \bigcap_{i = 1}^k  (\mathbf{z}(i)\in D) } + \Pr{\left( \bigcup_{i = 1}^k (\mathbf{z}(i) < - \log(n)) \right) \cap \left( \bigcap_{i = 1}^k  (\mathbf{z}(i)\le \hattau ) \right)  }\\
            &\le \Pr{ \bigcap_{i = 1}^k  (\mathbf{z}(i)\in D) } + \Pr{ \bigcup_{i = 1}^k (\mathbf{z}(i) < - \log(n)) }\\
            &\le \Pr{ \bigcap_{i = 1}^k  (\mathbf{z}(i)\in D) } + kn^{-\omega(1)}
        \end{align*} by the tail bound from \Cref{lem:tailofz}. To bound the remaining probability, we apply \Cref{thm:maindensityapprox} and choose the parameters $\sone,\stwo,\bone,\btwo$ all large enough such that our approximation error satisfies \begin{align*}
            \sup_{\mathbf{x} \in \mathbb{R}^k}|f(\mathbf{x}) - \tilde{f}(\mathbf{x})|& = \bigO{ \max\left\{ \frac{1}{\sqrt{d}}, |\kappa| \right\}^{\stwo+1} \left( \frac{1}{\sqrt{d}}\right)^{(k-2)(\stwo+1)} + \left( \frac{1}{\sqrt{d}}  \right)^{\min\{\bone-2, \btwo-2, \sone+1\}}} \\&= \bigO{ \frac{p^k}{d^k} },
        \end{align*} which is possible since all bases of the above exponentials are at most $n^{-\varepsilon}$ and since $p \ge n^{-\bigO{1}}$. Now, after applying our main theorem, we get \begin{align*}
            \Pr{ \bigcap_{i = 1}^k(\mathbf{z}(i) \in D) } &\le \int_{D^k} \tilde{f}(\mathbf{x}) \text{d} \mathbf{x} + \bigO{\frac{p^k\log^k(n)}{d^k}}
        \end{align*} since we integrate over a set of volume $\bigO{\log^k(n)}$ while having an error of $\bigO{p^k/d^k}$. It remains to bound the above integral. Using the expression for $\tilde{f}$ from \Cref{thm:maindensityapprox}, we get \begin{align*}
            \int_{D^k} \tilde{f}(\mathbf{x}) \text{d} \mathbf{x} 
            &= \int_{D^k}  \find(\mathbf{x})\textup{d}\mathbf{x} + (-1)^k\kappa \left( \frac{1}{\sqrt{d}} \right)^{k-2}\int_{D^k} \prod_{j=1}^k \phi^{(1)}(\mathbf{x}(j)) \text{d} \mathbf{x} + \sum_{j=4}^{\sone\bone+ \stwo\btwo} \int_{D^k}\boldsymbol{\alpha}_j^\top \left( \boldsymbol{\nabla}^{\otimes j} \phi(\mathbf{x}) \right) \text{d} \mathbf{x}\\
            &\le \left( \int_{D}  \find(\mathbf{x})\textup{d}x \right)^{k} + |\kappa| \left( \frac{1}{\sqrt{d}} \right)^{k-2}\left(\int_{D} \phi^{(1)}(x) \text{d} x \right)^k + \sum_{j=4}^{\sone\bone+ \stwo\btwo} \int_{D^k}\boldsymbol{\alpha}_j^\top \left( \boldsymbol{\nabla}^{\otimes j} \phi(\mathbf{x}) \right) \text{d} \mathbf{x}.
        \end{align*} 
        Applying \Cref{clm:thefirstclaim}, \Cref{clm:thesecondclaim} and \Cref{clm:thethirdclaim}, we get that \begin{align*}
            \int_{D^k} \tilde{f}(\mathbf{x}) \text{d} \mathbf{x} &= \left( p + \bigO{\frac{p^k}{d^k}} \right)^k + Cp^k|\kappa|
        \left(\frac{1}{\sqrt{d}}\right)^{k-2} + \bigO{ p^k|\kappa| \left(\frac{\log(n)}{\sqrt{d}} \right)^{k-2} + p^k\log^2(n) \left(\frac{\log(n)}{\sqrt{d}} \right)^{k-1} } \\
        &= p^k + \bigO{p^k|\kappa| \left(\frac{\log(n)}{\sqrt{d}} \right)^{k-2} + p^k\log^2(n) \left(\frac{\log(n)}{\sqrt{d}} \right)^{k-1} }.
        \end{align*} In total, this means that \begin{align*}
            \Expected{\prod_{e\in H} \mathds{1}(e)}  &\le p^k + \bigO{p^k|\kappa| \left(\frac{\log(n)}{\sqrt{d}} \right)^{k-2} + p^k\log^2(n) \left(\frac{\log(n)}{\sqrt{d}} \right)^{k-1} } + \bigO{\frac{p^k}{d^{k}}} + \bigO{\frac{p^k\log^k(n)}{d^k}} + kn^{-\omega(1)} \\
            &\le p^k + \bigO{p^k|\kappa| \left(\frac{\log(n)}{\sqrt{d}} \right)^{k-2} + p^k\log^2(n) \left(\frac{\log(n)}{\sqrt{d}} \right)^{k-1} } 
        \end{align*} as desired for the case of a chain. In case of a cycle, we simply note that $|\kappa| = \bigO{1}$ and the result follows. 
\end{proof}

\noindent It only remains to prove \Cref{clm:thefirstclaim}, \Cref{clm:thesecondclaim}, and \Cref{clm:thethirdclaim}.

\begin{proof}[Proof of \Cref{clm:thefirstclaim}]
    Factorizing and using that all components of $\mathbf{z}$ marginally have the same distribution yields \begin{align*}
        \int_{D^k}  \prod_{j=1}^k f_{\mathbf{z}(j)}(\mathbf{x}(j)) \textup{d}\mathbf{x} \le \left( \int_{D} f_{\mathbf{z}(1)}(x) \textup{d}x\right)^k = \Pr{\mathbf{z}(i) \in D}^k \le \Pr{\mathbf{z}(i) \le \hattau}^k.
    \end{align*} Now, we get by \Cref{lem:influenceofnoise} and \Cref{cor:tauissmall} that \begin{align*}
        \Pr{\mathbf{z}(1) \le \hattau} \le \Pr{\Deltaboldhat{}(1) \le \hattau} + \bigO{\frac{\log(n)}{d^{\eta - 1}}} = p + \bigO{\frac{\log(n)}{d^{\eta - 1}}}
    \end{align*} and the proof is finished.
\end{proof}

\begin{proof}[Proof of \Cref{clm:thesecondclaim}]
    Factorizing yields that \begin{align*}
        |\kappa|\left(\frac{1}{\sqrt{d}}\right)^{k-2}\int_{D^k} \prod_{j=1}^k \phi^{(1)}(\mathbf{x}(j)) \text{d} \mathbf{x} &= |\kappa|\left(\frac{1}{\sqrt{d}}\right)^{k-2}\left( \int_{D} \phi^{(1)}(x) \text{d} x \right)^k \\ &\le |\kappa|\left(\frac{1}{\sqrt{d}}\right)^{k-2} \phi(\hattau)^k \le Cp^k|\kappa|\left(\frac{1}{\sqrt{d}}\right)^{k-2}
        \end{align*} where we used \Cref{cor:phip}.
\end{proof}

\begin{proof}[Proof of \Cref{clm:thethirdclaim}]
    We note that for any $\mathbf{x} = (x_1, \ldots, x_k)^\top $ 
    \begin{align}\label{alphaexpansion}
            \boldsymbol{\alpha}_j^\top \left( \boldsymbol{\nabla}^{\otimes j} \phi(\mathbf{x}) \right) &\le \sum_{\substack{s = (s_1, \ldots, s_k) \in \mathbb{N}^{\times k} \\ |s| = j}} \|\boldsymbol{\alpha}_j\|_\infty \left| \frac{\partial^{|s|}}{\partial^{s_1}x_1\cdots\partial^{s_k}x_k} \phi(\mathbf{x}) \right|\\& = \sum_{\substack{s = (s_1, \ldots, s_k) \in \mathbb{N}^{\times k} \\ |s| = j}} \|\boldsymbol{\alpha}_j\|_\infty |P_{s} (\mathbf{x}) | \phi(\mathbf{x}) 
        \end{align} 
        where the $P_s$ are a collection of polynomials each of which is a product of Hermite polynomials whose order sums up to $j$ (i.e. each $P_s$ is of order $j$). Since each $\mathbf{x} \in D^k$ fulfills $\|\mathbf{x}\|_\infty \le \log(n)$ by \Cref{cor:tauissmall}, we get that $|P_{s} (\mathbf{x}) | = \bigO{\log^j(n)}$ for  $\mathbf{x} \in D^k$. Thus, when integrating, we get  \begin{align*}
            \int_{D^k}\boldsymbol{\alpha}_j^\top \left( \boldsymbol{\nabla}^{\otimes j} \phi(\mathbf{x}) \right) \text{d} \mathbf{x} \le \bigO{\|\boldsymbol{\alpha}_j\|_\infty \log^j(n) } \int_{D^k} \phi(\mathbf{x}) \text{d} \mathbf{x} =  \bigO{ p^k \|\boldsymbol{\alpha}_j\|_\infty \log^j(n) }
        \end{align*} since \begin{align*}
            \int_{D^k} \phi(\mathbf{x}) \text{d} \mathbf{x} = \left( \int_{D} \phi(x) \text{d} x  \right)^k \le \Phi(\hattau)^k \le Cp^k
        \end{align*} by \Cref{cor:phip}. To bound $\|\boldsymbol{\alpha}_j\|_\infty \log^j(n)$, we use our bound from \Cref{thm:maindensityapprox} to obtain \begin{align*}
            \log^{k + j}(n)\|\boldsymbol{\alpha}_{k + j}\|_\infty &\le |\kappa| \log^2(n) \left(\frac{\log(n)}{\sqrt{d}} \right)^{k-2} \left( \frac{\log(n)}{d^{1/6}} \right)^{j} + \log^2(n) \left(\frac{\log(n)}{\sqrt{d}} \right)^{k-1} \left( \frac{\log(n)}{d^{1/6}} \right)^{j-1} \\
            &\le |\kappa| \left(\frac{\log(n)}{\sqrt{d}} \right)^{k-2} \left( \frac{\log^3(n)}{d^{1/6}} \right)^{j} + \log^2(n) \left(\frac{\log(n)}{\sqrt{d}} \right)^{k-1} \left( \frac{\log(n)}{d^{1/6}} \right)^{j-1}\\
            &\le |\kappa| \left(\frac{\log(n)}{\sqrt{d}} \right)^{k-2} + \log^2(n) \left(\frac{\log(n)}{\sqrt{d}} \right)^{k-1}
        \end{align*} where used that $\left( \frac{\log(n)}{d^{1/6}} \right)^{j} \le 1$ since $d \ge n^{\gamma}$. We remark that we are intentionally loose with the log factors here for the sake of simplicity.
\end{proof}

\section{Application 1: Tight Testing Thresholds via Signed Triangles}\label{sec:triangles}

In this section, we prove that signed triangles provide a test that distinguishes a RGG from a $G(n, p)$ if $d = \tilde{o}(n^3p^3)$. This is a significant improvement over the previous best algorithmic upper bound for detecting geometry in a $\RGG$ proven by Bangachev and Bresler \cite[Theorem 1.11]{Bangachev_Bresler_2024} which asserts only that $\TV{\RGG}{\Gnp} = 1- o(1)$ if $d = \tilde{o}(np)$. The argument in \cite{Bangachev_Bresler_2024} does not rely on signed triangles but instead uses an entropy-based argument built upon an explicit $\varepsilon$-net of $\mathbb{T}^d$ with specific properties, which has the drawback of not giving a computationally efficient test. We analyze the signed triangle count
 \begin{align*}
    \T(\Gbf) \coloneqq \sum_{i < j< k \in [n]} \Tijk{ijk} \text{ where } \Tijk{ijk} \coloneqq (\Gbf_{ij} - p)(\Gbf_{ik} - p)(\Gbf_{jk} - p)
\end{align*} yielding a more powerful -- but still efficient -- algorithm for detecting geometry. The result, stated in the following theorem, matches the information-theoretic lower bound in case $p = \Omega(1)$.

{\renewcommand{\thetheorem}{\ref{thm:signedtriangles}}
\begin{theorem}[Explicitly stated]
    For any fixed $1 \le q < \infty$ and any $d \ge n^{\gamma}$ for $\gamma>0$ arbitrarily small, the number of signed triangles distinguishes $\Gbf_1 \sim \RGG$ and $\Gbf_2 \sim \Gnp$ with probability $1 - o(1)$ whenever $d = o(n^3p^3)$, and fails at doing so whenever $d = \omega(n^3p^3)$. 
    Precisely, \begin{align*}
        \left|\Expectedsub{\Gbf \sim \RGG}{\T(\Gbf)} - \Expectedsub{\Gbf \sim \Gnp}{\T(\Gbf)}\right| - \max \left\{ \sqrt{\Varsub{\Gbf \sim \RGG}{\T(\Gbf)}}, \sqrt{\Varsub{\Gbf \sim \Gnp}{\T(\Gbf)}} \right\} \\= \begin{cases}
            \omega(1) &\text{if } d = o(n^3p^3)\\
            o(1) &\text{if } d = \tilde{\omega}(n^3p^3).
        \end{cases}
    \end{align*}
\end{theorem}}

To prove \cref{thm:signedtriangles}, we bound the expected signed triangle count from below, and its variance from above, and the result follows by Chebyshev's inequality.
Bounding the expectation is done using \Cref{thm:maindensityapprox}; the main challenge in contrast to the proof of \Cref{thm:fouriercoefficients} is having to control the magnitude and the sign of $\kappa$ to obtain a lower bound on $\Expected{\T(\Gbf)}$. In fact, we must prove that $\kappa$ is a negative constant for all $q \ge 1$ such that the main correction term $(-1)^k\kappa /\sqrt{d}$ from \Cref{thm:maindensityapprox} is positive and the expectation of $\T(\Gbf)$ becomes larger than in case of a $\Gnp$. To do this we rely on the following lemmas.

%Moreover, to use our approximation to the density of $\mathbf{z}$ from \Cref{thm:maindensityapprox}, we need to know more about the cumulant $\kappa_{(1,1,1)}$ since it governs our correction term. Due to the factor of $(-1)^k = -1$ for $k=3$ in front of the correction term, we need to prove that $\kappa_{(1,1,1)}$ is negative and bounded away from $0$ such that the correction leads to an increased density close to the origin. We capture this in the following lemmas.
\begin{restatable}{lemma}{momentsfortriangleanalysis}\label{lem:momentsfortriangleanalysis}
    For a cycle $C_k = \{v_1, \ldots, v_k\}$ of length $k$ with edges $e_1, e_2, \ldots, e_k$ where $e_j = \{v_j, v_{(j+1)\text{ }\text{mod } k}\}$, we have \begin{align*}
        \kappa_{(1,1, \ldots, 1)} = (\zeta\sigma)^{-k}\Expected{\prod_{j=1}^k\gamma(e_j)} \text{ where } \gamma(\{u,v\}) \coloneqq |x_u-x_v|_C^q - \Expected{|x_u-x_v|_C^q}
    \end{align*} and where $x_{v_j}$ is an independent sample from $\text{Unif}(-\frac{1}{2}, \frac{1}{2})$ for each $v_j$.
\end{restatable}

\correlation*
% \begin{restatable}{lemma}{correlation} \label[lemma]{lem:correlation}
%     Consider a triangle $C_3$ with vertices $v_1, v_2, v_3$ and edges $e_1, e_2, e_3$. For all $L_q$-norms with $q \ge 1$, we have that  $\Expected{\gamma(e_1)\gamma(e_2)\gamma(e_3)} < 0$ where $\gamma(\{u,v\}) \coloneqq |x_u-x_v|_C^q - \Expected{|x_u-x_v|_C^q}$ and $x_{v_1}, x_{v_2}, x_{v_3} \sim \unifhalf{}$ are the positions of $v_1, v_2, v_3$ in a fixed dimension.
% \end{restatable}
The proof of the above statements is deferred to \Cref{sec:deferredtriangles} because of its high technical complexity. On a high level, \Cref{lem:correlation} relies on splitting the expected value $\Expected{\gamma(e_1)\gamma(e_2)\gamma(e_3)} = \mathbb{E}_{y}\left[\gamma(e_1)\mathbb{E}[\gamma(e_2)\gamma(e_3) \mid y]\right]$ where $y$ is the distance associated to the endpoints of $e_1$ and then employing a monotonicity argument to show that $\gamma(e_1)$ is increasing in $y$ and that $\mathbb{E}[\gamma(e_2)\gamma(e_3) \mid y]$ is decreasing in $y$. The latter argument is based on carefully controlling the derivative of $\mathbb{E}[\gamma(e_2)\gamma(e_3) \mid y]$ using the Leibnitz integral rule and suitable substitutions.

\subsection{The Expected Number of Signed Triangles}

In this section, we prove a lower bound on $\Expected{\T(\Gbf)}$. Since $\Expected{\T(\Gbf)} = \binom{n}{3}\Expected{\SW(C_3)}$ where $C_3$ is an arbitrary but fixed triangle in $\Gbf$, the following lemma, based on \Cref{thm:maindensityapprox}, suffices.
\begin{lemma}\label{lem:expectedsignedtriangle}
    Consider a $\Gbf \sim \RGG$ for any fixed $1 \le q < \infty$ and fix a triangle $C_3$ with edges $e_1, e_2,e_3$. Then, there is a constant $C > 0$ such that for all $\frac{1}{n^c} \le p \le 1- \varepsilon$ and all $n^{\gamma} \le d \le n^{\bigO{1}}$ (where $\gamma > 0$ is an arbitrarily small constant), \begin{align*}
        \Expected{\sw(C_3)} = \Expected{ (\mathds{1}(e_1) - p)(\mathds{1}(e_2) - p)(\mathds{1}(e_3) - p) } \ge \frac{Cp^3}{\sqrt{d}}.
    \end{align*}
\end{lemma} 
\begin{proof}%[Proof of \Cref{lem:expectedsignedtriangle}]
    First note that by \Cref{lem:rearrangecycle}, we have that $
        \sw(C_3) = \Expected{\mathds{1}(e_1)\mathds{1}(e_2)\mathds{1}(e_3)} - p^3,
    $ i.e., the signed weight of a triangle is simply the probability that a fixed set of vertices is a triangle minus $p^3$. Thus we are left with bounding $\Expected{\mathds{1}(e_1)\mathds{1}(e_2)\mathds{1}(e_3)}$. To this end, we apply the technique described in \Cref{sec:approxdensity}. We note that by \Cref{lem:influenceofnoise}, we have \begin{align*}
        \Expected{\mathds{1}(e_1)\mathds{1}(e_2)\mathds{1}(e_3)} 
        &= \Pr{ \bigcap_{i = 1}^3(\Deltaboldhat{}(i) \le \hattau) } 
        \ge \Pr{ \bigcap_{i = 1}^3(\mathbf{z}(i) \le \hattau) } - \bigO{\frac{\log(n)}{d^{\eta - 1}}} \\
        &\ge \Pr{ \bigcap_{i = 1}^3(\mathbf{z}(i) \le \hattau) } - \bigO{\frac{p^3}{d}} ,
    \end{align*} since $|\tilde{\tau}| \le \log(n)$ by \Cref{prop:cdfpdf} and since we can choose $\eta$ to be $\ge 1+\eps/\gamma$ such that $\frac{\log(n)}{d^{\eta - 1}} \le \frac{p^3}{d}$, as $d \ge n^{\gamma}$ and $p \ge n^{-\bigO{1}}$. Now, we use our approximation $\tilde{f}$ to the density of $\mathbf{z}$ from \Cref{thm:maindensityapprox} for which we set the parameters $\sone,\stwo,\bone,\btwo$ to be sufficiently large constants such that the approximation error satisfies \begin{align*}
        \sup_{\mathbf{x} \in \mathbb{R}^k} |f(\mathbf{x}) - \tilde{f}(\mathbf{x})|& = \bigO{ \max\left\{ \frac{1}{\sqrt{d}}, |\kappa| \right\}^{\stwo+1} \left( \frac{1}{\sqrt{d}}\right)^{(k-2)(\stwo+1)}  +\left( \frac{1}{\sqrt{d}}  \right)^{\min\{\bone-2, \btwo-2, \sone+1\}}}\\
        & = \bigO{\frac{p^3}{d}}
    \end{align*} 
    which again is possible since $d \ge n^{\gamma}$. Now we can define the interval $D = [-\log(n), \tilde{\tau}]$ and bound \begin{align*}
        \Pr{ \bigcap_{i = 1}^3 (\mathbf{z}(i) \le \hattau) } &\ge \int_{D^3} \tilde{f}(\mathbf{x}) \text{d} \mathbf{x} - \bigO{\frac{p^3\log(n)^3}{d}}
    \end{align*} where the $\bigO{p^3\log(n)^3/d}$ term comes from the fact that we have an error of $\bigO{p^3/d}$ and integrate over a set of volume $\bigO{\log^3(n)}$. Now, noting that $\kappa = c\Expected{\gamma(e_1)\gamma(e_2)\gamma(e_3)}/\sqrt{d}$ for some constant $c > 0$ by \Cref{lem:momentsfortriangleanalysis}, using our expression for $\tilde{f}$ and factorizing the integrals, we get \begin{align*}
        \int_{D^3} \tilde{f}(\mathbf{x}) \text{d} \mathbf{x} &= \int_{D^3} \left( \prod_{j=1}^3 f_{\mathbf{z}(j)}(\mathbf{x}(j)) \right) \text{d} \mathbf{x} - \frac{c\Expected{\gamma(e_1)\gamma(e_2)\gamma(e_3)}}{\sqrt{d}} \int_{D^3} \left( \prod_{j=1}^3 \phi^{(1)}(\mathbf{x}(j)) \right) \text{d} \mathbf{x} + g(\mathbf{x}) \\
        &= \left(\Pr{\mathbf{z}(1) \in D}\right)^3 - \frac{c\Expected{\gamma(e_1)\gamma(e_2)\gamma(e_3)}}{\sqrt{d}}\left(\int_{D} \phi^{(1)}(x) \text{d} x\right)^3 \\&\hspace{5cm} + \underbrace{\sum_{j=4}^{\sone\bone + \stwo\btwo} \int_{D^3}\boldsymbol{\alpha}_j^\top \left( \boldsymbol{\nabla}^{\otimes j} \phi(\mathbf{x}) \right) \text{d} \mathbf{x}}_{\eqqcolon g(\vecx)}.
    \end{align*} Choosing $\eta$ sufficiently large, we get from \Cref{lem:influenceofnoise}  that \begin{align*}
        \Pr{\mathbf{z}(1) \in D} &\ge p - \bigO{\frac{\log(n)}{d^{\eta-1}}} \ge p - \bigO{\frac{p\log(n)}{d}} = p\left( 1 - o\left(\frac{1}{\sqrt{d}}\right) \right)
    \end{align*} and further \begin{align*}
        \int_{D} \phi^{(1)}(x) \text{d} x = \phi(\hattau) - \underbrace{\phi(-\log(n))}_{= n^{-\omega(1)}} \ge Cp
    \end{align*}where we used \Cref{cor:phip}. It remains to bound the influence of the higher order correction terms captured in $g(\mathbf{x})$. 
    To this end, 
    % we expand \begin{align*}
    %         \boldsymbol{\alpha}_j^\top \left( \boldsymbol{\nabla}^{\otimes j} \phi(\mathbf{x}) \right) &\le \sum_{\substack{s = (s_1, s_2, s_3) \in \mathbb{N}^{\times 3} \\ :|s| = j}} \|\boldsymbol{\alpha}_j\|_\infty \left| \frac{\partial^j}{\partial^{s_1}x_1\partial^{s_2}x_2\partial^{s_3}x_3} \phi(\mathbf{x}) \right| = \sum_{\substack{s = (s_1, s_2, s_3) \in \mathbb{N}^{\times 3} \\ :|s| = j}} \|\boldsymbol{\alpha}_j\|_\infty |P_{s} (\mathbf{x}) | \phi(\mathbf{x}) 
    % \end{align*} 
    we recall the expansion of $\boldsymbol{\alpha}_j^\top ( \boldsymbol{\nabla}^{\otimes j} \phi(\mathbf{x}) )$ in \eqref{alphaexpansion} into a sum over polynomials $P_s$ of degree $s$
    % where each $P_s$ is a polynomial of degree $s$ 
    that can be represented as a product of Hermite polynomials in the individual $x_1, \ldots, x_k$. 
    Because $|\hattau| \le \log(n)$ by 
    \Cref{cor:tauissmall}, we get that $|P_{s} (\mathbf{x}) | = \bigO{\log^s(n)}$ for all $\mathbf{x} \in D^3$, so using the bound on $\| \boldsymbol{\alpha}_j \|_\infty$ from \Cref{thm:maindensityapprox} and $|\kappa| = \bigO{1}$, we obtain \begin{align*}
        \|\boldsymbol{\alpha}_j\|_\infty |P_{s} (\mathbf{x}) | \le  \frac{C}{\sqrt{d}}  \left(\log(n)d^{-\frac{1}{6}}\right)^j  \text{ and thus }
        g(\mathbf{x}) \le \frac{C}{\sqrt{d}} \sum_{j=4}^{s_1b_1 + s_2b_2} \left(\frac{\log(n)}{d^{1/6}} \right)^j \int_{D^3} \phi(\mathbf{x}) \text{d} \mathbf{x}.
    \end{align*} We further obtain \begin{align*}
        \int_{D^3} \phi(\mathbf{x}) \text{d} \mathbf{x} = \left( \int_{D} \phi(\mathbf{x}) \text{d} \mathbf{x} \right)^3 \le \Phi(\hattau)^3 \le Cp^3
    \end{align*} since $\Phi(\hattau) \le (1+o(1))p$ if $p = o(1)$ by \Cref{cor:phip} and $\Phi(\hattau) = \Theta(1)$ if $p = \Omega(1)$ since in this case $\hattau = \Omega(1)$ by \Cref{cor:tauissmall}. Combining this yields \begin{align*}
        g(\mathbf{x}) \le \frac{Cp^3}{\sqrt{d}} \sum_{j=4}^{\sone\bone + \stwo\btwo} \left(\frac{\log(n)}{d^{1/6}} \right)^j = o\left(\frac{p^3}{\sqrt{d}}\right)
    \end{align*} since $d^{1/6} = \omega(\log(n))$. Using this bound for $g(\vecx)$ we get \begin{align*}
        \Expected{\mathds{1}(e_1)\mathds{1}(e_2)\mathds{1}(e_3)} &\ge p^3\left( 1 - o\left(\frac{1}{\sqrt{d}}\right) \right)^3 - \frac{Cp^3}{\sqrt{d}} \Expected{\gamma(e_1)\gamma(e_2)\gamma(e_3)} + o\left(\frac{p^3}{\sqrt{d}}\right) - \bigO{\frac{p^3}{d}} - \bigO{\frac{p^3\log(n)^3}{d}} \\
        &\ge p^3 - \frac{Cp^3}{\sqrt{d}} \Expected{\gamma(e_1)\gamma(e_2)\gamma(e_3)} + o\left(\frac{p^3}{\sqrt{d}}\right).
    \end{align*} Finally, using that $\Expected{\gamma(e_1)\gamma(e_2)\gamma(e_3)}$ is a negative constant by \Cref{lem:correlation} we conclude the proof.
\end{proof}

\subsection{Bounding the Variance}

\begin{lemma}\label{lem:variance}
    Consider $\Gbf \sim \RGG$ for any fixed $1 \le q < \infty$. Then, for all $\frac{1}{n^c} \le p \le 1- \varepsilon$ and all $n^{\gamma} \le d \le n^{\bigO{1}}$, \begin{align*}
        \Var{\T(\Gbf)} = \bigO{ p^3n^3 + \frac{p^3n^3\log(n)}{\sqrt{d}} + \frac{n^4p^5\log^2(n)}{d} + \frac{n^4p^6\log(n)}{\sqrt{d}} }.
    \end{align*}
\end{lemma} \begin{proof}
We proceed as in  \cite[Appendix A]{Liu_Mohanty_Schramm_Yang_2021} and define \begin{align*}
    \barTijk{ijk} \coloneqq \Tijk{ijk} - \Expected{ \Tijk{ijk} } \text{ and } \barT(\Gbf ) \coloneqq \T(\Gbf) - \Expected{\T(\Gbf)} = \sum_{i < j< k \in [n]} \barTijk{ijk}.
\end{align*} Now expanding the variance like in \cite[Appendix A]{Liu_Mohanty_Schramm_Yang_2021} yields \begin{align*}
    \Var{ \T(\Gbf) } = \Expected{\barT(\Gbf)^2} &= \sum_{i < j < k \in [n]} \Expected{\barTijk{ijk}^2} + 2 \sum_{i < j < k < \ell \in [n]} \Expected{3\barTijk{ijk} \barTijk{jk\ell}}\\
    &= \binom{n}{3} \Expected{\barTijk{ijk}^2} + 6\binom{n}{4} \Expected{\barTijk{ijk}\barTijk{jk\ell}},
\end{align*} where we used that $\Expected{\barTijk{ijk}\barTijk{\ell r s}} = 0$ if at most two of the $i, j, k, \ell, r, s$ are the same, by independence of edges. Now, following the same calculation as in \cite[Appendix A]{Liu_Mohanty_Schramm_Yang_2021} and denoting by $\varrho = \frac{1}{p^2} \Pr{\Gbf_{12}\Gbf_{13} = 1 \mid \Gbf_{23} = 1} - 1$, we get that
\begin{align*}
    \Expectednop{\barTijk{ijk}^2} = \bigO{(1 + \varrho)p^3} \text{ and } \Expected{\barTijk{ijk}\barTijk{jk\ell}} = \bigO{p^5\varrho^2 + p^6\varrho}. 
\end{align*} Noting that by our \Cref{thm:fouriercoefficients}, we get \begin{align*}
    \varrho = \frac{1}{p^2} \Pr{\Gbf_{12}\Gbf_{13} = 1 \mid \Gbf_{23} = 1} - 1 = \frac{1}{p^3} \Pr{\Gbf_{12}\Gbf_{13}\Gbf_{23}=1} - 1 = \bigO{\frac{\log(n)}{\sqrt{d}}},
\end{align*} we can conclude the proof.
\end{proof}

\noindent With this, we prove \Cref{thm:signedtriangles}. \begin{proof}[Proof of \Cref{thm:signedtriangles}]
    Noting that $\Expectedsub{\Gbf \sim \Gnp}{ \T(\Gbf) } = 0$ and $\Varsub{\Gbf \sim \Gnp}{ \T(\Gbf) } = \binom{n}{3}p^3(1-p)^3$, and using \Cref{lem:expectedsignedtriangle} as well as \Cref{lem:variance}, \begin{align*}
        \left|\Expectedsub{\Gbf \sim \RGG}{\T(\Gbf)} - \Expectedsub{\Gbf \sim \Gnp}{\T(\Gbf)}\right| = \Omega\left( \frac{n^3p^3}{\sqrt{d}} \right)
    \end{align*} while \begin{align*}
        \max \left\{ \sqrt{\Varsub{\Gbf \sim \RGG}{\T(\Gbf)}}, \sqrt{\Varsub{\Gbf \sim \Gnp}{\T(\Gbf)}} \right\} = \bigO{\sqrt{np}^{3}}.
    \end{align*} The result follows.
\end{proof}

\section{Application 2: Improved Information-Theoretic Lower Bounds for all Densities}

Another application of \Cref{thm:fouriercoefficients} is an improved algorithmic lower bound that works for all densities $p$. This improves upon the previous best lower bound from \cite[Theorem 1.10]{Bangachev_Bresler_2024} that only holds in the case of $p = 1/2$. The bound constructed here works for all fixed $q \ge 1$ and is (up to logarithmic factors) as strong as the current best information-theoretic lower bound for spherical RGGs from \cite{liu2022testing} in the case of general $p$.\footnote{In case of $p = c/n$ for a constant $c$, \cite{liu2022testing} give a stronger bound based on a specialized argument exploiting the locally-tree-like properties of sparse RGGs. However, this analysis breaks down in the case of general $p$ as the locally-tree-like properties are lost. In this case, their lower bounds are just as strong as ours.} Formally, we prove the following.
\statisticalindistinguishability*

\noindent The proof is based on a technique introduced by Liu and Rácz \cite{Liu_Racz_2023}, which -- after employing Pinsker's inequality -- relates the total variation distance between $\RGG$ and $\Gnp$ to a sum over the moments of the self-convolution of two specific edge indicator random variables. 
It is not hard to see that this bound can be expanded as a weighted sum over the expected signed weight of complete bipartite graphs with $k$ vertices on the left, and $2$ vertices on the right side of the bipartition, denoted by $K_{2,k}$. We then rely on the following bound on the expected signed weight of a $K_{2,k}$ based on another application of \Cref{thm:fouriercoefficients}. 
\begin{restatable}{proposition}{swtwoktwo}\label{prop:swk2k}
    Assume that $d = \omega(n\log(n))$. 
    Then, for any fixed $1 \le q < \infty$, there is a constant $C > 0$ such that for any $\frac{1}{n^c}\le p \le 1-\varepsilon$ and all $2 \le k \le n$, we have \begin{align*}
        \Expectedsub{\Gbf \sim \RGG}{ \textsc{Sw}(K_{2,k}) }^{1/k} \le C\log(d)\log^2(n)p^2\sqrt{\frac{k}{d}}.
    \end{align*}
\end{restatable}
In contrast to the other applications of \Cref{thm:fouriercoefficients}, here the bound on $\Expected{ \textsc{Sw}(K_{2,k}) }$ must be valid for $k$ up to the same order of magnitude as $n$ (other applications rely on the fact that the graphs over which we compute the expected signed weight are small). 
\footnote{Technically, the random variable $\SW(H)$ is only defined for labeled graphs $H$. However, as its law (and hence its expectation) is invariant under graph isomorphism, we can make sense of the expression $\Expected{\SW(K_{2,k})}$ by fixing an arbitrary vertex labeling from $[n]$.}
However, due to the fact that a $K_{2,k}$, conditioned on the position of the two vertices on the right, is a union of $k$ chains of length $2$, \Cref{thm:fouriercoefficients} is applicable in this setting as well, after taking care of some technical precautions. 
Before proving \Cref{prop:swk2k}, we show how it helps us in proving \cref{thm:indistinguishability} combined with the bound of Liu and Rácz. Afterwards (in \cref{sec:provingktwok}), we prove \cref{prop:swk2k}.

\subsection{Using the Bound of Liu and Rácz}

To illustrate why \cref{prop:swk2k} leads to an improved algorithmic lower bound, we refer to the work of Liu and Racz, who showed in \cite[Section 3]{Liu_Racz_2023} that \begin{align*}
    \TV{\RGG}{\Gnp} \le \sum_{k=0}^{n-1} \log \left( \Expectedsub{\mathbf{x}, \mathbf{y}} {\left(1 + \frac{\gamma(\mathbf{x},\mathbf{y})}{p(1-p)} \right)^k} \right), \\ \text{ with } \gamma(\mathbf{x},\mathbf{y}) = \mathbb{E}_{\mathbf{z}}[ (\mathds{1}_e(\mathbf{x}, \mathbf{z}) - p)(\mathds{1}_e(\mathbf{y}, \mathbf{z}) - p) ],
\end{align*} where $\mathds{1}_e(\mathbf{x}, \mathbf{z})$ denotes the indicator random variable for the event that $\|\mathbf{x} - \mathbf{z}\|_q \le \tau$ and that $\mathbf{x}, \mathbf{y}, \mathbf{z} \sim \unifhalf{d}$, independently. 
Expanding the above expectation, we obtain 
\begin{align*}
    \mathbb{E}_{\mathbf{x}, \mathbf{y}} \left[ \left(1 + \frac{\gamma(\mathbf{x},\mathbf{y})}{p(1-p)} \right)^k\right]& = \sum_{j=0}^{k} \binom{k}{j} \frac{1}{{p^j(1-p)^j}} \mathbb{E}_{\mathbf{x}, \mathbf{y}}\left[ \gamma(\mathbf{x},\mathbf{y})^j \right]\\& = \sum_{j=0}^{k} \binom{k}{j} \frac{1}{{p^j(1-p)^j}} \Expected{ \textsc{Sw}(K_{2,j}) },
\end{align*} 
and hence
\begin{align*}
    \TV{\RGG}{\Gnp} &\le \sum_{k=0}^{n-1} \log \left( \mathbb{E}_{\mathbf{x}, \mathbf{y}} \left[ \left(1 + \frac{\gamma(\mathbf{x},\mathbf{y})}{p(1-p)} \right)^k\right] \right)\\
    &= \sum_{k=0}^{n-1} \log \left(1 + \sum_{j=2}^{k} \binom{k}{j} \frac{1}{{p^j(1-p)^j}} \Expected{ \textsc{Sw}(K_{2,j}) }\right).
\end{align*} 
It is very important that the inner sum above starts at $j = 2$, which is true since $\Expected{ \textsc{Sw}(K_{2,1})}$ is clearly zero as the edges are independent in this case.
With this, and the inequality $\log(1 + x) \le x$, 
\begin{align*}
    \TV{\RGG}{\Gnp} &\le \sum_{k=0}^{n-1} \log \left(1 + \sum_{j=2}^{k} \binom{k}{j} \frac{1}{{p^j(1-p)^j}} \Expected{ \textsc{Sw}(K_{2,j}) }\right)\\
    &\le \sum_{k=0}^{n-1} \sum_{j=2}^{k} \binom{k}{j} \frac{1}{{p^j(1-p)^j}} \Expected{ \textsc{Sw}(K_{2,j})}\\
   & \le n \sum_{j=2}^{n} \binom{n}{j} \frac{1}{{p^j(1-p)^j}} \Expected{ \textsc{Sw}(K_{2,j}) }.
\end{align*} Now, with \Cref{prop:swk2k}, 
\begin{align*}
    \TV{\RGG}{\Gnp} &\le n \sum_{j=2}^{n} \binom{n}{j} \frac{1}{{p^j(1-p)^j}} \left(\frac{Cp^2 \log(d)\log^2(n) \sqrt{j}}{\sqrt{d}}\right)^j \\
    &\le n \sum_{j=2}^{n} \left( \frac{ne}{j} \right)^j \frac{1}{{p^j(1-p)^j}} \left(\frac{Cp^2\log(d)\log^2(n)j}{\sqrt{d}}\right)^j \\
    &= n \sum_{j=2}^{n} \left(Ce \frac{np\log(d)\log^2(n) }{(1-p)\sqrt{d}} \right)^j \\
    &\le n\left(Ce \frac{np \log(d)\log^2(n)}{(1-p)\sqrt{d}} \right)^2 \le \frac{C' n^3p^2 \log^2(d)\log^4(n)}{d}
\end{align*} where in the penultimate step we used that the sum is geometric if $d = \omega(\log^2(n)n^2p^2)$. The whole expression becomes $o(1)$ if $d = \omega(\log^6(n)n^3p^2) = \tilde{\Omega}(n^3p^2)$ as desired. It thus only remains to prove \Cref{prop:swk2k}. 

\subsection{Proving \Cref{prop:swk2k}}\label{sec:provingktwok}
\begin{proof}[Proof of \Cref{prop:swk2k}]
    For a $K_{2,k}$ with vertices $v_1$ on one side and vertices $u_1, \ldots, u_k$ on the other, we denote the edge from $v_i$ to $u_j$ by $e_{ij}$. If we denote the position of $v_1, v_2$ by $\mathbf{x}_1, \mathbf{x}_2$, respectively, and the position of $u_1$ by $\mathbf{y}$, we get that \begin{align*}
        \Expected{\textsc{Sw}(K_{2,k}) }  = \Expected{ \prod_{j=1}^k (\mathds{1}(e_{1j}) - p)(\mathds{1}(e_{2j}) - p) } = \Expectedsub{\mathbf{x}_{1}, \mathbf{x}_{2}} { \Expectedsub{\mathbf{y}}{(\mathds{1}(e_{11}) - p)(\mathds{1}(e_{21}) - p) \mid \mathbf{x}_1, \mathbf{x}_2 }^k }.
    \end{align*} The inner expectation is now essentially the signed weight of a chain of length $k = 2$ with endpoints $v_1,v_2$ raised to the power of $k$, so we naturally wish to apply our bound from \Cref{thm:fouriercoefficients}. This bound heavily depends on $|\kappa|$, which in this case is essentially equal to the covariance between the rescaled distances associated to $e_{11}$ and $e_{12}$. However, this covariance depends on the positioning of $v_1$ and $v_2$. We shall therefore split the above expectation into two parts where the first part that represents the case in which $v_1,v_2$ are positioned such that $|\kappa|$ is small and the second part represents the remaining case. To this end, we define $\mathcal{E}$ to be the event that a chain with endpoints $v_1,v_2$ is $(\alpha, \beta)$-good (cf. \Cref{def:goodchains}) for some constant $\alpha > 0$ and \begin{align*}
        \beta \le C' \log(d/p^2)\sqrt{k/d}
    \end{align*} 
    where $C' > 0$ is a suitable constant to be fixed later. We note that 
    \begin{align*}
        \kappa = \frac{1}{d} \sum_{i = 1}^d \kappa_{(1,1)}(\mathbf{z}_i) = \frac{1}{d} \sum_{i = 1}^d \Expectedsub{\mathbf{y}}{\mathbf{z}_i(1)\mathbf{z}_i(2) \mid \mathbf{x}_1, \mathbf{x}_2} = \frac{1}{\zeta^2d} 
        \sum_{i = 1}^d\Expectedsub{\mathbf{y}}{\Deltaboldhat{i}(1)\Deltaboldhat{i}(2) \mid \mathbf{x}_1, \mathbf{x}_2 }
    \end{align*} 
    by using the decomposition from \eqref{eq:znoise} in \Cref{sec:noise} and expanding the expectation. We note that $\zeta^2 d \kappa$ is a random variable in which the randomness comes solely from $\mathbf{x}_1, \mathbf{x}_2$. Furthermore, we have that 
    \begin{align*}
        \Expectedsub{\mathbf{x}_1, \mathbf{x}_2}{\zeta^2d \kappa} = \sum_{i = 1}^d\Expectedsub{\mathbf{x}_1, \mathbf{x}_2}{\Expectedsub{\mathbf{y}}{\Deltaboldhat{i}(1)\Deltaboldhat{i}(2) \mid \mathbf{x}_1, \mathbf{x}_2 }} = 0,
    \end{align*} so $\zeta^2 d \kappa $ is a mean zero random variable and further a sum of $d$ i.i.d.~random variables, each of which is bounded almost surely by some constant $M$. Hence, applying the Bernstein bound from \Cref{thm:bernstein} yields \begin{align*}
        \Pr{|\kappa| \ge t} \le 2\exp\left( \frac{-t^2}{2M^2 d + \frac{2}{3}Mt} \right) \le 2\exp\left( \max\left\{ \frac{-t^2}{4M^2 d}, \frac{-t}{\frac{4}{3}M} \right\} \right)
    \end{align*} so setting $t = C'\log(d/p^2) \sqrt{kd}$ for a suitable constant $C' > 0$ yields that \begin{align*}
        \Pr{|\kappa| \ge t} &\le 2\exp\left( \max\left\{ \frac{-C'^2\log^2(d/p^2)k}{4M^2 }, \frac{-C'\log(d/p^2)\sqrt{kd}}{\frac{4}{3}M} \right\} \right) \le \left(\frac{p^2}{d}\right)^{2k}
    \end{align*} 
    since we assume that $\omega(n \log(n))$ and $k \le n$ so $\sqrt{kd} \ge k$. 
    This shows that $|\kappa| \le \beta$ with probability $1 - \left(\frac{p^2}{d}\right)^{2k}$. Furthermore, we get from \Cref{lem:goodchains} that with probability $1 - \exp(-\Omega(d))$, there is a constant $\alpha > 0$ such that at least $\alpha d$ out of $d$ dimensions are such that \cramer is met. Both statements together imply that there is a constant $C'' \ge 0$ such that \begin{align*}
        \Prnop{\overline{\mathcal{E}}} \le \left(\frac{p^2}{d}\right)^{2k} + e^{-C''d} \le 2\left(\frac{p^2}{d}\right)^{2k}
    \end{align*} because $d = \omega(n \log(n))$ and $k \le n$. Now, denoting the law of $\mathbf{x}_1, \mathbf{x}_2$ conditional on $\mathcal{E}$ and $\overline{\mathcal{E}}$, respectively, by $\mathcal{L}(\mathcal{E})$ and $\mathcal{L}(\overline{\mathcal{E}})$, we get \begin{align*}
        \Expected{\textsc{Sw}(K_{2,k}) }  &=  \Pr{\mathcal{E}} \Expectedsub{\mathbf{x}_{1}, \mathbf{x}_{2} \sim \mathcal{L}(\mathcal{E})} { \Expectedsub{\mathbf{y}}{(\mathds{1}(e_{11}) - p)(\mathds{1}(e_{21}) - p)}^k }\\&\hspace{4cm} + \Prnop{\overline{\mathcal{E}}}\Expectedsub{\mathbf{x}_{1}, \mathbf{x}_{2} \sim \mathcal{L}(\overline{\mathcal{E})}} { \Expectedsub{\mathbf{y}}{(\mathds{1}(e_{11}) - p)(\mathds{1}(e_{21}) - p)}^k }\\
        &\le \hspace{.75cm} \Expectedsub{\mathbf{x}_{1}, \mathbf{x}_{2} \sim \mathcal{L}(\mathcal{E})} { \Expectedsub{\mathbf{y}}{(\mathds{1}(e_{11}) - p)(\mathds{1}(e_{21}) - p)}^k } + \Prnop{\overline{\mathcal{E}}}.
    \end{align*} Applying \Cref{thm:fouriercoefficients}, yields 
    \begin{align*}
        \Expected{\textsc{Sw}(K_{2,k}) } \le \left( C'' p^2 \left( \log(d/p^2)\sqrt{\frac{k}{d}} + \frac{\log^3(n)}{\sqrt{d}} \right)\right)^{k} + 2\left(\frac{p^2}{d}\right)^{2k} \le \left(C\log(d)\log^2(n)p^2\sqrt{\frac{k}{d}}\right)^k
    \end{align*} as desired.
\end{proof}

\section{Application 3: Eigenvalues and the Spectral Threshold}

We use our bound on Fourier coefficients of cycles and chains to prove tight bounds on the spectral gap of a typical $\Gbf \sim \RGG$. Our main result, which we prove in this section, is as follows. 
\spectralstuff*
\begin{proof}
    The statement follows from the upper bound from \Cref{lem:spectralupperbound} and the lower bound from \Cref{lem:spectrallowerbound} in case $d \ll np$, proved in the following subsections. For $d = \Omega(np)$, the lower bound of $C\sqrt{np}$ on $\lambda_2(\Adj)$ follows from the generalized Alon--Boppana bound from \cite[Theorem 8]{Jiang2019} using $r = 1$.
\end{proof}

\subsection{Upper Bound via the Trace Method}\label{sec:upperboundviathetracemethod}

In this section, we prove the following upper bound on the second largest eigenvalue of the adjacency matrix $\Adj$ of a $\Gbf \sim \RGG$. 
\begin{lemma}\label{lem:spectralupperbound}
    Consider a $\Gbf \sim \RGGp{n}{d}{p}$ for any $\frac{1}{n} \le p \le 1 - \varepsilon$. With probability $1 - o(1)$, \begin{align*}
        \max\{ |\lambda_2(\Adj)|, |\lambda_n(\Adj)| \} \le n^{o(1)} \max\left\{ \sqrt{np}, \frac{np}{\sqrt{d}} \right\}
    \end{align*} where $\lambda_1(\Adj) \ge \lambda_2(\Adj) \ge \ldots \ge \lambda_n(\Adj)$ are the eigenvalues of the adjacency matrix $\Adj$ of a $\Gbf$. Moreover, if $d \ll np^{1 + \varepsilon}$ for any $\varepsilon > 0$, the number of eigenvalues $\lambda_{i}(\Adj)$ with $|\lambda_i(\Adj)| \ge \frac{np}{a\sqrt{d}}$ is at most $\bigOtildenop{da^{\frac{2}{\varepsilon} + 4}}$.
\end{lemma}

The upper bound of \cref{thm:spectralstuff} and \Cref{thm:spectralstuffinfty} is proved via the trace method, which we briefly sketch here. As the name suggests, it can be used to prove bounds on the largest eigenvalue of a matrix via its trace. Since we are interested in the second largest eigenvalue, we consider the centered adjacency matrix $\Adjp = \Adj - p \mathds{1}\mathds{1}^\top$ where $\mathds{1}$ is the $n$-dimensional all-ones vector. This subtracts the eigenspace associated to the largest eigenvector of $\Adj$ (which is $\mathds{1}$). To bound the largest eigenvector of $\Adjpp$, we first recall the well-known relation 
\begin{align*}
    \tr{\Adjpp} = \sum_{i = 1}^n \lambda_i\left( \Adjpp \right).
\end{align*} 
Considering the $m$-th power of $\Adjpp$, we see that 
\begin{align*}
    \tr{\Adjppt{m}} =  \sum_{i = 1}^n \lambda_i\left( \Adjpp \right)^m \text{, so }
    \left|\lambda_1\left( \Adjpp \right)\right|^m \le \tr{\Adjppt{m}}
\end{align*} for every even $m$. By Markov's inequality, we have that for every $a>0$, \begin{align*}
    \Pr{ \left|\lambda_1\left( \Adjpp \right)\right|^m \ge a \Expected{ \tr{\Adjppt{m} }}} \le \frac{1}{a}.
\end{align*} 
Choosing any $a = \omega(1)$, it follows that 
\begin{align}
    \Pr{\left|\lambda_1\left( \Adjpp \right)\right| \ge a^{1/m} \Expected{ \tr{\Adjppt{m} }}^{1/m}}\leq 1/a = o(1)
\end{align} 
and we are left with bounding the above expectation. Conveniently, $\tr{\Adjppt{m}}$ has a combinatorial interpretation.
Namely, it can be expressed as a sum over all directed, closed walks of length $m$ on the complete graph with vertex set $[n]$, where each term in the sum corresponds to the product of the entries of $\Adjpp$ along the edges of the walk.
Formally, writing $\EW(n, m) \coloneqq \{(v_1, \dots, v_{m + 1}) \in [n]^{m + 1} \mid v_1 = v_{m+1}\}$ for the set of closed walks of length $m$, we have
\begin{align} \label{eq:traceAsWalks} 
    \tr{\Adjppt{m} }  = \sum_{(v_1, \dots, v_{m + 1}) \in \EW(n, m)} \prod_{j = 1}^{m} \big(\Adj_{v_j, v_{j+1}} - p \big) .
\end{align}
% We proceed in the next subsection with deriving an upper bound on the expectation of the right-hand side of \eqref{eq:traceAsWalks}, by representing it as signed weights of (multi-)graphs and analyzing the corresponding signed weights.
We proceed by deriving an upper bound on the expectation of the right-hand side of \eqref{eq:traceAsWalks}, by representing it as signed weights of (multi-)graphs and analyzing the corresponding signed weights in the following subsection. 

\paragraph{Bounding the signed weight of Eulerian multigraphs}\label{sec:fouriereuler}

We give a graphical interpretation of the right-hand side of \eqref{eq:traceAsWalks}.
To this end, we extend our definition of the signed weight of a graph to (undirected) multigraphs (in which there can be multiple edges between any pair of vertices).
Formally, given a multigraph $H$ with vertex set $V(H) \subseteq [n]$ and multiset of edges $E(H)$ over the universe $\binom{[V(H)]}{2}$, we define the signed weight of $H$ as the random variable
\[
    \SW(H) \coloneqq \prod_{e \in E(H)} (\ind{e} - p) = \prod_{e \in E(H)} (\Adj_e - p),
\]
where the product respects multiplicities of edges in $H$, and, for any edge $e = \{v, w\} \in \binom{[n]}{2}$, we write $\Adj_e = \Adj_{v, w}$.
Moreover, it holds that $H_{\pmb{v}}$ has at most $m$ edges in total.

Next, given a walk $\pmb{v} = (v_1, \dots, v_{m+1}) \in \EW(n, m)$, we define $H_{\pmb{v}}$ to be the multigraph on the vertex set $\{v_1, \dots, v_{m + 1}\}$ (duplications of vertices are removed), where we add an edge $\{v_j, v_{j+1}\}$ for every $j \in [m]$ with $v_j \neq v_{j+1}$ to the multiset $E(H_{\pmb{v}})$ (we preserve multiplicities but we do not allow self loops).
Note that the resulting multigraph $H$ is Eulerian, in the sense that $H$ is connected, and all vertices have even degree. We proceed by observing that, for every $\pmb{v} = (v_1, \dots, v_{m+1}) \in \EW(n, m)$, it holds that
$$
    \absolute{\Expected{\prod_{j = 1}^{m} \big(\Adj_{v_j, v_{j+1}} - p \big)}} = p^{m - |E(H_{\pmb{v}})|}  \absolute{\Expected{\SW(H_{\pmb{v}})}}
    \le \absolute{\Expected{\SW(H_{\pmb{v}})}},
$$
where we use that every self loop in the walk $\pmb{v}$ contributes an additional factor of $-p$ compared to $\Expected{\SW(H_{\pmb{v}})}$.
Hence, recalling \eqref{eq:traceAsWalks}, we note that 
\begin{align}\label{eq:traceAsSignedWeights}
    \Expected{ \tr{\Adjppt{m} } } \le 
    \sum_{\pmb{v} \in \EW(n, m)} \absolute{\Expected{\SW(H_{\pmb{v}})}}.
\end{align}
\noindent To bound the right-hand side of the above inequality, we employ a strategy similar to the one used in~\cite{Li_Schramm_2024} for the Gaussian mixture block model. That is, we bound the expected signed weight of a given Eulerian graph $H$ with $\ell \le m$ edges and no self-loops by considering edges within cycles and chains and edges outside of such structures separately. Precisely, we repeatedly remove cycles from $H$ and afterwards contract all chains into a single edge, which leaves us with a (multi)graph $\Hardcore$ called the \emph{core} of $H$. The core $\Hardcore$ is still Eulerian but now has minimum degree $4$, which limits the number of vertices it can contain to $m/2$. Now, conditional on the positions of vertices in $\Hardcore$, any removed cycle or contracted chain appears independently and its contribution to $\Expected{\SW(H)}$ can be bounded using \Cref{thm:fouriercoefficients}. Edges in $\Hardcore$ only have a limited contribution because the number of vertices in $\Hardcore$ is limited by $m/2$, which in turn limits the number of possible embeddings of a such $\Hardcore$ in $K_n$. We refer the reader to \Cref{sec:spectralexplanation} for an intuitive explanation of our argument and proceed by making our ideas formal.

\paragraph{Constructing the Core $\Hardcore$.} 
% We start by defining what we mean by a chain/a cycle of vertices that can be contracted.
The following definitions formalize which chains/cycles of vertices can be contracted/removed.
 \begin{definition}[Chains]
    Consider a Eulerian multigraph $H$ with $\ell$ edges. For every $k \ge 3$, we call a sequence $v_1, v_2, \ldots, v_k$ of \emph{distinct} vertices from $V(H)$ a \emph{chain} if $\{v_i, v_{i+1}\} \in E(H)$ for all $i \in [k-1]$, and if $\deg(v_i) = 2$ for all $2 \le i \le k-1$.  
\end{definition}
\begin{definition}[Cycles]
    Consider a Eulerian multigraph $H$ with $\ell$ edges. We call a sequence $v_1, v_2, \ldots, v_{k+1}$ of vertices a \emph{cycle} if $k\ge 3$, if $v_1 = v_{k+1}$, if $\{v_1,\ldots, v_{k}\}$ contains no duplicates, if $\{v_i, v_{i+1}\} \in E(H)$ for all $i \in [k-1]$, and if $\deg(v_i) = 2$ for all $i \in \{2, \ldots, k\}$. In particular, $v_1$ is allowed to have degree $\ge 2$ and is referred to as the \emph{anchor} of our cycle. Further, if $v$ is a vertex which only has a single neighbor $u$ (possibly there are multiple edges between $u$ and $v$), $u$ and $v$ are a cycle of length $k = 2$ and $u$ is the anchor. Cycles of length $k = 2$ are called \emph{degenerate}. 
\end{definition}

\noindent The algorithm that transforms an $\ell$-edge Eulerian multigraph $H$ into its core $\Hardcore$ works as follows. It begins by iteratively removing all cycles in $H$. That is, while there still exists a cycle in $H$, it removes all vertices on said cycle except for the anchor.\footnote{If there is no clear anchor, i.e., if all vertices have degree $2$, choose an arbitrary vertex on the cycle as the anchor.} After all cycles have been removed, the algorithm contracts chains: while there still exists chain in $H$, replace all intermediate vertices by a single edge between the two endpoints of the chain. The result of this procedure is the core $\Hardcore$.  

It is immediate that $\Hardcore$ is a Eulerian multigraph, which is either a single isolated vertex (which happens if all edges were removed while contracting cycles) or a multigraph which contains at least $2$ vertices and has minimum degree $4$. We split the set of edges of $\Hardcore$ into two disjoint parts $E_U$ and $E_C$,
where $E_U$ is the set of \emph{non-contracted} edges, i.e., those that are both in $\Hardcore$ and $H$, and $E_C$ is the set of edges created by contracting a chain into a single edge. We further denote by $L$ (which stands for \emph{loops}) the set of edges of $H$ on cycles that were removed in step 1 of the above algorithm. Moreover, we denote by $U$ the graph induced by all non-contracted edges, i.e., the graph obtained from $\Hardcore$ after deleting all edges not in $E_U$ and afterwards removing all isolated vertices. We can further relate the number of vertices of the core $\Hardcore$ and $H$ itself using the following straightforward combinatorial claim. 
\begin{lemma}\label{lem:vertices}
    In the setting above, it holds that 
    $
        |V(H)| = |E(H)| - |L| - |E(\Hardcore)| + |V(\Hardcore)|.
    $ 
    Moreover, $|V(\Hardcore)| \le \max\{ |E(\Hardcore)|/2, 1 \}$.
\end{lemma} 

Furthermore, when summing over the signed weight of all multigraphs $H$ corresponding to Eulerian walks as in \eqref{eq:traceAsSignedWeights}, we shall distinguish between multigraphs $H$ in which all edges are part of a cycle that can be contracted, and those for which this is not the case. To this end, we call the core $\Hardcore$ associated to a given $H$ \emph{trivial} if it only contains a single vertex, and \emph{non-trivial} otherwise. 

It is easy to see that a given $H$ has a trivial core if and only if all of its edges were removed after step 1 of the above algorithm, i.e., if and only if all its edges are part of some cycle in $L$. For this case, we define the \emph{skeleton} $\widetilde{H}$ of $H$ as the version of $H$ were all edges in degenerate cycles on vertices $u,v$ are replaced by exactly two edges between $u,v$. It is furthermore easy to see that $|\Expected{\SW(H)}| \le |\Expected{\SW(\widetilde{H})}|$ since \begin{align*}
    |\Expected{\SW(H)}| = \prod_{e \in L} |\Expected{\SW(C_e)}|
\end{align*} due to independence of the removed cycles, and since the signed weight of any degenerate cycle is at most as large as the signed weight of a degenerate cycle with exactly two edges. This observation allows us to simplify some of the following calculations as accounting for the correct number of edges becomes easier if we know that each degenerate cycle contains exactly two edges.

\paragraph{Bounding $\Expected{\SW(H)}$.}
For every contracted edge $e \in E_C(\Hardcore)$ we denote by $C_e = (v_1, \ldots, v_k)$ the chain in $H$ that was contracted into $e$. 
Similarly, for $e \in L$, we denote by $\SW(e)$ the signed weight of the contracted cycle represented by $e$. 
For every $e \in E_C$, we further denote by $\kappa_e$ the average first mixed cumulant across all $\mathbf{z}_i$ associated to a $\mathbf{z}$ (cf. \Cref{sec:concrete}) that represents the chain $C_e$ as defined in \Cref{thm:maindensityapprox}. 
Using these notations, we can rewrite \begin{align*}
    \Expected{\SW(H)} & = \Expected{\left( \prod_{e\in E_U} (\mathds{1}(e) - p) \right) \left( \prod_{\tilde{e} \in E_C \cup L} \hspace{.1cm} \prod_{e \in C_{\tilde{e}}} (\mathds{1}(e) - p) \right) }\\& = \Expected{\left( \prod_{e \in E_U} (\mathds{1}(e) - p) \right) \left( \prod_{e \in E_C \cup L} \hspace{.1cm} \SW(C_e) \right) }.
\end{align*} 
Now the plan is to use a high-probability event $\mathcal{G}$ 
defined in terms of the positions of the vertices in $\Hardcore$ such that---conditional on $\mathcal{G}$---the contribution of the second factor above is small and can be bounded using \Cref{thm:fouriercoefficients}. 
The statement of \Cref{thm:fouriercoefficients} directly suggests to define $\mathcal{G}$ as the event that $\kappa_e$ is small for all pairs of vertices in $\Hardcore$. 
In fact, the $\kappa_e$ associated to our contracted chains are typically\footnote{Typically here is to be understood w.r.t.~the randomness over the positions of the vertices in $U$.} and even with high probability of order $\bigOtildenop{1/\sqrt{d}}$ and thus much smaller than a worst-case placement of the vertices of $U$ would imply. This suggests defining $\mathcal{G}$ as \begin{align*}
    \mathcal{G} \coloneqq \bigcap_{e \in E_C} \left\{ |\kappa_e| \le \frac{\log^2(n)}{\sqrt{d}} \right\}.
\end{align*}

\noindent Using standard Bernstein bounds, we further conclude that $G$ happens with high probability. \begin{lemma} \label{lem:goodevent}
    We have $\Prnop{ \overline{\mathcal{G}}} \le \ell^2n^{-\Omega(\log(n))}.$
\end{lemma} \begin{proof}
    By \Cref{lem:chainshavesmallkappa}, we get that for a single $\kappa_e$, \begin{align*}
        \Prsub{ \mathbf{x}_1, \mathbf{x}_{k+1} }{|\kappa_e| \ge \frac{\log^2(n)}{\sqrt{d}}} \le n^{-\Omega(\log(n))}
    \end{align*}
    Doing a union bound over all $\le\ell^2$ pairs of vertices concludes the proof.
\end{proof}

\noindent Conditioning on $\mathcal{G}$ allows us to use the bounds from \Cref{thm:fouriercoefficients} in essentially their strongest possible form. 
However, \Cref{thm:fouriercoefficients} only allows us to bound the contribution of contracted chains and cycles; to handle the case of general $p$, we must also control the influence of the non-contracted edges in $E_U$. 
Our approach here is to note that---if $p$ is small---most edges in $E_U$ are not present, and every non-present edge only contributes a factor of magnitude $p$ to the product in $\SW(H)$. 
Hence $\SW(U)$ (where we recall that $U$ is the graph induced by all non-contracted edges of $H$, i.e., those in $E_U$) is typically of order $p^{t}$ if $t$ is the number of non-edges in $U$.  

To make this idea formal, we shall refine the \say{good} event $\mathcal{G}$ further.
To this end, we fix a suitable spanning subgraph $F$ of $U$. Precisely, $F$ is just a spanning forest of $U$, plus an arbitrary additional edge from $E_U$ closing a cycle whenever this is possible.

Denote the number of edges in $F$ by $t$ and define for every $i \in \{0, 1, \ldots, t\}$ the event $\mathcal{E}_i$ to be the event that the vertices of $\Hardcore$ are arranged such that exactly $i$ edges in $E(F)$ are present. We further denote by $\mathbf{X} = (\mathbf{x}_1, \mathbf{x}_2, \ldots, \mathbf{x}_{|V(\Hardcore)|})$ and $\overline{\mathbf{X}} = (\mathbf{x}_1, \mathbf{x}_2, \ldots, \mathbf{x}_{|V(H) \setminus V(\Hardcore)|})$ the positions of the vertices in $\Hardcore$ and outside of $\Hardcore$, respectively. Moreover, let $\law{\Xbf} = \unifhalf{|V(\Hardcore)|}$ and $\law{\overline{\Xbf}} = \unifhalf{|V(\Hardcore) \setminus V(H)|}$ denote the (uniform) law of $\Xbf$ and $\overline{\Xbf}$, respectively. %With this, we prove the following. 

\begin{lemma}[Bounding the Contribution of Non-Contracted Edges]\label{lem:uncontracted}
    Fix a spanning subgraph $F$ of $U$ as described above and denote by $t$ its number of edges. Further denote by $\law{ \Xbf \mid \mathcal{G} \cap \mathcal{E}_j}$ the law of $\Xbf$ conditional on $\mathcal{G} \cap \mathcal{E}_j$. Then
    \begin{align*}
        \absolute{\Expected{\SW(H)}} &= \absolute{\Expectedsub{(\Xbf, \overline{\Xbf}) \sim \law{\Xbf} \otimes \law{\overline{\Xbf}}}{\SW(H)}} 
        \\&\le 2p^t \sum_{j=0}^{t} \ell^j \Expectedsub{ \mathbf{X} \sim \law{ \Xbf \mid \mathcal{G} \cap \mathcal{E}_j}}{ \absolute{ \Expectedsub{\overline{\Xbf} \sim \law{\overline{\Xbf}}}{  \left. \prod_{e \in E_C \cup L} \SW(C_e) \hspace{.1cm} \;\right|\; \Xbf } } } + \Prnop{\overline{\mathcal{G}}}.
    \end{align*}
\end{lemma} \begin{proof}
First of all, note that
\begin{align*}
    \absolute{\Expectedsub{(\Xbf, \overline{\Xbf}) \sim \law{\Xbf} \otimes \law{\overline{\Xbf}}}{\SW(H)}} &\le \hspace{.6cm} \absolute{\Expectedsub{(\Xbf, \overline{\Xbf}) \sim \law{\Xbf} \otimes \law{\overline{\Xbf}}}{\SW(H) \mathds{1}(\mathcal{G}) }} \hspace{.7cm} + \Prnop{\overline{\mathcal{G}}}\\ 
    &\le \sum_{j=0}^{t} \absolute{\Expectedsub{(\Xbf, \overline{\Xbf}) \sim \law{\Xbf} \otimes \law{\overline{\Xbf}}}{\SW(H) \mathds{1}(\mathcal{G} \cap \mathcal{E}_j ) }} + \Prnop{\overline{\mathcal{G}}}.
\end{align*} Now, since $\mathcal{G}$ and $\mathcal{E}_j$ are defined only in terms of $\Xbf$ but not in terms of $\overline{\Xbf}$, we get \begin{align*}
    &\absolute{\Expectedsub{(\Xbf, \overline{\Xbf}) \sim \law{\Xbf} \otimes \law{\overline{\Xbf}}}{\SW(H) \mathds{1}(\mathcal{G} \cap \mathcal{E}_j ) }}
    \\&\hspace{1cm}= \Pr{\mathcal{G} \cap \mathcal{E}_j} \absolute{ \Expectedsubnop{ \mathbf{X} \sim \law{ \Xbf \mid \mathcal{G} \cap \mathcal{E}_j}}{ \hspace{.2cm} \Expectedsub{\overline{\Xbf} \sim \law{\overline{\Xbf}}}{ \SW(H) }} } \\
    &\hspace{1cm}= \Pr{\mathcal{G} \cap \mathcal{E}_j} \absolute{ \Expectedsubnop{ \mathbf{X} \sim \law{ \Xbf \mid \mathcal{G} \cap \mathcal{E}_j}}{ \hspace{.2cm} \Expectedsub{\overline{\Xbf} \sim \law{\overline{\Xbf}}}{ \left( \prod_{e \in E_U} (\mathds{1}(e) - p) \right) \left( \prod_{e \in E_C \cup L} \hspace{.1cm} \SW(C_e) \right) }}}\\
    &\hspace{1cm}= \Pr{\mathcal{G} \cap \mathcal{E}_j} \absolute{ \Expectedsub{ \mathbf{X} \sim \law{ \Xbf \mid \mathcal{G} \cap \mathcal{E}_j}}{ \prod_{e \in E_U} (\mathds{1}(e) - p) \Expectedsub{\overline{\Xbf} \sim \law{\overline{\Xbf}}}{ \left. \prod_{e \in E_C \cup L} \hspace{.1cm} \SW(C_e) \hspace{.1cm}\right| \hspace{.05cm} \Xbf }}}
\end{align*}
\begin{align*}
    &\hspace{1cm}\le \Pr{\mathcal{G} \cap \mathcal{E}_j}  \Expectedsub{ \mathbf{X} \sim \law{ \Xbf \mid \mathcal{G} \cap \mathcal{E}_j}}{ \absolute{ \prod_{e \in E_U} (\mathds{1}(e) - p)}\cdot \absolute{ \Expectedsub{\overline{\Xbf} \sim \law{\overline{\Xbf}}}{ \left. \prod_{e \in E_C \cup L} \hspace{.1cm} \SW(C_e) \hspace{.1cm}\right| \hspace{.05cm} \Xbf } } }\\
    &\hspace{1cm}\le \Pr{\mathcal{E}_j} p^{t-j} \Expectedsub{ \mathbf{X} \sim \law{ \Xbf \mid \mathcal{G} \cap \mathcal{E}_j}}{ \absolute{ \Expectedsub{\overline{\Xbf} \sim \law{\overline{\Xbf}}}{ \left. \prod_{e \in E_C \cup L} \hspace{.1cm} \SW(C_e) \hspace{.1cm}\right| \hspace{.05cm} \Xbf } } }
\end{align*} because $\Pr{\mathcal{G} \cap \mathcal{E}_j} \le \Pr{\mathcal{E}_j}$ and since for any $\Xbf$ in $\mathcal{G} \cap \mathcal{E}_j$, there are exactly $j$ edges in our spanning subgraph $F$ of $U$ (whose edges exists as a function of $\Xbf$), implying that $\absolute{\prod_{e \in E_U} (\mathds{1}(e) - p)} \le p^{t-j}$.
Regarding the probability of $\mathcal{E}_j$, we note that since $F$ is a forest plus a single edge closing a cycle, the edges outside the cycle appear independently of the edges within the cycle. Furthermore, for edges within the cycle, we can just apply \Cref{thm:fouriercoefficients} combined with \Cref{lem:rearrangecycle} and basic inclusion-exclusion to recover the probability that any subset of edges is present or not present. Carrying out this calculation yields that the number of edges present in $F$ -- denoted by $X_F$ -- is still very close to being binomial in the sense that 
\begin{align*}
    \Pr{\mathcal{E}_j} \le 2\binom{t}{j} p^j(1-p)^{t-j} \le 2(pt)^j \le 2(p\ell)^j,
\end{align*} 
and the result follows.
\end{proof}

\noindent Now, we can use \Cref{thm:fouriercoefficients} to bound the inner expectation above. In the following, we will denote the number of cycles contracted while constructing $\Hardcore$ by $s$, the number of contracted chains by $c$, and the number of edges in $\Hardcore$ by $r$. We further denote the number of degenerate cycles removed from $H$ by $s_d$ and the number of non-degenerate cycles by $s_n$, such that $s = s_d + s_n$. We recall that a cycle is degenerate if it consists of only two vertices, and that $\Hardcore$ is called trivial if it only consists of a single vertex. We proceed by bounding the signed weight of multigraphs $H$ with trivial and non-trivial core separately.
\begin{lemma}[Bounding the signed weight of $H$ with trivial core]\label{lem:swtrivial}
    Consider a Eulerian multigraph $H$ such that the core $\Hardcore$ is trivial and its skeleton $\widetilde{H}$ has $\ell$ edges. Assume further that exactly $s$ cycles where removed during construction of $\Hardcore$ ($|L| = s$) out of which $a$ are degenerate. 
    Then, there is a constant $C > 0$ such that \begin{align*}
        \Expected{ \SW(H) } \le C^{s}p^{\ell - a} \left(\frac{\log(n)}{\sqrt{d}}\right)^{\ell - 2s}.
    \end{align*}
\end{lemma}\begin{proof}
    It is easy to see that for a $H$ with trivial core all edges are part of some removed cycle. Furthermore, all removed cycles are arranged in a tree-like fashion and therefore independent of each other. Denoting by $L_d$ and $L_n$ the set of removed degenerate and non-degenerate cycles, respectively, we immediately get from \Cref{thm:fouriercoefficients} that \begin{align*}
        \Expected{\SW(H)} \le \Expected{\SW(\widetilde{H})} \le p^{a} \prod_{e \in L_n} C p^{|C_e|} \left( \frac{\log(n)}{\sqrt{d}} \right)^{|C_e|-2}
    \end{align*}
    where we accounted for the contribution of degenerate cycles using the factor of $p^{a}$ since two edges between two vertices $u,v$ in $\widetilde{H}$ that form a degenerate cycle refer to the same underlying edge in $\Gbf \sim \RGG$, implying that the associated signed weight is $p(1-p)^{2} + p^{2}(1-p) \le p$. Summing together the respective exponents then finishes the proof.
\end{proof}
\noindent For non-trivial $H$, we prove a slightly different bound using the following two lemmas.
\begin{lemma}[Bounding the contribution of contracted edges]\label{lem:swgood}
    Consider a Eulerian multigraph $H$ with $\ell$ edges such that the core $\Hardcore$ contains $r$ edges ($|E(\Hardcore)| = r$), there are $c$ edges obtained from contracting exactly $c$ chains ($|E_C| = c$), and such that $s$ cycles where removed during construction of $\Hardcore$ ($|L| = s$) out of which $s_d$ are degenerate. 
    Then, for any $\Xbf$ consistent with $\mathcal{G}$, there is a constant $C > 0$ such that \begin{align*}
        \Expectedsub{\overline{\Xbf} \sim \law{\overline{\Xbf}}  }{ \left. \prod_{e \in E_C \cup L} \hspace{.1cm} \SW(C_e) \hspace{.1cm}\right| \hspace{.05cm} \Xbf }  \le (C\log^{2}(n))^{c} p^{\ell - r + c - s} \left(\frac{\log(n)}{\sqrt{d}}\right)^{\ell - r - 2s}.
    \end{align*}
\end{lemma} \begin{proof}
    Note that exactly $\ell - r + c$ edges in $H$ are part of a cycle or a chain that was contracted. Since we assumed that the vertices in $\Hardcore$ are positioned such that $\mathcal{G}$ is met (i.e. since we assumed that $\Xbf$ is consistent with $\mathcal{G}$), we can apply \Cref{thm:fouriercoefficients} with $|\kappa| \le \log^2(n)/\sqrt{d}$. Using further that each chain and cycle represented by some $e \in E_C \cup L$ is independent of all other chains and cycles by the way we constructed $\Hardcore$, we can apply \Cref{thm:fouriercoefficients} to each $C_e$ separately and obtain
    \begin{align*}
        \Expectedsub{\overline{\Xbf} \sim \law{\overline{\Xbf}}  }{ \left. \prod_{e \in E_C \cup L} \hspace{.1cm} \SW(C_e) \hspace{.1cm}\right| \hspace{.05cm} \Xbf } &\le \prod_{e \in E_C} C p^{|C_e|} \log^2(n) \left(\frac{\log(n)}{\sqrt{d}}\right)^{|C_e|-1} \times \prod_{e \in L} C p^{|C_e|-1} \left(\frac{\log(n)}{\sqrt{d}}\right)^{|C_e| - 2}\\
        &\le (C\log^{2}(n))^{c} p^{\ell - r + c - s} \left(\frac{\log(n)}{\sqrt{d}}\right)^{\ell - r - 2s}.
    \end{align*} 
    Notice that we used $p^{|C_e|-1}$ for $e \in L$ above, instead of $p^{|C_e|}$ as \Cref{thm:fouriercoefficients} would allow for $k \ge 3$. The reason for this is that that we also have to account for cycles of length $k = 2$, where a single vertex $v$ with only one neighbor $u$ and $a \ge 2$ edges between $u,v$ is removed as a (degenerate) cycle. In this case only a single factor of $p$ is allowed instead of $p^{2}$, since the $a$ edges between $u,v$ refer to the same underlying edge in $\Gbf \sim \RGG$, implying that the associated signed weight is $p(1-p)^{a} + p^{a}(1-p) \le p(1-p)^{2} + p^{2}(1-p) \le p$. Subtracting $1$ in the exponent accounts for this in all cases. 
\end{proof}
\noindent Using this in combination with \Cref{lem:uncontracted}, we get the following general bound on $\Expected{\SW(H)}$ for $H$ with non-trivial core. 
\begin{lemma}[The signed weight of non-trivial Eulerian multigraphs]\label{lem:swboundeuler}
    Assume that $H$ is a Eulerian multigraph with $\ell$ edges and non-trivial core $\Hardcore$. Assume further that $s$ cycles were contracted when constructing $\Hardcore$, and that $ |E(\Hardcore)| = r$ and $|V(\Hardcore)| = v$. Then,
    \begin{align*}
        \absolute{\Expected{\SW(H)}} &\le  (C\ell\log^{2}(n))^{\ell} p^{s +  v} \left(\frac{p\log(n)}{\sqrt{d}}\right)^{\ell - r - 2s}
    \end{align*}
\end{lemma} \begin{proof}
    Using \Cref{lem:uncontracted} and \Cref{lem:swgood}  yields that for any $H \in \mathcal{H}(\ell, r, s, v)$, we have \begin{align*}
        \absolute{\Expected{\SW(H)}} &\le 2p^t \sum_{j=0}^{t} \ell^j \Expectedsub{ \mathbf{X} \sim \law{ \Xbf \mid \mathcal{G} \cap \mathcal{E}_j}}{ \absolute{ \Expectedsub{\overline{\Xbf} \sim \law{\overline{\Xbf}}}{  \left. \prod_{e \in E_C \cup L} \SW(C_e) \hspace{.1cm} \;\right|\; \Xbf } } } + \Prnop{\overline{\mathcal{G}}}\\ 
        &\le 2\ell^{t} (C\log^{2}(n))^{\ell} p^{\ell - r + c - s +  t} \left(\frac{\log(n)}{\sqrt{d}}\right)^{\ell - r - 2s} +  \ell^2n^{-\Omega(\log(n))}\\ 
        &\le 3\ell^{t} (C\log^{2}(n))^{\ell} p^{c + s +  t} \left(\frac{p\log(n)}{\sqrt{d}}\right)^{\ell - r - 2s},
\end{align*} 
where $t$ is the number of edges in a suitable spanning subgraph $F$ of $U = \Hardcore[E_U]$ and $c = |E_C|$ is the number of chains we contracted while constructing $\Hardcore$. To finish the proof, we claim that $c+ t \ge v$ and we use the following case distinction to prove this. In case 1, $U$ contains at least two components, so we know that $c + t \ge v$ since all (say $C(U)$) components of $U$ are joined together by $\ge C(U)$ edges from $E_C$, otherwise $\Hardcore$ would be a tree and hence not Eulerian. In case 2, $U$ only has one component. Here, we consider the following two sub-cases. Either $U$ itself is a tree (possibly containing some multi-edges), in which case $t = v-1$. But we also know that $c \ge 1$ as otherwise contracting cycles would have eliminated all edges and $\Hardcore$ would be an isolated vertex, contradicting the assumption that $\Hardcore$ is non-trivial. This implies $t + c \ge v$ as desired. If $U$ is not a tree, then $F$ can be constructed as a spanning tree plus an additional edge, implying that $t \ge v$ as well. Hence, we conclude that \begin{align*}
    \absolute{\Expected{\SW(H)}} &\le 3\ell^{t} (C\log^{2}(n))^{\ell} p^{s +  v} \left(\frac{p\log(n)}{\sqrt{d}}\right)^{\ell - r - 2s},
\end{align*} as desired.
\end{proof}

\paragraph{Putting everything together.}

We prove the following bound on $\Expected{ \tr{\Adjppt{m} } }$. 

\begin{lemma}[Bounding the $m$-th power of the centered adjacency matrix]\label{lem:tracingthetrace}
    There is a constant $C>0$ such that for all even $m \in \mathbb{N}$, $m \ge 4$, and $d \ll np \log(n)$, \begin{align*}
        \Expected{ \tr{\Adjppt{m} } } \le \begin{cases} (Cm)^{4m} \left( d + \log(n)^{2m + 4} \right) \left(\frac{np\log(n)}{\sqrt{d}}\right)^{m} + (Cm)^{3m+1}n \sqrt{np}^m & \text{if } d \ll np \log^2(n)\\
        (Cm)^{4m+2} \left( n + \log(n)^{2m} \right) \sqrt{np}^{m} &\text{if } d \gg np\log^2(n).
        \end{cases}
    \end{align*}
    In particular, for every fixed $m$, \begin{align*}
        \Expected{ \tr{\Adjppt{m} } } = \begin{cases} \bigO{d \left( \frac{np \log(n)}{\sqrt{d}}
        \right)^m + n \sqrt{np}^m} & \text{if } d \ll np \log^2(n)\\
            \bigO{n \sqrt{np}^m} &\text{if } d \gg np\log^2(n).
        \end{cases}
    \end{align*}
\end{lemma} 
\begin{proof}
We denote by $\mathcal{H}(\ell, r, s, v)$ the set of all Eulerian multigraphs $H$ with $V(H) \subseteq [n]$ such that $|E(H)| = \ell$, $|V(\Hardcore)| = v$, $|L| = s$, and $|E(\Hardcore)| = r$.
Further, let $\mathcal{U}(\ell, r, s, v) \coloneqq \mathcal{H}(\ell, r, s, v) \,\big/ \simeq$, the quotient set of $\mathcal{H}(\ell, r, s, v)$ with respect to graph isomorphism. 
Note that, for any $H, H' \in \mathcal{H}(\ell, r, s, v)$ with $H \simeq H'$, it holds that $\Expected{\SW(H)} = \Expected{\SW(H')}$ because, by symmetry, $\SW(H)$ and $\SW(H')$ are identically distributed.
Hence, given an equivalence class $[H] \in \mathcal{U}(\ell, r, s, v)$, we allow ourselves to write $\Expected{\SW([H])}$ for the expected signed weight of any representative of the class $[H]$. 

Our goal is to bound the right-hand side of \eqref{eq:traceAsSignedWeights} by summing over all possible choices of parameters $\ell, s, r,$ and $v$, and all equivalence classes $[H] \in \mathcal{U}(\ell, r, s, v)$. To this end, we split the right-hand side of \eqref{eq:traceAsSignedWeights} into the contribution of those $H$ with trivial, and non-trivial core, respectively. To this end, we further refine $\mathcal{H}$ and $\mathcal{U}$ by defining $\mathcal{H}_{\text{triv}}(\ell, s, a)$ as the set of all Eulerian multigraphs $H$ with $\le m$ edges and trivial core whose \emph{skeleton} $\widetilde{H}$ contains $\ell$ edges and such that $|L| = s$, and $|L_d| = a$ where $L_d$ is the set of removed degenerate cycles. Moreover, we let $\mathcal{H}_{\text{non-triv}}(\ell, r, s, v)$ be set of all Eulerian multigraphs $H$ with non-trivial core such that $|E(H)| = \ell$, $|V(\Hardcore)| = v$, $|L| = s$, and $|E(\Hardcore)| = r$. We further define the quotient sets w.r.t. to multigraph isomorphism $\mathcal{U}_{\text{triv}}(\ell, s, a) = \mathcal{H}_{\text{triv}}(\ell, s, a)\,\big/ \simeq$, and $\mathcal{U}_{\text{non-triv}}(\ell, r, s, v) = \mathcal{H}_{\text{non-triv}}(\ell, r, s, v)\,\big/ \simeq$. Note that any $H \in \mathcal{H}(\ell, r, s, v)$ is either in $\mathcal{H}_{\text{triv}}(\ell', s)$ or in $\mathcal{H}_{\text{non-triv}}(\ell, r, s, v)$ for some choice of parameters $\ell', r, s, v, a \in \{0, 1, \ldots, \ell\}$. With this in mind, writing $\#[H]$ for the number of walks $\pmb{v} \in \EW(n, m)$ such that $H_{\pmb{v}} \in [H]$, and $\#[\widetilde{H}]$ for the number of walks $\pmb{v} \in \EW(n, m)$ such that $\widetilde{H}_{\pmb{v}} \in [H]$ (where $\widetilde{H}_{\pmb{v}}$ is the \emph{skeleton} of $H_{\pmb{v}}$), we obtain from \eqref{eq:traceAsSignedWeights} that \begin{align*}
    \Expected{ \tr{\Adjppt{m} } } &\le \sum_{\pmb{v} \in \EW(n, m)} \absolute{\Expected{\SW(H_{\pmb{v}})}} \le S_{\text{triv}} + S_{\text{non-triv}} %\\
    %&\le \sum_{[H] \in \mathcal{U}_{\text{triv}}(2, 1, 1)} \#[\widetilde{H}] \absolute{ \Expected{\SW([H])}}\\ 
    %&\hspace{.5cm}+ \sum_{\ell = 4}^{m} \left( \sum_{[H] \in \mathcal{U}_{\text{triv}}(\ell, 1, 0)} \#[\widetilde{H}] \absolute{ \Expected{\SW([H])}} + \sum_{s = 2}^{\ell/2} \sum_{a = 0}^{s} \hspace{.2cm} \sum_{[H] \in \mathcal{U}_{\text{triv}}(\ell, s, a)} \#[\widetilde{H}] \absolute{ \Expected{\SW([H])}} \right) \\
    %&\hspace{1cm}
    %+ \sum_{\ell = 2}^{m} \sum_{s = 0}^{\ell/2} \sum_{r = 4}^{\ell - 2s} \sum_{v = 1}^{\max\{r/2, 1\}} \sum_{[H] \in \mathcal{U}_{\text{non-triv}}(\ell, r, s, v)} \#[H] \absolute{\Expected{\SW([H])}}.
\end{align*} where $S_{\text{triv}}$ and $S_{\text{non-triv}}$ account for all $H$ with trivial and non-trivial core, respectively. Precisely, we express \begin{align*}
    S_{\text{triv}} &\coloneqq \sum_{\ell = 2}^{m} \left( \sum_{s = 1}^{\ell/2} \sum_{a = 0}^{s-1} \hspace{.2cm} \sum_{[H] \in \mathcal{U}_{\text{triv}}(\ell, s, a)} \#[\widetilde{H}] \absolute{ \Expected{\SW([H])}} + \sum_{[H] \in \mathcal{U}_{\text{triv}}(\ell, \lfloor \ell/2 \rfloor, \lfloor \ell/2 \rfloor)} \#[\widetilde{H}] \absolute{ \Expected{\SW([H])}}\right)
\end{align*} where the first term in the parentheses accounts for all $H$ with trivial core whose skeleton $\widetilde{H}$ does not consist purely of degenerate cycles (which is why the second sum only goes until $s-1$). The second term accounts for those $H$ whose skeleton $\widetilde{H}$ consists only of degenerate cycles. Since it is easy to see that such $\widetilde{H}$ can only exist if $\ell$ is even and the number of cycles is $\ell/2$. Therefore, the above sums cover all cases for $H$ with trivial core. On the other hand, for those $H$ with non-trivial core, we use \begin{align*}
    S_{\text{non-triv}} &\coloneqq \sum_{\ell = 2}^{m} \sum_{s = 0}^{\ell/2} \sum_{r = 4}^{\ell - 2s} \sum_{v = 1}^{\max\{r/2, 1\}} \sum_{[H] \in \mathcal{U}_{\text{non-triv}}(\ell, r, s, v)} \#[H] \absolute{\Expected{\SW([H])}}
\end{align*} where we used that the number of cycles $s$ is at most $\ell/2$, and that the number of remaining edges $r$ in the core $\Hardcore$ is at most $\ell -2s$ since each cycle contains at least two edges. We further used that $\Hardcore$ contains at most $\max\{1, r/2\}$ vertices since it has minimum degree $4$. 

%Above, the first term accounts for those $H$ that have trivial core with only two vertices, the second term accounts for those $H$ that have trivial core and a skeleton with at least $4$ edges, and the third term accounts for those $H$ with non-trivial core. In particular, we used that a graph $H$ with skeleton that only has one cycle of length $\ge 4$, has no degenerate cycles. For non-trivial cores, we further used that there are at most $\ell/2$ removed cycles, and at most $\ell - 2s$ remaining edges in $\Hardcore$ since each removed cycle contains at least two edges. We further used that $V(\Hardcore) \le \max\{1, r/2\}$ by \Cref{lem:vertices}.

We proceed by bounding $S_{\text{triv}}$ and $S_{\text{non-triv}}$ separately. Starting with $S_{\text{triv}}$, we note that any $H \in \mathcal{H}_{\text{triv}}(\ell, s, a)$ has $1 + \ell - s$ vertices, and that there are $\le \binom{m}{\ell}$ ways of removing self-loops and excess edges in degenerate cycles when constructing $\widetilde{H}_{\pmb{v}}$. Hence, we get that $\#[\widetilde{H}] \le \binom{m}{\ell} n^{1 + \ell - s}$. Moreover, $|\mathcal{U}_{\text{triv}}(\ell, s, a)| \le m^{2m}$ by a crude upper bound on the number of unlabeled multigraphs with $m$ edges. Hence, applying \Cref{lem:swtrivial}, we get
\begin{align*}
    S_{\text{triv}} &\le (Cm)^{3m} \sum_{\ell = 2}^{m} \left( \sum_{s = 1}^{\ell/2} \sum_{a = 0}^{s-1} n^{1 + \ell - s}p^{\ell - a} \left( \frac{\log(n)}{\sqrt{d}}\right)^{\ell - 2s} + n^{1 + \ell/2 } p^{ \ell/2 } \right)\\
    &\le (Cm)^{3m} n \sum_{\ell = 2}^{m} \left( \sum_{s = 1}^{\ell/2} n^{\ell - s}p^{\ell - s + 1} \left( \frac{\log(n)}{\sqrt{d}}\right)^{\ell - 2s} + \sqrt{np}^{\ell} \right)\\
    &\le (Cm)^{3m} np \sum_{\ell = 2}^{m} \sum_{s = 1}^{\ell/2} \sqrt{np}^{2s} \left( \frac{np\log(n)}{\sqrt{d}}\right)^{\ell - 2s} + (Cm)^{3m + 1} n \sqrt{np}^{m}.
\end{align*} Similarly, applying \Cref{lem:swboundeuler} and using that any $H \in \mathcal{H}(\ell, r, s, v)$ has $\ell - r - s + v$ vertices by \Cref{lem:vertices}, we get for $S_{\text{non-triv}}$ that \begin{align*}
    S_{\text{non-triv}} = &\sum_{\ell = 2}^{m} \sum_{s = 0}^{\ell/2} \sum_{r = 4}^{\ell - 2s} \sum_{v = 1}^{\max\{r/2, 1\}} \sum_{[H] \in \mathcal{U}_{\text{non-triv}}(\ell, r, s, v)} \#[H] \absolute{\Expected{\SW([H])}} \\
    &\le m^{4m} \sum_{\ell = 2}^{m} \sum_{s = 0}^{\ell/2} \sum_{r = 4}^{\ell - 2s} \sum_{v = 1}^{r/2} n^{\ell - r - s + v} (C\log^{2}(n))^{\ell} p^{s + v} \left(\frac{p\log(n)}{\sqrt{d}}\right)^{\ell - r - 2s}\\
    &\le (Cm^{4} \log^2(n))^{m} \sum_{\ell = 2}^{m} \sum_{s = 0}^{\ell/2} \sum_{r = 4}^{\ell - 2s} \sum_{v = 1}^{r/2} (np)^{s + v} \left(\frac{np\log(n)}{\sqrt{d}}\right)^{\ell - r - 2s}\\
    &\le  (Cm^{4} \log^2(n))^{m} \sum_{\ell = 2}^{m} \sum_{s = 0}^{\ell/2} \sum_{r = 4}^{\ell - 2s} \sqrt{np}^{2s+r} \left(\frac{np\log(n)}{\sqrt{d}}\right)^{\ell - r - 2s}.
\end{align*}
Now, notice that both of the above upper bounds undergo a phase transition at $d = np\log(n)$: if $d \ll np\log^2(n)$, then $\frac{np\log(n)}{\sqrt{d}} \gg \sqrt{np}$, so our sums are maximized when $s, r$ are minimal. If on the other hand $d \gg np\log^2(n)$, then the opposite occurs and the sums are maximized when $s, r$ are maximized. Distinguishing these two cases, we obtain our statement.

\paragraph{The case $\boldsymbol{d \ll np\log^2(n)}$} Here, the sums are maximized when $s, r$ are minimal, so
\begin{align*}
     &\Expected{ \tr{\Adjppt{m} } } \\
     &\hspace{.5cm}\le  (Cm)^{3m} np \sum_{\ell = 2}^{m} np \left( \frac{np\log(n)}{\sqrt{d}}\right)^{\ell - 2} + (Cm)^{3m + 1} n \sqrt{np}^{m} \\
     &\hspace{6cm}+ (Cm^{4} \log^2(n))^{m} \sum_{\ell = 2}^{m} \sum_{s = 0}^{\ell/2} \sum_{r = 4}^{\ell - 2s} (\sqrt{np})^{2s+r} \left(\frac{np\log(n)}{\sqrt{d}}\right)^{\ell - r - 2s}\\
     &\hspace{.5cm} \le (Cm)^{3m+1} \left( n^2p^2 \left(\frac{\sqrt{d}}{np}\right)^2 \left(\frac{np\log(n)}{\sqrt{d}}\right)^{m} + n\sqrt{np}^{m} \right) + (Cm^{4} \log^2(n))^{m} \sum_{\ell = 2}^{m} \sum_{s = 0}^{\ell/2} (\sqrt{np})^{2s+4}  \left(\frac{np\log(n)}{\sqrt{d}}\right)^{\ell - 4 + 2s}\\
     &\hspace{.5cm} \le (Cm)^{3m+1} \left( d \left(\frac{np\log(n)}{\sqrt{d}}\right)^{m} + n\sqrt{np}^m \right) + (Cm^{4} \log^2(n))^{m} n^2p^2 \frac{d^2}{n^4p^4} \left(\frac{np\log(n)}{\sqrt{d}}\right)^{m}\\
     &\hspace{.5cm}\le (Cm)^{4m} \left( d + \log(n)^{2m + 4} \right) \left(\frac{np\log(n)}{\sqrt{d}}\right)^{m} + (Cm)^{3m+1}n \sqrt{np}^m = \bigO{d\left(\frac{np\log(n)}{\sqrt{d}}\right)^{m} + n \sqrt{np}^m  },
\end{align*} 
which confirms the intuition that there are $\approx d$ eigenvalues each of order $\bigOtilde{\frac{np}{\sqrt{d}}}$.

\paragraph{The case $\boldsymbol{d \gg np\log^2(n)}$} In this case, the sums are maximized when $s, r$ are maximal because $\left(\frac{np\log(n)}{\sqrt{d}}\right) \ll \sqrt{np}$. Hence
\begin{align*}
    \Expected{ \tr{\Adjppt{m} } } &\le  (Cm)^{3m} np \sum_{\ell = 2}^{m} \sum_{s = 1}^{\ell/2} \sqrt{np}^{2s} \left( \frac{np\log(n)}{\sqrt{d}}\right)^{\ell - 2s} + (Cm)^{3m + 1} n \sqrt{np}^{m} \\
    &\hspace{4cm}(Cm^{4} \log^2(n))^{m} \sum_{\ell = 2}^{m} \sum_{s = 0}^{\ell/2} \sum_{r = 4}^{\ell - 2s} (\sqrt{np})^{2s+r} \left(\frac{np\log(n)}{\sqrt{d}}\right)^{\ell - r - 2s}\\
    &\le  (Cm)^{3m} np \sum_{\ell = 2}^{m} \sqrt{np}^{\ell} + (Cm)^{3m + 1} n \sqrt{np}^{m} + (Cm^{4} \log^2(n))^{m} \sum_{\ell = 2}^{m} \sum_{s = 0}^{\ell/2}  (\sqrt{np})^{\ell}\\
    &\le  (Cm)^{3m+2} n \sqrt{np}^{m} + (Cm^{4} \log^2(n))^m m^2 (\sqrt{np})^{m}\\
    &\le (Cm)^{4m+2} \left( n + \log(n)^{2m} \right) \sqrt{np}^{m} = \bigO{ n\sqrt{np}^m },
\end{align*} which confirms our intuition that now there are $\approx n$ large eigenvalues, each of order $\sqrt{np}$ like in a $\Gnp$.
\end{proof}

\noindent Combining everything, we can prove \Cref{lem:spectralupperbound}. 
\begin{proof}[Proof of \Cref{lem:spectralupperbound}]
    By \Cref{lem:tracingthetrace} and \eqref{eq:traceMethod} from the discussion of the trace method in \Cref{sec:intro_psketch_spectral}, we have that \begin{align*}
        \left|\lambda_1\left( \Adjpp \right)\right| \ge a^{1/m} \Expected{ \tr{\Adjppt{m} }}^{1/m}
    \end{align*} with probability at most $1/a$. Choosing any $a = \omega(1)$ and applying \Cref{lem:tracingthetrace}, we get that for any even $m$ and assuming $d \ll np\log^2(n)$, with probability $1 - o(1)$, we have that
    \begin{align*}
        \left|\lambda_1\left( \Adjpp \right)\right| \le (Cm)^{4} \left( d^{1/m} + \log(n)^{3} \right) \frac{np\log(n)}{\sqrt{d}} + (Cm)^{4}n^{1/m} \sqrt{np}
    \end{align*} where we used that $(a + b)^{1/m} \le a^{1/m} + b^{1/m}$. Letting $m$ tend to infinity, this is $n^{o(1)}\frac{np}{\sqrt{d}}$ as desired. In case $d \gg np \log^2(n)$, a similar calculation involving \Cref{lem:tracingthetrace} and letting $m$ tend to infinity yields that $\left|\lambda_1\left( \Adjpp \right)\right| \le n^{o(1)} \sqrt{np}$. Finally, to show that the number of large eigenvalues is bounded, we note that for every fixed, even $m$ \begin{align*}
        \Expected{ \tr{\Adjppt{m} } } = \sum_{i=1}^n \lambda_i(\Adjpp)^m = \bigO{d\left(\frac{np\log(n)}{\sqrt{d}}\right)^{m} + n \sqrt{np}^m  }.
    \end{align*} Assuming that $d \le np^{1+\varepsilon}$, we can choose $m \ge \frac{2}{\varepsilon} + 2$ such that the above is $\bigO{d\left(\frac{np\log(n)}{\sqrt{d}}\right)}$. To further ensure that $m$ is even, we require $m \ge \frac{2}{\varepsilon} + 4$. Then, if there were more than $\bigOtilde{da^m}$ eigenvalues of magnitude $\ge \frac{np}{a\sqrt{d}}$, this would lead to a contradiction. \qedhere
    %a calculation of $\Expected{ \tr{\Adjppt{m}} }$ for the special case $m = 4$ and $d \ll np^2 $, shows that
    %\begin{align*}
    %     \Expected{ \tr{\Adjppt{4} } } & = \sum_{i=1}^n \lambda_i(\Adjpp)^4 = \bigOtilde{n^2p^3 + n^3p^2 + d\left(\frac{np}{\sqrt{d}}\right)^4} \\&= \bigOtilde{d\left(\frac{np}{\sqrt{d}}\right)^4}.
    %\end{align*} 
    %If there were more than $\bigOtilde{da^4}$ eigenvalues of magnitude $\ge np / (a\sqrt{d})$, it would contradict this bound.
\end{proof}

\subsection{A Different Threshold for $L_\infty$-norm}

We complement the results from \cref{sec:upperboundviathetracemethod} by studying the spectrum of $\Gbf \sim \RGGinfty$.
Here, it turns out that -- perhaps surprisingly -- the spectral threshold (i.e. the point at which the spectral gap is essentially as large as in a $\Gnp$) is rather different and in particular reached much earlier than in the $L_q$ model for fixed $1 \le q < \infty$. 
Concretely, we show that the second largest eigenvalue of a typical $\Gbf \sim \RGGinfty$ is of order $\tilde{\Theta}\left(\max\{ \sqrt{np}, np/d \}\right)$ implying that the spectral threshold is reached at $d \approx \sqrt{np}$, instead of $d \approx np$ as in the $L_q$ model. 
We can also characterize the spectrum of $\Gbf \sim \RGGinfty$ in a somewhat different way: Instead of finding $\tilde{\Theta}(d)$ eigenvectors with large eigenvalue if $d \ll np$, we can find $\tilde{\Theta}(d^2)$ such vectors, provided that $d \ll \sqrt{np}$ (this is done in \Cref{sec:linftynormlowerbound}). 
These results are combined the following theorem.
\spectralstuffinfty*
\begin{proof}
    The upper bound follows from \cref{lem:spectralupperboundinfty} and the lower bound from \cref{lem:spectrallowerboundinfty} if $d \ll \sqrt{np}$. Otherwise, the lower bound follows from the generalized Alon--Boppana bound from \cite[Theorem 8]{Jiang2019} using $r = 1$.
\end{proof}

\subsubsection{Bounding the Signed Weight of Cycles and Chains}

Our proof has the same outline as the proof for the $L_q$ model, so it essentially relies on the idea of bounding the signed weight of cycles and chains. The signed weight of a cycle has already been studied in previous work \cite{Bangachev_Bresler_2024}, which we extend to also account for chains. To this end, recall that $\xi$ is defined such that $(1-\xi)^d = p$, or alternatively such that $u \nsim v$ whenever $|\mathbf{x}_u(i) - \mathbf{x}_v(i)|_C  > \frac{1  - \xi}{2}$. With this, we define the \say{good} event $\mathcal{G}$ in terms of the position of the two endpoints $u,v$ of a fixed chain $H$ as \begin{align*}
    \mathcal{G} \coloneqq \{ \mathcal{G}_i \text{ occurs for}\ge d - \log^2(n) \text{ dimensions} \} \text{ where } \mathcal{G}_i \coloneqq \left\{ |\mathbf{x}_u(i) - \mathbf{x}_v(i)|_C \ge k \xi
    \right\}.
\end{align*} That is, under $\mathcal{G}$, there are at most $\log^2(n)$ dimensions in which the endpoints of our chain $H$ are within distance $\le k\nu$. We call dimensions for which $\mathcal{G}_i$ occurs \emph{good} and all other dimensions \emph{bad}. It can immediately be shown that it is very unlikely to have many bad dimensions. \begin{lemma}\label{lem:gprobbound}
    For all $k = o(\log^2(n))$ it holds that $\Prnop{\overline{\mathcal{G}}} \le n^{-\Omega(\log^2(n))}$.
\end{lemma} 

\begin{proof}
    Applying a Taylor series to $\xi$, it is easy to note that $\xi = \Theta(\log(1/p)/d)$. 
    Therefore, the probability that a single, fixed dimension is bad is $\bigO{k/d}$, and the number $X$ of bad dimensions is the sum of independent Bernoulli random variables, each with success probability $\bigO{k/d}$. Combining Remark 2.5 and Theorem 2.8 in \cite{randomgraphsjanson}, it follows that 
    $
        \Pr{X \ge t} \le \exp\left( -t\right) \text{ for all } t \ge 7\Expected{X}.
    $ 
    Noting that $\Expected{X} = \bigO{k}$ and $k = o(\log^2(n))$, the result follows after setting $t = \log^2(n)$. 
\end{proof}

While the authors of \cite{Bangachev_Bresler_2024} also give bounds on the signed weight of general graphs, these are not strong enough for our purposes since they do not yield tight bounds in the special case of chains. We prove the following while partly (i.e. in the case of cycles) relying on the bounds from \cite{Bangachev_Bresler_2024}.

\begin{lemma}\label{lem:fouriercoefficientslinfty}
    Consider $G \sim \RGGinfty$ with $\frac{1}{n} \le p \le 1-\varepsilon$. Then, for a cycle $C_k$ of length $k\ge3$ or a chain $H$ with $k\ge 2$ edges with endpoints $v_1, v_{k+1}$ and associated positions $\vecx_1, \vecx_{k+1}$ such that $\mathcal{G}$ is met, we have \begin{align}\label{eq:signedcyclelinfty}
        \Expected{\SW(C_k)} &\le 3\log(n)p^k\left(\frac{2\log(1/p)}{d}\right)^{k-2}\\
        \text{ and }\Expected{\SW(H) \mid \vecx_1, \vecx_{k+1}} &\le  3\log^2(n)p^k\left( \frac{3\log(1/p)}{d} \right)^{k-1}.
    \end{align}
\end{lemma} 
\begin{proof}
    Regarding the cycle, we rely on \cite[Proposition 3.5]{Bangachev_Bresler_2024}, which immediately yields that \begin{align*}
        \Expected{\SW(C_k)} \le p^k \left( 1 + (1+o(1))d\xi^{k-1} \right) \le p^k\left( 1 + 3\log(n)\left(\frac{2\log(1/p)}{d}\right)^{k-2} \right)
    \end{align*} since applying a Taylor series to $\xi = 1 - p^{1/d}$ yields that \begin{align}\label{eq:taylorxi}
        \xi = 1 - p^{1/d} = 1 - \left( 1 - \frac{\log(1/p)}{d} \pm \bigO{\frac{\log^2(n)}{d^2}} \right) = \frac{\log(1/p)}{d} \pm \bigO{\frac{\log^2(n)}{d^2}} \le \frac{2\log(1/p)}{d}.
    \end{align}
    
    For the chain $H$, we use a \say{cluster expansion} based approach based on the same idea as in \cite{Bangachev_Bresler_2024}. This is the same approach also used to derive \eqref{eq:signedcyclelinfty}, however, in case of a chain $H$, we have to use a slight modification to account for the fact that the endpoints of $H$ are in $\mathcal{G}$. To carry this out, we start by noting that our chain exists if and only if it exists in all $d$ dimensions. Formally, this allows us to rewrite \begin{align*}
        \Expectedsub{G\sim \RGGinfty}{\prod_{e \in H} \mathds{1}(e)} = \left(\Expectedsub{G\sim \RGGinftyp{n}{1}{p^{1/d}}}{\prod_{e \in H} \mathds{1}(e)} \right)^d.
    \end{align*} We bound $\Expectedsub{G\sim \RGGinftyp{n}{1}{p^{1/d}}}{\prod_{e \in H} \mathds{1}(e)}$ by expressing it as an alternating sum over all subsets $A \subseteq H$ where each term is (up to its sign) equal to the probability of the event that no edge in $A$ is present. Accordingly, we define \begin{align*}
        \chi(A) \coloneqq \Prsub{G\sim \RGGinftyp{n}{1}{p^{1/d}}}{\bigcap_{e \in E(A)} (\mathds{1}(e) = 0)}
    \end{align*} and note that by elementary inclusion-exclusion, \begin{align*}
        \Expectedsub{G\sim \RGGinftyp{n}{1}{p^{1/d}}}{\prod_{e \in H} \mathds{1}(e)} = \sum_{A \subseteq H} (-1)^{|E(A)|}\chi(A).
    \end{align*} Now to quantify how close this expression is to $p^k$, we measure \say{perturbations} from a $\Gnp$ by defining \begin{align*}
        \psi(A) = \chi(A) - \xi^{|E(A)|}.
    \end{align*} As the probability that no edge in $A$ is present in a $\Gnp$ of the same density (i.e. density $p^{1/d}$) is exactly $\xi^{|E(A)|}$, the quantity $\psi(A)$ measures how much the probabilities of the event in question differ between $\Gnp$ and $\RGGinftyp{n}{1}{p^{1/d}}$. We now handle good and bad dimensions separately. 
    
    \paragraph{Good dimensions.} Recall that dimension $i$ is good if the endpoints $u,v$ of our chain $H$ are such that $|\mathbf{x}_u(i) - \mathbf{x}_v(i)|_C \ge k \xi$. 
    Hence, we consider the $1$-dimensional model $\RGGinftyp{n}{1}{p^{1/d}}$ conditioned on the event that the dimension is good with respect to $H$. 
    We denote the associated random graph model by $\RGGinftyp{n}{1}{p^{1/d}} \big| \mathcal{G}_1$, and observe that, under its law, it holds that
    \begin{align*}
        \chi(A) = \begin{cases}
            \xi^{|E(A)|}  & \text{if } E(A) \neq E(H)\\
            0             & \text{otherwise}
        \end{cases} \text{ and thus } \psi(A) = \begin{cases}
            0             & \text{if } E(A) \neq E(H)\\
            -\xi^{|E(A)|} & \text{otherwise}.
        \end{cases}
    \end{align*} 
    if $d$ is sufficiently large. 
    The reason for this is that a large dimension implies that $\xi = \Theta(\log(1/p)/d)$ is very small and sufficiently distant endpoints (specifically endpoints at distance $\ge k\xi$) make it impossible to arrange all vertices such that all edges in $H$ are excluded. A similar argument was also used in \cite{Bangachev_Bresler_2024}, we specifically refer the interested reader to \cite[Claim 3.1, Item 4]{Bangachev_Bresler_2024} for further details.
    With this in mind, we obtain 
    \begin{align*}
        \Expectedsub{G\sim \left.\RGGinftyp{n}{1}{p^{1/d}} \right| \mathcal{G}_1 }{\prod_{e \in H} \mathds{1}(e)} 
        &= \sum_{\ell=0}^{k} (-1)^{\ell} \sum_{A \subseteq H: |A| = \ell} \left(\xi^{\ell} + \psi(A)\right)\\
        &= \left(\sum_{\ell=0}^{k} \binom{k}{\ell} (-1)^{\ell} \xi^{-\ell}\right) - (-1)^k \xi^{k} 
        = \left(1 - \xi\right)^k - (-1)^k\xi^{k}.
    \end{align*} Noting that $1-\xi = p^{1/d}$ implies that \begin{align*}
        \Expectedsub{G\sim \left.\RGGinftyp{n}{1}{p^{1/d}} \right| \mathcal{G}_1 }{\prod_{e \in H} \mathds{1}(e)} \le p^{1/d} \pm \xi^{k} = p^{1/d} \left( 1 \pm p^{-1/d}\xi^{k} \right).
    \end{align*} Finally, applying the Taylor series of $\xi = 1 - p^{1/d}$ from \eqref{eq:taylorxi} and using further that $p^{-1/d} = 1 + o(1)$ since $p \ge n^{-\bigO{1}}$ and $d = \omega(\log(n))$, we get that \begin{align*}
        \Expectedsub{G\sim \left.\RGGinftyp{n}{1}{p^{1/d}} \right| \mathcal{G}_1}{\prod_{e \in H} \mathds{1}(e)} \le  p^{k/d} \left( 1 + \left(\frac{3\log(1/p)}{d}\right)^{k} \right).
    \end{align*}
    
    \paragraph{Bad dimensions.} For all bad dimensions, we use the same strategy but rely on a slightly weaker bound on $\chi(H)$ obtained by fixing an arbitrary $A \subset H$ with $|E(A)| = k-1$ and noting that  \begin{align*}
        \chi(H) &= \Prsub{G\sim \left.\RGGinftyp{n}{1}{p^{1/d}} \right| \overline{\mathcal{G}}_1}{\bigcap_{e \in E(H)} (\mathds{1}(e) = 0)}\\& \le \Prsub{G\sim \left.\RGGinftyp{n}{1}{p^{1/d}} \right| \overline{\mathcal{G}}_1}{\bigcap_{e \in E(A)} (\mathds{1}(e) = 0)} = \xi^{k-1}
    \end{align*} since edges in all proper subsets of $H$ are still independent. Therefore, we still have $\psi(A)= 0$ if $A \subsetneq H$ and $\psi(A)= \xi^{k-1} - \xi^{k}$. Carrying out the same calculation as before yields that \begin{align*}
        \Expectedsub{G\sim \left.\RGGinftyp{n}{1}{p^{1/d}} \right| \overline{\mathcal{G}}_1}{\prod_{e \in H} \mathds{1}(e)} &= \left(\sum_{\ell=0}^{k} \binom{k}{\ell} (-1)^k \xi^k\right) + (-1)^k\left(\xi^{k-1} - \xi^{k} \right)\\ 
        &= \left(1 - \xi\right)^k + (-1)^k\left(\xi^{k-1} - \xi^{k} \right)\\
        &\le p^{1/d} + \xi^{k-1} \le p^{k/d} \left( 1 + \left(\frac{3\log(1/p)}{d}\right)^{k-1} \right).
    \end{align*}
    
    \paragraph{Putting everything together.} In total, we obtain
    \begin{align*}
        &\Expectedsub{\Gbf\sim \left.\RGGinfty\right|\mathcal{G}}{\prod_{e \in H} \mathds{1}(e)} \\&= \left(\Expectedsub{\Gbf\sim \left.\RGGinftyp{n}{1}{p^{1/d}} \right| \mathcal{G}_1}{\prod_{e \in H} \mathds{1}(e)} \right)^{d - \log^2(n)}\left(\Expectedsub{\Gbf\sim \left.\RGGinftyp{n}{1}{p^{1/d}} \right| \overline{\mathcal{G}}_1}{\prod_{e \in H} \mathds{1}(e)} \right)^{\log^2(n)}\\ 
        &\le \left( p^{k/d}  \left( 1 + \left(\frac{3\log(1/p)}{d}\right)^{k} \right) \right)^{d - \log^2(n)} \left(p^{k/d} \left( 1 + \left(\frac{3\log(1/p)}{d}\right)^{k-1} \right) \right)^{\log^2(n)} \\
        &\le p^{k} \left( 1 + d \left(\frac{3\log(1/p)}{d}\right)^{k} + \sum_{\ell = 2}^d \binom{d}{\ell} \left(\frac{3\log(1/p)}{d}\right)^{\ell k} \right) \\
        &\hspace{3cm}\times \left( 1 + \log^2(n) \left(\frac{3\log(1/p)}{d}\right)^{k-1} + \sum_{\ell = 2}^{\log^2(n)} \binom{\log^2(n)}{\ell} \left(\frac{3\log(1/p)}{d}\right)^{\ell (k-1)} \right) \\
        &\le  p^{k} \left( 1 + 2d \left(\frac{3\log(1/p)}{d}\right)^{k} \right) \left( 1 + 2\log^2(n) \left(\frac{3\log(1/p)}{d}\right)^{k-1}\right)\\
        &\le  p^{k} \left( 1 + 3\log^2(n)\left( \frac{3\log(1/p)}{d} \right)^{k-1} \right).
    \end{align*} 
    where the penultimate step used that the sums decay exponentially due to $d = n^{\Omega(1)}$ and $p \ge 1/n$.
\end{proof}

\subsubsection{Using The Trace Method (Again)}

Here, we prove the following upper bound on the second largest eigenvalue of $\RGGinfty$.
\begin{lemma}\label{lem:spectralupperboundinfty}
    Consider a $\Gbf \sim \RGGp{n}{d}{p}$. With probability $1 - o(1)$, \begin{align*}
        \max\{ |\lambda_2(\Adj)|, |\lambda_n(\Adj)| \} \le \log^{11}(n)\max\left\{ \sqrt{np}, \frac{np}{d} \right\}
    \end{align*} where $\lambda_1(\Adj) \ge \lambda_2(\Adj) \ge \ldots \ge \lambda_n(\Adj)$ are the eigenvalues of the adjacency matrix $\Adj$ of a $\Gbf$. Moreover, if $d \ll p^\varepsilon\sqrt{np}$, the number of Eigenvalues $\lambda_{i}(\Adj)$ with $|\lambda_i(\Adj)| \ge \frac{np}{ad}$ is at most $\bigOtilde{d^2a^{\frac{1}{\varepsilon} + 4}}$.
\end{lemma}
\noindent The proof is essentially the same as described in more detail in \Cref{sec:fouriereuler}: Expanding the expectation $\Expected{\tr{\Adjppt{m}}}$, we get a sum over the signed weight of Eulerian multigraphs $H$ with $\ell \le m$ edges. 
For every such $H$, we construct a core $\Hardcore$ with minimum degree $4$ by repeatedly contracting all cycles and chains. Then, in order to bound the signed weight of $H$, we treat the non-contracted edges of $\Hardcore$ and the contracted cycles and chains from $H$ separately. In this regard, it is not hard to see that \Cref{lem:uncontracted} carries over verbatim to the $L_\infty$ model. Then, it is not hard to obtain the following analogs of \Cref{lem:swtrivial} and \Cref{lem:swboundeuler}, essentially by replacing \Cref{thm:fouriercoefficients} by \Cref{lem:fouriercoefficientslinfty}. 
\begin{lemma}[Bounding the signed weight of $H$ with trivial core]\label{lem:swtrivialinfty}
    Consider a Eulerian multigraph $H$ such that the core $\Hardcore$ is trivial and its skeleton $\widetilde{H}$ has $\ell$ edges. Assume further that exactly $s$ cycles where removed during construction of $\Hardcore$ ($|L| = s$) out of which $a$ are degenerate. 
    Then, \begin{align*}
        \Expected{ \SW(H) } \le (3\log(n))^{s} p^{\ell - a} \left(\frac{2\log(1/p)}{d}\right)^{\ell - 2s}.
    \end{align*}
\end{lemma}\begin{proof}
    The proof is analogous to the proof of \Cref{lem:swtrivial} with \Cref{lem:fouriercoefficientslinfty} instead of \Cref{thm:fouriercoefficients}.
\end{proof}

\begin{lemma}[The signed weight of non-trivial Eulerian multigraphs]\label{lem:swboundeulerinfty}
    Assume that $H$ is a Eulerian multigraph with $\ell$ edges and non-trivial core $\Hardcore$. Assume further that $s$ cycles were contracted when constructing $\Hardcore$, and that $ |E(\Hardcore)| = r$ and $|V(\Hardcore)| = v$. Then,
    \begin{align*}
        \absolute{\Expected{\SW(H)}} &\le (C\ell\log^{3}(n))^{\ell} p^{s +  v} \left(\frac{3p\log(1/p)}{d}\right)^{\ell - r - 2s}
    \end{align*}
\end{lemma} \begin{proof}
    The proof is analogous to the proof of \Cref{lem:swboundeuler} with \Cref{lem:fouriercoefficientslinfty} instead of \Cref{thm:fouriercoefficients}.
\end{proof}

\noindent With this, we are ready to trace the trace once again.
\begin{lemma}[Bounding the $m$-th power of the Centered Adjacency Matrix (Again)]\label{lem:tracingthetracetoinfinity}
    There is a constant $C>0$ such that for all even $m \in \mathbb{N}$, \begin{align*}
        \Expected{ \tr{\Adjppt{m} } } \le \begin{cases}
            (C\log(n) m)^{4m + 4}\left( d^2\left(\frac{np\log(1/p)}{d}\right)^{m} + n\sqrt{np}^{m} \right)  & \text{if } d \ll \sqrt{np}\log(1/p)\\
            (C\log(n)m)^{4m+2} \sqrt{np}^{m} &\text{if } d \gg \sqrt{np}\log(1/p).
        \end{cases}
    \end{align*}
\end{lemma} 
\begin{proof} We define $S_{\text{triv}}$ and $S_{\text{non-triv}}$ as in the proof of \cref{lem:tracingthetrace}. Applying \Cref{lem:swtrivialinfty} and \Cref{lem:swboundeulerinfty}, then yields \begin{align*}
    S_{\text{triv}} &\le (C\log(n)m)^{3m} np \sum_{\ell = 2}^{m} \sum_{s = 1}^{\ell/2} \sqrt{np}^{2s} \left( \frac{np\log(1/p)}{d}\right)^{\ell - 2s} + (C\log(n)m)^{3m + 1} n \sqrt{np}^{m}
\end{align*} and \begin{align*}
    S_{\text{non-triv}} &\le  (Cm^{4} \log^3(n))^{m} \sum_{\ell = 2}^{m} \sum_{s = 0}^{\ell/2} \sum_{r = 4}^{\ell - 2s} \sqrt{np}^{2s+r} \left(\frac{np\log(1/p)}{d}\right)^{\ell - r - 2s}.
\end{align*}
Again, both of these bounds show a phase transition, this time at $d \approx \sqrt{np}\log(1/p)$.
\paragraph{The Case $d \ll \sqrt{np}\log(1/p)$} Here, the sums are maximized when $s, r$ are minimal, so
\begin{align*}
     &\Expected{ \tr{\Adjppt{m} } } \\
     &\hspace{.5cm}\le  (C\log(n)m)^{3m} np \sum_{\ell = 2}^{m} np \left( \frac{np\log(1/p)}{d}\right)^{\ell - 2} + (C\log(n)m)^{3m + 1} n \sqrt{np}^{m} \\
     &\hspace{6cm}+ (Cm^{4} \log^3(n))^{m} \sum_{\ell = 2}^{m} \sum_{s = 0}^{\ell/2} \sum_{r = 4}^{\ell - 2s} (\sqrt{np})^{2s+r} \left(\frac{np\log(1/p)}{d}\right)^{\ell - r - 2s}\\
     &\hspace{.5cm} \le (C\log(n)m)^{3m+1} \left( n^2p^2 \left(\frac{d}{np}\right)^2 \left(\frac{np\log(1/p)}{d}\right)^{m} + n\sqrt{np}^{m} \right) \\
     &\hspace{6cm}+ (Cm^{4} \log^4(n))^{m} \sum_{\ell = 2}^{m} n^2p^2\frac{d^4}{n^4p^4}  \left(\frac{np\log(1/p)}{d}\right)^{\ell}\\
     &\hspace{.5cm} \le (C\log(n) m)^{4m}( d^2 + \log^4(n) )\left(\frac{np\log(1/p)}{d}\right)^{m} + (C\log(n)m)^{3m + 1} n \sqrt{np}^{m} \\
     &\hspace{.5cm} \le (C\log(n) m)^{4m + 4}\left( d^2\left(\frac{np\log(1/p)}{d}\right)^{m} + n\sqrt{np}^{m} \right) \\
     &\hspace{.5cm} = \bigOtilde{ d^2 \left(\frac{np\log(1/p)}{d}\right)^{\ell} + n \sqrt{np}^{m} }.
\end{align*}

\paragraph{The Case $d \gg \sqrt{np}\log(1/p)$} In this case, sums are maximized when $s,r$ are large, so 
\begin{align*}
    \Expected{ \tr{\Adjppt{m} } } &\le  (C\log(n)m)^{3m} mnp \sqrt{np}^{\ell} + (C\log(n)m)^{3m + 1} n \sqrt{np}^{m} + (Cm^{4} \log^3(n))^{m} m^2 (\sqrt{np})^{m}\\
    &\le (C\log(n)m)^{4m+2} \sqrt{np}^{m} = \bigO{ n\sqrt{np}^m },
\end{align*} as desired.
\end{proof}

\noindent Combining everything, we prove \Cref{lem:spectralupperboundinfty}. \begin{proof}[Proof of \Cref{lem:spectralupperboundinfty}]
    By the discussion about the trace method (\eqref{eq:traceMethod} from \Cref{sec:intro_psketch_spectral}), we have that \begin{align*}
        \left|\lambda_1\left( \Adjpp \right)\right| \ge a^{1/m} \Expected{ \tr{\Adjppt{m} }}^{1/m}
    \end{align*} with probability at most $1/a$. Choosing any $a = \omega(1)$ and applying \Cref{lem:tracingthetrace}, we get that for any even $m$ and assuming $d \ll np\log(n)$, with probability $1 - o(1)$, we have \begin{align*}
        \left|\lambda_1\left( \Adjpp \right)\right| \le (C\log(n) m)^{5}\left( d^{2/m}\frac{np\log(1/p)}{d} + n^{1/m}\sqrt{np} \right)
    \end{align*} where $C > 0$ is a constant. Choosing $m = \Omega(\log(n))$, this is $\bigO{\log^{11}(n)np/d}$ as desired. In case $d \gg \sqrt{np} \log(1/p)$, \Cref{lem:tracingthetracetoinfinity} yields that $\left|\lambda_1\left( \Adjpp \right)\right| \le \bigOtilde{ \log^{10}(n) \sqrt{np}}$. Finally, to show that the number of large eigenvalues is bounded, we note that if $d \le p^\varepsilon \sqrt{np}$, then for every fixed even $m$, \begin{align*}
        \Expected{ \tr{\Adjppt{m} } } = \sum_{i=1}^n \lambda_i(\Adjpp)^m = \bigOtilde{d^2\left(\frac{np\log(1/p)}{d}\right)^{m} + n \sqrt{np}^m  }.
    \end{align*} Assuming that $d \le p^{\varepsilon}\sqrt{np}$, we can choose $m \ge \frac{1}{\varepsilon} + 2$ such that the above is $\bigOtilde{d^2\left(\frac{np\log(1/p)}{d}\right)^m}$. To further ensure that $m$ is even, we require $m \ge \frac{1}{\varepsilon} + 4$. Then, if there were more than $\bigOtilde{d^2a^m}$ eigenvalues of magnitude $\ge \frac{np}{ad}$, this would lead to a contradiction. \qedhere 
\end{proof}

\subsection{Lower Bound: Finding Explicit Eigenvectors}\label{sec:lowerboundeigenvalues}

In this section, we show that the upper bounds from \Cref{lem:spectralupperbound} and \Cref{lem:spectralupperboundinfty} are tight up to lower order factors. %In the (relatively) low-dimensional case, we achieve this by explicitly constructing a sequence of eigenvalues, while in the higher dimensional case, this follows by a generalized Alon-Boppana bound on the second eigenvalue of an almost regular graph from .

\subsubsection{$L_q$-Norms for $q < \infty$}

We complement the results from the previous subsection by showing that we can in fact find $d$ approximate eigenvectors with eigenvalue $\Theta\left(\frac{np}{\sqrt{d}}\right)$ provided that $d \ll np$. Precisely, what we mean by \say{approximate eigenvectors} is that we can find $d$ unit vectors $\mathbf{y}_1, \mathbf{y}_2, \ldots, \mathbf{y}_d$ such that $$\mathbf{y}_i^\top \left(\Adjpp\right) \mathbf{y}_i = \Theta\left(  \frac{np}{\sqrt{d}} \right)$$ and the $\mathbf{y}_1, \mathbf{y}_2, \ldots, \mathbf{y}_d$ have pairwise inner products close to $0$, in the order of $1/\sqrt{n}$. It is easy to see that one can turn such a collection into $\Theta(d)$ \say{real} eigenvectors, by applying Gram-Schmidt.

\begin{lemma}\label{lem:spectrallowerbound}
    Consider a $\Gbf \sim \RGGp{n}{d}{p}$ and assume that $d \ll np$. With probability $1 - o(1)$, there are unit vectors $\mathbf{y}_1, \mathbf{y}_2, \ldots, \mathbf{y}_d$ such that \begin{align*}
        \mathbf{y}_i^\top \left(\Adjpp\right) \mathbf{y}_i = \Theta\left(  \frac{np}{\sqrt{d}} \right) \text{ and } \mathbf{y}_i^\top\mathbf{y}_j = \bigOtilde{\frac{1}{\sqrt{n}}} \text{ for all } i \neq j.
    \end{align*} In particular this implies that 
    \begin{align*}
        \max\{ |\lambda_2(\Adj)|, |\lambda_n(\Adj)| \} = \Omega\left( \frac{np}{\sqrt{d}} \right).
    \end{align*}
\end{lemma}
The idea behind our construction is to define one vector $\mathbf{y}_i$ for every dimension $i \in [d]$. Now, sample the position of all vertices in dimension $i$ and call the result $\mathbf{x}_{(i)}$, i.e., the $n$-dimensional vector containing the position of all vertices in dimension $i$. Now, fix a sufficiently small constant $c > 0$ and define the set $A$ as the set of all vertices which have distance at most $c$ from the origin in dimension $i$, and similarly the set $B$ as the set of all vertices with distance at least $\frac{1}{2} -c$ from the origin. More formally, we set $A = \{v \mid |\mathbf{x}_{(i)}(v)|_{C} \le c\}$ and $B = \{v \mid |\mathbf{x}_{(i)}(v)|_{C} \ge \frac{1}{2}- c\}$. With this, we define the vector $\tilde{\mathbf{y}}_i \in \{\pm1, 0\}$ by setting \begin{align*}
    \tilde{\mathbf{y}}_i(v) = \begin{cases}
        +1 &\text{if } v \in A\\
        -1 &\text{if } v \in B\\
        \hspace{.23cm}0  &\text{otherwise.}
    \end{cases}
\end{align*}
Finally, $\mathbf{y}_i$ is defined as $\tilde{\mathbf{y}}_i$ normalized to unit length (i.e., $\mathbf{y}_i \eqqcolon \mathbf{y}_i/\|\tilde{\mathbf{y}}_i\|$). This construction is depicted in \Cref{fig:evq}, left.
The reason why this procedure yields an approximate eigenvector w.h.p.~is that the quadratic form $\tilde{\mathbf{y}}_i^\top \Adj \tilde{\mathbf{y}}_i$ is equal to the difference of the number of edges with both endpoints in either $A$ or $B$ (called \emph{intra-cluster edges}), and the number of edges with one endpoint in $A$ and the other in $B$ (called \emph{inter-cluster edges}). 
Because intra-cluster edges are slightly likelier to be present than inter-cluster edges (w.r.t. the randomness incurred by all dimensions except $i$), this difference has expected size $\approx \frac{np}{\sqrt{d}}$.
Since $d \ll np$, this is much larger than the magnitude of the random fluctuations in the number of inter- and intra-cluster edges (which is of order $\sqrt{np}$), so we can prove that $\mathbf{y}_i^\top \Adj \mathbf{y}_i = \Omega(np/\sqrt{d})$ holds w.h.p., as desired. 
It is furthermore easy to see that the $\mathbf{y}_i$ constructed this way are almost orthogonal.

\begin{figure}
    \centering
    \includegraphics[width=0.65\linewidth]{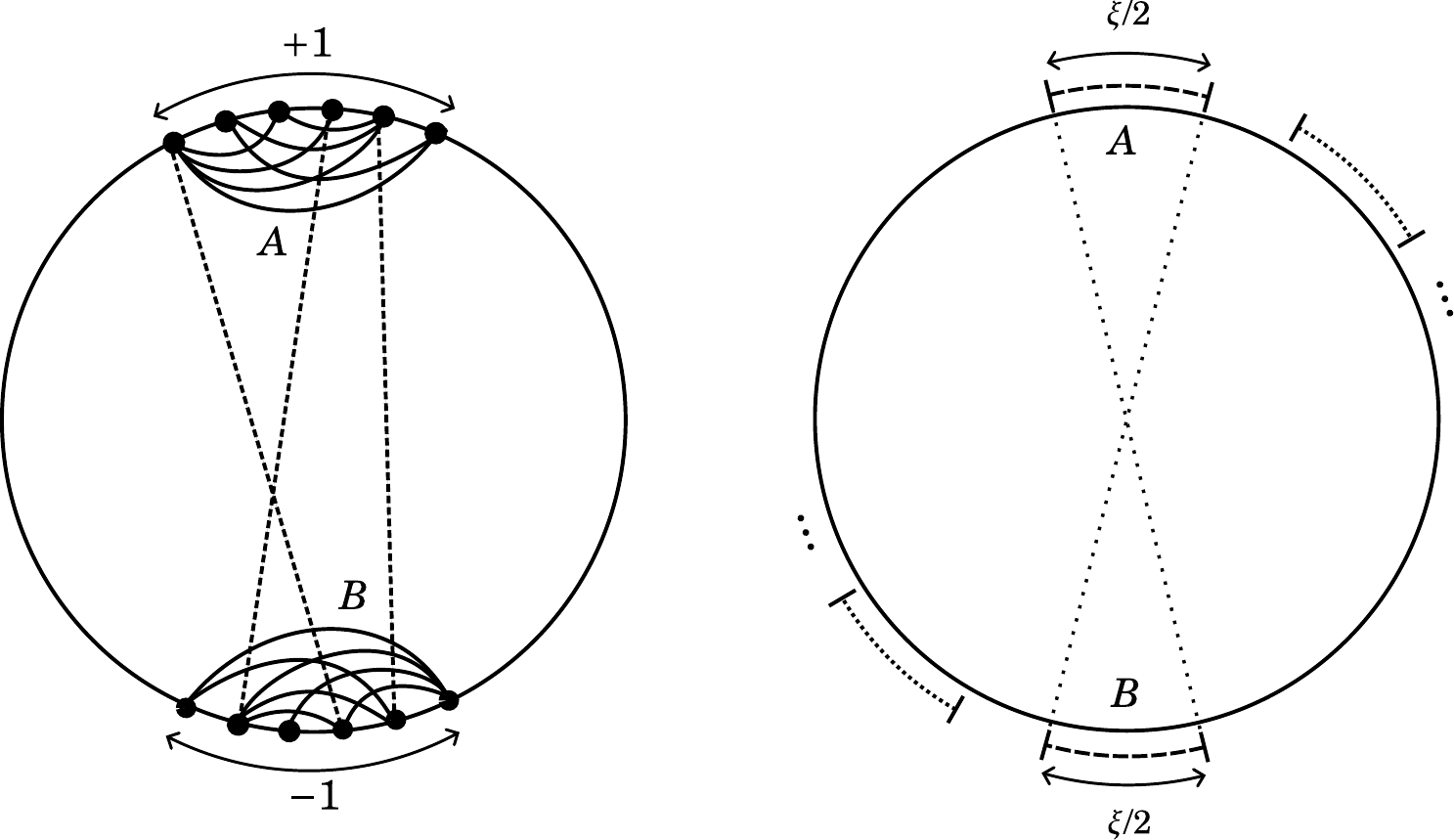}
    \caption{Illustration for the construction of explicit eigenvectors for a $\RGG$ and $ \RGGinfty$. Left: The case of $\RGG$. It is indicated that intra-cluster edges (solid) are likelier than inter-cluster edges (dashed), making the vector constructed this way an approximate eigenvector. Right: The changes required for $\RGGinfty$. The sets $A, B$ become much smaller and only take up a \say{volume} of $\xi = \Theta(\log(1/p)/d$. Now we can fit $\tilde{\Omega}(d)$ vectors in a single dimension. Furthermore, inter-cluster do not exist.  }
    \label{fig:evq}
\end{figure}

The crucial fact upon which the above discussion relies is that our intra-cluster edges are likelier to be present than inter-cluster edges. 
This difference is of order $p/\sqrt{d}$. 
We capture this formally in the following lemma. \begin{lemma}\label{lem:probforeigenvalues}
    For all $\mathbf{x}_{(1)}$, we have that \begin{align*}
    \forall \{u,v\} \in \binom{A}{2} \cup \binom{B}{2}: \Pr{u \sim v \mid \mathbf{x}_{(1)}} &\ge p + \frac{Cp}{\sqrt{d}} \quad \text{ and }\\ \forall \{u,v\} \in A \times B: \Pr{u \sim v \mid \mathbf{x}_{(1)}} &\le p - \frac{Cp}{\sqrt{d}}.
    \end{align*}% where $\RGG | \mathbf{x}_{(1)}$ denotes the law of $\RGG$ conditional on fixing the positions of all vertices in dimension $1$ according to $\mathbf{x}_{(1)}$.
\end{lemma}
\begin{proof}
    Denote by $r = \sum_{i=1}^d \frac{\Delta_i(\{u, v\}) - \mu}{\sigma \sqrt{d}}$ the rescaled $L_q$-distance (raised to the power of $q$) between two vertices $u$ and $v$. 
    Define further $\bar{r}$ by leaving out the first dimension in the previous sum. 
    We essentially use that, if $u,v$ are either both in $A$ or both in $B$, it holds that
    \begin{align*}
        \Pr{u \sim v \mid \mathbf{x}_{(1)}} = \Pr{r \le \hattau \mid \mathbf{x}_{(1)}} &\ge \Pr{\bar{r} + \frac{c - \mu}{\sigma\sqrt{d}} \le \hattau} = \Pr{\bar{r} \le \hattau - \frac{c - \mu}{\sigma\sqrt{d}}}
    \end{align*} because $u,v$ both being in either $A$ or $B$ implies that $\Delta_1(\{u,v\}) \le c$. 
    We choose $c$ sufficiently small (i.e. smaller than $\mu$) such that $\beta =  -\frac{c - \mu}{\sigma\sqrt{d}}$ is positive and focus on bounding $\Pr{\bar{r} \le \hattau + \beta}$ from below. 
    We apply an Edgeworth expansion to $\sqrt{\frac{d}{d-1}} \bar{r}$ using \Cref{thm:univariateedgeworthdensity}, where we added the factor of $\sqrt{\frac{d}{d-1}}$ to ensure that the variance remains $1$. 
    This yields that, for all $s \ge 2$,
    \begin{align}
        f_{\sqrt{\frac{d}{d-1}} \bar{r}} (x) &= \left( \phi(x) + \phi(x)\sum_{j=1}^{s-2} \frac{p_j(x)}{(d-1)^{j/2}} \right) \pm  o\left( \frac{1}{d^{(s-2)/2}} \right)\\
        &= \left( \phi(x) + \phi(x)\sum_{j=1}^{s-2} \frac{p_j(x)}{d^{j/2}} \left(\frac{d}{d-1} \right)^{j/2} \right) \pm  o\left( \frac{1}{d^{(s-2)/2}} \right)
    \end{align} Using a Taylor expansion, we get that \begin{align*}
        \left(\frac{d}{d-1}\right)^{j/2} = \left( 1  + \frac{1}{d-1} \right)^{j/2} = 1 + \frac{j}{2(d-1)} \pm \bigO{\frac{1}{d^2}}.
    \end{align*} 
    Hence, for every $s \ge 2$ there is a constant $C > 0$ such that  \begin{align}\label{eq:stuff}
        f_{\sqrt{\frac{d}{d-1}} \bar{r}} (x) &\ge \left( \phi(x) + \phi(x)\sum_{j=1}^{s-2} \frac{p_j(x)}{d^{j/2}} \right) - \phi(x)\sum_{j=1}^{s-2} \frac{C|x|^{3j}}{d^{1 + j/2}} -  o\left( \frac{1}{d^{(s-2)/2}} \right) \nonumber \\
        &\ge f_r(x) - \phi(x)\sum_{j=1}^{s-2} \frac{C|x|^{3j}}{d^{1 + j/2}} -  o\left( \frac{1}{d^{(s-2)/2}} \right),
    \end{align} 
    where we noted that the expression in the parentheses is just the $s$-th order Edgeworth expansion of $r$. 
    Now, we use that \begin{align*}
        \Pr{\bar{r} \le \hattau + \beta} \ge \Pr{\sqrt{\frac{d}{d-1}} \bar{r} \le  \hattau + \beta} = \Pr{\sqrt{\frac{d}{d-1}} \bar{r} \le - \log(n)} + \Pr{ -\log(n) \le \sqrt{\frac{d}{d-1}} \bar{r} \le  \hattau + \beta}.
    \end{align*} 
    From the tail bound in \Cref{lem:tailofz}, we get that the first term is $n^{-\omega(1)}$.
    Recalling that $|\hattau| \le \log(n)$, defining $\alpha = -\log(n)$ and integrating \eqref{eq:stuff}, we get that 
    \begin{align*}
        \Pr{ -\log(n) \le \sqrt{\frac{d}{d-1}} \bar{r} \le  \hattau + \beta} &\ge \Pr{ \alpha \le r \le  \hattau + \beta} -
        \sum_{j=1}^{s-2} \frac{C\log(n)^{3j}}{d^{1 + j/2}} \int_{\alpha}^{\hattau + \beta} \phi(x) \text{d} x - o\left( \frac{\log(n)}{d^{(s-2)/2}} \right) \\
        &\ge \Pr{ \alpha \le r \le  \hattau + \beta} - \frac{1}{d}\sum_{j=1}^{s-2} \bigO{ \left( \frac{\log^3(n)}{\sqrt{d}} \right)^j } \Phi(\hattau + \beta) - \frac{p}{d} \\
        &\ge \Pr{ \alpha \le r \le  \hattau + \beta} - (1 + o(1))\frac{\Phi(\hattau + \beta)}{d} - \frac{p}{d}
    \end{align*} if we choose $s$ to be a sufficiently large constant such that $\frac{\log(n)}{d^{(s-2)/2}} \le \frac{p}{d}$, which is possible since $d \ge n^\gamma$. We further note that \begin{align*}
        \Phi(\hattau + \beta) &\le \Phi(\hattau) + \beta\sup_{x \in [\hattau, \hattau+\beta]}\phi(x)\\ 
        &= p + \beta\sup_{x \in [\hattau, \hattau+\beta]}\phi(x)\\ 
        &= p + \beta \phi(\hattau) \sup_{y \in [0, \beta]} \exp\left( -\frac{1}{2}\left(2y\hattau + y^2\right)  \right) =p + (1-o(1))\beta\phi(\hattau)
    \end{align*}
    as $\absolute{\hattau} \le \log(n)$ and $\beta = \bigO{\frac{1}{\sqrt{d}}} = n^{-\Omega(1)}$. 
    Further, we have that $\phi(\hattau) \le Cp$ by \Cref{cor:phip}. 
    Combined this yields that $\Phi(\hattau + \beta) \le p ( 1 + O(1/\sqrt{d}) )$, so \begin{align*}
        \Pr{\bar{r} \le \hattau + \beta} &\ge \Pr{ \alpha \le r \le  \hattau + \beta} - \bigO{\frac{p}{d}} \\
        &\ge \Pr{ r \le  \hattau} - \Pr{r \le \alpha} + \Pr{ \hattau \le r \le \hattau + \beta} - \bigO{\frac{p}{d}}\\
        &\ge p - n^{-\omega(1)} - \bigO{\frac{p}{d}} +\Pr{ \hattau \le r \le \hattau + \beta}
    \end{align*} 
    by the tail bound from \Cref{lem:tailofz}. 
    Now, applying another Edgeworth expansion to $f_r$ using \Cref{thm:univariateedgeworthdensity} yields that $\sup_{ x \in [\hattau, \hattau + \beta] } \left| \frac{f_r(x)}{\phi(x)}  \right| = 1 \pm o(1)$, so \begin{align*}
        \Pr{ \hattau \le r \le \hattau + \beta} \ge (1-o(1))\beta\sup_{x \in [\hattau, \hattau+\beta]}\phi(x) \ge (1-o(1))\beta\phi(\hattau) = \Omega\left(\frac{p}{\sqrt{d}}\right)
    \end{align*} as noted previously. This shows the first part of our claim, i.e., $\Prnop{u \sim v \mid \mathbf{x}_{(1)}} \ge p + \frac{Cp}{\sqrt{d}}$ for $u,v$ being either both in $A$ or both in $B$. The second part follows by an analogous calculation.
\end{proof}

\noindent In total, this allows us to prove \Cref{lem:spectrallowerbound}. \begin{proof}[Proof of \Cref{lem:spectrallowerbound}]
    For a set $M\subseteq V$, we denote by $\deg(u \rightarrow M)$ the number of neighbors of $u$ in the set $M$. 
    It is not hard to notice that \begin{align*}
        \mathbf{y}_1^\top \Adj \mathbf{y}_1 &= \frac{\tilde{\mathbf{y}}_1^\top \Adj \tilde{\mathbf{y}}_1}{\|\tilde{\mathbf{y}}_1\|^2} \\&= \frac{1}{\|\tilde{\mathbf{y}}_1\|^2} \left(\sum_{u\in A} \left( \deg(u \rightarrow A) - \deg(u \rightarrow B) \right) +\sum_{u\in B} \left( \deg(u \rightarrow B) - \deg(u \rightarrow A) \right) \right).
    \end{align*} 
    In other words, evaluating the quadratic form $\tilde{\mathbf{y}}_i^\top \Adj \tilde{\mathbf{y}}_i$ is the difference of the number of intra-cluster edges and inter-cluster edges. 
    To estimate the sum, we recall that the sets $A, B$ are defined as a function of the position of all vertices in our fixed, first dimension, denoted by $\mathbf{x}_{(1)}$.
    We call $\mathbf{x}_{(1)}$ \emph{good} if \begin{align*}
        &\absolute{|A| - \Expectedsub{\mathbf{x}_{(1)} \sim \text{Unif}([-1/2,1/2]^n)}{|A|}} \le 10\sqrt{n\log(n)} \text{ and}\\& \absolute {|B| - \Expectedsub{\mathbf{x}_{(1)} \sim \text{Unif}([-1/2,1/2]^n)}{|B|}} \le  10\sqrt{n\log(n)}.
    \end{align*} 
    Applying the Bernstein bound (\Cref{thm:bernstein}), we see that $\mathbf{x}_{(1)}$ is good with probability $1 - 1/n^2$, as  
    \begin{align*}
       &\Prsub{\mathbf{x}_{(1)} \sim \text{Unif}([-1/2,1/2]^n)}{ \left| |A| - \Expectedsub{\mathbf{x}_{(1)} \sim \text{Unif}([-1/2,1/2]^n)}{|A|} \right| \ge 10\sqrt{n\log(n)} } \\ & \hspace{7cm}\le 2\exp\left( \frac{-100n\log(n)}{2n + 10\sqrt{n\log(n)}} \right)
        \le \frac{1}{n^2}.
    \end{align*} 
    Now, we can apply \Cref{lem:probforeigenvalues} to estimate $\deg(u \rightarrow A)$ and the related terms. To this end, we first note that there is a constant $0  < \alpha < 1$ such that \begin{align*}
        \Expectedsub{\mathbf{x}_{(1)} \sim \text{Unif}\left([-1/2, 1/2]^n\right)}{|A|} = \Expectedsub{\mathbf{x}_{(1)} \sim \text{Unif}\left([-1/2, 1/2]^n\right)}{|B|} = \alpha n.
    \end{align*} 
    Further, assuming that $\mathbf{x}_{(1)}$ is good, we get by \Cref{lem:probforeigenvalues} that, for any $u \in A$, 
    \begin{align*}
        \Expectedsub{\Gbf \sim \RGG | \mathbf{x}_{(1)}}{\deg(u \rightarrow A)} \ge p(\alpha n - 10\sqrt{n\log(n)}) + \Omega\left(\frac{np}{\sqrt{d}}\right) \ge \alpha p n + \Omega\left(\frac{np}{\sqrt{d}}\right)
    \end{align*} 
    since $d \ll (np)^{1-\varepsilon}$.  Moreover, it is easy to see that the random variable $\deg(u \rightarrow A) - \Expected{\deg(u \rightarrow A)}$ (where the randomness is over the position of all vertices in dimensions $\{2,\ldots,d\}$ while the positions in dimension $1$ are fixed)  can be represented as a sum of independent zero mean random variables $X_1,\ldots, X_{|A|}$ bounded in absolute value by $1$. Furthermore, each $X_i$ can be represented as $X_i = Y_i - p_i$ where $Y_i$ is a Bernoulli RV with expectation $p_i$, while $p_i = (1 + o(1))p$ by \Cref{lem:probforeigenvalues}. Hence, $X_i$ takes a value of $1 - p_i$ w.p. $p_i$ and a value of $-p_i$ w.p. $1- p_i$. This implies that \begin{align*}
        \Expectedsub{\Gbf \sim \RGG | \mathbf{x}_{(1)}}{X_i^2} \le p_i(1-p_i)^2 + (1- p_i)p_i^2 \le 2p_i \le 3 p.
    \end{align*} Thus, by a Bernstein bound (\Cref{thm:bernstein}), we have \begin{align*}
        &\Prsub{\Gbf \sim \RGG | \mathbf{x}_{(1)}}{ \left|\deg(u \rightarrow A) - \Expectedsub{\Gbf \sim \RGG | \mathbf{x}_{(1)}}{\deg(u \rightarrow A)}\right| \ge 10\sqrt{np\log(n)}}\\ 
        &\hspace{2cm}\le 2 \exp\left( -\frac{-100np\log(n)}{6p|A| + 10\sqrt{np\log(n)}} \right) \le 2 \exp\left( -\frac{-100np\log(n)}{6np + 10\sqrt{np\log(n)}} \right) \le \frac{1}{n^2}.
    \end{align*} This certifies that there is a constant $C > 0$ such that  
    \begin{align*}
        \deg(u \rightarrow A) &\ge \alpha np + \Omega\left(\frac{np}{\sqrt{d}}\right) - 10\sqrt{np\log(n)} \ge \alpha np  + \frac{Cnp}{\sqrt{d}}
    \end{align*} holds with probability $1 - 1/n^2$ over the randomness in $\RGG | \mathbf{x}_{(1)}$ whenever we have a good $\mathbf{x}_{(1)}$. Note that we also applied the fact that $d \ll (np)^{1-\varepsilon}$ above. An entirely analogous calculation shows that $\deg(u \rightarrow B) \le \alpha np - \frac{Cnp}{\sqrt{d}}$ holds as well. Noting that $\|\tilde{\mathbf{y}}_1\|^2 \le n$ and that $\mathbf{x}_{(1)}$ is good w.p. $1 - 1/n^2$, we get that \begin{align*}
        \mathbf{y}_1^\top \Adj \mathbf{y}_1 \ge \frac{(|A| + |B|)2Cnp}{n\sqrt{d}} = \Omega\left(\frac{np}{\sqrt{d}}\right)
    \end{align*} holds with probability $1 - 2/n^2$. Since this is true for each $\mathbf{y}_i$, we get from a union bound that it also holds for all $\mathbf{y}_i$ w.h.p.

    We finish the proof by arguing that the constructed random vectors are w.h.p. pairwise almost orthogonal and that we can lift our lower bound on their quadratic forms from $\Adj$ to $\Adjpp$.
    For the former, note that, for each $j$, $\|\tilde{\mathbf{y}}_j\|^2$ is a Binomial random variable with success probability bounded away from $0$ and $n$ trails, and therefore $\|\tilde{\mathbf{y}}_j\|^2 = \Omega(n)$ with probability $1 - \bigO{n^{-3}}$. 
    Moreover, for every $i \neq j$, $\tilde{\mathbf{y}}_j^\top \tilde{\mathbf{y}}_i$ is a sum of $n$ independent random variables in $\{-1, 0, 1\}$, each with mean $0$. 
    Hence, by McDiarmid's inequality, we get that $|\tilde{\mathbf{y}}_j^\top \tilde{\mathbf{y}}_i| = \bigOtilde{\sqrt{n}}$ with probability $1 - \bigO{n^{-3}}$.
    Combining both yields that $|\mathbf{y}_j^\top \mathbf{y}_i| = \bigOtilde{1/\sqrt{n}}$ for each $i \neq j$ with probability $1 - \bigO{n^{-3}}$, and applying a union bound proves that w.h.p. all pairs of vectors are almost orthogonal.
    To lift our bound on the quadratic form from $\Adj$ to $\Adjpp$, we further observe that, by McDiarmid's inequality, $|\mathds{1}^\top \tilde{\mathbf{y}}_i| = \bigOtilde{\sqrt{n}}$ with probability $1 - \bigO{n^{-2}}$.
    Hence, using again that $\tilde{\mathbf{y}}_j\|^2 = \Omega(n)$ with probability $1 - \bigO{n^{-3}}$, we can apply the union bound to obtain that w.h.p.
    \[
        |\mathbf{y}_i^\top(p \mathds{1} \mathds{1}^\top)\mathbf{y}_i| = p |\mathds{1}^\top \mathbf{y}_i|^2 = \bigOtilde{p} \le o\left(\frac{np}{\sqrt{d}}\right) 
    \]
    for all $i$.
\end{proof}

\subsubsection{$L_\infty$-Norm}\label{sec:linftynormlowerbound}

For $L_\infty$-norm, we prove the following. \begin{lemma}\label{lem:spectrallowerboundinfty}
    Consider a $\Gbf \sim \RGGinftyp{n}{d}{p}$ and assume that $d \ll \sqrt{np}$. With probability $1 - o(1)$, there are unit vectors $\mathbf{y}_1, \mathbf{y}_2, \ldots, \mathbf{y}_r$ with $r = \Theta(d^2)$ such that \begin{align*}
        \mathbf{y}_1^\top \left(\Adjpp\right) \mathbf{y}_1 = \Theta\left(  \frac{np}{d} \right) \text{ and } \mathbf{y}_i^\top\mathbf{y}_j = \bigOtilde{\frac{1}{\sqrt{n}}} \text{ for all } i \neq j.
    \end{align*} In particular this implies that 
    \begin{align*}
        \max\{ |\lambda_2(\Adj)|, |\lambda_n(\Adj)| \} = \Omega\left( \frac{np}{d} \right).
    \end{align*} 
\end{lemma}

\noindent We construct the vectors $\mathbf{y}_1, \mathbf{y}_2, \ldots, \mathbf{y}_r$ in a similar way as before. The difference will be that we do not get one vector per dimension, but $\tilde{\Theta}(d)$ many. Again, we construct our vectors by sampling the position of all vertices in a fixed dimension, which we denote by $\mathbf{x}_{(1)}$. 
Now, we shall define a region within our fixed dimension that consists of two opposing circular arcs, each of length $\xi/2 = \Theta(\log(1/p)/d)$ as depicted in \Cref{fig:evq}, right. 
Similarly to the notation used previously, we call the set of vertices falling into one of the arcs $A$, and the set of vertices falling into the other $B$. 
Then, we define a vector $\tilde{\mathbf{y}}$ just like before, by setting
\begin{align*}
    \tilde{\mathbf{y}}(v) = \begin{cases}
        +1 &\text{if } v \in A\\
        -1 &\text{if } v \in B\\
        \hspace{.23cm}0  &\text{otherwise.}
    \end{cases}
\end{align*}
Finally, our approximate eigenvector $\mathbf{y}$ will be the normalized version of the above, i.e., $\mathbf{y} = \tilde{\mathbf{y}}/ \|\tilde{\mathbf{y}}\|$. 
It is easy to see that---since the length of our circular arcs is only $\xi/2 = \Theta(\log(1/p)/d)$---we can define $\Theta(d)$ orthogonal such $\mathbf{y}$ \emph{for each} dimension, such that in the end, we end up with $\Theta(d^2)$ vectors in total. 
To prove that the vectors constructed this way are indeed approximate eigenvectors, we again show that inter-cluster edges are less likely than intra-cluster edges. In case of inter-cluster edges, this is in fact trivial, since it is easy to see that for $u\in A$ and $v \in B$, the $L_\infty$-distance between $u,v$ is above the threshold, no matter what happens in the remaining $d -1$ dimensions. Accordingly, by construction of $A$ and $B$, there are no inter-cluster edges. However, for intra-cluster edges, we prove the following simple lemma.
\begin{lemma}\label{lem:probforeigenvaluesinfty}
     For all $\mathbf{x}_{(1)}$, we have that \begin{align*}
        \forall \{u,v\} \in \binom{A}{2} \cup \binom{B}{2}: \Prsub{\Gbf \sim \RGGinfty | \mathbf{x}_{(1)} }{u \sim v} = p^{\frac{d-1}{d}} = p\left(1 + \Omega\left(\frac{1}{d}\right)\right).
    \end{align*} where $\RGGinfty | \mathbf{x}_{(1)}$ denotes the law of $\RGGinfty$ conditional on fixing the positions of all vertices in dimension $1$ according to $\mathbf{x}_{(1)}$.
\end{lemma} \begin{proof}
    Conditional on $u,v \in A$ or $u,v \in B$ for a fixed dimension, the probability that $u$ and $v$ are below distance $1 - \xi$ in each of the remaining $d-1$ dimensions is exactly $(1-\xi)^{d-1} = p^{\frac{d-1}{d}}$. Applying a Taylor series then yields the desired result.
\end{proof}

%\begin{figure}
%    \centering
%    \includegraphics[width=0.25\linewidth]{torusinfty.pdf}
%    \caption{Illustration for the construction of explicit eigenvectors for a $\Gbf\sim \RGGinfty$.}
%    \label{fig:ev}
%\end{figure}

\noindent With this, we can prove \Cref{lem:spectrallowerboundinfty}. \begin{proof}[Proof of \Cref{lem:spectrallowerboundinfty}] 
    As in the proof of \Cref{lem:spectrallowerbound}, for any set $M\subseteq V$, we denote by $\deg(u \rightarrow M)$ the number of neighbors of $u$ in $M$. 
    Furthermore, \begin{align*}
        \mathbf{y}_1^\top \Adj \mathbf{y}_1 & = \frac{\tilde{\mathbf{y}}_1^\top \Adj \tilde{\mathbf{y}}_1}{\|\tilde{\mathbf{y}}_1\|^2} \\&= \frac{1}{\|\tilde{\mathbf{y}}_1\|^2} \left(\sum_{u\in A} \left( \deg(u \rightarrow A) - \deg(u \rightarrow B) \right) +\sum_{u\in B} \left( \deg(u \rightarrow B) - \deg(u \rightarrow A) \right) \right)\\
        &=\frac{1}{\|\tilde{\mathbf{y}}_1\|^2} \left(\sum_{u\in A} \deg(u \rightarrow A) + \sum_{u\in B} \deg(u \rightarrow B) \right)
    \end{align*} because no inter-cluster edges exist. To estimate the sum, we recall that the sets $A, B$ are defined as a function of the position of all vertices in our fixed, first dimension, denoted by $\mathbf{x}_{(1)}$. We call $\mathbf{x}_{(1)}$ \emph{good} if \begin{align*}
       &\absolute{|A| - \Expectedsub{\mathbf{x}_{(1)} \sim \text{Unif}([-1/2,1/2]^n)}{|A|}} \le 10\log(n)\sqrt{\frac{n}{d}} \text{ and} \\& \absolute{|B| - \Expectedsub{\mathbf{x}_{(1)} \sim \text{Unif}([-1/2,1/2]^n)}{|B|}} \le 10\log(n)\sqrt{\frac{n}{d}}.
    \end{align*} 
    Note that $\mathbf{x}_{(1)}$ is good with probability $1 - 1/n^2$ as, by the Bernstein bound (\Cref{thm:bernstein}), 
    \begin{align*}
        &\Prsub{\mathbf{x}_{(1)} \sim \text{Unif}([-1/2,1/2]^n)}{ \left| |A| - \Expectedsub{\mathbf{x}_{(1)} \sim \text{Unif}([-1/2,1/2]^n)}{|A|} \right| \ge 10\log(n)\sqrt{\frac{n}{d}} } \\
        &\hspace{6.5cm} \le 2\exp\left( \frac{- \frac{100n\log^2(n)}{d}}{n\xi + 10\log(n)\sqrt{\frac{n}{d}}} \right)
        \le \frac{1}{n^2}
    \end{align*} since $\xi = \Theta(\log(1/p)/d)$. Moreover, it is easy to note that \begin{align*}
        \Expectedsub{\mathbf{x}_{(1)} \sim \text{Unif}\left([-1/2, 1/2]^n\right)}{|A|} = \Expectedsub{\mathbf{x}_{(1)} \sim \text{Unif}\left([-1/2, 1/2]^n\right)}{|B|} = \frac{1}{2}\xi n.
    \end{align*} Now, applying \Cref{lem:probforeigenvaluesinfty}, we get that for every good $\mathbf{x}_{(1)}$ and any $u \in A$, \begin{align*}
        \Expectedsub{\Gbf \sim \RGGinfty | \mathbf{x}_{(1)}}{\deg(u \rightarrow A)}& \ge p|A| \ge p\left(\frac{1}{2}\xi n - 10\log(n)\sqrt{\frac{n}{d}}\right) \\&\ge \frac{1}{2}\xi np - 10p\log(n)\sqrt{\frac{n}{d}} = \Omega\left( \frac{np}{d} \right)
    \end{align*} since $d \ll \sqrt{np}$. Moreover, it is easy to see that the random variable $\deg(u \rightarrow A) - \Expected{\deg(u \rightarrow A)}$ (where the randomness is over the position of all vertices in dimensions $\{2,\ldots,d\}$ while the positions in dimension $1$ are fixed)  can be represented as a sum of independent zero mean random variables $X_1,\ldots, X_{|A|}$ bounded in absolute value by $1$. Furthermore, each $X_i$ can be represented as $X_i = Y_i - p_i$ where $Y_i$ is a Bernoulli RV with expectation $p_i$, while $p_i = (1 + o(1))p$ by \Cref{lem:probforeigenvaluesinfty}. Hence, $X_i$ takes a value of $1 - p_i$ w.p. $p_i$ and a value of $-p_i$ w.p. $1- p_i$. This implies that \begin{align*}
        \Expectedsub{\Gbf \sim \RGGinfty| \mathbf{x}_{(1)}}{X_i^2} \le p_i(1-p_i)^2 + (1- p_i)p_i^2 \le 2p_i \le 3 p.
    \end{align*} Thus, by a Bernstein bound (\Cref{thm:bernstein}), we have \begin{align*}
        &\Prsub{\Gbf \sim \RGGinfty | \mathbf{x}_{(1)}}{ \left|\deg(u \rightarrow A) - \Expectedsub{\Gbf \sim \RGGinfty | \mathbf{x}_{(1)}}{\deg(u \rightarrow A)}\right| \ge 10\log(n)\sqrt{\frac{np}{d}}}\\ 
        &\hspace{3cm}\le 2 \exp\left( -\frac{-100\frac{np\log^2(n)}{d}}{6p|A| + 10\log(n)\sqrt{\frac{np}{d}}} \right) \le 2 \exp\left( -\frac{-100\frac{np\log^2(n)}{d}}{ 6\xi np + 10\log(n)\sqrt{\frac{np}{d}}} \right) \le \frac{1}{n^2}
    \end{align*} Where we used that $|A| \le \xi np$ because $\mathbf{x}_{(1)}$ is good. This certifies that 
    \begin{align*}
        \deg(u \rightarrow A) &\ge \frac{1}{2}\xi np-10\log(n)\sqrt{\frac{np}{d}} - 10p\log(n)\sqrt{\frac{n}{d}} = \Omega\left(\frac{np}{d}\right)
    \end{align*} holds with probability $1 - 1/n^2$ whenever we have a good $\mathbf{x}_{(1)}$. Note that we also applied the fact that $d \ll \sqrt{np}$ above. An entirely analogous calculation shows that $\deg(u \rightarrow B) = \Omega\left(\frac{np}{d}\right)$ holds as well. Noting that $\|\tilde{\mathbf{y}}_1\|^2 = \bigO{\frac{n}{d}}$ for good $\mathbf{x}_{(1)}$, and that $\mathbf{x}_{(1)}$ is good w.p. $1 - 1/n^2$, we get that \begin{align*}
        \mathbf{y}_1^\top \Adj \mathbf{y}_1 \ge \Omega\left(  \frac{np}{d}\right)
    \end{align*} holds with probability $1 - 2/n^2$. Since this is true for each $\mathbf{y}_i$, we get from a union bound that it also holds for all $\mathbf{y}_i$ w.h.p.
\end{proof}

\section{Acknowledgments}
We would like to thank Kiril Bangachev for fruitful discussions during early stages of this work.

\bibliographystyle{abbrv}	
\bibliography{bibliography}

\appendix
\section{Deferred Proofs}

\subsection{Deferred Proofs from \Cref{sec:probestimates}}\label{sec:deferredprelims}

\thresholdestimate*
\begin{proof}
    We prove that choosing $b = \frac{Cp\log^{4}(n)}{\sqrt{d}}$ for some suitable $C > 0$ ensures that \begin{align*}
        \Pr{\Deltaboldhat{}(1) \le \Phi^{-1}(p) - b} < p \text{ and } \Pr{\Deltaboldhat{}(1) \le \Phi^{-1}(p) + b} > p
    \end{align*} which implies our lemma. Since $\Deltaboldhat{}(1)$ is a sum of i.i.d. continuous random variables of bounded density, we get from \Cref{thm:univariateedgeworthdensity} that the density $g$ of $\Deltaboldhat{}(1)$ satisfies \begin{align*}
        \left| g(x) - \left( \phi(x) + \phi(x)\sum_{j=1}^{s-2} \frac{p_j(x)}{d^{j/2}} \right) \right| =  o\left( \frac{1}{d^{(s-2)/2}} \right).
    \end{align*} We choose $s$ sufficiently large such that $\frac{1}{d^{(s-2)/2}} \le \frac{p}{d}$ which is possible since we assume $p \ge 1/n$ and $d \ge n^{\gamma}$ for some arbitrarily small but constant $\gamma > 0$. Setting $a= -\log(n)$ and applying the above, we thus get that for any $b > 0$ sufficiently small, \begin{align*}
        \Pr{\Deltaboldhat{}(1) \le \Phi^{-1}(p) - b} &= \Pr{\Deltaboldhat{}(1) \le a} + \int_{a}^{\Phi^{-1}(p) - b} g(x) \text{d} x.
    \end{align*} We immediately get that the first term is $\le n^{-\omega(1)}$ by \Cref{lem:tailofz}. The second term can be bounded as\begin{align}\label{eq:iseriouslycantthinkofasuitablenameforthisrightnow}
        \int_{a}^{\Phi^{-1}(p) - b} g(x) \text{d} x &\le \int_{a}^{\Phi^{-1}(p) - b} \left( 1 +  \sum_{j=1}^{s-2} \frac{p_j(x)}{d^{j/2}}\right) \phi(x)\text{d} x + \bigO{\frac{\log(n)}{d^{(s-2)/2}}} \nonumber \\
        &\le \left( 1 +  \sum_{j=1}^{s-2} \bigO{\frac{\log^{3j}(n)}{d^{j/2}}}\right)\int_{a}^{\Phi^{-1}(p) - b} \phi(x)\text{d} x + \bigO{\frac{\log(n)}{d^{(s-2)/2}}} \nonumber \\
        &\le \left( 1 +  \bigO{\frac{\log^{3}(n)}{\sqrt{d}}}\right)\int_{a}^{\Phi^{-1}(p) - b} \phi(x)\text{d} x + \bigO{\frac{p\log(n)}{\sqrt{d}}}
    \end{align} where we used that $\Phi^{-1}(p) = \bigO{\log(n)}$, $p_j$ is a polynomial of degree $3j$, so $|p_j(x)| \le \bigO{|a|^{3j}} = \bigO{\log^{3j}(n)}$. Now, assuming further that $b \le n^{-\varepsilon}$ for an arbitrary $\varepsilon > 0$, we get that \begin{align*}
        \int_{a}^{\Phi^{-1}(p) - b} \phi(x)\text{d} x &= \int_{-\infty}^{\Phi^{-1}(p)} \phi(x)\text{d} x - \int_{\Phi^{-1}(p) - b}^{\Phi^{-1}(p)} \phi(x)\text{d} x \\
        &\le p - b\phi\left( \Phi^{-1}(p) - b \right)\\
        &\le p - b\phi\left(\Phi^{-1}(p)\right) \exp\left(b\Phi^{-1}(p) - b^2/2\right)\\
        &= p - (1-o(1))b\phi\left(\Phi^{-1}(p)\right)\\
        &\le p - C'bp
    \end{align*} where we used that $\exp\left(b\Phi^{-1}(p) - b^2/2\right) = 1 - o(1)$ since the absolute value of the  exponent is $o(1)$ because $|\Phi^{-1}(p)| \le \log(n) \text{ and } b \le n^{-\varepsilon}$. In the last step, we further used that $\phi\left(\Phi^{-1}(p)\right) \ge C'\Phi\left(\Phi^{-1}(p)\right) = C'p$ for some constant $C'$ by \Cref{prop:cdfpdf}. Thus, going back to (\ref{eq:iseriouslycantthinkofasuitablenameforthisrightnow}), we get \begin{align*}
        \Pr{\Deltaboldhat{}(1) \le \Phi^{-1}(p) - b} \le p + \bigO{\frac{p\log^{3}(n)}{\sqrt{d}} + \frac{p\log(n)} {\sqrt{d}}} - C'pb + n^{-\omega(1)}.
    \end{align*} Since $p \ge 1/n$, choosing $b = \frac{Cp\log^{3}(n)}{\sqrt{d}}$ for sufficiently large constant $C > 0$ ensures that \[\Pr{\hat{\mathbf{z}}(1) \le \Phi^{-1}(p) - b} < p\] as desired. For the other direction, we proceed similarly and note that for every  sufficiently small $b > 0$, \begin{align*}
        \Pr{\Deltaboldhat{}(1) \le \Phi^{-1}(p) + b} \ge \left( 1 - \bigO{\frac{\log^{3}(n)}{\sqrt{d}}}\right)\int_{a}^{\Phi^{-1}(p) + b} \phi(x)\text{d} x - \bigO{\frac{p\log(n)}{\sqrt{d}}}
    \end{align*} and \begin{align*}
        \int_{a}^{\Phi^{-1}(p) + b} \phi(x)\text{d} x &\ge p - \int_{-\infty}^{a} \phi(x)\text{d} x + \int_{\Phi^{-1}(p)}^{\Phi^{-1}(p) + b} \phi(x)\text{d} x \\
        &\ge p - n^{-\omega(1)} + b\phi\left(\Phi^{-1}(p)\right)\\
        &\ge p - n^{-\omega(1)} + bp\frac{1 - e^{-3/2}}{2\log(n)}
    \end{align*} by \Cref{prop:cdfpdf}. In total, \begin{align*}
        \Pr{\Deltaboldhat{}(1) \le \Phi^{-1}(p) + b} \ge  p - n^{-\omega(1)} - \bigO{\frac{p\log^{3}(n)}{\sqrt{d}} + \frac{p\log(n)} {\sqrt{d}}} + bp\frac{1 - e^{-3/2}}{2\log(n)}
    \end{align*} so choosing $b = \frac{Cp\log^{4}(n)}{\sqrt{d}}$ for sufficiently large constant $C > 0$ ensures that 
    $$\Pr{\Deltaboldhat{}(1) \le \Phi^{-1}(p) + b} > p$$  as desired.
    
    %where $X \sim \mathcal{N}(0,1)$. Now, since $p$ is a constant, we know that the density of $X$ in the interval $[p, p + \varepsilon]$ is $\Omega(1)$, so the above probability is at least $c_p \cdot C/\sqrt{d}$ for some constant $c_p > 0$ that can depend on $p$. Thus, choosing $C$ large enough yields that 
    %$\mathbb{P}(Z \le \Phi^{-1}(p) + C/\sqrt{d}) \ge p$ which implies  that $\tilde{\tau} \le \Phi^{-1}(p) + C/\sqrt{d}$. Entirely analogously, we can also show that $\tilde{\tau} \ge \Phi^{-1}(p) - C/\sqrt{d}$ for some $C > 0$ large enough.
\end{proof}

\tauissmall*
\begin{proof}
    If $p = \Theta(1), p < 1-\varepsilon$, then it is easy to see that $\Phi^{-1}(p) = \Theta(1)$, so we get $\tilde{\tau} = \Theta(1)$ by \Cref{lem:thresholdestimate} as well. On the other hand, if $p = o(1)$, we note that $\Phi^{-1}(p) < 0$ and we get from \Cref{prop:cdfpdf} that for all $x \le -\frac{1}{2}\log(n)$ \begin{align*}
        \Phi(x) \le  \phi(x) \le \frac{1}{\sqrt{2\pi}}\exp\left( -\frac{\log^2(n)}{8}\right) = n^{-\omega(1)}. 
    \end{align*} On the other hand, we know that $p = \Phi(\Phi^{-1}(p))$, so $\Phi^{-1}(p) \le -\frac{1}{2}\log(n)$ would contradict $p \ge n^{-c}$. Applying \Cref{lem:thresholdestimate} then implies that $|\tilde{\tau}| \le \log(n)$ holds as well.
\end{proof}

\phip*
\begin{proof}
    We first note that $\Phi^{-1}(p) = -\omega(1)$ because $p = o(1)$, and that $|\Phi^{-1}(p)| \le \log(n)$ by \Cref{cor:tauissmall}.
    Hence, by \Cref{lem:thresholdestimate}, we have \begin{align*}
        \phi(\tilde{\tau}) &\ge \phi\left(\Phi^{-1}(p) - \frac{Cp\log^4(n)}{\sqrt{d}}\right) = \exp\left( \frac{C\Phi^{-1}(p)p\log^4(n)}{\sqrt{d}} - \frac{1}{2}\left(\frac{Cp\log^4(n)}{\sqrt{d}}\right)^2\right) 
        \phi\left(\Phi^{-1}(p)\right)\\ &= (1 - o(1)) \phi\left(\Phi^{-1}(p)\right),
    \end{align*} since $|\Phi^{-1}(p)| \le \log(n)$ and $d \ge n^{\gamma}$. Now, since $\Phi^{-1}(p) \le -1$ for sufficiently large $n$, we get by \Cref{prop:cdfpdf} that $\phi\left(\Phi^{-1}(p)\right) \ge \Phi\left(\Phi^{-1}(p)\right) = p$, so the first statement follows. The fact that $\Phi(\hattau) \le \phi(\hattau) \le (1 + o(1))p$ follows by \Cref{prop:cdfpdf} and noting that $\hattau = o(1)$ for $p = o(1)$. Finally, The fact that $C_1p \le \phi(\hattau) \le C_2p \text{ and } C_1p \le \Phi(\hattau) \le C_2p$ for all $\frac{1}{n^c} \le p \le 1-\varepsilon$ follows for $p = o(1)$ by the above combined with \Cref{prop:cdfpdf}, and for $p= \Omega(1)$ since $|\hattau| = \Theta(1)$ due to \Cref{cor:tauissmall}.
\end{proof}

\tailofz*
\begin{proof}
    It suffices to show the statement for $j = 1$ as all components of $\Deltaboldhat{}$ have the same marginal distribution.
    We note that \begin{align*}
        \sqrt{d}\Deltaboldhat{}(1) = \sum_{i = 1}^d \frac{\boldsymbol{\Delta}_i(\{u, v\})-\mu}{\sigma}
    \end{align*} for some pair of vertices $u, v$, and this is a sum of i.i.d. RVs with mean $0$ and absolute value bounded by some constant $M$ a.s. Therefore, applying \Cref{thm:bernstein} yields that \begin{align*}
        \Pr{\Deltaboldhat{}(1) \ge c \log(n)} &= \Pr{\sqrt{d} \Deltaboldhat{}(1) \ge c \sqrt{d} \log(n)}\\
        &\le \exp \left( - \frac{c^2 d \log(n)^2}{2(d\sigma^2 + \frac{1}{3}Mc\sqrt{d}\log(n))} \right)\\
        &\le \exp \left( - \min\left\{ \frac{c^2\log(n)^2}{4M^2} , \frac{c\sqrt{d}\log(n)}{4 M} \right\} \right) = n^{-\omega(1)}.
    \end{align*} Doing the same for the lower tail finishes the proof of the first part. 
    To prove that the same statement holds for $\mathbf{z}$, we use the decomposition from \eqref{eq:znoise} ($\mathbf{z} = \Deltaboldhat{}/\zeta + \hat{\boldsymbol{\eta}}$) to conclude that \begin{align*}
        \Pr{\mathbf{z}(1) \le -c\log(n)} \le \Pr{|\Deltaboldhat{}(1) / \zeta| \le \frac{c}{2}\log(n)} + \Pr{|\etahat(1)| \le \frac{c}{2}\log(n)} \le n^{-\omega(1)}
    \end{align*} since $\zeta$ is at most a constant and $\eta \ge 1$ so both terms are $n^{-\omega(1)}$. 
\end{proof}

\subsection{Deferred Proofs from \Cref{sec:approxdensity}}\label{sec:deferrederror}
We further prove that a cycle and a chain in a single dimension meet \cramer (for the chain this holds with probability $\Omega(1)$). 
\begin{lemma}\label{lem:cramer}
    For a cycle or a chain of length $k \ge 3$, there is a function $\varepsilon(\delta)$ such that for any $j \in [d]$, we have that \begin{align*}
        |C_{\Deltaboldhat{j}}(\mathbf{t})| \le 1 - \varepsilon(\delta)
    \end{align*} whenever $\|\mathbf{t}\| \ge \delta$. For a chain of length $k = 2$, the two endpoints are arranged such that the above holds with constant probability. More precisely, there is an event $\mathcal{E}$ happening with probability $\Omega(1)$ under which there is a function $\varepsilon: \mathbb{R} \rightarrow \mathbb{R}$ such that $|C_{\Deltaboldhat{j}}(\mathbf{t})| \le 1 - \varepsilon(\delta)$ whenever $\|\mathbf{t}\| \ge \delta$.
\end{lemma}\begin{proof}
    We define an event $\mathcal{E}$ and split \begin{align*}
        C_{\Deltaboldhat{j}}(\mathbf{t}) = \Expected{\exp(i \mathbf{t}^\top\Deltaboldhat{j})} = \Pr{\mathcal{E}} \Expected{\exp(i \mathbf{t}^\top\Deltaboldhat{j}) \mid \mathcal{E}} + \Pr{\bar{\mathcal{E}}} \Expected{\exp(i \mathbf{t}^\top\hat{\mathbf{z}}) \mid \bar{\mathcal{E}}},
    \end{align*} so \begin{align*}
        |C_{\Deltaboldhat{j}}(\mathbf{t})| \le \Pr{\mathcal{E}} \left|\Expected{\exp(i \mathbf{t}^\top\Deltaboldhat{j}) \mid \mathcal{E}}\right| + \Pr{\bar{\mathcal{E}}}.
    \end{align*} Thus, it suffices to show that $\left|\Expected{\exp(i \mathbf{t}^\top\Deltaboldhat{j}) \mid \mathcal{E}}\right| \le 1 - \varepsilon(\delta)$ for some $\mathcal{E}$ with $\Pr{\mathcal{E}} = \Omega(1)$. We give an explicit event $\mathcal{E}$ for which this is met. To this end, denote the vertices of our cycle or chain by $v_1, \ldots, v_{k+1}$\footnote{In case of a cycle $v_1 = v_{k+1}$.} and the edges by $e_1, \ldots, e_k$. Note that the vector $\mathbf{t}$ associates the value $\mathbf{t}(j)$ to edge $e_j$. Since $\|\mathbf{t}\| \ge \delta$, there is some $j$ with $|\mathbf{t}(j)| \ge \delta/\sqrt{k}$. We assume w.l.o.g. that $j=1$ such that $|\mathbf{t}(1)| \ge \delta/\sqrt{k}$ is associated to $e_1 = \{v_1, v_2\}$ and $\mathbf{t}(2)$ is associated to $e_2 = \{v_2, v_3\}$. We consider two cases depending on the sign of $\mathbf{t}(1)$ and $\mathbf{t}(2)$.

    Case 1: $\text{sgn}(\mathbf{t}(1)) = \text{sgn}(\mathbf{t}(2))$. In this case we consider the event that $x_1 \in [-c, c] \eqqcolon I_1$, that $x_3 \in \left[\frac{1}{4} - c, \frac{1}{4} + c\right] \eqqcolon I_3$, and that $x_2 \in \left[\frac{3}{8} - c, \frac{3}{8} + c \right] \eqqcolon I_2$ for some sufficiently small constant $c > 0$. We further denote the expectation in the conditional probability space obtained by conditioning on $\mathcal{E}$ by $\tilde{\mathbb{E}}$ and use that \begin{align*}
        \left|\Expected{\exp(i \mathbf{t}^\top\Deltaboldhat{j}) \mid \mathcal{E}}\right| &= \left|\tilde{\mathbb{E}}_{x_1, x_3, \ldots, x_{k+1}}\left[ \tilde{\mathbb{E}}_{x_2} \left[ \exp(i \mathbf{t}^\top\Deltaboldhat{j}) \right] \right]\right| \\
        &= \left|\tilde{\mathbb{E}}_{x_1, x_3, \ldots, x_k}\left[ \exp\left(i\sum_{\ell=3}^k \mathbf{t}(\ell)\Deltaboldhat{j}(\ell)\right) \tilde{\mathbb{E}}_{x_2} \left[ \exp\left(i\sum_{\ell=1}^2 \mathbf{t}(\ell)\Deltaboldhat{j}(\ell)]\right) \right] \right]\right| \\
        &\le \tilde{\mathbb{E}}_{x_1,x_3 \ldots, x_k}\left[ \left|\exp\left(i\sum_{\ell=3}^k \mathbf{t}(\ell)\Deltaboldhat{j}(\ell)]\right)\right|\left| \tilde{\mathbb{E}}_{x_2} \left[ \exp\left(i\sum_{\ell=1}^2 \mathbf{t}(\ell)\Deltaboldhat{j}(\ell)]\right) \right] \right| \right]\\
        &= \tilde{\mathbb{E}}_{x_1,x_3, \ldots, x_k}\left[\left| \tilde{\mathbb{E}}_{x_2} \left[ \exp\left(i\sum_{\ell=1}^2 \mathbf{t}(\ell)\Deltaboldhat{j}(\ell)]\right) \right] \right| \right].
    \end{align*} To show that this is strictly smaller than $1$, it thus suffices to show that \begin{align}\label{eq:sup}
        \sup_{x_1, x_3,\ldots, x_k} \left| \tilde{\mathbb{E}}_{x_2} \left[ \exp\left(i\sum_{j=1}^2 \mathbf{t}(j)\Deltaboldhat{j}(\ell)]\right) \right] \right| < 1
    \end{align} where the supremum goes over all positions $x_1, x_3, \ldots, x_k$ that are in accordance with $\mathcal{E}$. 
    Recall that by definition of $\Deltaboldhat{j}$, we have $\Deltaboldhat{j}(\ell) = \frac{\Delta_j(e_\ell) - \mu}{\sigma}$ where $\Delta_1(\{u,v\}) = |x_u - x_v|_C^q$. We define \begin{align*}
        g_{x_1,x_3,\lambda}(x_2 ) \coloneqq \Deltaboldhat{j}(1) + \lambda\Deltaboldhat{j}(2) =  \frac{|x_1 - x_2|_C^q - \mu}{\sigma} + \lambda \frac{|x_2 - x_3|_C^q - \mu}{\sigma}
    \end{align*} and we note that \begin{align*}
        \sum_{\ell=1}^2 \mathbf{t}(\ell)\Deltaboldhat{j}(\ell) = \mathbf{t}(1)g_{x_1,x_3,\lambda}(x_2) \text{ for } \lambda = \frac{\mathbf{t}(2)}{\mathbf{t}(1)} 
    \end{align*} while $\frac{\mathbf{t}(2)}{\mathbf{t}(1)} \in [0,1]$ due to our assumptions. Now, it is easy to see that for all $\lambda \in [0,1]$, $g_{x_1,x_3}(x_2, \lambda)$ is a function strictly monotonically increasing in $x_2$ since for increasing $x_2 \in \left[\frac{3}{8} - c, \frac{3}{8} + c\right]$, both $|x_1 - x_2|_C^q$ and $|x_2 - x_3|_C^q$ are strictly increasing. Recalling that $x_2$ is uniform in $\left[\frac{3}{8} - c, \frac{3}{8} + c\right]$, we can treat $Y_{x_1,x_3,\lambda} = g_{x_1,x_3}(x_2, \lambda )$ as a random variable (for fixed $x_1,x_3,\lambda$) where the randomness comes solely from $x_2$. We claim that $Y_{x_1,x_3,\lambda}$ has a uniformly bounded variance and density over all $x_1 \in I_1, x_3 \in I_3$ and $\lambda \in [0,1]$. \begin{claim}\label{clm:boundeddensityandvariance}
        Denote by $f_{x_1,x_3,\lambda}(x)$ the density of $Y_{x_1,x_3,\lambda} = g_{x_1,x_3}(x_2, \lambda )$ for $x_2 \sim \text{Unif}\left(I_2\right)$. Then, there exist constants $M, s$ such that 
        \begin{align*}
            \sup_{\substack{x_1\in I_1, x_3 \in I_3 \\ \lambda \in [0,1], x \in \mathbb{R}}} f_{x_1,x_3,\lambda}(x) \le M \hspace{.5cm}\text{ and }\hspace{.5cm}    \sup_{\substack{x_1\in I_1, x_3 \in I_3 \\ \lambda \in [0,1]}} \Var{Y_{x_1,x_3,\lambda}} \le s
        \end{align*}
    \end{claim} 
    \noindent With the above claim, our result follows using the following proposition from \cite{Bobkov_Chistyakov_Goetze_2012}. \begin{proposition}[Theorem 1 in \cite{Bobkov_Chistyakov_Goetze_2012}]
        For a random variable $Y$ which has a density that is almost everywhere bounded by $M$ and variance bounded by $s$, we have that \begin{align*}
            |\Expected{ \exp(itY)}| \le 1 - \min \left\{ \frac{c_1}{M^2s^2}, \frac{c_2t^2}{M^2}\right\}
        \end{align*} where $c_1,c_2 > 0$ are universal constants.
    \end{proposition} 
    \noindent With the above and \Cref{clm:boundeddensityandvariance}, we get that \begin{align*}
        \sup_{x_1,x_3,\ldots, x_k} \left| \tilde{\mathbb{E}}_{x_1} \left[ \exp\left(i\sum_{\ell=1}^2 \mathbf{t}(\ell)\Deltaboldhat{j}(\ell)]\right) \right] \right|& = \sup_{x_1, x_3,\ldots, x_k} \left| \mathbb{E} \left[ \exp\left(i \mathbf{t}(1) Y_{x_1,x_3,\frac{\mathbf{t}(2)}{\mathbf{t}(1)}} \right) \right] \right| \\& \le 1 - \min \left\{ \frac{c_1}{M^2s^2}, \frac{c_2\delta^2}{M^2k}\right\},
    \end{align*} which proves (\ref{eq:sup}). So it only remains to prove \Cref{clm:boundeddensityandvariance}. \begin{proof}[Proof of \Cref{clm:boundeddensityandvariance}] The fact that the variance is uniformly bounded follows directly from the easy fact that $|g_{x_2, x_k}(x_1, \lambda)|$ is bounded by $|g_{x_2, x_k}(x_1, \lambda)| \le 1/\sigma$. To show that the density is bounded, we note that for any fixed $x_1,x_3, \lambda$, we have (if we abbreviate $g_{x_1,x_3,\lambda}$ by $g$) \begin{align*}
        \Pr{g(x_1) \le y} &= \Pr{x_1 \le g^{-1}(y)} \\
        \text{ so } f_{g(x_1)}(y) &= \frac{\text{d}}{\text{d} y} \Pr{x_1 \le g^{-1}(y)} \\
        &= \frac{\text{d} \Pr{x_1 \le g^{-1}(y)}}{\text{d} g^{-1}(y)}\frac{\text{d} g^{-1}(y)}{\text{d} y}\\
        &= \left.\frac{\text{d} \Pr{x_1 \le z}}{\text{d} z} \right|_{z=g^{-1}(y)} \left(\left.\frac{\text{d} g(z)}{\text{d} z}\right|_{z = g^{-1}(y)} \right)^{-1}\\
        &= \begin{cases}
            \frac{1}{|I_2|} \left(\frac{q|g^{-1}(y) - x_2|_C^{q-1} + q\lambda|g^{-1}(y) - x_k|_C^{q-1}}{\sigma}\right)^{-1} &\text{if } y \in \left[g\left(\frac{3}{8}-c\right), g\left(\frac{3}{8}+c\right)\right] \\
            0 & \text{else}
        \end{cases}
    \end{align*} which follows by the chain rule and the inverse function rule. Now, it is easy to see that this is uniformly bounded from above by some constant $M$ for all $x_2 \in I_2, x_k \in I_k, \lambda \in [0,1]$, provided only that $c$ was chosen small enough.
    \end{proof}

    Case 2: $\text{sgn}(\mathbf{t}(1)) \neq \text{sgn}(\mathbf{t}(2))$. In this case we can proceed analogously to case 1, the only difference is that we define the event $\mathcal{E}$ as the event that $x_1 \in [-c,c] \eqqcolon I_1, x_3 \in \left[\frac{1}{4} - c, \frac{1}{4} + c\right] \eqqcolon I_3$ (as before) and $x_2 \in \left[\frac{1}{8} - c, \frac{1}{8} + c\right]$ (this has changed compared to case 1). Then again \begin{align*}
        g_{x_1,x_3}(x_2, \lambda ) \coloneqq \Deltaboldhat{j}(1) + c\Deltaboldhat{j}(2) = \frac{|x_1 - x_2|_C^q - \mu}{\sigma} + \lambda \frac{|x_1 - x_3|_C^q - \mu}{\sigma}
    \end{align*} is a strictly increasing function for $\lambda \in [-1, 0]$ and $x_2 \in [-c,c], x_k \in \left[\frac{2}{3} - c, \frac{2}{3} + c\right]$.  After noting this, all other steps carry over analogously. In the case that $\mathbf{t}(2) = 0$, the argument from both case 1 and case 2 applies.

    This can be done whenever we deal with a chain or a cycle of length $k\ge 3$. For a chain of length $k = 2$ with endpoints $v_1, v_3$, and with intermediate vertex $v_2$, everything still works if the $v_1, v_3$ are placed in $I_1, I_3$, respectively. Since this occurs with probability $\Omega(1)$ over the (uniform) draw of the endpoints $v_1,v_3$, the result follows.
\end{proof}
Finally, we prove the following. 

\cgfderivative* \begin{proof}
    First of all, we define \begin{align*}
        \partial \Kpure(\mathbf{z}_j) &\coloneqq \left.\boldsymbol{\nabla}^{\otimes b}\Kpure\left(\mathbf{x}, \mathbf{z}_j\right)\right|_{\mathbf{x} = \frac{c\mathbf{t}}{\sqrt{d}}} - \left.\boldsymbol{\nabla}^{\otimes b}\Kpure\left(\mathbf{x}, \mathbf{z}_j\right)\right|_{\mathbf{x} = \mathbf{0}} \\
        \text{ and } \partial \Kmix(\mathbf{z}_j) &\coloneqq \left.\boldsymbol{\nabla}^{\otimes b}\Kmix\left(\mathbf{x}, \mathbf{z}_j\right)\right|_{\mathbf{x} = \frac{c\mathbf{t}}{\sqrt{d}}} - \left.\boldsymbol{\nabla}^{\otimes b}\Kmix\left(\mathbf{x}, \mathbf{z}_j\right)\right|_{\mathbf{x} = \mathbf{0}}
    \end{align*}
    Now, it is not hard to see that each entry of any $\mathbf{z}_j$ has the same marginal distribution. By definition of $\Kpure$, this means that $\Kpure\left(\mathbf{x}, \mathbf{z}_j\right)$ is the same function for each $j \in [d]$. Since $\boldsymbol{\nabla}^{\otimes b}\Kpure\left(\mathbf{x}, \mathbf{z}_j\right)$ is a continuous function, our statement follows in case of $\Kpure$.

    In case of $\Kmix$, the individual functions $\Kmix(\vecx, \mathbf{z}_j)$ are not the same for different $j$. Even though we could apply the same argument for each fixed $j$, this does not imply that $\left\| \partial \Kmix(\mathbf{z}_j) \right\|$ is bounded by  \emph{the same} $\varepsilon$ for all $j \in [d]$. To prove that our statement holds for the same $\varepsilon$ over all $j \in [d]$, we therefore derive an explicit upper bound on $\left\| \partial\Kmix(\mathbf{z}_j) \right\|$ using the relationship between cumulants and mixed moments combined with mixed moments that signal $\Deltaboldhat{j}$ underlying $\mathbf{z}$ is bounded almost surely. Before this, however, we note that for any $\vecx \in \mathbb{R}^k$, we have $\Kmix(\vecx, \mathbf{z}_j) = K(\vecx, \mathbf{z}_j) - \Kpure(\vecx, \mathbf{z}_j)$. Thus, \begin{align*}
        \|\partial \Kmix(\mathbf{z}_j)\|\! &=\! \left\| \left.\boldsymbol{\nabla}^{\otimes b}\left(K\left(\mathbf{x}, \mathbf{z}_j\right) - \Kpure\left(\mathbf{x}, \mathbf{z}_j\right)\right)\right|_{\mathbf{x} = \frac{c\mathbf{t}}{\sqrt{d}}}\!\!\! - \left.\boldsymbol{\nabla}^{\otimes b}\left(K\left(\mathbf{x}, \mathbf{z}_j\right) - \Kpure\left(\mathbf{x}, \mathbf{z}_j\right)\right)\right|_{\mathbf{x} = \mathbf{0}} \right\| \\
        &\le\! \left\| \left.\boldsymbol{\nabla}^{\otimes b}K\left(\mathbf{x}, \mathbf{z}_j\right)\right|_{\mathbf{x} = \frac{c\mathbf{t}}{\sqrt{d}}}\!\!\! - \left.\boldsymbol{\nabla}^{\otimes b}K\left(\mathbf{x}, \mathbf{z}_j\right)\right|_{\mathbf{x} = \mathbf{0}} \right\| \\&\qquad\qquad\qquad\qquad\qquad + \left\| \left.\boldsymbol{\nabla}^{\otimes b}K\left(\mathbf{x}, \mathbf{z}_j\right)\right|_{\mathbf{x} = \frac{c\mathbf{t}}{\sqrt{d}}}\!\!\! - \left.\boldsymbol{\nabla}^{\otimes b}\Kpure\left(\mathbf{x}, \mathbf{z}_j\right)\right|_{\mathbf{x} = \mathbf{0}} \right\|.
    \end{align*} Since we already showed that the second term is at most $\varepsilon/2$ for sufficiently small $\delta$, it suffices to bound only the first term. This end, we recall that
    \begin{align*}    
    \mathbf{z}_j = \Deltaboldhat{j}/\zeta + \tilde{\boldsymbol{\eta}}_j \text{ where } \Deltaboldhat{j} \coloneqq \frac{\boldsymbol{\Delta}_j - \mu}{\sigma} \text{ and }  \tilde{\boldsymbol{\eta}}_j \coloneqq  \frac{\boldsymbol{\eta}_j}{\zeta\sigma} \sim \mathcal{N} \left( 0, \frac{d^{-2\eta}}{(\zeta\sigma)^2} \mathbf{I}_k \right)
    \end{align*} as used in \Cref{lem:integrability} as well. Since $\Deltaboldhat{j}$ and $\tilde{\boldsymbol{\eta}}_j$ are independent, we get that \begin{align*}
        K(\mathbf{t}, \mathbf{z}_j) = K(\mathbf{t}/\zeta, \Deltaboldhat{j}) + K(\mathbf{t}, \tilde{\boldsymbol{\eta}}_j),
    \end{align*} so it suffices to prove \begin{align*}
       & \left\| \left.\boldsymbol{\nabla}^{\otimes b_1}K\left(\mathbf{x}, \tilde{\boldsymbol{\eta}}_j \right)\right|_{\mathbf{x} = \frac{c\mathbf{t}}{\sqrt{d}}} - \left.\boldsymbol{\nabla}^{\otimes b_1}K\left(\mathbf{x}, \tilde{\boldsymbol{\eta}}_j \right)\right|_{\mathbf{x} = \mathbf{0}} \right\| \le \varepsilon/2 \text{ and }\\& \left\| \left.\boldsymbol{\nabla}^{\otimes b_1}K\left(\mathbf{x}, \Deltaboldhat{j} \right)\right|_{\mathbf{x} = \frac{c\mathbf{t}}{\sqrt{d}}} - \left.\boldsymbol{\nabla}^{\otimes b_1}K\left(\mathbf{x}, \Deltaboldhat{j} \right)\right|_{\mathbf{x} = \mathbf{0}} \right\| \le \varepsilon/2.
     \end{align*} The first statement is immediate since all the $\tilde{\boldsymbol{\eta}}_j$ have the same distribution and threrfore the result follows because $K\left(\mathbf{x}, \tilde{\boldsymbol{\eta}}_j \right)$ and all its derivatives are a continuous function. To prove the second result, we use that each entry of $\Deltaboldhat{j}$ is bounded almost surely by some constant $M$, which implies that any mixed moment in the entries of $\Deltaboldhat{j}$ of order $r$ is at most $M^r$. At this point, we recall that the CGF has the representation \begin{align*}
        K_{\mathbf{z}}((t_1, \ldots, t_k)^\top) = \sum_{\substack{s = (s_1, \ldots, s_k) \in \mathbb{N}^{\times k}}} \kappa_{s}(\mathbf{z}) \frac{(it_1)^{s_1}(it_2)^{s_2}\cdots (it_k)^{s_k}}{s_1!s_2!\cdots s_k!}.
    \end{align*} and that \begin{align*}
        \kappa_{s}(\mathbf{z}) &= \kappa(\underbrace{\mathbf{z}(1), \ldots, \mathbf{z}(1)}_{s_1 \text{ times}}, \underbrace{\mathbf{z}(2), \ldots, \mathbf{z}(2)}_{s_2 \text{ times}}, \ldots, \underbrace{\mathbf{z}(k), \ldots, \mathbf{z}(k)}_{s_k \text{ times}}) \\
        \text{where }  \kappa(X_1, X_2, \ldots, X_k) &= \sum_\pi (|\pi|-1)!(-1)^{|\pi|-1}\prod_{B \in \pi}\Expected{\prod_{j\in B} X_j}.
    \end{align*} It is easy to see that for any $1 \le \ell \le |s|$, the number of partitions of $[|s|]$ with $\ell$ blocks is at most $\frac{|s|!}{\ell!} \binom{|s|}{\ell} \le \frac{|s|!}{\ell!}2^{|s|}$. Using that the $X_j$ are bounded a.s. by some constant $C$, we get that \begin{align*}
        |\kappa_{s}(\Deltaboldhat{j})| = \sum_{\ell = 1}^{|s|} \sum_{\pi : |\pi| = \ell} (\ell-1)!(-1)^{\ell-1}\prod_{B \in \pi}\Expected{\prod_{j\in B} X_j} \le \sum_{\ell = 1}^{|s|} |s|! C^{|s|} = |s|! |s| C^{|s|} \le |s|! M^{|s|} 
    \end{align*} for some constant $M > C$. Moreover, it is easy to see that for every $s = (s_1, \ldots, s_k)$ such that $|s| = r \ge k$, we have that \begin{align*}
        s_1!s_2!\cdots s_k! \ge \left( \left \lfloor r/k\right\rfloor !\right)^k \ge \left( \sqrt{ \lfloor r/k \rfloor } \left( \frac{r}{2ek} \right)^{\lfloor r/k \rfloor}\right)^k \ge \left( \frac{2ek}{r} \left( \frac{r}{2ek} \right)^{r/k} \right)^k \ge \left( \frac{1}{r} \right)^{k} \left( \frac{r}{2ek} \right)^{r}
    \end{align*} by Stirling's approximation and using that $\lfloor r/k \rfloor \ge (r/k) - 1$. Similarly, we further have that \begin{align*}
        |s|! \le 2\sqrt{2 \pi r} \left( \frac{r}{e}\right)^{r}. 
    \end{align*} Accordingly, for every $s = (s_1, \ldots, s_k) \in \mathbb{N}^{\times k}$ with $|s| = r$, we get that \begin{align*}
        \frac{\kappa_s(\mathbf{z})}{s_1!s_2!\cdots s_k!}\le 2 r^k \sqrt{2 \pi r} M^r(2k)^r,
    \end{align*} which grows merely exponentially in $r$. Now, let $D$ denote an operator representing an arbitrary partial derivative of order $b$. It is easy to see that \begin{align*}
        \left| D K(\mathbf{x}, \Deltaboldhat{j})|_{\mathbf{x} = \mathbf{t}} - D K(\mathbf{x}, \Deltaboldhat{j})|_{\mathbf{x} = \mathbf{0}}  \right| &\le \sum_{\substack{s = (s_1, \ldots, s_k) \in \mathbb{N}^{\times k} \\ |s| > b}} \frac{|\kappa_{s}(\mathbf{z})|}{s_1!s_2!\cdots s_k!} |s|^b \left( \| \boldsymbol{t} \|_\infty\right)^{|s| - b} \\
        &\le \sum_{\substack{s = (s_1, \ldots, s_k) \in \mathbb{N}^{\times k} \\ |s| > b}} \frac{|\kappa_{s}(\mathbf{z})|}{s_1!s_2!\cdots s_k!} |s|^b \left( \| \boldsymbol{t} \|_\infty\right)^{|s| - b} \\
        &\le \sum_{r = b+1}^\infty 2\sqrt{2 \pi r}k^rM^r(2k)^r r^b \left( \| \boldsymbol{t} \|_\infty\right)^{|s| - b} \\
        &\le C_1 \sum_{r = b+1}^\infty \exp\left( C_2 r + C_3 \log(r) + (r-b)\log(\|\mathbf{t}\|_\infty) \right) \\
        &\le C_1 \sum_{r = b+1}^\infty \exp\left( (C_4 + C_5 \log(\|\mathbf{t}\|_\infty))r \right)
    \end{align*} which can be made smaller than any $\varepsilon > 0$ by for all $\|\mathbf{t}\|_\infty$ small enough. The statement follows.
\end{proof}

\subsection{Deferred Proofs from \Cref{sec:triangles}}\label{sec:deferredtriangles}

\momentsfortriangleanalysis*
\begin{proof}
    We get from \Cref{lem:cumulantasnicemixedmoment} and the definition of $\mathbf{z}$ that \begin{align*}
        \kappa_{(1,1,\ldots,1)} = \Expected{ \prod_{j = 1}^k \mathbf{z}_1(j) } = \Expected{ \prod_{j = 1}^k \frac{\Delta(e_j) - \mu + \boldsymbol{\eta}(j)}{\zeta \sigma} } = (\zeta \sigma)^{-k} \Expected{ \prod_{j = 1}^k (\gamma(e_j) + \boldsymbol{\eta}(j)) }
    \end{align*} where $\mu = \Expected{\Delta(e_1)}$ and $\boldsymbol\eta \sim \mathcal{N}(0, d^{-2\eta}\mathbf{I}_k)$. At this point, it only remains to quantify the influence of $\boldsymbol{\eta}$. But since $\boldsymbol{\eta}$ is independent of all the other randomness, we get that \begin{align*}
        \Expected{ \prod_{j = 1}^k (\gamma(e_j) + \boldsymbol{\eta}(j)) } = \sum_{S \subseteq [k]}^k \Expected{\prod_{j\in S}\gamma(e_j)}\Expected{\prod_{j\notin S}\boldsymbol{\eta}(j)} = \Expected{\prod_{j=1}^k \gamma(e_j)}
    \end{align*} since all terms corresponding to $S \neq \emptyset$ are zero.
\end{proof}

Let us prove that $\Expected{\gamma(e_1)\gamma(e_2)\gamma(e_3)} < 0$ for all $L_q$-norms.

\correlation
\begin{proof}
    For an edge $e = \{u, v\}$, we have (up to a constant factor) $\gamma(e) = |x_u - x_v|_C^q - \mu$ with $\mu = \Expected{|x_u - x_v|_C^q}$. We define $\delta(e) = |x_u - x_v|_C$ and write $y \coloneqq \delta(e_1)$ to decompose \begin{align*}
        \Expected{\gamma(e_1)\gamma(e_2)\gamma(e_3)} &= \mathbb{E}_{y}\left[\gamma(e_1)\mathbb{E}[\gamma(e_2)\gamma(e_3) \mid y]\right]
    \end{align*} and note that both $\gamma(e_1)$ and $\mathbb{E}[\gamma(e_2)\gamma(e_3) \mid y]$ are a function of $y$. Since $\gamma(e_1) = y^q - \mu$ it is easy to see that $\gamma(e_1)$ is strictly monotonically increasing in $y$. We prove that $\mathbb{E}[\gamma(e_2)\gamma(e_3) \mid y]$ is monotonically decreasing in $y$. 

    To this end, we note that \begin{align*}
        \mathbb{E}[\gamma(e_2)\gamma(e_3) \mid y] &= \mathbb{E}[(\Delta(e_2) - \mu)(\Delta(e_3) - \mu) \mid y]\\
        &= \mathbb{E}[\Delta(e_2)\Delta(e_3) \mid y] - \mu \Expected{\Delta(e_2)} - \mu\Expected{\Delta(e_3)} + \mu^2\\
        &= \underbrace{\mathbb{E}[\Delta(e_2)\Delta(e_3) \mid y]}_{\eqqcolon F(y)} - \mu^2
    \end{align*} so it suffices to show that $F(y)$ is decreasing in $y$. We note that \begin{align*}
        F(y) = \int_{0}^{1/2} \left( x^q|x-y|^q + \left(1/2 - x\right)^q\left(1/2 - |x-y|\right)^q \right) \text{d} x
    \end{align*} and prove that the derivative w.r.t. $y$ is $< 0$ for all $y \in (0,1/2)$. Note that by the Leibniz integral rule \begin{align*}
        \frac{\text{d}}{\text{d} y} F(y) &= \int_{0}^{1/2} \frac{\text{d}}{\text{d} y}\left( x^q|x-y|^q + \left(1/2 - x\right)^q\left(1/2 - |x-y|\right)^q \right) \text{d} x\\ 
        &=\int_{0}^{1/2} \left( \text{sgn}(y-x)qx^q|x-y|_C^{q-1} + \text{sgn}(x-y)q\left(1/2 - x\right)^q\left(1/2 - |x-y|_C\right)^{q-1} \right) \text{d} x\\
        &= q \left( \int_{0}^y x^q(y-x)^{q-1} \text{d} x - \int_{y}^{1/2} x^q(x-y)^{q-1} \text{d} x \right) \\
        &\hspace{0.8cm} + q\left( -\int_{0}^y (1/2 - x)^q(1/2 - y + x)^{q-1} \text{d} x + \int_{y}^{1/2} (1/2 - x)^q(1/2 + y - x)^{q-1} \text{d} x \right).
    \end{align*} Rearranging yields \begin{align*}
        \frac{\text{d}}{\text{d} y} F(y) &= q \underbrace{\left( \int_{0}^y x^q(y-x)^{q-1} \text{d} x -\int_{0}^y (1/2 - x)^q(1/2 - y + x)^{q-1} \text{d} x \right)}_{\eqqcolon g_1(y)} \\
        &\hspace{3cm} + q\underbrace{\left( \int_{y}^{1/2} (1/2 - x)^q(1/2 + y - x)^{q-1} \text{d} x - \int_{y}^{1/2} x^q(x-y)^{q-1} \text{d} x \right)}_{\eqqcolon g_2(y)}.
    \end{align*} We prove that $g_1(x), g_2(x) \le 0$ for all $y\in[0,1/2]$ which implies what we want. Starting with $g_2$, we see that \begin{align*}
        g_2(y) &= \int_{y}^{1/2} (1/2 - x)^q(1/2 + y - x)^{q-1} \text{d} x - \int_{y}^{1/2} x^q(x-y)^{q-1} \text{d} x\\
        &= \int_{0}^{1/2-y} z^q(z+y)^{q-1} \text{d} z - \int_{0}^{1/2-y} (z+y)^qz^{q-1} \text{d} z
    \end{align*} where we substituted $z = 1/2 - x$ in the first integral and $z = x - y$ in the second integral. Combining the integrals and factoring out then yields \begin{align*}
        g_2(y) = \int_{0}^{1/2-y} z^{q-1}(z+y)^{q-1} \left( z - (z + y) \right) \text{d} z = -y \int_{0}^{1/2-y} z^{q-1}(z+y)^{q-1} \text{d} z.
    \end{align*} Since the above integral and $y$ are clearly non-negative, we conclude that $g_2(y) < 0$ whenever $y \neq 0$ and $y \neq 1/2$ as desired.

    Proceeding with $g_1$, we note that \begin{align*}
        g_1(y) &=  \int_{0}^y x^q(y-x)^{q-1} \text{d} x -\int_{0}^y (1/2 - x)^q(1/2 - y + x)^{q-1} \text{d} x\\
        &\le \int_{0}^y x^q(1/2-x)^{q-1} \text{d} x - \int_{0}^y (1/2 - x)^qx^{q-1} \text{d} x\\
        &= \int_{0}^y x^{q-1}(1/2-x)^{q-1} \left(x - (1/2-x) \right) \text{d} x \\
        &= \int_{0}^y \underbrace{x^{q-1}(1/2-x)^{q-1} \left(2x - 1/2 \right)}_{\eqqcolon h(x)} \text{d} x.
    \end{align*} To show that this is negative, we consider two cases. Case 1: $y \le 1/4$. Then $\int_{0}^yh(x)\text{d}x \le 0$  since the integrand $h(x)$ is non-positive in the whole domain. In Case 2, where $y > 1/4$, we note that \begin{align*}
        \int_{0}^yh(x)\text{d}x = \int_{0}^{1/2}h(x)\text{d}x - \int_{y}^{1/2}h(x)\text{d}x.
    \end{align*} However, we also have that \begin{align*}
        \int_{0}^{1/2} h(x)\text{d}y &= \int_{0}^{1/4}h(x)\text{d}x + \int_{1/4}^{1/2}h(x)\text{d}yx\\
        &= \int_{0}^{1/4}h(x)\text{d}y + \int_{0}^{1/4}h(1/2-x)\text{d}x\\
        &= \int_{0}^{1/4}(h(x) + h(1/2-x))\text{d}x.
    \end{align*} Now, it is easy to see that $h(x) = -h(1/2 - x)$ (by definition of $h$), so the integral above is $0$. Going back, this means that \begin{align*}
        \int_{0}^yh(x)\text{d}x = - \int_{y}^{1/2}h(x)\text{d}x.
    \end{align*} and now (since $y > 1/4$), the above integrand is non-negative, so (due to the minus in front of the above expression) $g_1(y) = \int_{0}^yh(x)\text{d}x \le 0$ holds in both case 1 and 2. Assuming further that $y \in (0,1/2)$, the inequality can be made strict, as desired.

    In total, we have thus shown that $\mathbb{E}[\gamma(e_2)\gamma(e_3) \mid y]$ is strictly monotonically decreasing in $y$. Now, we can use this fact to show that $\mathbb{E}[\gamma(e_1)\gamma(e_2)\gamma(e_3) ] = \mathbb{E}_{y}\left[\gamma(e_1)\mathbb{E}[\gamma(e_2)\gamma(e_3) \mid y]\right] < 0$ since we know that $f_1(y) \coloneqq \gamma(e_1)$ is a function monotonically increasing in $y$, and since $f_2(y) \coloneqq \mathbb{E}[\gamma(e_2)\gamma(e_3) \mid y]$ is decreasing in $y$. In particular, we know that $f_1(0) < 0$ while $f_1(1/2) > 0$ and we know that $f_2(0) > 0$ while $f_2(1/2) < 0$. Due to strict monotonicity and since $f_1,f_2$ are continuous, this implies that $f_1, f_2$ each have exactly one zero in the interval $(0,1/2)$. We call these positions $y_1, y_2$, respectively such that $f_1(y) > 0$ for $y \in (y_1, 1/2]$ and $f_2(y) < 0$ for $y \in (y_2, 1/2]$. Depending on which of $y_1,y_2$ is larger, we distinguish three cases.
    
    Case 1: $y_2 < y_1$. Here, we can split \begin{align*}
        \mathbb{E}[\gamma(e_1)\gamma(e_2)\gamma(e_3) ] &= 2\int_0^{1/2}f_1(y) f_2(y) \text{d} y \\
        &= 2 \left(\underbrace{\int_0^{y_2}f_1(y) f_2(y) \text{d} y}_{\eqqcolon I_1} + \underbrace{\int_{y_2}^{y_1}f_1(y) f_2(y) \text{d} y}_{\eqqcolon I_2} + \underbrace{\int_{y_1}^{1/2}f_1(y) f_2(y) \text{d} y}_{\eqqcolon I_3} \right).
    \end{align*} Now, $I_1 \le 0$ since $f_1(y) \le 0$ and $f_2(y) \ge 0$ for $y\in[0,y_2]$ by monotonicity. Furthermore, we have $f_1(y)f_2(y) \le f_1(y)f_2(y_1)$ for all $y \in [y_2, 1/2]$ because $f_2(y_1)$ is the smallest ('most negative') value $f_2$ attains for $y \in [y_2, y_1]$ (where $f_1,f_2$ are both negative), and the largest ('least negative') value $f_2$ attains for $y \in [y_1, 1/2]$  (where $f_1$ is positive and $f_2$ is negative). Hence, leaving out $I_1$ because it is $< 0$ anyways and applying the above fact, we get \begin{align*}
        \frac{1}{2}\mathbb{E}[\gamma(e_1)\gamma(e_2)\gamma(e_3) ] &\le \int_{y_2}^{y_1}f_1(y) f_2(y) \text{d} y + \int_{y_1}^{1/2}f_1(y) f_2(y) \text{d} y \\
        &\le f_2(y_1) \int_{y_2}^{1/2}f_1(y) \text{d} y.
    \end{align*} Now, $f_2(y_1) < 0$ due to $y_2 < y_1$, and  $\int_{y_2}^{1/2}f_1(y) \text{d}y > 0$ since $\int_{0}^{1/2}f_1(y) \text{d}y = 0$ while $f_1(y)$ is strictly increasing and $y_2 > 0$. This implies that $ f_2(y_1) \int_{x_2}^{1/2}f_1(y) \text{d} < 0$ as desired.

    Case 2: $y_2 > y_1$. Using analogous arguments as used in case 1, we can split \begin{align*}
        \frac{1}{2}\mathbb{E}[\gamma(e_1)\gamma(e_2)\gamma(e_3) ] &\le \int_{0}^{y_1}f_1(y) f_2(y) \text{d} y + \int_{y_1}^{y_2}f_1(y) f_2(y) \text{d} y \\
        &\le f_2(y_1) \int_0^{y_2}f_1(y) \text{d} y.
    \end{align*} This time $f_2(y_1) > 0$ since $y_1 < y_2$, but $\int_0^{y_2}f_1(y) \text{d} y < 0$ since $\int_{0}^{1/2}f_1(y) \text{d}y = 0$ while $f_1(y)$ is strictly increasing and $y_2 > 0$. Our statement follows.

    Case 3: $y_1 = y_2$: In this case, our statement follows even more directly since $f_1$ and $f_2$ have opposite sign for all $y \neq y_1=y_2$ so \begin{align*}
        \mathbb{E}[\gamma(e_1)\gamma(e_2)\gamma(e_3) ] &= 2\int_0^{1/2}f_1(y) f_2(y) \text{d} y
    \end{align*} is negative.
\end{proof}

\end{document}